\documentclass[reqno]{amsart}

\usepackage[english]{babel}

\usepackage{standalone,amsmath,amssymb,amsthm,amsfonts,enumerate,mathtools,verbatim,comment,pdflscape,makecell,rotating}
\usepackage{csquotes,tikz-cd,booktabs}

\usepackage[shortlabels]{enumitem}
    \setenumerate[0]{label={\rm (\roman*)}, leftmargin=*}
\usepackage[T1]{fontenc}
\setlist{nolistsep}

\usepackage[textsize=footnotesize,textwidth=20ex,colorinlistoftodos]{todonotes}

\usetikzlibrary{decorations.markings}
\tikzset{%
	myvertex/.style = {circle, draw, fill, very thick, outer sep=3.5pt, inner sep=0pt, minimum size=#1},
	myvertex/.default = 3.5pt,
	myblob/.style = {myvertex, fill=white, minimum size=#1},
	myedge/.style = {draw=black, very thick, line cap=round, line join=round},
	mysquare/.style = {rectangle, draw, fill, thick, outer sep=3.5pt, inner sep=3.5pt, fill=white, minimum size=#1},
}

\definecolor{ugentblue}{RGB}{30,100,200}
\definecolor{ugentyellow}{RGB}{120,190,0}
\definecolor{ugentred}{RGB}{220,78,40}

\usepackage{hyperref}
\hypersetup{
    colorlinks=true,
    linkcolor=ugentblue,
    filecolor=magenta,
    citecolor=ugentblue,
    urlcolor=cyan,
}

\usepackage[capitalize]{cleveref}

\newtheorem{proposition}{Proposition}[section]
\newtheorem{lemma}[proposition]{Lemma}
\newtheorem{corollary}[proposition]{Corollary}
\newtheorem{theorem}[proposition]{Theorem}
\newtheorem{maintheorem}{Theorem}

\theoremstyle{definition}
\newtheorem{definition}[proposition]{Definition}
\newtheorem{notation}[proposition]{Notation}
\newtheorem{example}[proposition]{Example}
\newtheorem{convention}[proposition]{Convention}
\newtheorem{remark}[proposition]{Remark}

\AddToHook{env/lemma/begin}{\crefalias{proposition}{lemma}}
\AddToHook{env/corollary/begin}{\crefalias{proposition}{corollary}}
\AddToHook{env/theorem/begin}{\crefalias{proposition}{theorem}}
\AddToHook{env/maintheorem/begin}{\crefalias{proposition}{maintheorem}}
\AddToHook{env/definition/begin}{\crefalias{proposition}{definition}}
\AddToHook{env/notation/begin}{\crefalias{proposition}{notation}}
\AddToHook{env/example/begin}{\crefalias{proposition}{example}}
\AddToHook{env/convention/begin}{\crefalias{proposition}{convention}}
\AddToHook{env/remark/begin}{\crefalias{proposition}{remark}}
\AddToHook{env/construction/begin}{\crefalias{proposition}{construction}}
\AddToHook{env/assumption/begin}{\crefalias{proposition}{assumption}}
\AddToHook{env/motivation/begin}{\crefalias{proposition}{motivation}}

\crefname{enumi}{}{}
\Crefname{enumi}{Item}{Items}
\crefname{equation}{}{}
\Crefname{equation}{Equation}{Equations}
\crefname{notation}{Notation}{Notations}
\crefname{maintheorem}{Theorem}{Theorems}

\newlength{\fixmidfigure}
\newcommand{\premidfigure}{\setlength{\fixmidfigure}{\lastskip}\addvspace{-\lastskip}}
\newcommand{\postmidfigure}{\addvspace{\fixmidfigure}}

\DeclareMathOperator{\Der}{Der}
\DeclareMathOperator{\Aut}{Aut}

\DeclareMathOperator{\TKK}{TKK}
\DeclareMathOperator{\id}{id}
\DeclareMathOperator{\im}{im}
\DeclareMathOperator{\ad}{ad}

\DeclareMathOperator{\Mat}{Mat}

\DeclareMathOperator{\Jac}{Jac}
\DeclareMathOperator{\Her}{Her}
\DeclareMathOperator{\Rad}{Rad}
\DeclareMathOperator{\Weyl}{Weyl}

\DeclarePairedDelimiterX{\brackets}[1]{(}{)}{#1}
\DeclarePairedDelimiterX{\sqbrackets}[1]{[}{]}{#1}
\DeclarePairedDelimiterX{\braces}[1]{ \{ }{ \} }{#1}

\newcommand{\Z}{\mathbb{Z}}
\newcommand{\gl}{\mathfrak{gl}}
\newcommand{\dd}{\mathbf{d}}

\newcommand{\coloneq}{\coloneqq}

\makeatletter
\DeclareRobustCommand{\iotainv}{{\mathord{\mathpalette\iotamirror@{}}}}
\newcommand{\iotamirror@}[2]{%
  \begingroup
    \setbox0=\hbox{$\m@th#1\iota$}%
    \reflectbox{\box0}%
  \endgroup
}
\makeatother

\newcommand{\backslasharrow}{{\ensuremath\operatorname{\rotatebox{120}{\!$_\leftrightarrow$}}}}
\newcommand{\slasharrow}{{\ensuremath\operatorname{\rotatebox{45}{\!$_\leftrightarrow$}}}}
\newcommand{\phihor}{\varphi_\leftrightarrow}
\newcommand{\phisl}{\varphi_\slasharrow}
\newcommand{\phiver}{\varphi_\updownarrow}
\newcommand{\phibs}{\varphi_\backslasharrow}
\newcommand{\psibs}{\psi_\backslasharrow}

\newcommand{\tmin}{\mathrm{min}}
\newcommand{\tmax}{\mathrm{max}}

\DeclareMathOperator{\End}{End}

\newcommand{\cartan}[2]{\langle #1 \vert #2 \rangle}
\newcommand{\rootproj}{\pi} 
\newcommand{\rootint}[2]{\mathord{]#1, #2[}}
\newcommand{\refl}[1]{\sigma_{#1}}
\newcommand{\reflbr}[1]{\sigma(#1)}
\DeclareMathOperator{\rootht}{ht}

\newcommand{\overbar}[1]{\mkern 1.5mu\overline{\mkern-1.5mu#1\mkern-1.5mu}\mkern 1.5mu} 
\newcommand{\conconj}[1]{\overbar{#1}} 
\newcommand{\contr}{t}
\newcommand{\connorm}{n}
\newcommand{\cubel}[2]{#1[#2]}

\newcommand{\cbas}[1]{x_{#1}} 
\newcommand{\pargrp}[1]{P_{#1}}
\newcommand{\rhomcub}[1]{\vartheta_{#1}^+} 
\newcommand{\rhomcubop}[1]{\vartheta_{#1}^-} 
\newcommand{\rhomcubvar}[2]{\vartheta_{#1}^{#2}}
\newcommand{\rhombr}[1]{\vartheta_{#1}} 
\newcommand{\rhomlien}[1]{\bar{\vartheta}_{#1}} 
\newcommand{\rhomlie}[1]{\vartheta_{#1}} 
\newcommand{\rhomlieG}[1]{\vartheta_{#1}} 
\newcommand{\rhomgrp}[1]{\theta_{#1}} 

\newcommand{\etac}{\eta^{\text{c}}}
\newcommand{\etae}{\eta^{\text{e}}}

\newcommand{\dash}{\nobreakdash-\hspace{0pt}}

\allowdisplaybreaks

\begin{document}

\title[From cubic norm pairs to $G_2$- and $F_4$-graded groups and Lie algebras]{From cubic norm pairs to $\mathbf{G_2}$- and $\mathbf{F_4}$-graded groups and Lie algebras}

\author{Tom De Medts}
\author{Torben Wiedemann}


\begin{abstract}
	We construct Lie algebras arising from cubic norm pairs over arbitrary commutative base rings.
	Such Lie algebras admit a grading by a root system of type $G_2$, and when the cubic norm pair is a cubic Jordan matrix algebra, the $G_2$-grading can be further refined to an $F_4$-grading.
	We then use these Lie algebras and their gradings to construct corresponding root graded groups.
	Along the way, we produce many results providing detailed information about the structure of these Lie algebras and groups.
\end{abstract}

\maketitle

\section*{Introduction}

The study of Lie algebras graded by root systems goes back to the influential work of Stephen Berman and Robert V.~Moody from 1992 \cite{BM92}.
More precisely, Berman and Moody classified all Lie algebras over a base field $k$ of characteristic~$0$ graded by simply-laced finite root systems of rank $\geq 2$, up to central isogeny.
Over the following decade, this result was extended to the types $A_1$, $B_n$, $C_n$, $ BC_r $, $F_4$ and $G_2$ for $ n \ge 2 $, $ r \ge 1 $ and strengthened to a classification up to isomorphism by the combined work of Bruce Allison, Georgia Benkart, Yun Gao, Oleg Smirnov and Efim Zel'manov \cite{BZ96,ABG00,ABG02,BS03}. A different, uniform approach to the classification for $3$-graded root systems (an assumption which excludes exactly the types $E_8$, $F_4$ and $G_2$) which avoids case-by-case considerations was given by Erhard Neher \cite{Neher96}.
 
A different approach for $G_2$-graded Lie algebras, now over arbitrary fields $k$ (including fields of characteristic $2$ and $3$) was taken in \cite{DMM25}, based on the theory of extremal elements in Lie algebras and intersecting two different $5$-gradings on the Lie algebra.

\medskip

The goals of the current paper are the following.
\begin{enumerate}[label={\rm (\arabic*)}]
	\item Construct and investigate $G_2$- and $F_4$-graded Lie algebras over arbitrary commutative base rings $k$.
	\item Develop the theory of \emph{cubic norm pairs} over arbitrary base rings $k$, required for (1).
	\item Construct and investigate $G_2$- and $F_4$-graded groups.
\end{enumerate}
Our description of the Lie algebras is similar as in \cite{DMM25} in the sense that we coordinatize  each ``piece'' of the grading separately and explicitly describe the Lie bracket in terms of this coordinatization. 

For $G_2$-graded Lie algebras, we encounter the same sort of algebraic structures as the ones that play a key role in \cite{DMM25}, namely \emph{cubic norm pairs}, a twin version of the more common cubic norm structures.
When $k$ is a field, such a cubic norm pair either has trivial norm or it arises from a genuine cubic norm structure, but this dichotomy fails to hold over base rings. At the same time, the occurrence of such a twin structure is no surprise and is simply the correct analogue of a cubic norm structure, very much like Jordan pairs form a twin version of Jordan algebras.
Nevertheless, not much has been written in the existing literature about these cubic norm pairs, and our only traces back are John Faulkner's paper \cite{Faulkner2001} and Michiel Smet's recent preprint \cite{Smet25}.
We use the opportunity to spell out the basic theory of cubic norm pairs. We follow a similar setup as in the recent excellent monograph \cite{GPR24}, which develops the theory of cubic norm \emph{structures} over arbitrary commutative rings.
The results from \cite{GPR24} will play a crucial role throughout the whole paper.

A purely combinatorial construction, which relies on a certain Tits index and an associated surjection $ F_4 \to G_2 \cup \{0\} $ of root systems, makes it possible to regard any $F_4$-graded Lie algebra as a $G_2$-graded Lie algebra. For $G_2$\dash graded Lie algebras defined by a specific kind of cubic norm pairs, namely \emph{cubic Jordan matrix algebras}, we can reverse this procedure: that is, we refine the $G_2$-grading to obtain an $F_4$-graded Lie algebra.
Again, we will be able to give a precise description of each ``piece'' of the grading.
This turns out to be rather delicate for the ``middle piece'' of the original $G_2$-grading.

\medskip

In the next step, we pass from root graded Lie algebras to root graded groups. These are groups equipped with a family of subgroups $(U_\alpha)_{\alpha \in \Phi}$ for a finite root system $\Phi$ such that certain \emph{commutator relations} $[U_\alpha, U_\beta] \subseteq U_{\rootint{\alpha}{\beta}}$ for non-proportional $\alpha, \beta \in \Phi$ are satisfied and such that so-called \emph{Weyl elements}, whose conjugation action permutes the root groups, exist for all roots.
The terminology of groups graded by root systems appears for the first time in \cite{Shi93}, where Zhiyong Shi proves coordinatization results for root graded groups of simply-laced type with rank at least~3. These results are similar in spirit to the recognition theorems for root graded Lie algebras. Partial progress towards a comparable result for type $C_n$ with $n \ge 3$ was made by Zezhou Zhang in \cite{ZhangPhD}. Complete coordinatization results for root graded groups of the remaining types $B_n$, $C_n$, $F_4$ and $BC_n$ for $n \ge 3$ were proven in \cite{WiedemannPhD}, and they were extended to the types $H_3$ and $H_4$ in \cite{BW25} (with a slightly different axiomatic setup to account for the non-crystallographicness of these root systems). We will follow the terminology introduced in \cite{WiedemannPhD} for root graded groups (of crystallographic types), and point to its preface for further historical details.


It should also be noted that Jacques Tits' notion of RGD-systems, introduced in \cite{Tits92}, is a special case of root graded groups. For these objects, a geometrical interpretation in terms of thick spherical Moufang buildings is available. In particular, Tits' classification of thick spherical Moufang buildings of rank at least~3 \cite{Tits74} provides a classification of RGD-systems of rank at least~3. This classification was extended to the rank~2 case---the so-called Moufang polygons---by Jacques Tits and Richard Weiss \cite{TW02}.

The main structure theorem for $F_4$-graded groups is \cite[10.7.8]{WiedemannPhD}: Every $F_4$\dash graded group is parametrized by what is called in \cite{GPR24} a \emph{multiplicative conic alternative algebra} $C$ over a commutative ring $k$, which means that the structure of the group is largely determined by $C$ and $k$ (though not up to isomorphism). Multiplicative conic alternative algebras should be thought of as a suitable generalization of composition algebras---such as octonions or quaternions---over rings, and they are precisely the kind of algebras over which the aforementioned cubic Jordan matrix algebras are defined. Our main goal in the setting of root graded groups---which was, in fact, our original motivation to undertake this project---is to prove a converse existence result: For every multiplicative conic alternative algebra~$C$ over any commutative ring $k$, there actually exists an $F_4$-graded group which is coordinatized by $C$. Along we way, we will also show that there exists a $G_2$-graded group for every cubic norm \emph{structure} (while the group defined by a cubic norm \emph{pair} may lack Weyl elements for the short roots). We remark that a similar existence question is answered positively in \cite{Zhang14}, where Zezhou Zhang constructs a $C_3$-graded group that is coordinatized by an arbitrary alternative ring with a nuclear involution.

We construct the desired $G_2$-graded group as a subgroup of the automorphism group of our $G_2$-graded Lie algebra. In principle, this amounts to ``exponentiating'' the appropriate adjoint maps in the Lie algebra, but since we are dealing with arbitrary commutative base rings (where $2$ and $3$ are not necessarily invertible), this is not straightforward at all, and simply defining the exponential maps along with the required verifications will occupy us for some time.
This can be compared with the definition of exponential maps over arbitrary base \emph{fields} introduced in \cite[Theorem A]{DMM25}, but we cannot simply extend the ideas of \emph{loc.\@ cit.} to commutative base rings.
One new technical tool that we use is the explicit formulation of what we call ``higher Leibniz rules''.


The exponential maps in the $F_4$-graded group may be defined by specializing the exponential maps in the $G_2$-grading, but only for those roots that do not \enquote{map to zero under the surjection $F_4 \to G_2 \cup \{0\}$}. The exponential maps defined in this way allow us to define automorphisms of the Lie algebra that will later be proven to be Weyl elements, and the exponential maps for the \enquote{problematic} roots in $F_4$ may be defined as conjugates of other exponential maps by a suitable choice of these automorphisms. The Weyl elements in the $G_2$-grading can be defined in a more straightforward way as certain \emph{reflections in the Lie algebra}, which are given by explicit formulas inspired by \cite{DMM25}.

Once we have the exponential maps at our disposal, we can define root groups as $U_\alpha \coloneq \exp_\alpha(L_\alpha)$. It then remains to show that these satisfy the desired commutator relations and that our candidates for the Weyl elements permute the root groups, which is a straightforward but long computation. For the $F_4$-grading in particular, we have to show that each Weyl element permutes the $48$ root groups in the desired way, for which we use computer assistance.


\subsection*{Outline of the paper.}

We start by recalling the required background on root graded Lie algebras and root graded groups in \cref{sec:root graded}.
In \cref{sec:CNP}, we develop the theory of cubic norm pairs, following a similar approach as in \cite{GPR24}. In the process, we collect many identities that we will need later.

We start the construction of the $G_2$-graded Lie algebra from a cubic norm pair in \cref{sec:TKK} by first producing the ``classical'' Tits--Kantor--Koecher Lie algebra of the cubic norm pair. In \cref{sec:lie const}, we make this TKK Lie algebra into a larger, $G_2$-graded, Lie algebra. This involves an important subtlety where we first have to modify the $0$-graded part of the TKK construction.
It requires some computational effort to explicitly write down the Lie bracket and verify that it satisfies the Jacobi identity, but this pays off, since we will later have to refine the construction to produce an $F_4$-graded Lie algebra.

We take an intermezzo in \cref{sec:reflections,sec:simple}.
First, in \cref{sec:reflections}, we show that surjective homotopies between cubic norm pairs give rise to surjective homomorphism between the corresponding Lie algebras, and we introduce reflections $\phihor$, $\phisl$, $\phibs$ and $\phiver$ of the Lie algebra.
Then, in \cref{sec:simple}, we investigate the ideal structure of the Lie algebra. The main results up to this point can be summarized as follows:

\begin{maintheorem}[\ref{thm:G2 Lie}, \ref{le:reflections auto}, \ref{thm:L simple}]\label{main:A}
	Let $(J,J')$ be a cubic norm pair over a commutative ring $k$. Then there exists a $G_2$-graded Lie algebra $L=L(J,J')$ whose root spaces $L_\alpha$ for $\alpha \in G_2$ are isomorphic to $k$, $J$ or $J'$. Further, the construction of $ L $ is functorial in $(J,J')$ (with respect to surjective homotopies), and $L$ is simple if and only if $k$ is a field and the trace of $(J,J')$ is non-degenerate.
\end{maintheorem}

We are then ready to pass from Lie algebras to groups.
In \cref{sec:G2 exp}, we introduce the notion of $\alpha$-parametrizations and $\alpha$-exponential automorphisms (where $\alpha$ is a root).
We then proceed to explicitly define such maps for the $G_2$-graded Lie algebras arising from any cubic norm pair.
The precise verification that these maps define automorphisms of the Lie algebra is rather computational, and we have decided to extract it from this section into a separate Appendix~\ref{sec:leibniz}.
In addition, some verifications involving Weyl elements have been moved into a second Appendix~\ref{sec:weyl G2}.
As a final result of \cref{sec:G2 exp}, we obtain:

\begin{maintheorem}[\ref{thm:G2 graded group}]\label{main:B}
	Let $J$ be a cubic norm structure over a commutative ring~$k$ and let $L \coloneq L(J, J)$ be as in \cref{main:A}. Then $\Aut(L)$ contains a $G_2$-graded group $G(J)$ with root groups $U_\alpha \cong L_\alpha$ for all $\alpha \in G_2$.
\end{maintheorem}

We then move on to $F_4$-gradings. We begin by laying out the combinatorics of the Tits index that relates $ F_4 $-gradings to $ G_2 $-gradings in \cref{sec:index}.
In a short \cref{sec:CJMA}, we collect the required background on cubic Jordan matrix algebras, following \cite{GPR24}, expressed in our setting of cubic norm pairs (rather than cubic norm structures).
Next, in \cref{sec:lie F4}, we show that for this specific class of cubic norm pairs, the $G_2$-grading can be refined to an $F_4$-grading, which we obtain by introducing three new $5$-gradings that we then intersect.
Finally, in \cref{sec:refining aut}, we show that also the $G_2$-grading in the automorphism group of the Lie algebra can be refined to an $F_4$-grading, leading to our final result.

\begin{maintheorem}[\ref{thm:lie F4 grading}, \ref{thm:F4 graded group}]
	Let $C$ be a multiplicative conic alternative algebra (e.g., a composition algebra) over a commutative ring $k$, let $\Gamma = \operatorname{diag}(\gamma_1, \gamma_2, \gamma_3) \in \Mat_3(k)$ be invertible and let $J = \Her(C, \Gamma)$ be the associated cubic Jordan matrix algebra.
	Let $L \coloneq L(J,J)$ be as in \cref{main:A} and $G \coloneq G(J)$ be as in \cref{main:B}.
	
	Then $L$ is $F_4$-graded with root spaces $L_\alpha \cong k$ for all long $\alpha \in F_4$ and $L_\alpha \cong C$ for all short $\alpha \in F_4$. Further, $G$ is an $F_4$-graded group with root groups $U_\alpha \cong L_\alpha$ for all $\alpha \in F_4$.
\end{maintheorem}

\subsection*{Acknowledgments.}

We are grateful to Holger Petersson for pointing out the reference \cite{Faulkner2001} to us and to Bernhard Mühlherr for several insightful discussions about Tits indices. The work of the second author was partially supported by a DAAD postdoctoral research scholarship and by the Deutsche Forschungsgemeinschaft (DFG, German Research Foundation), Project-ID 286237555, TRR 195.

\section{Root graded Lie algebras and groups}\label{sec:root graded}

\begin{convention}
	Throughout this paper, we will use the terminology of \cite{HumphreysLie} for root systems. In other words, all root systems are finite, reduced and crystallographic. Further, for any root system $ \Phi $, we put $ \Phi^0 \coloneq \Phi \cup \{0\} $.
\end{convention}

\subsection{Root gradings}

We first introduce the central objects that we intend to construct in this paper: Root gradings of Lie algebras and groups.

\begin{definition}
	Let $ k $ be a commutative ring and let $ L $ be a Lie algebra over $ k $. For an abelian group $ (A,+) $, an \emph{$ A $-grading of $ L $} is a family $ (L_a)_{a \in A} $ of $ k $-submodules of $ L $ such that
	\[ L = \bigoplus_{a \in A} L_a \]
	and $ [L_a, L_b] \le L_{a+b} $ for all $ a,b \in A $. For $ n \in \mathbb{N} $, a \emph{$ (2n+1) $-grading of $ L $} is a family $ (L_i)_{-n \le i \le n} $ such that $ (L_i)_{i \in \Z} $ is a $ \Z $-grading of $ L $ where $ L_i \coloneq 0 $ for $ \lvert i \rvert > n $. For a root system $ \Phi $, a \emph{$ \Phi $-grading of $ L $} is a family $ (L_\alpha)_{\alpha \in \Phi^0} $ such that $ (L_\alpha)_{\alpha \in \langle \Phi \rangle_\Z} $ is a $ \langle \Phi \rangle_\Z $-grading where $ L_\alpha \coloneq 0 $ for all $ \alpha \in \langle \Phi \rangle_\Z \setminus \Phi^0 $.
\end{definition}

\begin{remark}\label{rem:stronger grading}
	Our notion of root graded Lie algebras is more general than the one in \cite{BM92}. A $\Phi$-graded Lie algebra $L = \bigoplus_{\alpha \in \Phi^0} L_\alpha$ in our sense is $\Phi$-graded in the sense of \cite[1.1]{BM92} if it contains a copy $\bar{L}$ of the split Lie algebra of type $\Phi$ such that the root spaces $(\bar{L})_{\alpha \in \Phi^0}$ of $\bar{L}$ with respect to some split Cartan subalgebra satisfy $\bar{L}_\alpha \subseteq L_\alpha$ for all $\alpha \in \Phi^0$. (Actually, root graded Lie algebras in the sense of \cite{BM92} are usually only considered over base fields and not rings, but since the split Chevalley Lie algebras are defined over $\Z$ (see \cite[Section~25]{HumphreysLie}), the notion makes sense over rings as well.) The Lie algebras that we construct in this paper are often root graded in this stronger sense, and they are \enquote{close to being so} in some other cases. We will discuss this in more detail in \cref{rem:is stronger grading}.
\end{remark}

\begin{convention}
	For $ x,y $ in an arbitrary group $ G $, we define their commutator to be $ [x,y] \coloneq x^{-1} y^{-1} xy $.
\end{convention}

\begin{definition}[{\cite[2.5.2]{WiedemannPhD}}]\label{def:rgg}
	Let $ \Phi $ be a root system and let $ G $ be a group. A \emph{$ \Phi $-grading of $ G $} is a family $ (U_\alpha)_{\alpha \in \Phi} $ of subgroups of $ G $ with the following properties, where we write $ U_A \coloneq \langle U_\alpha \mid \alpha \in A \rangle $ for any subset $ A $ of $ \Phi $.
	\begin{enumerate}
		\item \label{def:rgg:gen}$ G $ is generated by $ (U_{\alpha})_{\alpha \in \Phi} $ and $ U_\alpha \ne 1 $ for all $ \alpha \in \Phi $.
		\item \label{def:rgg:comm}For all non-proportional $ \alpha, \beta \in \Phi $, we have $ [U_\alpha, U_\beta] \subseteq U_{\rootint{\alpha}{\beta}} $ where $ \rootint{\alpha}{\beta} \coloneq \Phi \cap \{i\alpha + j\beta \mid i,j \in \Z_{>0}\} $.
		\item \label{def:rgg:weyl}For all $ \alpha \in \Phi $, the set of so-called \emph{$ \alpha $-Weyl elements}
		\[ M_\alpha \coloneq \{w \in U_{-\alpha} U_\alpha U_{-\alpha} \mid U_\beta^w = U_{\beta^{\reflbr{\alpha}}} \text{ for all } \beta \in \Phi\} \]
		is not empty. Here $ \refl{\alpha} \colon \Phi \to \Phi $ denotes the reflection along $ \alpha^\perp $, which we let act on $\Phi$ from the right-hand side.
		\item \label{def:rgg:nondeg}If $ \Pi $ is a positive system in $ \Phi $, then $ U_{-\alpha} \cap U_\Pi = 1 $ for all $ \alpha \in \Pi $.
	\end{enumerate}
\end{definition}

\begin{remark}
	What we call a $ \Phi $-grading here is actually called a \emph{crystallographic $ \Phi $-grading} in \cite{WiedemannPhD}. However, in the context of $\Phi$-graded groups arising from $\Phi$-graded Lie algebras, only crystallographic $\Phi$-graded groups are of interest, and hence we will always drop the word \enquote{crystallographic} in this paper.
\end{remark}

\begin{remark}\label{rem:rgg abelian}
	It is a non-trivial but elementary observation that root groups in $A_2$-graded groups are necessarily abelian (see \cite[5.4.9]{WiedemannPhD}). Since $G_2$ and $F_4$ have the property that every root is contained in a subsystem of type $A_2$, it follows that all root groups in root graded groups of these types are abelian as well.
\end{remark}

\begin{remark}[{\cite[2.2.6]{WiedemannPhD}}]\label{rem:weyl basic}
	Let $ \Phi $ be a root system, let $ G $ be a group with a family of subgroups $ (U_\alpha)_{\alpha \in \Phi} $ and let $ \alpha, \beta \in \Phi $. Then $ M_\alpha = M_\alpha^{-1} = M_{-\alpha} $ and $ M_\alpha^{w_\beta} = M_{\alpha^{\reflbr{\beta}}} $ for all $ w_\beta \in M_\beta $. It follows from the second property that if $ M_\delta \ne \emptyset $ for all $ \delta $ in a system of simple roots of $ \Phi $, then $ M_\alpha \ne \emptyset $ for all $ \alpha \in \Phi $.
\end{remark}

\subsection{\texorpdfstring{$F_4$}{F4}-graded groups}

We now turn to the special case of $F_4$-graded groups. One of the main results of \cite{WiedemannPhD} is that every such group is \enquote{coordinatized} by what is called a multiplicative conic alternative algebra (\cref{thm:rgg-param}). One of our main results in this paper is that for every multiplicative conic alternative algebra, there exists an $F_4$-graded groups coordinatized by this algebra (\cref{thm:F4 graded group}). In the following, we develop the necessary terminology to make these two statements precise. For this we will freely use the terminology of multiplicative conic alternative algebras that we will only later introduce in \cref{sec:CJMA}. For the moment, it suffices to know that a multiplicative conic alternative algebra is an algebra $C$ over some commutative ring $k$ which is equipped with a conjugation $\conconj{\:\cdot\:} \colon C \to C$, a norm $n \colon C \to k$ (which has a linearization $n \colon C \times C \to k$) and a trace $\contr \colon C \to k$. However, the reader who wishes to do so may safely skip ahead to \cref{sec:CJMA} to read the precise definitions.

\begin{definition}\label{def:F4 basis order}
	A \emph{standard system of simple roots in $ F_4 $} is a tuple $ (\delta_1, \delta_2, \delta_3, \delta_4) $ such that $ \{\delta_1, \delta_2, \delta_3, \delta_4\} $ is a simple system of roots, $ \delta_1, \delta_2 $ are longer that $ \delta_3, \delta_4 $ and the subsystems spanned by $ \{\delta_1, \delta_2\} $, $ \{\delta_2, \delta_3\} $, $ \{\delta_3, \delta_4\} $ are of type $ A_2 $, $ B_2 $, $ A_2 $, respectively.
\end{definition}

\begin{definition}[{\cite[4.3.2, 10.1.17]{WiedemannPhD}}]\label{def:twistgrp}
	Let $C$ be a multiplicative conic alternative algebra over a commutative ring $k$. The \emph{twisting group (of $ (k,C) $)} is the group $ T \coloneq \{\pm 1\} \times \{\pm 1\} $ together with the actions on (the sets) $ k $ and $ C $ defined by
	\begin{align*}
		(-1,1).t &\coloneq -t, \quad (1,-1).t \coloneq t, \quad (-1,1).c \coloneq -c, \quad (1,-1).c \coloneq \conconj{c}
	\end{align*}
	for all $ t \in k $, $ c \in C $.
\end{definition}

\begin{definition}[{\cite[4.2.4]{WiedemannPhD}}]\label{def:parmap}
	Let $\Delta$ be a system of simple roots in $F_4$ and let $T$ be any abelian group. A \emph{$\Delta$-parity map with values in $T$} is any map $\eta \colon F_4 \times \Delta \to T \colon (\alpha, \delta) \mapsto \eta_{\alpha,\delta}$. For any sequence $\delta_1, \ldots, \delta_\ell \in \Delta$ with $\ell \ge 2$, we put
	\[ \eta_{\alpha,\delta_1 \cdots \delta_\ell} \coloneq \prod_{i=1}^\ell \eta_{\alpha^{\reflbr{\delta_1} \cdots \reflbr{\delta_{i-1}}}, \delta_i}. \]
\end{definition}

\begin{definition}[{\cite[4.3.4]{WiedemannPhD}}]\label{def:rgg-param}
	Let $C$ be a multiplicative conic alternative algebra over a commutative ring $k$, let $T$ be the twisting group of $(k,C)$, let $\Delta$ be a system of simple roots in $F_4$ and let $\eta$ be a $\Delta$-parity map with values in $T$. Further, let $G$ be a group with an $F_4$-grading $(U_\alpha)_{\alpha \in F_4}$ and let $w_\delta$ be a $\delta$-Weyl element for all $\delta \in \Delta$. Put $P_\alpha \coloneq k$ for all long $\alpha \in F_4$ and $P_\alpha \coloneq C$ for all short $\alpha \in F_4$. Then a \emph{parametrization of $(G, (U_\alpha)_{\alpha \in F_4})$ by $(k,C)$ with respect to $(w_\delta)_{\delta \in \Delta}$ and $\eta$} is a family of isomorphisms $(\theta_\alpha \colon P_\alpha \to U_\alpha)_{\alpha \in F_4}$ such that
	\[ \theta_\alpha(p)^{w_\delta} = \theta_{\alpha^{\reflbr{\delta}}}(\eta_{\alpha,\delta}.p) \quad \text{for all } p \in P_\alpha. \]
\end{definition}

\begin{remark}\label{rem:parmap prod}
	A parametrization automatically satisfies
	\[ \theta_\alpha(p)^{w_{\delta_1} \cdots w_{\delta_\ell}} = \theta_{\alpha^{\reflbr{\delta_1} \cdots \reflbr{\delta_\ell}}}(\eta_{\alpha,\delta}.p) \quad \text{for all } p \in P_\alpha. \]
\end{remark}

\begin{remark}\label{rem:coordinatization parmap}
	Let $\Delta$ be a system of simple roots in $F_4$. Not every parity map (with values in the twisting group of some pair $ (k,C) $) is reasonable in the sense that it may arise in a parametrization of an $ F_4 $-graded group: For example, it is known that for many possible choices of $ \alpha \in F_4 $, $ \delta \in \Delta $, the action of the square of any $ \delta $-Weyl element on $ U_\alpha $ must be the inversion map (see \cite[10.5.5]{WiedemannPhD}), and hence any reasonable parity map $ \eta $ should satisfy $ \eta_{\alpha, \delta \delta} = (-1,1) $. An example of a reasonable parity map $\etac \colon F_4 \times \Delta \to \{\pm 1\}^2$ is given in \cite[Figure~10.4]{WiedemannPhD}. Here the letter \enquote{c} stands for the fact that this is the parity map that is used in the Coordinatization Theorem~\ref{thm:rgg-param} for $ F_4 $-graded groups.
\end{remark}

\begin{theorem}[{\cite[10.7.8]{WiedemannPhD}}]\label{thm:rgg-param}
	Let $ G $ be a group with an $ F_4 $-grading $ (U_\alpha)_{\alpha \in F_4} $. Choose a standard system $ \Delta = (\delta_1, \delta_2, \delta_3, \delta_4) $ of simple roots in $ F_4 $ and a $ \delta $-Weyl element $ w_\delta $ for all $ \delta \in \Delta $. Then there exist a commutative ring $ k $, a multiplicative conic alternative algebra $ C $ over $ k $ and a parametrization $(\theta_\alpha)_{\alpha \in F_4}$ of $ (G, (U_\alpha)_{\alpha \in F_4}) $ by $ (k,C) $ with respect to $(w_\delta)_{\delta \in \Delta}$ and the standard parity map $\etac$ from \cref{rem:coordinatization parmap} (whose codomain we identify with the twisting group $T$ of $(k,C)$) such that the following commutator relations hold for all $ s,t \in k $ and all $ c,d \in C $:
	\begin{align*}
		[\theta_{\delta_1}(s), \theta_{\delta_2}(t)] &= \theta_{\delta_1 + \delta_2}(-st), \\
		[\theta_{\delta_2}(s), \theta_{\delta_3}(c)] &= \theta_{\delta_2 + \delta_3}(-sc) \theta_{\delta_2 + 2\delta_3}\brackets[\big]{-s \connorm(c)}, \\
		[\theta_{\delta_2 + \delta_3}(c), \theta_{\delta_3}(d)] &= \theta_{\delta_2+2\delta_3}\brackets[\big]{\connorm(c,d)} = \theta_{\delta_2+2\delta_3}\brackets[\big]{\contr(c \conconj{d})}, \\
		[\theta_{\delta_4}(c), \theta_{\delta_3}(d)] &= \theta_{\delta_3+\delta_4}(cd).
	\end{align*}
\end{theorem}

\begin{remark}\label{rem:rgg F4 comrel}
	Keep the notation of \cref{thm:rgg-param} and put
	\[ X \coloneq \{(\delta_1, \delta_2), (\delta_2, \delta_3), (\delta_2 + \delta_3, \delta_3), (\delta_2, \delta_2 + 2\delta_3), (\delta_3, \delta_4)\}. \]
	For all $ (\alpha, \beta) \in X $, \cref{thm:rgg-param} provides an explicit commutator formula for $ [U_\alpha, U_\beta] $. Now let $ \alpha, \beta \in F_4 $ be arbitrary non-proportional roots. Then by the transitivity properties of the Weyl group, there exists $ u $ in the Weyl group of $ F_4 $ such that $ (\alpha^u, \beta^u) \in X $. Thus we have a commutator formula for $ [U_{\alpha^u}, U_{\beta^u}] $. Conjugating this formula by a sequence $ w_{\delta_{i_\ell}}, \ldots, w_{\delta_{i_1}} $ of Weyl elements where $ u = \refl{\delta_{i_1}} \cdots \refl{\delta_{i_\ell}} $, we obtain a commutator formula for $ [U_\alpha, U_\beta] $. For example, it follows from the formula $ [\theta_{\delta_3}(c), \theta_{\delta_3}(d)] = \theta_{\delta_3+\delta_4}(cd) $ that
	\[ [\theta_{\delta_4^{\reflbr{\delta}}}(\etac_{\delta_4,\delta}.c), \theta_{\delta_3^{\reflbr{\delta}}}(\etac_{\delta_3,\delta}.d)] = \theta_{\delta_3^{\reflbr{\delta}}+\delta_4^{\reflbr{\delta}}}(\etac_{\delta_3+\delta_4,\delta}.cd) \]
	for all $ \delta \in \Delta $ and all $ c,d \in C $. Thus explicit knowledge of $\etac$ grants explicit knowledge of all commutator relations.
\end{remark}

\begin{remark}\label{rem:different signs}
	We refer to the family $(\rhomgrp{\alpha})_{\alpha \in F_4}$ in \cref{thm:rgg-param} as a \emph{coordinatization}, meaning that it is a parametrization for which explicit commutator formulas (involving the structural maps of $(k,C)$) are known. In \cite[10.4.14, 10.4.15]{WiedemannPhD}, the parity map $\etac$ and $(\rhomgrp{\alpha})_{\alpha \in F_4}$ are called the \enquote{standard parity map} and a \enquote{coordinatization with standard signs}, respectively. However, we avoid the terminology of \enquote{standard signs} in this paper because there is nothing canonical about the parity map $\etac$. For example, the $ F_4 $-graded group in our Existence Theorem~\ref{thm:F4 graded group} will, a priori, be coordinatized with respect to a different parity map $ \etae $, and the commutator relations in this group (see \cref{pr:F4 comm}) will be the same as in \cref{thm:rgg-param} only \enquote{up to twisting}, that is, \enquote{up to signs and conjugation}. However, we will show in \cref{rem:sign twist} how a coordinatization of one type can be modified to obtain a coordinatization of the other type, which illustrates that this difference in the choice of signs is not meaningful.
\end{remark}

\section{Cubic norm pairs}\label{sec:CNP}

\subsection{Definition and identities}

Let $k$ be an arbitrary commutative ring. To define cubic norm pairs, we first need to clarify the notion of a cubic form.

\begin{definition}[{\cite{Faulkner2000}, \cite[12.43]{GPR24}}]\label{def:cubic map}
	Let $M,N$ be $k$-modules. A \emph{cubic map from $M$ to $N$ (over $k$)} is a pair $(f,g)$ of maps
	\[ f \colon M \to N \quad \text{and} \quad g \colon M \times M \to N \]
	with the following properties.
	\begin{enumerate}
		\item $f(\lambda x) = \lambda^3 f(x)$ for all $x \in M$, $\lambda \in k$.
		
		\item $g$ is quadratic in the first component and linear in the second component.
		
		\item $f(x+y) = f(x) + g(x,y) + g(y,x) + f(y)$ for all $x,y \in M$.
		
		\item \emph{Euler's differential equation:} $g(x,x) = 3f(x)$ for all $x \in M$.
	\end{enumerate}
	We will also write $ (f,g) \colon M \to N $. If $N=k$, we also call $(f,g)$ a \emph{cubic form on $M$}.
\end{definition}

The reason to include the \enquote{$ (2,1) $-linearization $ g $ of $ f $} in the definition of a cubic map is to ensure the existence and uniqueness of scalar extensions in the following sense.

\begin{lemma}[{\cite[Theorem~35]{Faulkner2000}, \cite[12.43]{GPR24}}]\label{le:cubic ext}
	Let $M,N$ be $k$-modules, let $ R $ be a $ k $-algebra and let $ (f,g) \colon M \to N $ be a cubic map. Then there exists a unique cubic map $ (f,g)_R = (f_R, g_R) \colon M_R = M \otimes_k R \to N_R = N \otimes R $, called the \emph{$ R $-cubic extension of $ (f,g) $}, such that the following diagrams commute:
	\[ \begin{tikzcd}
		M \arrow[r, "f"] \arrow[d, "m \mapsto m \otimes 1"'] & N \arrow[d, "n \mapsto n \otimes 1"] \\
		M_R \arrow[r, "f_R"] & N_R,
	\end{tikzcd} \qquad \begin{tikzcd}
		M \times M \arrow[r, "g"] \arrow[d, "{(u,v) \mapsto (u \otimes 1, v \otimes 1)}"'] & N \arrow[d, "n \mapsto n \times 1"] \\
		M_R \times N_R \arrow[r, "g_R"] & N_R.
	\end{tikzcd} \]
\end{lemma}

We are now ready to define cubic norm pairs.

\begin{definition}\label{def:CNP new}
	A \emph{cubic norm pair (over $k$)} is a pair $(J,J')$ of $k$-modules which are equipped with maps
	\begin{align*}
		\times &\colon J \times J \to J', &
		\times' &\colon J' \times J' \to J, \\
		T &\colon J \times J' \to k, &
		T' &\colon J' \times J \to k, \\
		N &\colon J \to k, &
		N' &\colon J' \to k, \\
		\sharp &\colon J \to J', &
		\sharp' &\colon J' \to J
	\end{align*}
	such that the following conditions are satisfied.
	\begin{enumerate}
		\item \label{eq:CNP bilinear}$\times$ and $T$ are $k$-bilinear maps, $\times$ is symmetric, and $T'(a',a) = T(a,a')$ for all $a \in J$, $a' \in J'$.
		
		\item $\sharp$ is a quadratic map with linearization $\times$. That is, $(\lambda a)^\sharp = \lambda^2 a^\sharp$ and $(a+b)^\sharp = a^\sharp + b^\sharp + a \times b$ for all $a,b \in J$ and $\lambda \in k$.
		
		\item \label{eq:CNP cubic}Put $ g \colon J \times J \to k \colon (a,b) \mapsto T'(a^\sharp,b)$. Then $(N,g)$ is a cubic form on $M$. Explicitly, $N(\lambda a) = \lambda^3 N(a)$, $T'(a^\sharp,a) = 3N(a)$ and the \emph{gradient identity}
		\[ N(a+b) = N(a) + T'(a^\sharp, b) + T'(b^\sharp, a) + N(b) \]
		hold for all $a,b \in J$ and $\lambda \in k$.
		
		\item \label{eq:sharp twice}The \emph{adjoint identity} $(a^\sharp)^{\sharp'} = N(a)a$ holds strictly. That is, $(a^\sharp)^{\sharp'} = \tilde{N}_R(a)a$ for all $k$-algebras $R$ and all $a \in J_R$.
		
		\item \label{eq:CNP 18}$N'(a^\sharp) = N(a)^2$ holds strictly. That is, $N_R'(a^\sharp) = N_R(a)^2$ for all $k$-algebras $R$ and all $a \in J_R$ where $ N_R $, $ N_R' $ are as in \cref{le:cubic ext}.
		
		\item \label{eq:CNP 19}$a \times (a^\sharp \times' b') = N(a)b' + T(a,b') a^\sharp$ holds strictly. That is, $a \times (a^\sharp \times' b') = \tilde{N}_R(a)b' + T(a,b') a^\sharp$ for all $k$-algebras $R$ and all $a \in J_R$, $b' \in J'_R$.
		
		\item All the previous conditions are still satisfied if we interchange the roles of $(J,\times, T, N, \sharp)$ and $(J', \times', T', N', \sharp')$.
	\end{enumerate}
	We will usually simply write $ \times $, $ T $, $ N $, $ \sharp $ for the maps $ \times' $, $ T' $, $ N' $, $ \sharp' $. We will refer to $T$ as the \emph{(bilinear) trace} and to $N$ as the \emph{norm}.
\end{definition}

\begin{definition}[{\cite[33.4]{GPR24}}]\label{def:CNS}
	A \emph{cubic norm structure} is a $k$-module $J$ which is equipped with maps 
	\begin{align*}
		\times &\colon J \times J \to J, \quad T \colon J \times J \to k, \quad N \colon J \to k, \quad \sharp \colon J \to J
	\end{align*}
	and a distinguished element $1_J \in J$, called the \emph{base point}, such that the following conditions are satisfied.
	\begin{enumerate}
		\item $(J,J)$ is a cubic norm pair with structural maps $ \times $, $ T $, $ N $, $ \sharp $ and $\times' \coloneq \times$, $T' \coloneq T$, $N' \coloneq N$, $\sharp' \coloneq \sharp$. It is called the \emph{cubic norm pair associated to (or induced by) $J$}.
		
		\item \label{def:CNS:1}The base point 
		satisfies $1_J^\sharp = 1_J$, $N(1_J) = 1_k$ and $1_J \times a = T(1_J, a) 1_J - a$ for all $a \in J$.
	\end{enumerate}
	We will also sometimes refer to the cubic norm pair $ (J,J) $ induced by $ J $ as a cubic norm structure.
\end{definition}

\begin{remark}\label{rem:CNP axioms}
	For the curious reader, we compare our definition of cubic norm pairs to some related ones that appear in the literature, and comment on the necessity of Axioms~\cref{eq:CNP 18,eq:CNP 19} in \cref{def:CNP new}. These facts will not be used in the sequel.
	\begin{enumerate}[(1)]
		\item \label{rem:CNP axioms:inv free}If there exist $a,b \in J$ such that $N(a)$ is invertible and $kb$ is free, then Axioms \cref{eq:CNP 18,eq:CNP 19} follow from the remaining ones. Namely, by similar arguments as in \cite[33.7]{GPR24}, the existence of $b$ implies that $kc$ is free for all $c \in J$ for which $N(c)$ is invertible, and with this observation, \cref{eq:CNP 18,eq:CNP 19} can be derived as in \cite[(33a.18), (33a.19)]{GPR24}.
		
		\item \label{rem:CNP axioms:unimod}The condition in~\ref{rem:CNP axioms:inv free} on the existence of $a,b$ is satisfied if there exists a base point $ 1_J $ that is \emph{unimodular}, meaning that $ k 1_J $ is a free direct summand of $ J $. The unimodularity of the base point is required in the definition of cubic norm structures in \cite[33.1, 33.4]{GPR24}, and hence Axioms \cref{eq:CNP 18,eq:CNP 19} are absent from \emph{loc.\@ cit.}. The condition in~\ref{rem:CNP axioms:inv free} is also satisfied if $k$ is a field and $N \ne 0$, which is how \cref{eq:CNP 18} is proven in \cite[(ADJ3')]{Faulkner2001}.
		
		\item Axiom \cref{eq:CNP 19} does not appear in \cite{Faulkner2001}, but follows from \cref{eq:CNP 18} if $k$ is a field. Further, it is shown in \cite[p.~4644]{Faulkner2001} that \cref{eq:CNP 18} holds if $T$ is non-degenerate in the sense that $T(a,J') = 0$ implies $a = 0_J$.
		
		\item In \cite[2.1.5]{Smet25}, the identity in \cref{le:CNP identities}\cref{le:CNP identities:21} is required as an axiom in place of our Axiom~\cref{eq:CNP 18}. In particular, it follows from \cref{le:CNP identities}\cref{le:CNP identities:21} that a cubic norm pair in our sense is a cubic norm pair in the sense of \cite[2.1.5]{Smet25}.
	\end{enumerate}
\end{remark}

\begin{remark}[Linearizations]\label{rem:CNP def lin}
	By \cite[12.10, 12.11]{GPR24}, Axiom~\ref{def:CNP new}\cref{eq:sharp twice} is equivalent to the requirement that the identity $(a^\sharp)^\sharp = N(a)a$ and its formal linearizations
	\begin{align}
		a^\sharp \times (a \times b) &= N(a)b + T(a^\sharp, b) a \quad \text{and} \label{eq:sharp twice lin 1} \\
		\quad a^\sharp \times b^\sharp + (a \times b)^\sharp &= T(a^\sharp, b)b + T(a, b^\sharp) a \label{eq:sharp twice lin 2}
	\end{align}
	hold for all $a,b \in J$. (See also \cite[33.14]{GPR24}.) Further, by linearizing with respect to $b$, it follows from~\eqref{eq:sharp twice lin 2} that
	\[ a^\sharp \times (b \times c) + (a \times b) \times (a \times c) = T(a^\sharp, b) c + T(a^\sharp, c)b + T(a,b \times c)a \]
	for all $a,b,c \in J$, which is also the linearization of~\eqref{eq:sharp twice lin 1} with respect to $a$.
	
	Similarly, Axiom~\ref{def:CNP new}\cref{eq:CNP 19} is equivalent to the requirement that $a \times (a^\sharp \times b') = N(a)b' + T(a,b') a^\sharp$ and the formal linearization
	\begin{align*}
		a \times \brackets[\big]{(a \times c) \times b'} + c \times (a^\sharp \times b') &= T(a^\sharp, c)b + T(a,b') a \times c + T(c,b') a^\sharp
	\end{align*}
	hold for all $a,c \in J$, $b' \in J'$.
	
	Finally, the formal linearizations of Axiom~\ref{def:CNP new}\cref{eq:CNP 18} of respective multi-degrees $(5,1)$, $(4,2)$, $(3,3)$ in $(a,b) \in J^2$ are the identities
	\begin{align*}
		T\brackets[\big]{(a^\sharp)^\sharp, a \times b} &= 2 T(a^\sharp, b) N(a), \\
		T\brackets[\big]{(a^\sharp)^\sharp, b^\sharp} + T\brackets[\big]{a^\sharp, (a \times b)^\sharp} &= 2N(a) T(a,b^\sharp) + T(a^\sharp, b)^2, \\
		N(a) T(b,b^\sharp) + T(a^\sharp, b) T(a,b^\sharp) + N(a \times b) &= 2N(a) N(b) + 2T(a^\sharp, b) T(a, b^\sharp).
	\end{align*}
	These formulas are obtained by computing $D^i(N \circ \sharp) $ and $ D^i(N \cdot N)$ for $i \in \{1,2,3\}$ with the chain and product rules \cite[12.17, 12.42]{GPR24}, and equating the resulting expressions. The first two of these identities are consequences of \cref{le:CNP identities}\cref{le:CNP identities:7,le:CNP identities:10}, whose proof is independent of Axiom~\ref{def:CNP new}\cref{eq:CNP 18}, while the third one is equivalent to
	\begin{equation}\label{eq:CNP 18:lin 33}
		N(a \times b) = T(a^\sharp, b) T(a, b^\sharp) - N(a) N(b)
	\end{equation}
	because $T(b,b^\sharp) = 3N(b)$. Further, the formal linearization of \eqref{eq:CNP 18:lin 33} of multi-degree $(3,2,1)$ in $(a,b,c) \in J^3$ is
	\begin{align*}
		T\brackets[\big]{(a \times b)^\sharp, a \times c} &= T(a^\sharp, b) T(a, b \times c) - N(a) T(b^\sharp, c) + T(a^\sharp, c) T(a, b^\sharp),
	\end{align*}
	which is also a consequence of \cref{le:CNP identities}\cref{le:CNP identities:10}. We conclude that we could replace Axiom~\ref{def:CNP new}\cref{eq:CNP 18} by the requirement that $N(a^\sharp) = N(a)^2$ and \eqref{eq:CNP 18:lin 33} hold for all $a,b \in J$.
\end{remark}

For the rest of this section, we assume that $(J,J')$ is a cubic norm pair.

\begin{definition}\label{def:U-operators}
	The \emph{U-operators} $U_a \colon J' \to J$ are given by
	\[ U_a (a') \coloneq T(a,a')a - a^\sharp \times a' , \]
	for all $a \in J$ and $a' \in J'$, with linearization
	\[ U_{a,b} (a') = T(a,a')b + T(b,a')a - (a \times b) \times a' \]
	for all $a,b \in J$ and $a' \in J'$.
	We also write
	\[ D_{a,a'}(b) := \{ a, a', b \} \coloneq U_{a,b} (a') . \]
	In particular, for each $a \in J$ and $a' \in J'$, this defines an endomorphism $D_{a,a'}$ of~$J$.
	Similarly, we have endomorphisms $D_{a',a}$ of $J'$.
	(There is no danger of confusion in notation, because $a \in J$ and $a' \in J'$.)
	We will see in \cref{le:CNP is JP} that the $U$-operators equip $(J,J')$ with the structure of a Jordan pair. We will sometimes write $ D(a,a') $ for $ D_{a,a'} $.
\end{definition}

The following lemma says that all identities which are proven in \cite[33.8]{GPR24} for cubic norm structures and which do not refer to the distinguished base point remain valid. We keep the same numbering as in \emph{loc.\@ cit.}.

\begin{lemma}\label{le:CNP identities}
	The following identities hold for all $ a,b,c,d \in J $, $ b',c',d' \in J' $ and also for all $ a,b,c,d \in J' $, $ b',c',d' \in J $:
	\begin{enumerate}[(1)]
		\setcounter{enumi}{6}
		
		\item \label{le:CNP identities:7}$ T(a \times b, c) = T(a, b \times c) $.
		\item \label{le:CNP identities:8}$ a^\sharp \times (a \times b) = N(a)b + T(a^\sharp, b) a $.
		\item \label{le:CNP identities:9}$ (a \times b) \times (a \times c) + a^\sharp \times (b \times c) = T(a^\sharp, b)c + T(a^\sharp, c)b + T(a \times b, c)a $.
		\item \label{le:CNP identities:10}$ a^\sharp \times b^\sharp + (a \times b)^\sharp = T(a^\sharp, b)b + T(a, b^\sharp) a $.
		\setcounter{enumi}{15}
		\item \label{le:CNP identities:16}$ U_a(a \times b) = T(a^\sharp, b)a - N(a)b $.
		\item \label{le:CNP identities:17}$ \{a, a^\sharp, b\} = 2N(a)b $.
		\setcounter{enumi}{19}
		\item \label{le:CNP identities:20}$ (U_a b')^\sharp = U_{a^\sharp}((b')^\sharp) $.
		\item \label{le:CNP identities:21}$ N(U_a b') = N(a)^2 N(b') $.
	
		\setcounter{enumi}{24}
		\item \label{le:CNP identities:25}$ \begin{aligned}[t]
			&(a \times b) \times (c \times d) + (b \times c) \times (a \times d) + (c \times a) \times (b \times d) \\
			&\qquad = T(a \times b, c)d + T(b \times c, d) a + T(c \times d, a) b + T(d \times a, b)c.
		\end{aligned} $
		\item \label{le:CNP identities:26}$ \begin{aligned}[t]
			&a \times (b' \times (a \times c)) = T(a,b')a \times c + T(a^\sharp, c)b' + T(b',c)a^\sharp - (a^\sharp \times b') \times c \\
			&\qquad = (U_a b') \times c + T(a^\sharp, c)b' + T(b',c) a^\sharp
		\end{aligned} $
		\item \label{le:CNP identities:27}$ \begin{aligned}[t]
			&a \times (b' \times (c \times d)) + c \times (b' \times (d \times a)) + d \times (b' \times (a \times c)) \\
			&\qquad = T(a,b')c \times d + T(c,b') d \times a + T(d,b') a \times c + T(a \times c, d)b'.
		\end{aligned} $
		\setcounter{enumi}{28}
		\item \label{le:CNP identities:29}$ N(a \times b) = T(a^\sharp, b) T(a, b^\sharp) - N(a) N(b) $.
		\item \label{le:CNP identities:30}$ T(U_a b', c') = T(b', U_a c') $.
		\item \label{le:CNP identities:31}$ T(\{a,b',c\}, d') = T(c, \{b',a,d'\}) $.
	\end{enumerate}
\end{lemma}

\begin{proof}
	The same arguments and computations as in \cite[33.8]{GPR24} for cubic norm structures work in our more general setting. Note that, unlike in \cite[33.8]{GPR24}, we cannot use the Zariski-density argument \cite[12.24]{GPR24} to restrict our attention to elements with invertible norm because it is not guaranteed that the set
	\[ \mathbf{D}(N)(k) \coloneq \{a \in J \mid N(a) \text{ invertible}\} \]
	is non-empty. Note also that we have already seen~\cref{le:CNP identities:8,le:CNP identities:9,le:CNP identities:10,le:CNP identities:26,le:CNP identities:29} in \cref{rem:CNP def lin}.
\end{proof}

\begin{lemma}[{\cite[Proposition~1]{Faulkner2001}}]\label{le:CNP is JP}
	The $U$-operators from \cref{def:U-operators} equip $(J,J')$ with the structure of a Jordan pair. In other words, the identities
	\begin{gather*}
		\{a, b', U_a c'\} = U_a(\{b', a, c'\}), \\
		\{U_a b', b', c\} = \{a, U_{b'} a, c\} \quad \text{and} \quad U_{U_a b'} c' = U_a U_{b'} U_a c'
	\end{gather*}
	hold for all $a,c \in J$, $b', c' \in J'$.
\end{lemma}
\begin{proof}
	This holds by the computations in \cite[Proposition~1]{Faulkner2001}. Alternatively, we may cite \cite[1.2.9]{Smet25}. The first and third identity also follow from the computations in \cite[33.9]{GPR24}. All identities required in the proof hold by \cref{le:CNP identities}.
\end{proof}

We briefly summarize the relationship between cubic norm pairs and cubic norm structures, which naturally leads to the notion of homotopies.

\begin{definition}\label{def:CNP hom}
	Let $(J_1, J_1')$, $(J_2, J_2')$ be two cubic norm pairs, let $\varphi \colon J_1 \to J_2 $, $\varphi' \colon J_1' \to J_2' $ be two homomorphisms of $k$-modules and let $ t \in k $ be invertible. Then the pair $(\varphi,\varphi')$ is called a \emph{$ t $-homotopy of cubic norm pairs} if
	\begin{align*}
		N_2\bigl(\varphi(a)\bigr) &= tN_1(a), & \varphi(a)^{\sharp_2} &= t \varphi'(a^{\sharp_1}), & T_2\brackets[\big]{\varphi(a), \varphi'(a')} &= T_1(a,a'), \\
		N_2\bigl(\varphi'(a')\bigr) &= t^{-1} N_1(a'), & \varphi'(a')^{\sharp_2} &= t^{-1} \varphi\bigl((a')^{\sharp_1}\bigr)
	\end{align*}
	for all $a \in J_1$, $a' \in J_1'$. A \emph{homotopy} is a $ t $-homotopy for some invertible $ t \in k $, and a \emph{homomorphism} is a $ 1 $-homotopy. We say that $(\varphi,\varphi')$ is surjective, injective or bijective if both $\varphi$ and $\varphi'$ have the respective property. Bijective homotopies are also called isotopies and bijective homomorphisms are also called isomorphisms.
\end{definition}

\begin{remark}\label{rem:homotopy times}
	A $ t $-homotopy $(\varphi, \varphi') \colon (J_1, J_1') \to (J_2, J_2')$ of cubic norm pairs automatically satisfies
	\begin{align*}
		\varphi'(a \times_1 b) &= \varphi'\brackets[\big]{(a+b)^\sharp - a^\sharp - b^\sharp} = t^{-1} \brackets[\big]{(\varphi(a) + \varphi(b))^\sharp - \varphi(a)^\sharp - \varphi(b)^\sharp} \\
		&= t^{-1}\varphi(a) \times_2 \varphi(a)
	\end{align*}
	and similarly $	\varphi(a' \times_1 b') = t \varphi'(a') \times_2 \varphi'(b') $
	for all $a,b \in J_1$, $ a',b' \in J_1' $. Further, if $ (\varphi, \varphi') $ is a bijective $ t $-homotopy, then $ (\varphi^{-1}, (\varphi')^{-1}) $ is a bijective $ t^{-1} $-homotopy, and the composition of a $ t_1 $-homotopy with a $ t_2 $-homotopy is a $ t_1 t_2 $-homotopy.
\end{remark}

\begin{example}
	Let $t,s \in k$ be invertible. Then $(t \id_J, t^{-1} \id_{J'}) \colon (J,J') \to (J,J')$ is a $t^3$-isotopy. It follows that for any $s$-isotopy $(\phi, \phi')$, the pair $(t\phi, t^{-1}\phi') = (\phi, \phi') \circ (t \id_J, t^{-1}\id_{J'})$ is a $t^3s$-isotopy.
\end{example}

\begin{remark}\label{rem:CNP hom jordan}
	Any homomotopy $(\varphi, \varphi') \colon (J_1,J_1') \to (J_2, J_2')$ of cubic norm pairs is also a homomorphism of Jordan pairs: If $ (\varphi, \varphi') $ is a $ t $-homotopy for $ t \in k $ invertible, we have
	\begin{align*}
		\varphi'\brackets[\big]{U_a (a')} &=T(a,a') \varphi(a) - \varphi(a^\sharp \times a') = T\brackets[\big]{\varphi(a), \varphi'(a')} \varphi(a) - t \varphi'(a^\sharp) \times \varphi'(a') \\
		&= T\brackets[\big]{\varphi(a), \varphi'(a')} \varphi(a) - \varphi(a)^\sharp \times \varphi'(a') = U_{\varphi(a)}\brackets[\big]{\varphi'(a')}
	\end{align*}
	for all $ a \in J_1 $, $ a' \in J_1' $, and similarly, $ \varphi(U_{a'}(a)) = U_{\varphi'(a')} \varphi(a) $. It follows that $ \varphi(\{a, a', b\}) = \{\varphi(a), \varphi'(a'), \varphi(b)\} $ and $ \varphi'(\{a', a, b'\}) = \{\varphi'(a'), \varphi(a), \varphi'(b')\} $ for all $ a,b \in J_1 $, $ a',b' \in J_1' $. We conclude that we have an implication chain \enquote{homomorphism of cubic norm pairs} $ \implies $ \enquote{homotopy of cubic norm pairs} $ \implies $ \enquote{homomorphism of Jordan pairs}.
\end{remark}

\begin{remark}\label{rem:CNS homotopy switch}
	Let $(\varphi, \varphi') \colon (J_1,J_1) \to (J_2, J_2')$ be a $t$-homotopy of cubic norm pairs (for some invertible $t \in k$) where $(J_1,J_1)$ is the cubic norm pair induced by a cubic norm structure $J_1$. Then $(\varphi', \varphi) \colon (J_1, J_1) \to (J_2', J_2)$ is a $t^{-1}$-homotopy.
\end{remark}

\begin{definition}
	Let $t \in k$ be invertible. A \emph{$t$-involution} of a cubic norm pair $(J,J')$ is a $t$-isotopy $(\varphi, \varphi') \colon (J,J') \to (J', J)$ of cubic norm pairs such that $\varphi' = \varphi^{-1}$. An \emph{involution} is a $1$-involution.
\end{definition}

\begin{remark}
	Any cubic norm pair that is induced by a cubic norm structure has an involution. However, a cubic norm pair with an involution is not necessarily induced by a cubic norm structure because it need not have a base point in the sense of \cref{def:CNS}.
\end{remark}

\begin{remark}\label{rem:isotopes}
	Assume that there exists $p \in J$ such that $N(p)$ is invertible, and put $p' \coloneq N(p)^{-1}p^\sharp$. Then by \cite[1.1.10]{Smet25} or by the same computations as in \cite[33.10]{GPR24}, $p$ is invertible in the sense of Jordan pairs with inverse $p'$ (meaning that $U_p \colon J' \to J$ is invertible with inverse $U_{p'}$). Moreover, we may define a cubic norm structure with base point $p$ on the $k$-module $J^{(p)} \coloneq J$ by the formulas in \cite[1.1.10]{Smet25} or \cite[33.11]{GPR24}. It is called the \emph{$p$-isotope of $(J,J')$} and its structural maps are given by
	\begin{align*}
		N^{(p)} \colon J \to k &\colon a \mapsto N(p') N(a) = N(p')^{-1} N(U_{p'} a), \\
		\sharp^{(p)} \colon J \to J &\colon a \mapsto N(p')^{-1} U_{p'}(a)^\sharp = N(p') U_{p'}^{-1}(a^\sharp), \\
		T^{(p)} \colon J \times J \to k &\colon (a,b) \mapsto T\brackets[\big]{U_{p'} a, b} = T\brackets[\big]{a, U_{p'} b}, \\
		\times^{(p)} \colon J \times J \to k &\colon (a,b) \mapsto N(p') U_{p'}^{-1} (a \times b) = N(p')^{-1} \brackets[\big]{U_{p'}(a) \times U_{p'}(b)}.
	\end{align*}
	In this setup, a straightforward computation shows that the pair
	\[ (\id_J, U_p^{-1}) = (\id_J, U_{p'}) \colon (J^{(p)}, J^{(p)}) \to (J,J') \]
	is an $N(p)$-isotopy. By \cref{rem:CNS homotopy switch}, this implies that
	\[ (U_p^{-1}, \id_J) = (U_{p'}, \id_J) \colon (J^{(p)}, J^{(p)}) \to (J',J) \]
	is an $N(p)^{-1}$-isotopy. In particular, the composition
	\[ (U_p^{-1}, U_p) = (U_p^{-1}, \id_J) \circ (\id_J, U_p^{-1})^{-1} \colon (J,J') \to (J', J) \]
	is an $N(p)^{-2}$-involution, and hence
	\[ (N(p)U_p^{-1}, N(p)^{-1}U_p) \colon (J,J') \to (J',J) \]
	is an $N(p)$-involution.
	
	We conclude that a cubic norm pair $ (J,J') $ is isotopic to the cubic norm pair induced by a cubic norm structure if and only if there exists $ p \in J $ such that $ N(p) $ is invertible. Further, for a second cubic norm pair $(J_2, J_2')$ and a pair $(\varphi, \varphi') \colon (J_2, J_2') \to (J,J')$ of $k$-linear maps, another straightforward computation shows that $(\varphi, \varphi')$ is an $N(p)$-homomoty if and only if the composition $ (\id_J, U_{p}) \circ (\varphi, \varphi') $ is a homomorphism of cubic norm pairs (from $ (J_2, J_2') $ to $ (J^{(p)}, J^{(p)}) $). Thus the notion \enquote{$t$-homotopy to $(J,J')$ for some invertible $t \in k \cap \im(N)$} coincides with the notion \enquote{homomorphism to an isotope of $(J,J')$}. The second notion is precisely the definition of homotopies in the theory of Jordan algebras (see \cite[31.17]{GPR24}).
\end{remark}

We conclude this section with some identities that will be needed later.

\begin{lemma}\label{le:triple D}
	For all $ a,b,c \in J $, the following identities hold.
	\begin{enumerate}
		\item \label{le:triple D:1} $ D_{a,a^\sharp} = 2N(a) \id_J $.
		\item \label{le:triple D:2} $ D_{a^\sharp, a} = 2N(a) \id_{J'} $.
		\item \label{le:triple D:3} $ D_{a,b^\sharp} + D_{b, a \times b} = 2T(a,b^\sharp) \id_J $.
		\item \label{le:triple D:4} $ D_{a^\sharp, b} + D_{a \times b, a} = 2T(a^\sharp, b) \id_{J'} $.
		\item \label{le:triple D:r} $D_{a, b \times c} + D_{b, c \times a} + D_{c, a \times b} = 2 T(a, b \times c) \id_{J}$.
		\item \label{le:triple D:l} $D_{a \times b, c} + D_{b\times c, a} + D_{c \times a, b} = 2 T(a \times b, c) \id_{J'}$.
	\end{enumerate}
	The same identities hold for all $ a,b,c \in J' $ if the roles of $ \id_J $ and $ \id_{J'} $ are interchanged.
\end{lemma}
\begin{proof}
	Let $ a,b,c,d \in J $.
	By \cref{def:CNP new}\cref{eq:CNP cubic} and \cref{le:CNP identities}\cref{le:CNP identities:8}, we have
	\[ D_{a,a^\sharp}(c) = T(a,a^\sharp)c + T(c,a^\sharp)a - (a \times c) \times a^\sharp = 2N(a) c , \]
	proving \cref{le:triple D:1}.
	Next, by \cref{le:CNP identities}\cref{le:CNP identities:7,le:CNP identities:9}, we have
	\begin{align*}
		D_{a,b^\sharp}(c) + D_{b, a \times b}(c)
		&= T(a,b^\sharp)c + T(c,b^\sharp)a - (a \times c) \times b^\sharp \\
		&\hspace*{8ex} + T(b,a \times b)c + T(c,a \times b)b - (b \times a) \times (b \times c) \\
		&= 2 T(a, b^\sharp) c ,
	\end{align*}
	proving \cref{le:triple D:3}.
	Now \cref{le:triple D:r} follows immediately from \cref{le:triple D:3} by linearizing in $b$.
%
	
	The proofs of \cref{le:triple D:2}, \cref{le:triple D:4} and \cref{le:triple D:l} are completely similar.
\end{proof}

\begin{lemma}\label{le:CNP new}
	For all $ a,b \in J $ and $ c' \in J' $, the following identities hold.
	\begin{enumerate}
		\item \label{eq:new 4} $U_a(a^\sharp) = N(a) a$.
		\item \label{eq:new 5} $T(U_a(c'), a^\sharp) = N(a) T(a, c')$.
		\item \label{eq:new 6} $U_a(c') \times a = T(a, c') a^\sharp - N(a) c'$.
		\item \label{eq:new 1} $T\bigl((a \times b) \times c', a^\sharp \bigr) = N(a) T(b, c') + T(b, a^\sharp) T(a, c')$.
		\item \label{eq:new 2} $T\bigl((a \times b)^\sharp, a^\sharp \bigr) = N(a) T(a, b^\sharp) + T(b, a^\sharp)^2$.
		\item \label{eq:new 3} $(a \times b)^\sharp \times a = T(a, b^\sharp) a^\sharp + T(b, a^\sharp) a \times b - N(a) b^\sharp$.
	\end{enumerate}
\end{lemma}
\begin{proof}
	Let $a,b \in J$ and $c' \in J'$.
	\begin{enumerate}
		\item
		Using \cref{def:CNP new}\cref{eq:CNP cubic,eq:sharp twice}, we have
		$U_a(a^\sharp) = T(a, a^\sharp) a - a^\sharp \times a^\sharp = 3N(a)a - 2 (a^\sharp)^\sharp = N(a)a$.
		\item
		Again using \cref{def:CNP new}\cref{eq:CNP cubic,eq:sharp twice}, we get
		\begin{align*}
			T(U_a(c'), a^\sharp)
			&= T(a, c') T(a, a^\sharp) - T(a^\sharp \times c', a^\sharp) \\
			&= 3N(a) T(a, c') - T(a^\sharp \times a^\sharp, c') = N(a) T(a, c') .
		\end{align*}
		\item
		By \cref{def:CNP new}\cref{eq:CNP 19}, we have
		$U_a(c') \times a = T(a, c') a \times a - (a^\sharp \times c') \times a = T(a, c') a^\sharp - N(a) c'$.
		\item
		Using \cref{le:CNP identities}\cref{le:CNP identities:8}, we get
		$T\bigl((a \times b) \times c', a^\sharp \bigr) = T\bigl((a \times b) \times a^\sharp, c' \bigr) = N(a) T(b, c') + T(b, a^\sharp) T(a, c')$.
		\item
		Using \cref{def:CNP new}\cref{eq:CNP cubic,eq:sharp twice} and \cref{le:CNP identities}\cref{le:CNP identities:10}, we get
		\begin{align*}
			T\bigl((a \times b)^\sharp, a^\sharp \bigr)
			&= T(b, a^\sharp) T(b, a^\sharp) + T(a, b^\sharp) T(a, a^\sharp) - T(a^\sharp \times b^\sharp, a^\sharp) \\
			&= T(b, a^\sharp)^2 + 3 N(a) T(a, b^\sharp) - T(a^\sharp \times a^\sharp, b^\sharp) \\
			&= N(a) T(a, b^\sharp) + T(b, a^\sharp)^2 .
		\end{align*}
		\item
		Using \cref{def:CNP new}\cref{eq:CNP 19} and \cref{le:CNP identities}\cref{le:CNP identities:10}, we get
		\begin{align*}
			(a \times b)^\sharp \times a
			&= T(b, a^\sharp) b \times a + T(a, b^\sharp) a \times a - (a^\sharp \times b^\sharp) \times a \\ 
			&= T(b, a^\sharp) a \times b + 2 T(a, b^\sharp) a^\sharp - N(a) b^\sharp - T(a, b^\sharp) a^\sharp \\ 
			&= T(a, b^\sharp) a^\sharp + T(b, a^\sharp) a \times b - N(a) b^\sharp .
		\qedhere
		\end{align*}
	\end{enumerate}
\end{proof}

\subsection{Linear cubic norm pairs}\label{subsec:linear CNP}

One of our main goals, which will be accomplished in \cref{sec:lie const}, is to construct a $ G_2 $-graded Lie algebra $ L=L(J,J') $ from an arbitrary cubic norm pair $ (J,J') $. Since a Lie algebra is a \enquote{linear object}, it is not surprising that $ L $ depends only on the ring $ k $, the modules $ J $, $ J' $ and the bilinear maps $ \times $ and $ T $. In other words, the maps $ N $ and $ \sharp $ of higher degree will appear neither in the definition of the module $ L $ nor in its multiplication table (\cref{ta:Lie}). This observation, which may seem like a mere curiosity at first glance, will be needed explicitly in \cref{sec:simple} (specifically, in \cref{rem:ideal quot}). The purpose of this section is to make the aforementioned observation more precise by introducing the notion of \emph{linear cubic norm pairs}, which capture the linear structure of (and hence generalize) cubic norm pairs. It is only this linear structure that will be needed in \cref{sec:TKK,sec:lie const,sec:reflections,sec:simple}. In contrast, the definition of exponential maps on the Lie algebra $ L $ in \cref{sec:G2 exp} involves the maps $ \sharp $ and $ N $. Hence our construction of $ G_2 $-graded groups works only for cubic norm pairs but not for arbitrary linear cubic norm pairs.

The reader who is enthusiastic to get to our main results may want to skip this somewhat technical section until later.

\begin{definition}\label{def:linear CNP}
	A \emph{linear cubic norm pair (over $k$)} is a pair $(J,J')$ of $k$-modules which are equipped with maps
	\begin{align*}
		\times &\colon J \times J \to J', &
		\times' &\colon J' \times J' \to J, \\
		T &\colon J \times J' \to k, &
		T' &\colon J' \times J \to k
	\end{align*}
	such that the following conditions are satisfied.
	\begin{enumerate}
		\item $\times$ and $T$ are $k$-bilinear maps, $\times$ is symmetric, and $T'(a',a) = T(a,a')$ for all $a \in J$, $a' \in J'$.
		
		\item The identities in \cref{le:CNP identities}\cref{le:CNP identities:7,le:CNP identities:25,le:CNP identities:27,le:CNP identities:31} hold:
		\begin{align*}
			&T'(a \times b, c) = T(a, b \times c), \\
			&a \times (b' \times' (a \times c)) = T(a,b')a \times c + T'(a^\sharp, c)b' + T'(b',c)a^\sharp - (a^\sharp \times' b') \times c \\
			&\qquad = (U_a b') \times c + T'(a^\sharp, c)b' + T'(b',c) a^\sharp, \\
			&a \times (b' \times' (c \times d)) + c \times (b' \times' (d \times a)) + d \times (b' \times' (a \times c)) \\
			&\qquad = T(a,b')c \times d + T(c,b') d \times a + T(d,b') a \times c + T'(a \times c, d)b', \\
			&T(\{a,b',c\}, d') = T(c, \{b',a,d'\})
		\end{align*}
		for all $ a,b,c \in J $ and $ b',d' \in J' $.		
		\item All the previous conditions are still satisfied if we interchange the roles of $(J,\times, T)$ and $(J', \times', T')$.
	\end{enumerate}
	We will usually simply write $ \times $, $ T $, $ N $, $ \sharp $ for the maps $ \times' $, $ T' $, $ N' $, $ \sharp' $. We will sometimes refer to cubic norm pairs as \emph{proper} cubic norm pairs to emphasize the distinction from linear cubic norm pairs.
\end{definition}

\begin{example}
	Any proper cubic norm pair $ (J,J') $ with maps $ \times $, $ \times' $, $ T $, $ T' $, $ N $, $ N' $, $ \sharp $, $ \sharp' $ induces a linear cubic norm pair $ (J,J') $ with maps $ \times $, $ \times' $, $ T $, $ T' $. Certain quotients of cubic norm pairs, which we will encounter in \cref{rem:ideal quot}, provide examples of linear cubic norm pairs that are not (proven to be) proper cubic norm pairs.
\end{example}

For the rest of this section, $ (J,J') $ denotes a linear cubic norm pair over $ k $.

\begin{definition}
	For all $ a,b \in J $, we define a \emph{$ U $-operator} $ U_{a,b} \colon J' \to J $ by 
	\[ U_{a,b} (a') \coloneq T(a,a')b + T(b,a')a - (a \times b) \times a' \]
	for all $a' \in J'$.
	We also write
	\[ D_{a,a'}(b) := \{ a, a', b \} \coloneq U_{a,b} (a') . \]
\end{definition}

\begin{remark}
	The set of axioms in \cref{def:linear CNP} is adequate for our purposes, but it might not be sufficient to build a satisfying theory of linear cubic norm pairs (which we have no interest in doing). For example, we do not claim that the $ U $-operators equip $ (J,J') $ with the structure of a linear Jordan pair, and this seems cumbersome to prove (if it is even true).
\end{remark}

\begin{definition}\label{def:linear CNP homotopy}
	Let $(J_1, J_1')$, $(J_2, J_2')$ be two linear cubic norm pairs, let $\varphi \colon J_1 \to J_2 $, $\varphi' \colon J_1' \to J_2' $ be two homomorphisms of $k$-modules and let $ t \in k $ be invertible. Then the pair $(\varphi,\varphi')$ is called a \emph{$ t $-homotopy of linear cubic norm pairs} if
	\begin{gather*}
		T_2\brackets[\big]{\varphi(a), \varphi'(a')} = T_1(a,a'), \\
		\varphi'(a \times_1 b) = t^{-1}\varphi(a) \times_2 \varphi(a), \quad \varphi(a' \times_1 b') = t \varphi'(a') \times_2 \varphi'(b')
	\end{gather*}
	for all $a,b \in J_1$, $a',b' \in J_1'$. A \emph{homotopy} is a $ t $-homotopy for some invertible $ t \in k $, and a \emph{homomorphism} is a $ 1 $-homotopy. We say that $(\varphi,\varphi')$ is surjective, injective or bijective if both $\varphi$ and $\varphi'$ have the respective property. Bijective homotopies are also called isotopies and bijective homomorphisms are also called isomorphisms.
\end{definition}

\begin{remark}
	Any homotopy of linear cubic norm pairs satisfies
	\[ \varphi(\{a, a', b\}) = \{\varphi(a), \varphi'(a'), \varphi(b)\} \quad \text{and} \quad \varphi'(\{a', a, b'\}) = \{\varphi'(a'), \varphi(a), \varphi'(b')\} \]
	for all $ a,b \in J_1 $, $ a',b' \in J_1' $.
\end{remark}

\begin{remark}
	The identities
	\begin{align*}
		D_{a, b \times c} + D_{b, c \times a} + D_{c, a \times b} &= 2 T(a, b \times c) \id_{J}, \\
		D_{a \times b, c} + D_{b\times c, a} + D_{c \times a, b} &= 2 T(a \times b, c) \id_{J'}
	\end{align*}
	from \cref{le:triple D} remain valid for linear cubic norm pairs because they can be proven using only \cref{le:CNP identities}\cref{le:CNP identities:7,le:CNP identities:25,le:CNP identities:27}.
\end{remark}

\section{The Tits--Kantor--Koecher Lie algebra for cubic norm pairs}\label{sec:TKK}

Throughout this section, we assume that $(J, J')$ is a cubic norm pair (or, more generally, a linear cubic norm pair) over the commutative ring $ k $. We first introduce derivations and inner derivations.
(We have not seen this notion before in the literature.)
\begin{definition}\label{def:der}
	\begin{enumerate}
		\item
			A \emph{derivation} of $(J, J')$ is a pair of $k$-linear maps $\delta = (\delta_+, \delta_-)$ where $\delta_+ \colon J \to J$ and $\delta_- \colon J' \to J'$ such that there is a constant $\lambda \in k$ for which
			\begin{alignat}{3}				
				\delta_-(a \times b) &= \delta_+(a) \times b &&+ a \times \delta_+(b) &&- \lambda(a \times b) , \label{eq:der times} \\
				\delta_+(a' \times b') &= \delta_-(a') \times b' &&+ a' \times \delta_-(b') &&+ \lambda(a' \times b') , \label{eq:der times'} \\
				0 &= T(\delta_+(a), a') &&+ T(a, \delta_-(a')) &&, \label{eq:der T}
			\end{alignat}
			for all $a,b \in J$ and $a',b' \in J'$.
		\item
			The collection of all derivations of $(J, J')$ forms a Lie algebra that we denote by $\Der(J, J')$.
		\item
			Let $a \in J$ and $a' \in J'$. Then the pair
			\[ \delta_{a,a'} := (D_{a,a'}, -D_{a',a}) \]
			defines an \emph{inner derivation} of the cubic norm pair.
			We will show in \cref{le:inner} below that $\delta_{a,a'}$ is indeed a derivation of $(J, J')$, with $\lambda = 2T(a, a')$.
		\item\label{def:der:gamma}
			We define the \emph{grading derivation}
			\[ \gamma := (\id_J, -\id_{J'}) . \]
			This is indeed a derivation of $(J, J')$, with $\lambda = 3$.
	\end{enumerate}
\end{definition}
\begin{remark}\label{rem:JP}
	If $\delta = (\delta_+, \delta_-)$ is a derivation, then
	\[ \delta_+ \bigl( (a \times b) \times a' \bigr)
		= \bigl( \delta_+(a) \times b \bigr) \times a' + \bigl( a \times \delta_+(b) \bigr) \times a' + (a \times b) \times \delta_-(a') \]
	for all $a,b \in J$ and $a' \in J'$.
	Together with \cref{eq:der T}, it follows that
	\[ \delta_+ \{ a, a', b \} = \{ \delta_+(a), a', b \} + \{ a, \delta_-(a'), b \} + \{ a, a', \delta_+(b) \} \]
	for all $a,b \in J$ and $a' \in J'$.
	In the language of \emph{Jordan pairs}, this tells exactly that $\delta = (\delta_+, \delta_-)$ is a derivation of the Jordan pair $(J, J')$.
	See, for instance, \cite[\S 6.8]{LoosNeher}.
\end{remark}
\begin{lemma}\label{le:inner}
	The pair $\delta_{a,a'} := (D_{a,a'}, -D_{a',a})$ is a derivation of $(J, J')$. More precisely, it satisfies \cref{eq:der times,eq:der times',eq:der T} with $\lambda = 2 T(a, a')$.
\end{lemma}
\begin{proof}
	By \cref{le:CNP identities}\cref{le:CNP identities:27} (or \cref{def:linear CNP}),
	\begin{multline*}
		\bigl( (a \times b) \times a' \bigr) \times c 
		+ \bigl( (b \times c) \times a' \bigr) \times a 
		+ \bigl( (c \times a) \times a' \bigr) \times b \\
		= T(a, b \times c)a' + T(a, b') b \times c + T(b, b') c \times a + T(c, b') a \times b
	\end{multline*}
	for all $a,b,c \in J$ and $a' \in J'$.
%
	Using this identity, it is easy to verify that
	\begin{equation}\label{eq:D is der}
		D_{a',a}(b \times c) + D_{a,a'}(b) \times c + b \times D_{a,a'}(c) = 2 T(a, a') b \times c,
	\end{equation}
	proving \cref{eq:der times} with $\lambda = 2 T(a, a')$.
	The proof of \cref{eq:der times'} is similar, and \cref{eq:der T} holds by \cref{le:CNP identities}\cref{le:CNP identities:31} (or \cref{def:linear CNP}).
\end{proof}

We can now define the Tits--Kantor--Koecher Lie algebra $\TKK(J,J')$ of the cubic norm pair $(J, J')$.
First, let
\[ I \coloneq \TKK_0(J,J') \coloneq \langle \delta_{a,a'} \mid a \in J, a' \in J' \rangle + k \gamma \leq \Der(J, J') . \]
We then define the \emph{TKK Lie algebra}
\[ \TKK(J, J') := J' \oplus I \oplus J , \]
equipped with the usual Lie bracket on $I$ inherited from $\Der(J, J')$ extended by the rules
\begin{equation}\label{eq:TKK rules}
	[J,J] = 0, \quad [J', J'] = 0, \quad [\delta, a] = \delta_+(a), \quad [\delta, a'] = \delta_-(a'), \quad [a',a] = \delta_{a,a'}
\end{equation}
for all $a \in J$, $a' \in J'$ and $\delta \in I$.

To show that this defines a $3$\dash graded Lie algebra, we could make a detour via Jordan pairs (see \cite{Neher94}), but we believe that it is interesting to provide a (short and easy) direct proof of this fact.
\begin{proposition}
	The algebra $\TKK(J, J') = J' \oplus I \oplus J$ defined above is a $3$-graded Lie algebra, with $J'$ of degree $-1$, $I$ of degree $0$ and $J$ of degree $1$.
\end{proposition}
\begin{proof}
	We first observe that $I$ is, in fact, an ideal of $\Der(J, J')$.
	Indeed, let $\delta = (\delta_+, \delta_-) \in \Der(J, J')$ and let $a \in J$ and $a' \in J'$.
	Then using \cref{eq:der T,rem:JP}, we compute that
	\begin{equation}\label{eq:delta bracket}
		[\delta, \delta_{a,a'}] = \delta_{\delta_+(a), a'} + \delta_{a, \delta_-(a')} .
	\end{equation}
	Moreover, $\gamma$ is central in $\Der(J, J')$, so indeed $I = \langle \delta_{a,a'} \mid a \in J, a' \in J' \rangle + k \gamma$ is an ideal of $\Der(J, J')$. In particular, $I$ is itself a Lie algebra.
	
	Next, we claim that the identities \cref{eq:TKK rules} make $J$ (and similarly, $J'$) into a module for the Lie algebra~$I$.
	In fact, if we allow $\delta \in \Der(J, J')$ in these identities (rather than $\delta \in I$), then $J$ becomes a module for the full Lie algebra $\Der(J, J')$. Indeed, let $\delta = (\delta_+, \delta_-)$ and $\epsilon = (\epsilon_+, \epsilon_-)$ be two elements of $\Der(J, J')$ and $a \in J$, then we have to verify that $[[\delta, \epsilon], a] = [\delta, [\epsilon, a]] - [\epsilon, [\delta, a]]$, but since $[\delta, \epsilon] = (\delta_+ \epsilon_+ - \epsilon_+ \delta_+, \delta_- \epsilon_- - \epsilon_- \delta_-)$, this is obvious from the identities \cref{eq:TKK rules}.
	
	Finally, we have to check that the Jacobi identity holds.
	There are only two remaining non-trivial cases to consider.
	We first consider triples of the form $(a, b, a') \in J \times J \times J'$. Then $[a,b] = 0$ and
	\begin{align*}
		&[[a', a], b] = [\delta_{a,a'}, b] = D_{a,a'}(b) \\
		&[[a', b], a] = [\delta_{b,a'}, a] = D_{b,a'}(a),
	\end{align*}
	and by definition, both expressions are equal to $U_{a,b}(a')$.
	The case of triples in $J' \times J' \times J$ is, of course, completely similar.
	
	Next, we consider triples of the form $(a, a', \delta) \in J \times J' \times \Der(J, J')$.
	By \cref{eq:delta bracket}, we have
	\[ [\delta, [a', a]] = [\delta, \delta_{a, a'}] = \delta_{\delta_+(a), a'} + \delta_{a, \delta_-(a')} . \]
	On the other hand, we have
	\begin{align*}
		&[[\delta, a'], a] = [\delta_-(a'), a] = \delta_{a, \delta_-(a')} \quad \text{and} \\
		&[[\delta, a], a'] = [\delta_+(a), a'] = - \delta_{\delta_+(a), a'} ,
	\end{align*}
	and we see that indeed the Jacobi identity holds also in this case.
\end{proof}
\begin{remark}
	In fact, we have shown that also the larger algebra
	\[ \widetilde\TKK(J, J') := J' \oplus \Der(J, J') \oplus J \]
	is a $3$-graded Lie algebra, but we will not need this.
\end{remark}

\section{Constructing the Lie algebra}\label{sec:lie const}

\subsection{Constructing the bracket}

We continue to assume that $(J, J')$ is a cubic norm pair (or, more generally, a linear cubic norm pair) over the commutative ring~$ k $. Our aim is to extend $\widetilde L_0 := \TKK(J, J')$ to a larger $5$-graded Lie algebra
\[ L_{-2} \oplus L_{-1} \oplus \widetilde L_0 \oplus L_1 \oplus L_2 \]
with
\[ L_{-2} \cong L_2 \cong k \quad \text{and} \quad L_{-1} \cong L_1 \cong k \oplus J \oplus J' \oplus k, \]
with a finer $5 \times 5$-grading as in \cref{fig:grading}. This Lie algebra will be denoted by $ L $, or (in later sections) by $ L(J,J') $ to emphasize its dependence on $ (J,J') $.

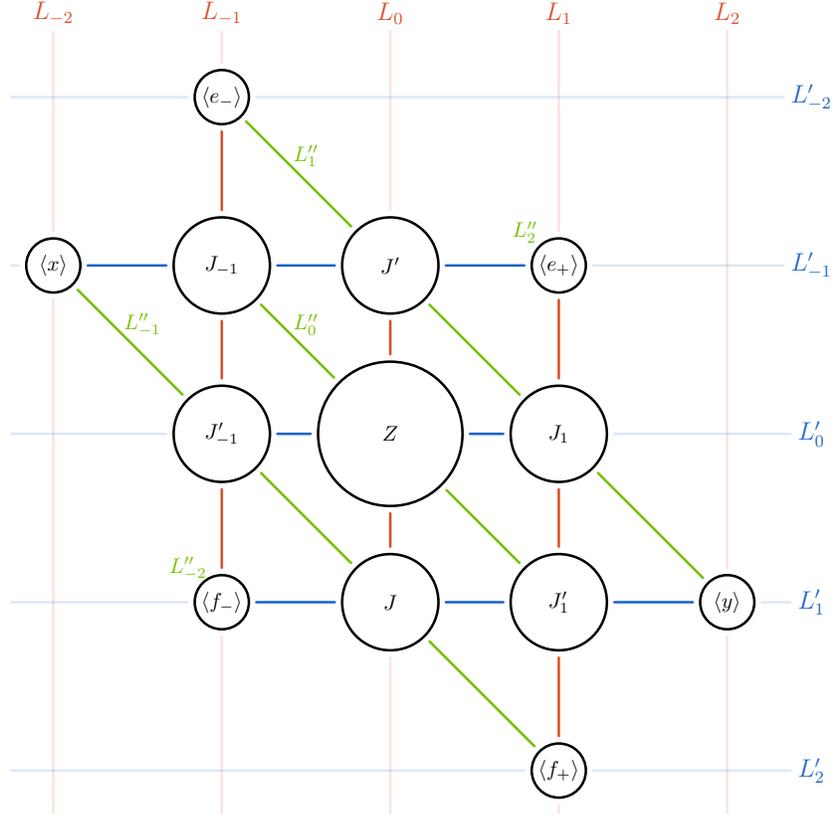
\begin{figure}[ht!]
\[
\scalebox{.8}{%
	\begin{tikzpicture}[x=28mm, y=28mm, 	label distance=-3pt]
		\node[myblob=9mm] at (0,3) (N03) {$\langle x\rangle$};
		\node[myblob=9mm] at (1,1) (N11) {$\langle f_- \rangle$};
		\node[myblob=16mm] at (1,2) (N12) {$J'_{-1}$};
		\node[myblob=16mm] at (1,3) (N13) {$J_{-1}$};
		\node[myblob=9mm] at (1,4) (N14) {$\langle e_- \rangle$};
		\node[myblob=16mm] at (2,1) (N21) {$J$};
		\node[myblob=24mm] at (2,2) (N22) {$Z$};
		\node[myblob=16mm] at (2,3) (N23) {$J'$};
		\node[myblob=9mm] at (3,0) (N30) {$\langle f_+ \rangle$};
		\node[myblob=16mm] at (3,1) (N31) {$J'_1$};
		\node[myblob=16mm] at (3,2) (N32) {$J_1$};
		\node[myblob=9mm] at (3,3) (N33) {$\langle e_+ \rangle$};
		\node[myblob=9mm] at (4,1) (N41) {$\langle y\rangle$};
		\path[ugentred]
			(0,4.5) node (C0) {\Large $L_{-2}$}
			(1,4.5) node (C1) {\Large $L_{-1}$}
			(2,4.5) node (C2) {\Large $L_{0}$}
			(3,4.5) node (C3) {\Large $L_{1}$}
			(4,4.5) node (C4) {\Large $L_{2}$};
		\path[ugentblue]
			(4.5,0) node (R0) {\Large $L'_{2}$}
			(4.5,1) node (R1) {\Large $L'_{1}$}
			(4.5,2) node (R2) {\Large $L'_{0}$}
			(4.5,3) node (R3) {\Large $L'_{-1}$}
			(4.5,4) node (R4) {\Large $L'_{-2}$};
		\path[ugentyellow]
			(2.8,3.2) node (H0) {$L''_{2}$}
			(1.5,3.65) node (H1) {$L''_{1}$}
			(1.5,2.65) node (H2) {$L''_{0}$}
			(.53,2.65) node (H3) {$L''_{-1}$}
			(.8,1.2) node (H4) {$L''_{-2}$};
		\draw[myedge,ugentblue]
			(N11) -- (N21) -- (N31) -- (N41)
			(N12) -- (N22) -- (N32)
			(N03) -- (N13) -- (N23) -- (N33);
		\draw[myedge,ugentblue,opacity=.15]
			(-.25,0) -- (N30) -- (R0)
			(-.25,1) -- (N11) (N41) -- (R1)
			(-.25,2) -- (N12) (N32) -- (R2)
			(-.25,3) -- (N03) (N33) -- (R3)
			(-.25,4) -- (N14) -- (R4);
		\draw[myedge,ugentred]
			(N11) -- (N12) -- (N13) -- (N14)
			(N21) -- (N22) -- (N23)
			(N30) -- (N31) -- (N32) -- (N33);
		\draw[myedge,ugentred,opacity=.15]
			(0,-.25) -- (N03) -- (C0)
			(1,-.25) -- (N11) (N14) -- (C1)
			(2,-.25) -- (N21) (N23) -- (C2)
			(3,-.25) -- (N30) (N33) -- (C3)
			(4,-.25) -- (N41) -- (C4);
		\draw[myedge,ugentyellow]
			(N03) --  (N12) -- (N21) -- (N30)
			(N13) -- (N22) -- (N31)
			(N14) --  (N23) -- (N32) -- (N41);
	\end{tikzpicture}
}
\]
\caption{Construction of a $5 \times 5$-graded Lie algebra}\label{fig:grading}
\end{figure}

This involves an important subtlety, as we now explain, and in particular, we will only be able to construct such a Lie algebra after modifying $\widetilde L_0$
(which is also why the middle part in \cref{fig:grading} is called $Z$ rather than $I$).

We will first need to ensure the presence of a \emph{grading derivation} $\xi$ for this new grading, which we can do, in principle, by defining
$\overline L_0 := \TKK(J, J') \oplus k \xi$, a central extension of $\TKK(J, J')$ by a free $k$-module spanned by a new element $\xi$ of degree~$0$.
However, in the resulting Lie algebra that we will construct, the element $\xi$ is not necessarily linearly independent from $\TKK(J, J')$, so we will replace $\overline L_0$ by the Lie algebra $L_0$ consisting of the corresponding inner derivations acting on $\bigoplus_{i \neq 0} L_i$, following a typical idea that has been worked out in detail in \cite[Construction 4.1.2]{BDMS19}.
In addition, this will also solve any problems that we might have if we would try to directly define an action of $\TKK(J, J')$ on the $L_i$.
(See also \cref{rem:TKK no action} below.)

We start by setting up the notation for the $k$-modules $L_i$ for $i \neq 0$.

\begin{definition}\label{def:L0 bar}
	\begin{enumerate}
		\item 
			We define four $k$-modules
			\begin{align*}
				L_{-2} &:= kx, \\
				L_{-1} &:= \{ (\lambda, b, b', \mu)_- \mid \lambda,\mu \in k, b \in J, b' \in J' \}, \\
				L_{1} &:= \{ (\lambda, b, b', \mu)_+ \mid \lambda,\mu \in k, b \in J, b' \in J' \}, \\
				L_{2} &:= ky.
			\end{align*}
		\item
			We will further decompose
			\begin{alignat*}{4}
				L_{-1} &= L_{-1,-2} &&\oplus L_{-1,-1} &&\oplus L_{-1,0} &&\oplus L_{-1,1}, \\
				L_{1} &= L_{1,-1} &&\oplus L_{1,0} &&\oplus L_{1,1} &&\oplus L_{1,2},
			\end{alignat*}
			where
			\begin{align*}
				L_{-1,-2} &:= \{ (\lambda, 0, 0, 0)_- \mid \lambda \in k \}, &
				L_{1,-1} &:= \{ (\lambda, 0, 0, 0)_+ \mid \lambda \in k \}, \\
				L_{-1,-1} &:= \{ (0, b, 0, 0)_- \mid b \in J \}, &
				L_{1,0} &:= \{ (0, b, 0, 0)_+ \mid b \in J \}, \\
				L_{-1,0} &:= \{ (0, 0, b', 0)_- \mid b' \in J' \}, &
				L_{1,1} &:= \{ (0, 0, b', 0)_+ \mid b' \in J' \}, \\
				L_{-1,1} &:= \{ (0, 0, 0, \mu)_- \mid \mu \in k \}, &
				L_{1,2} &:= \{ (0, 0, 0, \mu)_+ \mid \mu \in k \}.
			\end{align*}
			We also write $L_{-2,-1} := L_{-2}$ and $L_{2,1} := L_2$.
			We call elements of $L_{i,j}$, whenever $L_{i,j}$ is defined, \emph{homogeneous of $\xi$-degree $i$ and $\zeta$-degree $j$}.
		\item
			We define the Lie algebra $M := \bigoplus_{i \neq 0} \gl(L_i)$.
			In \cref{def:dd and L0}, we will be able to define the Lie algebra $L_0$ as a subalgebra of $M$, an idea that also appears in \cite[Construction~4.1.2]{BDMS19}.
			Notice that trivially, each $L_i$ is a module for the Lie algebra $M$.
		\item\label{item:zeta and xi}
			We define two specific elements $\zeta$ and $\xi$ in $M$ by setting
			\[ \xi(l) = il \quad \text{and} \quad \zeta(l) = jl \quad \text{whenever } l \in L_{i,j} . \]
			Explicitly,
			\begin{align*}
				\zeta(x) &:= -x , \\
				\zeta\bigl((\lambda, b, b', \mu)_- \bigr) &:= (-2 \lambda, -b, 0, \mu)_- , \\
				\zeta\bigl((\lambda, b, b', \mu)_+ \bigr) &:= (-\lambda, 0, b', 2 \mu)_+ , \\
				\zeta(y) &:= y , \\
				\xi(l_i) &:= i l_i \quad \text{for all } l_i \in L_i, i \neq 0 .
			\end{align*}
		\item\label{item:zeta-degree in M}
			Let $\varphi = (\varphi_i)_{i \neq 0} \in M$.
			We say that $\varphi$ \emph{has $\zeta$-degree $j$} if each $\varphi_i$ maps all elements of $\zeta$-degree $i$ to elements of $\zeta$-degree $i+j$.
			(Visually, this means that $\varphi$ shifts the parts in \cref{fig:grading} down by $j$.)
			In particular, $\zeta$ and $\xi$ have $\zeta$-degree~$0$.
		\item
			Since our goal is to construct a Lie algebra $L_{-2} \oplus L_{-1} \oplus L_0 \oplus L_1 \oplus L_2$ with $L_0 \leq M$, we will already denote actions of elements of $M$ on each $L_i$ using a Lie bracket, i.e., we will write%
			\footnote{In principle, each $\varphi \in M$ is a $4$-tuple $\varphi = (\varphi_i)_{i \neq 0}$, so $[\varphi, l_i] := \varphi_i(l_i)$, but we will omit the subscript $i$ and simply write $\varphi(l_i)$ without danger of confusion.}
			$[\varphi, l_i] := \varphi(l_i)$ for all $\varphi \in M$ and $l_i \in L_i$, $i \neq 0$.
	\end{enumerate}
\end{definition}

\begin{remark}\label{rem:skew-symmetric}
	Whenever we define a Lie bracket between certain elements, it is silently understood that it is skew-symmetric.
	In other words, once we define an element $[l,m]$, we silently assume that we also define $[m,l] := - [l,m]$.
\end{remark}

\begin{remark}\label{rem:all gradings}
	A graphical depiction of the subspaces of $ L $ defined in \cref{def:L0 bar} is given in \cref{fig:grading} (where $ Z $ is a certain subspace of $ M $ that will be defined in \cref{def:dd and L0}). Our goal is to define a Lie bracket on $ L $ such that these subspaces constitute a $ 5 \times 5 $-grading on $ L $, meaning that $ [L_{ij}, L_{pq}] \subseteq L_{i+p,j+q} $ for all $ i,j,p,q \in \Z $ where $ L_{st} \coloneq 0 $ for all $ s,t \in \Z $ for which no such space $ L_{st} $ has been defined previously. In particular, the \enquote{vertical lines} $ (L_i)_{-2 \le i \le 2} $, the \enquote{horizontal lines} $ (L_i')_{-2 \le i \le 2} $ and the \enquote{diagonal lines} $ (L_i'')_{-2 \le i \le 2} $ that can be read off from \cref{fig:grading} will each constitute a 5-grading of $ L $. When $k$ is a field, these 5-gradings correspond to the so-called \enquote{hyperbolic pairs} $(x,y)$, $(e_-, f_+)$ and $(-f_-, e_+)$ in $L$, respectively; see \cite[\S 4, page~28]{DMM25}. Note that $ L_{i,j} = L_i \cap L_j' $ for all $ i,j \in \{-2, \ldots, 2\} $. Further, we have a $ 7 $-grading $(L^\backslasharrow_i)_{-3 \le i \le 3}$ of $ L $ defined by
	\[
		L^{\backslasharrow}_{i} \coloneq \sum_{p \in \mathbb{Z}} L_{p,i-p}
	\]
	for all $i \in \{-3, \ldots, 3\}$. (This grading, however, does not correspond to a hyperbolic pair.)
\end{remark}

\begin{remark}\label{rem:G2 grid}
	Let $\{ \gamma,\delta \}$ be a choice of simple roots for the root system $G_2$, where $\gamma$ is long and $\delta$ is short.
	Then there is a unique group homomorphism $\omega$ from $\Z G_2$ to $\Z^2$ mapping $\gamma$ to $(-2, -1)$ and $\delta$ to $(1, 0)$.
	Now a straightforward computation shows that $ \omega(G_2) $ is precisely the set of indices $ (i,j) \in \Z^2 $ for which we defined a module $ L_{i,j} $ in \cref{def:L0 bar}.
	Hence $ \omega(G_2^0) $ is precisely the set of indices for which a non-trivial submodule of $ L $ is depicted in \cref{fig:grading}.
	We conclude that any $ 5 \times 5 $-grading of the form in \cref{fig:grading} may equivalently be regarded as a $ G_2 $-grading.
	For this reason, we will sometimes identify $ G_2^0 $ with the set $ \omega(G_2^0) \subseteq \Z^2 $, so that we have a decomposition
	\[ L = \bigoplus_{\alpha \in G_2^0} L_\alpha. \]
	(This means that we implictly choose a system of simple roots $ \{\gamma, \delta\} $ in $ G_2 $ in such situations.)
\end{remark}

\begin{remark}\label{rem:G2 orbits}
	Let $ W $ be the subgroup of the Weyl group of $ G_2 $ which is generated by the reflections corresponding to long roots. Then there are three orbits of $ G_2 $ under $ W $: One orbit consisting of all six long roots and two orbits consisting of three short roots in each. The first one corresponds to the root spaces of $ L $ parametrized by $ k $ while the other two correspond to the root spaces canonically parametrized by $ J $ or $ J' $, respectively. (For the last statement to be true, we have to consider $J$ and $J'$ as distinct parametrizing spaces even if $J=J'$, which we underline by saying \enquote{canonically parametrized}.)
\end{remark}

\begin{remark}\label{rem:2zetaxi formula}
	The element $2\zeta - \xi$ will appear at several points. It is the unique element $ \varphi \in L_0 $ that satisfies
	\[ \varphi(L_{\pm 2}) = 0 \quad \text{and} \quad \varphi\brackets[\big]{(\lambda, b, b', \mu)_\pm} = (-3\lambda, -b, b', 3\mu)_\pm \]
	for all $ (\lambda, b, b', \mu)_\pm \in L_{\pm 1} $.
\end{remark}

Intuitively, our next goal is to make each of the four parts $L_{-2}$, $L_{-1}$, $L_1$ and~$L_2$ into a module for the Lie algebra $\TKK(J, J')$.
This is not quite possible, so instead, we first make them into a module for the abelian Lie algebras $J'$ and $J$.

\begin{definition}\label{def:Li as J-module}
	\begin{enumerate}
		\item 
			We define $L_{-2}$ and $L_2$ as trivial $J$- and $J'$-modules, i.e.,
			we set $[a,x] = [a,y] = 0$ and $[a',x] = [a',y] = 0$ for all $a \in J$, $a' \in J'$.
		\item 
			To define the action of $J$ and $J'$ on $L_{-1}$ and $L_1$, we set
			\begin{align*}
				\bigl[ a, (\lambda, b, b', \mu)_\pm \bigr] &:= \bigl( 0, \ \lambda a, \ a \times b, \ T(a, b') \bigr)_\pm \ , \\
				\bigl[ a', (\lambda, b, b', \mu)_\pm \bigr] &:= - \bigl( T(b, a'), \ a' \times b', \ \mu a', \ 0 \bigr)_\pm \ ,
			\end{align*}
			for all $a,b \in J$, $a',b' \in J'$ and $\lambda,\mu \in k$.
		\item\label{item:ad}
			Together, this associates to each $a \in J$ and to each $a' \in J'$ an element of $M$, which we will denote by $\ad_a$ and $\ad_{a'}$, respectively. Since the map $J \to M \colon a \mapsto \ad_a$ is clearly injective, we can safely identify $a$ and $\ad_a$ (and similarly for $J'$).
			Notice that the elements $\ad_{a'}$ with $a' \in J'$ have $\zeta$-degree $-1$ and that the elements $\ad_{a}$ with $a \in J$ have $\zeta$-degree $1$.
	\end{enumerate}
\end{definition}
\begin{proposition}\label{pr:J-module}
	\Cref{def:Li as J-module} makes each $L_i \ (i \neq 0)$ into a module for the abelian Lie algebras $J$ and $J'$.
\end{proposition}
\begin{proof}
	This is obvious for the trivial modules $L_{-2}$ and $L_2$.
	We will now verify that $L_1$ is a $J$-module; the remaining three cases are completely similar.
	So let $a,c \in J$ and $z = (\lambda, b, b', \mu)_+ \in L_1$. We compute
	\begin{align*}
		[a,[c,z]]
		&= \bigl[ a, \bigl( 0, \ \lambda c, \ c \times b, \ T(c, b') \bigr)_+ \bigr] \\
		&= \bigl( 0, \ 0, \ \lambda a \times c, \ T(a, c \times b) \bigr)_+ ,
	\end{align*}
	and since $T(a, c \times b) = T(c, a \times b)$ by \cref{le:CNP identities}\cref{le:CNP identities:7}, we get $[a, [c,z]] = [c, [a,z]]$ as required.
\end{proof}
We can rephrase \cref{pr:J-module} as the fact that the maps $J \to M \colon a \mapsto \ad_a$ and $J' \to M \colon a' \mapsto \ad_{a'}$ are homomorphisms of Lie algebras.
(As in \cref{def:Li as J-module}\cref{item:ad}, this now implies that we may view $J$ and $J'$ as Lie subalgebras of $M$.)
Interestingly, this does \emph{not} extend to a Lie algebra homomorphism $\TKK(J, J') \to M$. Instead, we do the following.
\begin{definition}\label{def:dd and L0}
	For each $a \in J$ and $a' \in J'$, we define
	\[ \dd(a,a') \coloneq \dd_{a,a'} \coloneq [\ad_{a'}, \ad_{a}] \in M . \]
	We also set
	\[ Z := L_{0,0} := \langle \dd_{a,a'} \mid a \in J, a' \in J' \rangle + k\zeta + k\xi \leq M \]
	and finally also
	\[ L_0 := J' \oplus Z \oplus J \leq M. \]
	Further, we denote by $ L_{0,-1} $, $ L_{0,0} $ and $ L_{0,1} $ the subspaces $ J' $, $ Z $ and $ J $ of $ L_0 $, respectively.
\end{definition}
\begin{remark}
	Notice that the sum in the definition of $L_0$ is indeed a direct sum.
	Indeed, as we observed in \cref{def:Li as J-module}\cref{item:ad}, the elements of $J'$ have $\zeta$-degree $-1$ and the elements of $J$ have $\zeta$-degree $1$.
	On the other hand, the elements of $Z$ all have $\zeta$-degree $0$.
	However, the sum
	\[ \langle \dd_{a,a'} \mid a \in J, a' \in J' \rangle + k\zeta + k\xi \]
	is never a direct sum (cf. \cref{le:lin comb}).
\end{remark}
\begin{remark}\label{rem:xi zeta formulas}
	Observe that $ \xi $ centralizes $ M $ and that $ \zeta $ centralizes $ Z $. Further, since elements of $ J $ have $ \zeta $-degree $ 1 $, we have $ [\zeta, \ad_a] = \ad_a $ for all $ a \in J $. Similarly, $ [\zeta, \ad_{a'}] = -\ad_{a'} $ for all $ a' \in J' $.
\end{remark}
We first explicitly compute the action of $\dd_{a,a'}$ on each $L_i$.
\begin{proposition}\label{pr:delta z}
	Let $a \in J$ and $a' \in J'$. Then $\dd_{a,a'}(L_{\pm 2}) = 0$ and
	\begin{multline*}
		\bigl[ \dd_{a, a'}, (\lambda, b, b', \mu)_\pm \bigr] \\
		= \bigl( - \lambda T(a, a'), \ D_{a,a'}(b) - T(a, a')b, \ - D_{a',a}(b') + T(a, a')b', \ \mu T(a, a') \bigr)_\pm		
	\end{multline*}
	for all $\lambda,\mu \in k$, $b \in J$, $b' \in J'$.
\end{proposition}
\begin{proof}
	Let $z = (\lambda, b, b', \mu)_\pm$.
	Since $\dd_{a,a'} = \ad_{a'} \ad_{a} - \ad_{a} \ad_{a'}$, we can use \cref{def:Li as J-module} to get
	\begin{align*}
		[\dd_{a,a'}, z]
		&= [a', [a,z]] - [a, [a',z]] \\
		&= \begin{multlined}[t][.8\textwidth]
			- \bigl( \lambda T(a, a'), \ a' \times (a \times b), \ T(a, b')a', \ 0 \bigr)_\pm \\
			+ \bigl( 0,\ T(b, a')a, \ a \times (a' \times b'), \ \mu T(a, a') \bigr)_\pm ,
		\end{multlined}
	\end{align*}
	and using the definition of $D_{a,a'}$ and $D_{a',a}$, we get the required formula.
\end{proof}
\begin{corollary}\label{cor:Z is sublie}
	Let $a,b \in J$ and $a',b' \in J'$.
	Then
	\begin{equation}\label{eq:d bracket}
		[\dd_{a,a'}, \dd_{b,b'}] = \dd_{D_{a,a'}(b), b'} - \dd_{b, D_{a',a}(b')} .
	\end{equation}
	In particular, $Z$ is a Lie subalgebra of $M$.
\end{corollary}
\begin{proof}
	First observe that by \cref{eq:delta bracket}, we have
	\[ [\delta_{a,a'}, \delta_{b,b'}] = \delta_{D_{a,a'}(b), b'} - \delta_{b, D_{a',a}(b')} , \]
	so in particular,
	\[ [D_{a,a'}, D_{b,b'}] = D_{D_{a,a'}(b), b'} - D_{b, D_{a',a}(b')} . \]
	By \cref{pr:delta z}, on the other hand, we have
	\[ \bigl[ [\dd_{a,a'}, \dd_{b,b'}], (\lambda, c, c', \mu)_\pm \bigr] = \bigl( 0, \ [D_{a,a'}, D_{b,b'}](c), \ - [D_{a',a}, D_{b',b}](c'), \ 0 \bigr)_\pm . \]
	Since
	\begin{equation}\label{eq:der T explicit}
		T(D_{a,a'}(b), b') - T(b, D_{a',a}(b')) = 0
	\end{equation}
	by \cref{eq:der T}, we conclude that both sides of \cref{eq:d bracket} have the same action on $L_{\pm 1}$.
	Since they both act trivially on $L_{\pm 2}$, they have the same action on each $L_i$ ($i \neq 0$), and hence they represent the same element of $M$.
	This proves \cref{eq:d bracket} and shows, in particular, that $\langle \dd_{a,a'} \mid a \in J, a' \in J' \rangle$ is a Lie subalgebra of $M$.
	Together with \cref{rem:xi zeta formulas}, it follows that $Z$ is indeed a Lie subalgebra of $M$.
\end{proof}

\begin{proposition}\label{pr:Jd}
	Let $ a,c \in J $ and $ a',c' \in J' $. Then
	\begin{equation*}
		\sqbrackets{\dd_{a,a'}, \ad_c} = \ad_{D_{a,a'}(c)} \quad \text{and} \quad \sqbrackets{\dd_{a,a'}, \ad_{c'}} = -\ad_{D_{a',a}(c')}.
	\end{equation*}
\end{proposition}
\begin{proof}
	We compute the action of $\sqbrackets{\dd_{a,a'}, \ad_c}$ on an arbitrary element $ (\lambda, b, b', \mu)_\pm \in L_{\pm 1} $:
	\begin{align*}
		&\sqbrackets[\big]{\dd_{a,a'}, \sqbrackets{c, \brackets{\lambda, b, b', \mu}_\pm}} \\
		&= \sqbrackets[\big]{\dd_{a,a'}, \brackets[\big]{0, \lambda c, c \times b, T(c,b')}_\pm} \\
		&= \brackets[\big]{0, \lambda \brackets[\big]{D_{a,a'}(c) - T(a,a') c}, -D_{a',a}(c \times b) + T(a,a')(c \times b), T(c,b') T(a,a')}_\pm, \\[1.5ex]
		&\sqbrackets[\big]{c, \sqbrackets{\dd_{a,a'}, \brackets{\lambda, b, b', \mu}_\pm}} \\
		&= \sqbrackets[\big]{c, \brackets[\big]{-\lambda T(a,a'), D_{a,a'}(b) - T(a,a') b, -D_{a',a}(b') + T(a,a')b', \mu T(a,a')}_\pm} \\
		&= \brackets[\big]{0, -\lambda T(a,a')c, c \times \brackets[\big]{D_{a,a'}(b) - T(a,a') b}, T\brackets[\big]{c, -D_{a',a}(b') + T(a,a') b'}}_\pm.
	\end{align*}
	Using \cref{eq:D is der,eq:der T explicit}, we conclude that
	\begin{align*}
		\sqbrackets[\big]{\sqbrackets{\dd_{a,a'}, \ad_c},\ (\lambda, b, b', \mu)_\pm}
			&= \brackets[\big]{0,\ \lambda D_{a,a'}(c),\ D_{a,a'}(c) \times b,\ T\brackets{D_{a, a'}(c), b'}} \\
			&= \sqbrackets[\big]{D_{a,a'}(c), \brackets{\lambda, b, b', \mu}_\pm} .
	\end{align*}
	This shows that $\sqbrackets{\dd_{a,a'}, \ad_c}$ and $\ad_{D_{a,a'}(c)}$ have the same action on $L_1$ and $L_{-1}$.
	Since both elements also act trivially on $L_2$ and $L_{-2}$, they represent the same element of $M$, i.e., they coincide. The proof of the second assertion is completely similar.
\end{proof}

\begin{proposition}\label{pr:L0 is sublie}
	$L_0$ is a Lie subalgebra of $M$.
\end{proposition}
\begin{proof}
	By \cref{cor:Z is sublie}, we already know that $Z$ is a Lie subalgebra of $M$.
	Further, $ [J, k\zeta + k \xi] \subseteq L_0 $ and $ [J', k \zeta + k\xi] \subseteq L_0 $ by \cref{rem:xi zeta formulas}.
	Finally, by \cref{pr:Jd}, $ [\dd_{a,a'}, J] \subseteq L_0 $ and $ [\dd_{a,a'}, J'] \subseteq L_0 $ for all $a \in J$, $a' \in J'$.
\end{proof}
\begin{remark}\label{rem:Li is L0-module}
	Notice that since each $L_i$ is an $M$-module, it is automatically also an $L_0$-module, now that we know that $L_0$ is a Lie subalgebra of $M$.
\end{remark}

We will need the following lemma in the proof of \cref{pr:L is Lie algebra} below.
At the same time, it may serve as a good example to explain the subtle difference between the $\delta_{a,a'} \in \TKK(J, J')$ and the $\dd_{a,a'} \in M$ that we have just defined.
\begin{lemma}\label{le:triple}
	Let $a,b,c \in J$ and $ a',b',c' \in J' $.
	Then the following hold.
	\begin{enumerate}
		\item \label{le:triple:aasharp}$\dd_{a, a^\sharp} = N(a) \cdot (2 \zeta - \xi)$.
		\item $\dd_{(a')^\sharp, a} = N(a') \cdot (2 \zeta - \xi)$.
		\item \label{le:triple:norm r}$\dd_{a, b^\sharp} + \dd_{b, a \times b} = T(a, b^\sharp) \cdot (2 \zeta - \xi)$.
		\item $\dd_{(a')^\sharp, b} + \dd_{a' \times b', a'} = T((a')^\sharp, b) \cdot (2 \zeta - \xi)$
		\item \label{le:triple:r}$\dd_{a, b \times c} + \dd_{b, c \times a} + \dd_{c, a \times b} = T(c, a \times b) \cdot (2 \zeta - \xi)$.
		\item \label{le:triple:l}$\dd_{a' \times b', c'} + \dd_{b' \times c',a'} + \dd_{c' \times a', b'} = T(c', a' \times b') \cdot (2\zeta - \xi)$.
	\end{enumerate}
\end{lemma}
\begin{proof}
	Notice that these are identities in $L_0$, and by construction, two such elements coincide if and only if they have the same action on $L_i$ for all $i \neq 0$.
	Since each $\dd_{a,a'}$ acts trivially on $L_2$ and $L_{-2}$ and the grading derivation $2\zeta - \xi \in M$ kills both $L_2$ and $L_{-2}$, it suffices to verify this for $i = \pm 1$.
	
	Recall from \cref{le:triple D}\cref{le:triple D:1} that
	\[
		D_{a, a^\sharp}(d) = 2 N(a) d .
	\]
	Together with \cref{pr:delta z}, this yields, using $T(a, a^\sharp) = 3N(a)$, that
	\[
		[\dd_{a, a^\sharp}, \ (\lambda, d, d', \mu)_+]
		= \bigl( -3 \lambda N(a), \ - N(a) d, \ N(a) d', \ 3 \mu N(a) \bigr)_+.
	\]
	Using \cref{rem:2zetaxi formula}, we conclude that both sides of \cref{le:triple:aasharp} have the same action on $L_1$. The verification for~$L_{-1}$ is similar, so \cref{le:triple:aasharp} holds.
	
	The proofs of the remaining identities follow in a completely similar way from the other identities in \cref{le:triple D}.
%
\end{proof}

\begin{remark}\label{rem:TKK no action}
	In contrast, notice that in $\TKK(J, J')$, we have
	\[ \delta_{a, a^\sharp} = 2 N(a) \cdot \gamma , \]
	by \cref{le:triple D}\cref{le:triple D:1,le:triple D:2}.
	(See \cref{def:der}\cref{def:der:gamma} and \cref{le:inner}.)
	Comparing this with \cref{le:triple} shows, in particular, that the action introduced in \cref{def:Li as J-module} does not extend to an action of $\TKK(J, J')$ on $L_i$, as we mentioned just before \cref{def:dd and L0}.
	See, however, \cref{co:central ext} below.
\end{remark}

The appearance of a multiple of $2 \zeta - \xi$ in \cref{le:triple} is not a coincidence.
In the following lemma and its proof, all summations are understood to run over $a \in J$ and $a' \in J'$.

\begin{lemma}\label{le:lin comb}
	Let $f \in L_{0,0}$ be arbitrary and write
	\begin{equation}\label{eq:lin comb}
		f=\sum \lambda_{a,a'} \dd_{a,a'} + \lambda_\zeta \zeta + \lambda_\xi \xi
	\end{equation}
	for certain scalars $\lambda_{a,a'}, \lambda_\zeta, \lambda_\xi \in k$. Then $f=0$ if and only if the following assertions are satisfied.
	\begin{enumerate}
		\item $\lambda_\zeta + 2\lambda_\xi = 0$.
		
		\item $\textstyle\sum \lambda_{a,a'} T(a,a') = 3 \lambda_\xi $.
		
		\item $\textstyle\sum \lambda_{a,a'} D_{a,a'} = 2\lambda_\xi \id_J $ and $ \textstyle\sum \lambda_{a,a'} D_{a',a} = 2\lambda_\xi \id_{J'}$.
	\end{enumerate}
	In this situation, we have
	\[ \sum \lambda_{a,a'} \dd_{a,a'} = \lambda_\xi \cdot (2 \zeta - \xi) \quad \text{and} \quad \sum \lambda_{a,a'} \delta_{a,a'} = \lambda_\xi \cdot 2 \gamma. \]
\end{lemma}
\begin{proof}
	Let $t := \textstyle\sum \lambda_{a,a'} T(a,a')$. Then
	\begin{align*}
		[f, (1,0,0,0)_+] &= (-t + \lambda_\xi - \lambda_\zeta, 0, 0, 0)_+, \\
		[f, (1,0,0,0)_-] &= (-t -\lambda_\xi -2\lambda_\zeta), \\
		[f, (0,b,0,0)_+] &= \brackets*{0, \textstyle\sum \lambda_{a,a'} D_{a,a'}(b) + (-t + \lambda_\xi) b, 0, 0}_+, \\
		[f, (0,b,0,0)_-] &= \brackets*{0, \textstyle\sum \lambda_{a,a'} D_{a,a'}(b) + (-t - \lambda_\xi - \lambda_\zeta)b}_-, \\
		[f, (0,0,b',0)_+] &= \brackets*{- \textstyle\sum \lambda_{a,a'} D_{a',a}(b') + (t + \lambda_\xi + \lambda_\zeta) b', 0}_+, \\
		[f, (0,0,b',0)_-] &= \brackets*{- \textstyle\sum \lambda_{a,a'} D_{a',a}(b') + (t - \lambda_\xi)b', 0}_-, \\
		[f, (0,0,0,1)_+] &= (0, 0, 0, t + \lambda_\xi + 2\lambda_\zeta), \\
		[f, (0,0,0,1)_-] &= (0,0,0, t - \lambda_\xi + \lambda_\zeta), \\
		[f, x] &= (-2\lambda_\xi - \lambda_\zeta)x, \\
		[f, y] &= (2\lambda_\xi + \lambda_\zeta)y.
	\end{align*}
	Since $f=0$ if and only if all these expressions are zero, the equivalence in the assertion follows. The remaining equations are then a straightforward consequence, using that $\gamma = (\id_J, -\id_{J'})$ and $\delta_{a,a'} = (D_{a,a'}, -D_{a',a})$.
\end{proof}

\begin{corollary}\label{co:central ext}
	The linear map 
	\[ \pi \colon Z \to \TKK_0(J,J') \colon \sum \lambda_{a,a'} \dd_{a,a'} + \lambda_\zeta \zeta + \lambda_\xi \xi \mapsto \sum \lambda_{a,a'} \delta_{a,a'} + \lambda_\zeta \gamma  \]
	is well-defined. It is an epimorphism of Lie algebras, which extends to an epimorphism $L_0 \to \TKK(J, J')$.
	In particular, $L_0$ is a central extension of $\TKK(J, J')$.
\end{corollary}
\begin{proof}
	If $\sum \lambda_{a,a'} \dd_{a,a'} + \lambda_\zeta \zeta + \lambda_\xi \xi = 0$, then also $\sum \lambda_{a,a'} \delta_{a,a'} + \lambda_\zeta \gamma = 0$ by \cref{le:lin comb}, so $\pi$ is well-defined. Further, since $\xi, \zeta$ are central in $Z$ and $\gamma$ is central in $\TKK_0(J,J')$, it follows from the comparison of \cref{eq:delta bracket,cor:Z is sublie} that $\pi$ is a homomorphism of Lie algebras. Clearly, $\pi$ is surjective. Since $L_0 = J \oplus Z \oplus J'$ and $\TKK(J,J') = J \oplus \TKK_0(J,J') \oplus J'$, we can extend $\pi$ to an epimorphism $L_0 \to \TKK(J,J')$ of $k$-modules in a natural way. This defines an epimorphism of Lie algebras because for all $a,b \in J$, $a',b' \in J'$, we have $[\dd_{a,a'}, b] = D_{a,a'}(b)$ and $[\dd_{a,a'},b'] = -D_{a',a}(b')$ in $L_0$ by \cref{pr:Jd} and $[\delta_{a,a'},b] = (\delta_{a,a'})_+(b) = D_{a,a'}(b)$ and $[\delta_{a,a'}, b'] = (\delta_{a,a'})_-(b') = -D_{a',a}(b')$ in $\TKK(J,J')$ by \cref{eq:TKK rules}.
	
	It remains to show that the kernel of $\pi$ is central. Let
	\[ f = \sum \lambda_{a,a'} \dd_{a,a'} + \lambda_\xi \xi + \lambda_\zeta \zeta \]
	denote an arbitrary element of the kernel and let $c \in J$, $c' \in J'$ be arbitrary. Then
	\begin{align*}
		[f, \dd_{c,c'}] &= \sqbrackets*{\sum \lambda_{a,a'} \dd_{a,a'}, \dd_{c,c'}} = \sum \lambda_{a,a'} \brackets[\big]{\dd(D_{a,a'}(c), c') - \dd(c, D_{a',a}(c'))} \\
		&= \dd\brackets[\big]{\sum \lambda_{a,a'} D_{a,a'}(c), c'} - \dd\brackets[\big]{c, \sum \lambda_{a,a'} D_{a',a}(c')}.
	\end{align*}
	Since $f$ lies in the kernel of $\pi$, we have $\sum \lambda_{a,a'} \delta_{a,a'} + \lambda_\zeta \gamma = 0$ or, in other words,
	\[ \sum \lambda_{a,a'} D_{a,a'} = -\lambda_\zeta \id_J \quad \text{and} \quad \sum \lambda_{a,a'} D_{a',a} = -\lambda_\zeta \id_{J'}. \]
	Hence
	\begin{align*}
		[f, \dd_{c,c'}] = \dd(-\lambda_\zeta c, c') - \dd(c, -\lambda_\zeta c') = 0.
	\end{align*}
	Further, $[f,\xi] = 0 = [f, \zeta]$. Thus the kernel of $\pi$ is indeed central.
\end{proof}

We are now ready to make the $k$-module
\begin{equation}\label{eq:def L}
	L := L_{-2} \oplus L_{-1} \oplus L_0 \oplus L_1 \oplus L_2
\end{equation}
into a Lie algebra.
It remains to describe the Lie bracket between elements of $L_i$ and $L_j$ for $i,j \neq 0$, followed by a somewhat lengthy verification that this does indeed define a Lie algebra.
Before we proceed, we also fix the second grading on $L$ that was long visible on \cref{fig:grading}.
\begin{definition}\label{def:second grading}
	For each $j \in \Z$, we define
	\[ L'_j := \bigoplus_i L_{i,j} , \]
	where the $L_{i,j}$ are as in \cref{def:L0 bar,def:dd and L0}, with the convention that $L_{i,j} = 0$ if it has not been defined.
	In particular, $L$ now admits two decompositions
	\begin{align*}
		L &= L_{-2} \oplus L_{-1} \oplus L_0 \oplus L_1 \oplus L_2 \quad \text{and} \\
		L &= L'_{-2} \oplus L'_{-1} \oplus L'_0 \oplus L'_1 \oplus L'_2 .
	\end{align*}
	(Observe that both $L_i = 0$ and $L'_i = 0$ whenever $|i| > 2$.)
	
	Notice that $[\xi, l] = il$ for all $l \in L_i$ and $[\zeta, l] = jl$ for all $l \in L'_j$.
	However, we cannot use these identities to \emph{define} $L_i$ and $L'_j$ in general because our base ring $k$ is arbitrary (so e.g. we might have $2 = 0$ in $k$).
\end{definition}

\begin{definition}\label{def:L bracket}
	Let
	\begin{align*}
		[x, (\lambda, b, b', \mu)_+] &:= -(\lambda, b, b', \mu)_-\,, \\
		[y, (\lambda, b, b', \mu)_-] &:= (\lambda, b, b', \mu)_+\,, \\
		[x, y] &:= \xi , \\
		\bigl[ (\lambda, b, b', \mu)_-, \ (\nu, c, c', \rho)_- \bigr] &:= \bigl( T(b, c') - T(c, b') + \mu \nu - \lambda \rho \bigr) x, \\
		\bigl[ (\lambda, b, b', \mu)_+, \ (\nu, c, c', \rho)_+ \bigr] &:= \bigl( T(b, c') - T(c, b') + \mu \nu - \lambda \rho \bigr) y,
	\end{align*}
	for all $\lambda,\mu,\nu,\rho \in k$, $a,b \in J$ and $a',b' \in J'$.
	Of course, we define $[L_i, L_j] := 0$ whenever $|i+j| > 2$.
	It only remains to define $[L_{-1}, L_1]$, which is the most complicated case:
	\begin{multline}\label{eq:L-1 L1}
		\bigl[ (\lambda, b, b', \mu)_-, \ (\nu, c, c', \rho)_+ \bigr] \\*
		\shoveleft \quad := \bigl( - \lambda c' + b \times c - \nu b' \bigr) \\
		+ \bigl( (\lambda \rho - T(b, c')) \cdot \zeta + (T(c, b') - \mu \nu) \cdot (\xi - \zeta) + \dd_{b, c'} + \dd_{c, b'}  \bigr) \\
		+ \bigl( - \rho b + b' \times c' - \mu c \bigr) .
	\end{multline}
\end{definition}

We first observe that all these definitions respect both gradings.
\begin{lemma}\label{le:gradings are preserved}
	Let $l \in L_i \cap L'_{i'}$ and $m \in L_j \cap L'_{j'}$.
	Then $[l,m] \in L_{i+j} \cap L'_{i' + j'}$.
\end{lemma}
\begin{proof}
	This is a straightforward verification, going through the different cases, using \cref{def:Li as J-module,def:dd and L0,def:L bracket}.
\end{proof}

We have summarized the complete definition of the Lie bracket on $L = L_{-2} \oplus L_{-1} \oplus L_0 \oplus L_1 \oplus L_2$ in \cref{ta:Lie}.

\begin{sidewaystable}
\vspace{70ex}
\resizebox{\textwidth}{!}{
	
\renewcommand{\arraystretch}{2.5}

$
\begin{array}{c||c|c|ccccc|c|c|}
	[ \,\downarrow\,, \rightarrow\,] & x & (\nu, c, c', \rho)_- & c' & \dd_{e,e'} & \xi & \zeta & c & (\nu, c, c', \rho)_+ & y \\
	\hline \hline
	x & 0 & 0 & 0 & 0 & 2x & x & 0 & -(\nu, c, c', \rho)_- & \xi \\
	\hline
	(\lambda, b, b', \mu)_- & 0 & \bigl( T(b, c') - T(c, b') + \mu \nu - \lambda \rho \bigr) x & \bigl( T(b, c'), \ c' \times b', \ \mu c', \ 0 \bigr)_- &
		\makecell*[l]{\bigl( \lambda T(e, e'), \\[.3ex] \quad - D_{e,e'}(b) + T(e, e')b, \\[.3ex] \quad  D_{e',e}(b') - T(e, e')b', \\[.3ex] \hspace*{12ex} -\mu T(e, e') \bigr)_-}
		& (\lambda, b, b', \mu)_- & (2\lambda, b, 0, -\mu)_- & -\bigl( 0, \ \lambda c, \ c \times b, \ T(c, b') \bigr)_- &
		\makecell*[l]{\ \bigl( - \lambda c' + b \times c - \nu b' \bigr) \\[.8ex]
		\quad {} + (\lambda \rho - T(b, c')) \cdot \zeta \\
		\quad {} + (T(c, b') - \mu \nu) \cdot (\xi - \zeta) \\
		\quad {} + \dd_{b, c'} + \dd_{c, b'} \\[.6ex]
		\hspace*{9ex} {} + \bigl( - \rho b + b' \times c' - \mu c \bigr) }
		& -(\lambda, b, b', \mu)_+  \\
	\hline
	b' & 0 & - \bigl( T(c, b'), \ b' \times c', \ \rho b', \ 0 \bigr)_- & 0 & D_{e',e}(b') & 0 & b' & \dd_{c, b'} & - \bigl( T(c, b'), \ b' \times c', \ \rho b', \ 0 \bigr)_+ & 0\\
	\dd_{a,a'} & 0 &
			\makecell*[l]{\bigl( - \nu T(a, a'), \\[.3ex] \quad D_{a,a'}(c) - T(a, a')c, \\[.3ex] \quad - D_{a',a}(c') + T(a, a')c', \\[.3ex] \hspace*{15ex} \rho T(a, a') \bigr)_-}
		& -D_{a', a}(c') & \dd_{D_{a,a'}(e), e'} - \dd_{e, D_{a',a}(e')} & 0 & 0 & D_{a,a'}(c) &
			\makecell*[l]{\bigl( - \nu T(a, a'), \\[.3ex] \quad D_{a,a'}(c) - T(a, a')c, \\[.3ex] \quad - D_{a',a}(c') + T(a, a')c', \\[.3ex] \hspace*{15ex} \rho T(a, a') \bigr)_+}
		& 0 \\
	\xi & -2x & -(\nu, c, c', \rho)_- & 0 & 0 & 0 & 0 & 0 & (\nu, c, c', \rho)_+ & 2y \\
	\zeta & -x & (-2 \nu, - c, 0, \rho)_- & -c' & 0 & 0 & 0 & c & (-\nu, 0, c', 2 \rho)_+ & y \\
	b & 0 & \bigl( 0, \ \nu b, \ b \times c, \ T(b, c') \bigr)_- & - \dd_{b, c'} & -D_{e,e'}(b) & 0 & -b & 0 & \bigl( 0, \ \nu b, \ b \times c, \ T(b, c') \bigr)_+ & 0\\
	\hline
	(\lambda, b, b', \mu)_+ & (\lambda, b, b', \mu)_- &
		\makecell*[l]{\ \bigl( \lambda c' - b \times c + \nu b' \bigr) \\[.8ex]
		\quad {} + (T(c, b') - \mu \nu) \cdot \zeta \\
		\quad {} + (\lambda \rho - T(b, c')) \cdot (\xi - \zeta) \\
		\quad {} - \dd_{b, c'} - \dd_{c, b'} \\[.6ex]
		\hspace*{9ex} {}+ \bigl( \rho b - b' \times c' + \mu c \bigr) }
		& \bigl( T(b, c'), \ c' \times b', \ \mu c', \ 0 \bigr)_+ &
		\makecell*[l]{\bigl( \lambda T(e, e'), \\[.3ex] \quad - D_{e,e'}(b) + T(e, e')b, \\[.3ex] \quad D_{e',e}(b') - T(e, e')b', \\[.3ex] \hspace*{12ex} -\mu T(e, e') \bigr)_+}

		& -(\lambda, b, b', \mu)_+ & (\lambda, 0, - b', -2 \mu)_+ & -\bigl( 0, \ \lambda c, \ c \times b, \ T(c, b') \bigr)_+ & \bigl( T(b, c') - T(c, b') + \mu \nu - \lambda \rho \bigr) y & 0 \\
	\hline
	y & -\xi & (\nu, c, c', \rho)_+ & 0 & 0 & -2y & -y & 0 & 0 & 0 \\
	\hline
\end{array}
$

}
\bigskip
\caption{The Lie bracket on $L = L_{-2} \oplus L_{-1} \oplus L_0 \oplus L_1 \oplus L_2$}\label{ta:Lie}
\end{sidewaystable}

\subsection{Verification of the Jacobi identity}

To verify the Jacobi identity on $L$, we use a little trick to simplify some computations.
The following auxiliary definition is a useful shorthand.
\begin{definition}
	Let $a,b,c \in L$, where $L$ is a $k$-module equipped with a $k$-bilinear skew-symmetric operation $[\cdot, \cdot]$.
	We write $\Jac(a,b,c)$ if the Jacobi identity holds for the triple $(a,b,c)$, i.e.,
	\[ [[a,b], c] + [[b,c], a] + [[c,a], b] = 0 . \]
	Notice that by the skew-symmetry, $\Jac(a,b,c)$ implies $\Jac(x,y,z)$ for any permutation $(x,y,z)$ of $(a,b,c)$.
	Similarly, if $A$, $B$ and $C$ are subsets (usually subspaces) of $L$, we write $\Jac(A,B,C)$ if $\Jac(a,b,c)$ for all $a \in A$, $b \in B$ and $c \in C$. We also write, for example, $\Jac(a,B,C)$ as a shorthand for $\Jac(\{ a \}, B, C)$ when $a \in L$ and $B$ and $C$ are subsets of $L$.
\end{definition}
\begin{lemma}\label{le:Jacobi trick}
	Let $L$ be a $k$-algebra with a $\Z$-grading $L = \oplus_{i \in \Z} L_i$, denote the multiplication in $L$ by $[ \cdot, \cdot ]$
	and assume that it is skew-symmetric.
	Fix $i,j \in \Z$ and $l,m \in L_0$.
	Assume that $\Jac(l, L_i, L_j)$, $\Jac(m, L_i, L_j)$ and $\Jac(l, m, L_i \cup L_j \cup L_{i+j})$.
	
	Then also $\Jac([l,m], L_i, L_j)$.
\end{lemma}
\begin{proof}
	This is a straighforward verification, using the fact that when $z \in L_i$, then also $[l,z], [m,z] \in L_i$, because of the grading.
\end{proof}

\begin{remark}\label{rem:grading}
	By \cref{def:second grading}, we have
	\[ [\xi, \ell] = i\ell \quad \text{and} \quad [\zeta, \ell] = j\ell \quad \text{whenever } \ell \in L_i \cap L'_j . \]
	As a consequence, the Jacobi identity will be automatically satisfied for triples involving $\xi$ or $\zeta$.
		More precisely, if $\ell \in L_{i} \cap L'_{j}$ and $\ell' \in L_{i'} \cap L'_{j'}$, then by \cref{le:gradings are preserved}, we have $[\ell, \ell'] \in L_{i + i'} \cap L'_{j + j'}$, so we have
		\begin{align*}
			[[\zeta, \ell], \ell'] + [[\ell, \ell'], \zeta] + [[\ell', \zeta], \ell]
			&= [j\ell, \ell'] - [\zeta, [\ell, \ell']] + [\ell, j'\ell'] \\
			&= j [\ell, \ell'] - (j + j') [\ell, \ell'] + j' [\ell, \ell'] = 0 ,
		\end{align*}
		and similarly for $\xi$.	
\end{remark}

\begin{proposition}\label{pr:L is Lie algebra}
	The algebra $L = L_{-2} \oplus L_{-1} \oplus L_0 \oplus L_1 \oplus L_2$ equipped with the Lie brackets from \cref{def:Li as J-module,def:dd and L0,def:L bracket} is a (5-graded) Lie algebra.
\end{proposition}
\begin{proof}
	We have to check the Jacobi identity for all triples of elements in $L$.
	By linearity, we may of course assume that each of these elements is homogeneous, i.e., belongs to some $L_i$.
	It is clear from our definitions that the proof for triples in $L_i$, $L_{i'}$, $L_{i''}$ can be translated \emph{mutatis mutandis} to a proof for triples in $L_{-i}$, $L_{-i'}$, $L_{-i''}$.
	In addition, the Jacobi identity is trivially satisfied if $|i + i' + i''| > 2$, and by \cref{rem:Li is L0-module}, we already know that it is satisfied if at least two of the three indices are $0$.
	Recall also that by \cref{rem:skew-symmetric}, the Lie bracket is skew-symmetric.
	That leaves us with the following nine cases for $(i, i', i'')$:
	\begin{align*}
		&&&(0, 1, 1), && (0, 1, -1), && (0, 1, -2), && (0, 2, -2), &&\\
		&&& (1, 1, -1), && (1, 1, -2), &&(1, 2, -1), && (1, 2, -2), \\
		&&&(2, 2, -2).
	\end{align*}
	Unavoidably, we will deal with these different cases one by one.
	In each case, we will assume that $\ell \in L_i$, $\ell' \in L_{i'}$ and $\ell'' \in L_{i''}$ and our goal is to show the Jacobi identity for the triple $(\ell, \ell', \ell'')$.
	
	Before we proceed with our case-by-case analysis, we point out that whenever we want to check the Jacobi identity for all triples $(\ell, \ell', \ell'')$ with $\ell \in L_0$, it suffices to consider $\ell \in J \cup J'$. Indeed, the case $\ell = \dd_{a,a'}$ then follows from \cref{le:Jacobi trick}, and the cases $\ell = \zeta$ and $\ell = \xi$ are automatically satisfied by \cref{rem:grading}.
		In addition, since the cases $\ell \in J$ and $\ell \in J'$ are completely similar in each case, we will always assume that $\ell = a \in J$ whenever $i=0$.
	
	\begin{description}\setlength{\itemsep}{1.5ex}
        \item[Case $(0,1,1)$]
        	Write $\ell = a$, $\ell' = (\lambda, b, b', \mu)_+$ and $\ell'' = (\nu, c, c', \rho)_+$.
        	Notice that $[\ell', \ell''] \in L_2$,
        	so $[a, [\ell', \ell'']] = 0$.
        	We have to show that $[[a, \ell'], \ell''] = [[a, \ell''], \ell']$.
        	We compute
        	\begin{align*}
        		[[a, \ell'], \ell'']
        		&= \bigl[ [a, (\lambda, b, b', \mu)_+], (\nu, c, c', \rho)_+ \bigr] \\
        		&= \bigl[ \bigl( 0, \ \lambda a, \ a \times b, \ T(a, b') \bigr)_+, (\nu, c, c', \rho)_+ \bigr] \\
        		&= \bigl( \lambda T(a, c') - T(c, a\times b) + \nu T(a, b') \bigr) \cdot y,
        	\end{align*}
        	and since $T(c, a\times b) = T(b, a \times c)$, this is indeed symmetric in $\ell'$ and $\ell''$.
        \item[Case $(0,1,-1)$]
        	This is the most interesting case.
        	Write $\ell = a$, $\ell' = (\nu, c, c', \rho)_+$ and $\ell'' = (\lambda, b, b', \mu)_-$.
        	We compute
        	\begin{align*}
        		&[[a, \ell''], \ell'] = \bigl[ [a, (\lambda, b, b', \mu)_-], (\nu, c, c', \rho)_+ \bigr] \\
        		&= \bigl[ \bigl( 0, \ \lambda a, \ a \times b, \ T(a, b') \bigr)_-, (\nu, c, c', \rho)_+ \bigr] \\
        		&= (\lambda a \times c - \nu a \times b) \\
	        		&\hspace*{6ex} + \bigl( - \lambda T(a, c') \cdot \zeta + (T(c, a \times b) - \nu T(a, b')) \cdot (\xi - \zeta) + \lambda \dd_{a,c'} + \dd_{c, a \times b} \bigr) \\
	        		&\hspace*{6ex} + \bigl( - \lambda \rho a + (a \times b) \times c' - T(a,b')c \bigr) ; \\[1.5ex]
        		&[[\ell', a], \ell''] = [\ell'', [a, \ell']] = \bigl[ (\lambda, b, b', \mu)_-, [a, (\nu, c, c', \rho)_+] \bigr] \\
        		&= \bigl[ (\lambda, b, b', \mu)_-, \bigl( 0, \ \nu a, \ a \times c, \ T(a, c') \bigr)_+ \bigr] \\
        		&= (-\lambda a \times c + \nu a \times b) \\
	        		&\hspace*{6ex} + \bigl( (\lambda T(a, c') - T(b, a \times c)) \cdot \zeta + \nu T(a, b') \cdot (\xi - \zeta) + \dd_{b,a \times c} + \nu \dd_{a, b'} \bigr) \\
	        		&\hspace*{6ex} + \bigl( - T(a, c') b + (a \times c) \times b' - \mu \nu a \bigr) ; \\[1ex]
	        	&[[\ell'', \ell'], a] = \Bigl[
	        		\bigl( - \lambda c' + b \times c - \nu b' \bigr) \\
					&\hspace*{16ex}+ \bigl( (\lambda \rho - T(b, c')) \cdot \zeta + (T(c, b') - \mu \nu) \cdot (\xi - \zeta) + \dd_{b, c'} + \dd_{c, b'}  \bigr) \\
					&\hspace*{16ex}+ \bigl( - \rho b + b' \times c' - \mu c \bigr), \ a
					\Bigr] \\
				&= (-\lambda \dd_{a, c'} + \dd_{a, b \times c} - \nu \dd_{a, b'}) \\
				&\hspace*{6ex} + \bigl( ( \lambda \rho + \mu \nu - T(b, c') - T(c, b')) a + D_{b, c'}(a) + D_{c, b'}(a)  \bigr) \\
				&= (-\lambda \dd_{a, c'} + \dd_{a, b \times c} - \nu \dd_{a, b'}) \\
				&\hspace*{6ex} + \bigl( (\lambda \rho + \mu \nu) a +  T(a, c')b + T(a, b')c - (a \times b) \times c' - (a \times c) \times b' \bigr) . \\[-2ex]
        	\end{align*}
        	The Jacobi identity $[[a, \ell''], \ell'] + [[\ell', a], \ell''] + [[\ell'', \ell'], a] = 0$ now follows from \cref{le:triple}.
        \item[Case $(0,1,-2)$]
        	Write $\ell = a$, $\ell' = (\lambda, b, b', \mu)_+$. Without loss of generality,  $\ell'' = x$.
        	Notice that $[a, x] = 0$, and we quickly verify that
        	\begin{align*}
				&[a, [x, (\lambda, b, b', \mu)_+]] - [x, [a, (\lambda, b, b', \mu)_+]] \\
				&= - [a, (\lambda, b, b', \mu)_-] - [x, \bigl( 0, \ \lambda a, \ a \times b, \ T(a, b') \bigr)_+] \\
				&= - \bigl( 0, \ \lambda a, \ a \times b, \ T(a, b') \bigr)_- + \bigl( 0, \ \lambda a, \ a \times b, \ T(a, b') \bigr)_- = 0.
        	\end{align*}
        \item[Case $(0,2,-2)$]
        	Write $\ell = a$. Without loss of generality, $\ell' = y$ and $\ell'' = x$.
        	We have $[a, x] = 0$ and $[a, y] = 0$. Moreover, $[[x,y], a] = [\xi, a] = 0$.
        \item[Case $(1,1,-1)$]
        	This is the lengthiest computation, but there are no surprises.
        	Write $\ell = (\nu_1, c_1, c_1', \rho_1)_+$, $\ell' = (\nu_2, c_2, c_2', \rho_2)_+$ and $\ell'' = (\lambda, b, b', \mu)_-$.
        	First, we have
        	\begin{align*}
        		[\ell'', [\ell, \ell']]
        		&= \bigl[ (\lambda, b, b', \mu)_-, \bigl( T(c_1, c_2') - T(c_2, c_1') + \rho_1 \nu_2 - \nu_1 \rho_2 \bigr) \cdot y \bigr] \\
        		&= \bigl( - T(c_1, c_2') + T(c_2, c_1') - \rho_1 \nu_2 + \nu_1 \rho_2 \bigr) \cdot (\lambda, b, b', \mu)_+ .
        	\end{align*}
        	To compute $[[\ell'', \ell], \ell']$, we use \cref{eq:L-1 L1,def:Li as J-module}, and after working out all the different terms, we get
        	\begin{align*}
        		&[[\ell'', \ell], \ell'] \\*
        		&= \Bigl(
        		\begin{multlined}[t][.94\textwidth]
	        		\lambda T(c_2, c_1') - \lambda T(c_2, b \times c_1) + \nu_1 T(c_2, b') - \lambda \rho_1 \nu_2 + \nu_2 T(c_1, b') - 2 \mu \nu_1 \nu_2, \\[1ex]
	        		\shoveleft{\lambda c_1' \times c_2' - (b \times c_1) \times c_2' + \nu_1 b' \times c_2' - \nu_2 \rho_1 b + \nu_2 b' \times c_1' - \nu_2 \mu c_1 + T(c_1, b') c_2} \\
	        		\shoveright{- \mu \nu_1 c_2 - c_1' \times (b \times c_2) + T(c_2, c_1')b - (c_1 \times c_2) \times b' + T(c_2, b') c_1,} \\[1ex]
	        		\shoveleft{\lambda \rho_2 c_1' - \rho_2 b \times c_1 + \nu_1 \rho_2 b' - \rho_1 b \times c_2 + (b' \times c_1') \times c_2 - \mu c_1 \times c_2 + \lambda \rho_1 c_2'} \\
	        		\shoveright{- T(b, c_1') c_2' - T(b, c_2') c_1' + b \times (c_1' \times c_2') - T(c_1, c_2') b' + c_1 \times (b' \times c_2'),} \\[0ex]
	        		-\rho_1 T(b, c_2') + T(b' \times c_1', c_2') - \mu T(c_1, c_2') + 2 \rho_1 \rho_2 - T(b, c_1') \rho_2 + \mu \nu_1 \rho_2
        		\Bigr)_+ \, .
        		\end{multlined}
        	\end{align*}
        	The expression for $[[\ell'', \ell'], \ell]$ is, of course, exactly the same but with all subscripts $1$ and $2$ interchanged.
        	It is now pleasant to observe that when we compute the difference $[[\ell'', \ell], \ell'] - [[\ell'', \ell'], \ell]$, many terms cancel out, and what is left is precisely the expression for $[\ell'', [\ell, \ell']]$ that we have computed above.
        \item[Case $(1,1,-2)$]
        	Write $\ell = (\lambda, b, b', \mu)_+$ and $\ell' = (\nu, c, c', \rho)_+$, and assume without loss of generality that $\ell'' = x$.
        	We have
        	\begin{align*}
        		[[x, \ell], \ell']
        		&= \bigl[ [x, (\lambda, b, b', \mu)_+], (\nu, c, c', \rho)_+ \bigr] \\
        		&= - \bigl[ (\lambda, b, b', \mu)_-, (\nu, c, c', \rho)_+ \bigr] ,
        	\end{align*}
        	so we deduce from \cref{eq:L-1 L1} that
        	\begin{align*}
        		&[[x, \ell], \ell'] - [[x, \ell'], \ell] \\
        		&= \bigl[ (\nu, c, c', \rho)_-, (\lambda, b, b', \mu)_+ \bigr] - \bigl[ (\lambda, b, b', \mu)_-, (\nu, c, c', \rho)_+ \bigr] \\
        		&= \bigl( T(b, c') - T(c, b') + \mu \nu - \lambda \rho \bigr) \cdot \xi .
        	\end{align*}
        	On the other hand,
        	\begin{align*}
        		[x, [\ell, \ell']]
        		&= \bigl[ x, \bigl( T(b, c') - T(c, b') + \mu \nu - \lambda \rho \bigr) y \bigr] \\
        		&= \bigl( T(b, c') - T(c, b') + \mu \nu - \lambda \rho \bigr) \cdot [x,y] ,
        	\end{align*}
        	and we conclude that indeed $[x, [\ell, \ell']] = [[x, \ell], \ell'] - [[x, \ell'], \ell]$.
        \item[Case $(1,2,-1)$]
        	Write $\ell = (\nu, c, c', \rho)_+$ and $\ell'' = (\lambda, b, b', \mu)_-$, and assume without loss of generality that $\ell' = y$.
        	Notice that $[\ell, y] = 0$.
        	Since $[y, J] = [y, J'] = [y, [J', J]] = 0$, we have
        	\begin{align*}
        		[y, [\ell'', \ell]]
        		&= \bigl[ y, (\lambda \rho - T(b, c')) \cdot \zeta + (T(c, b') - \mu \nu) \cdot (\xi - \zeta) \bigr] \\
        		&= \bigl( - \lambda \rho + T(b, c') - T(c, b') + \mu \nu \bigr) \cdot y .
        	\end{align*}
        	On the other hand, we also have
        	\[
        		[[y, \ell''], \ell] = \bigl[ (\lambda, b, b', \mu)_+, (\nu, c, c', \rho)_+ \bigr] = \bigl( T(b, c') - T(c, b') + \mu \nu - \lambda \rho \bigr) y .
        	\]
        \item[Case $(1,2,-2)$]
        	Write $\ell = (\lambda, b, b', \mu)_+$ and assume without loss of generality that $\ell' = y$ and $\ell'' = x$.
        	Notice that $[\ell, y] = 0$ and that $[x, y] = \xi$. The result now follows from the facts that $[\xi, \ell] = \ell$ and $[y, [x, \ell]] = - \ell$.
        \item[Case $(2,2,-2)$]
        	Without loss of generality, $\ell = \ell' = y$ and $\ell'' = x$, so the Jacobi identity is trivially satisfied.
        \qedhere
	\end{description}
\end{proof}

%

\begin{proposition}\label{pr:G2-graded}
	The decomposition
	\[ L = \bigoplus_{\alpha \in G_2^0} L_\alpha \]
	from \cref{rem:G2 grid} is a $G_2$-grading of the Lie algebra $ L $.
\end{proposition}
\begin{proof}
	We can read off from \cref{ta:Lie} that $ [L_\alpha, L_\beta] \le L_{\alpha+\beta} $ for all $ \alpha, \beta \in G_2^0 $.
\end{proof}

\begin{remark}
	In the case that $ J = 0 = J' $, there are only six non-trivial root spaces in $ L $, and they constitute an $ A_2 $-grading of $ L $.
\end{remark}

\begin{theorem}\label{thm:G2 Lie}
	Let $(J,J')$ be a cubic norm pair over a commutative ring $k$. Then there exists a $G_2$-graded Lie algebra $L$ such that $L_\alpha \cong (k,+)$ for all long roots $\alpha \in G_2$, $L_\beta \cong (J,+)$ for three short roots $\beta \in G_2$ which lie in the same of the three orbits of $G_2$ in \cref{rem:G2 orbits} and $L_\gamma \cong (J',+)$ for the remaining three short roots $\gamma$.
\end{theorem}
\begin{proof}
	As a module, $L$ was constructed in \cref{def:L0 bar,def:dd and L0,eq:def L}, and its root spaces have the desired isomorphism type. The Lie bracket was constructed in \cref{def:L0 bar,def:Li as J-module,def:L bracket} and is summarized in \cref{ta:Lie}. We showed in \cref{pr:L is Lie algebra,pr:G2-graded} that $L$ is a $G_2$-graded Lie algebra.
\end{proof}

\begin{remark}
	In \cite{Faulkner1971}, John Faulkner constructs Lie algebras from ``ternary algebras'' over arbitrary commutative base rings.
	Cubic norm structures provide examples of those, and the resulting Lie algebras are probably isomorphic to the Lie algebras that we have constructed (but we have not explicitly carried out this verification).
	
	In any case, we need our more ``refined'' description to obtain the $G_2$-grading (and we want to refine it even further to an $F_4$-grading later).
	In fact, our approach is more direct for cubic norm structures, because Faulkner uses the triple product---see \cite[(1.6) on p.\@~399]{Faulkner1971}---which in itself is already quite complicated. On the other hand, we recognize his bilinear form (1.5) directly in our formulas in \cref{def:L bracket}.
\end{remark}

\section{Functoriality and reflections}\label{sec:reflections}

In the previous section, we have constructed a Lie algebra $ L=L(J,J') $ for every cubic norm pair $ (J,J') $ (in fact, for every linear cubic norm pair). The current section is devoted to the construction of several interesting Lie algebra homomorphisms that are defined on $ L(J,J') $. We begin by establishing that $ L(J,J') $ is functorial in $ (J,J') $: For any surjective homomorphism $ (\iota, \iotainv) \colon (J,J') \to (J_2, J_2') $ of cubic norm pairs, we construct a surjective homomorphism $ L(\iota, \iotainv) \colon L(J,J') \to L(J_2, J_2') $ of Lie algebras that is bijective if and only if $ (\iota, \iotainv) $ is (\cref{le:reflections auto}). Here the surjectivity of $ (\iota, \iotainv) $ is a subtle requirement which ensures that $ L(\iota, \iotainv) $ is well-defined (see \cref{rem:refl wd}), and which is also present in the construction of (the functorial homomorphisms defined on) the Tits-Kantor-Koecher algebra of a Jordan pair (see \cite[\S 7.1]{LoosNeher}). We can even phrase our construction in a slightly more general way: Every surjective $ t $-homotopy $ (\iota, \iotainv) $ between cubic norm pairs induces a surjective homomorphism $ L(\iota, \iotainv, t) $ of Lie algebras. A corollary of this observation is that the isomorphism type of $ L(J,J') $ depends only on the isotopy type of $ (J,J') $. Another application of the functoriality of $ L(J,J') $ will be given in \cref{sec:simple} (in particular, \cref{rem:ideal quot}), where we investigate quotients of $ L(J,J') $ that are induced by \enquote{ideals of $ (k,J,J') $}.

We also define automorphisms $ \phihor $, $ \phiver $, $ \phisl $ of $ L $ that we call \enquote{reflections} because they permute the root spaces of $ L(J,J') $ in a way that is prescribed by a reflection of the root system $ G_2 $. Each of these automorphisms corresponds to one of the $ 5 $-gradings of $ L(J,J') $ described in \cref{rem:all gradings} in the sense that it reverses the grading: It maps the $ i $-graded part to the $ (-i) $-graded part. We can also define a fourth reflection $ \phibs $ which reverses the $ 7 $-grading of $ L(J,J') $, but only under the assumption that we are given an isomorphism (or, more generally, a $ t $-isotopy) $ (\iota, \iotainv) \colon (J,J') \to (J',J) $. This assumption is certainly necessary because a reflection that reverses the $ 7 $-grading necessarily maps root spaces parametrized by $ J $ isomorphically to root spaces parametrized by $ J' $, and vice versa. In the general situation, we can only construct an isomorphism $ \psibs \colon L(J,J') \to L(J',J) $, and $ \phibs $ is simply the composition of $ \psibs $ with $ L(\iota, \iotainv, t) $.

In \cref{sec:G2 exp}, we will construct root groups $ (U_\alpha)_{\alpha \in G_2} $ in $ \Aut(L) $. This would, if done naively, involve a large number of lengthy computations. The reflections described in the previous paragraph allow us to reduce the number of computations significantly by exploiting the symmetries of $ L $. For details, see \cref{pr:exp conj aut summary}. Moreover, we will show in \cref{le:phihor weyl,le:phibs weyl} that $ \phihor $ and $ \phibs $ are Weyl elements with respect to $ (U_\alpha)_{\alpha \in G_2} $ (under suitable assumptions in the case of $ \phibs $). This is one important step in proving that $ (U_\alpha)_{\alpha \in G_2} $ is a $ G_2 $-grading of a subgroup of $ \Aut(L) $.

\begin{notation}
	For any linear cubic norm pair $(J,J')$, we denote by $L(J,J')$ the corresponding Lie algebra constructed in \cref{pr:L is Lie algebra}. If no linear cubic norm pair is specified, we assume that $(J,J')$ is a fixed linear cubic norm pair and that $L \coloneq L(J,J')$ is the corresponding Lie algebra. Further, we will occasionally use the convention from \cref{rem:G2 grid} to identify $ G_2^0 $ with a subset of $ \Z^2 $.
\end{notation}

\begin{convention}
	We will always let automorphism of $ L $ act on $ L $ from the left.
\end{convention}

\begin{remark}\label{rem:refl wd}
	We will define all homomorphisms on $L$ by describing their action on all components. Since it is in general not true that the sum
	\[ \langle \dd_{a,a'} \mid a  \in J, a' \in J' \rangle + k\zeta + k\xi \]
	is direct, it is a priori not clear that these maps are well-defined on $ L_{0,0} $. This will be verified in \cref{le:reflections well-defined}.
\end{remark}

We begin with the functoriality of $ L(J,J') $ with respect to surjective homotopies.

\begin{definition}\label{def:L functorial}
	Let $(J_1, J_1')$ and $(J_2, J_2')$ be two linear cubic norm pairs (over the same ring $k$), let $ t \in k $ be invertible and let $(\iota, \iotainv) \colon (J_1, J_1') \to (J_2, J_2')$ be a surjective $ t $-homotopy of linear cubic norm pairs. Then we define a map
	\[ L(\iota,\iotainv, t) \colon L(J_1, J_1') \to L(J_2, J_2'), \]
	denoted by $\psi$ in the following set of equations, by setting
	\begin{align*}
		\psi(x) &\coloneq t^{-1}x, \\
		\psi((\nu, c, c', \rho)_-) &\coloneq (t^{-1}\nu, t^{-1}c^\iota, (c')^\iotainv, \rho)_-, \\
		\psi(c') &\coloneq (c')^\iotainv, \\
		\psi(\dd_{c,c'}) &\coloneq \dd_{c^\iota, (c')^\iotainv}, \quad \psi(\xi) \coloneq \xi, \quad \psi(\zeta) \coloneq \zeta, \\
		\psi(c) &\coloneq c^\iota, \\
		\psi((\nu, c, c', \rho)_+) &\coloneq (\nu, c^\iota, t(c')^\iotainv, t\rho)_+ \\
		\psi(y) &\coloneq ty
	\end{align*}
	for all $\nu, \rho \in k$, $c \in J$, $c' \in J'$. If $(\iota, \iotainv)$ is a homomorphism (i.e., if $ t=1 $), we simply put $L(\iota, \iotainv) \coloneq L(\iota, \iotainv, 1)$.
\end{definition}

\begin{remark}\label{rem:functor inverse}
	If $(\iota, \iotainv) \colon (J,J') \to (J', J)$ is a $t$-involution for some invertible $t \in k$, then $L(\iota, \iotainv, t)^{-1} = L(\iotainv, \iota, t^{-1})$.
\end{remark}

\begin{remark}
	The definition of $L(\iota, \iotainv, t)$ is not canonical: We made a choice \enquote{where to put $t$ and $t^{-1}$}. Our choice corresponds to the $5$-grading $(L''_i)_{-2 \le i \le 2}$ from \cref{rem:all gradings} in the sense that exactly the blocks of type $J$, $J'$ that lie in $L_{0}''$ and the blocks of type $k$ that lie in $L_{-1}'' + L_1''$ are multiplied by $t$ or $t^{-1}$. In a similar way, we could for example define a map $\psi \coloneq L'(\iota, \iotainv, t) \colon L(J_1, J_1') \to L(J_2, J_2')$ that corresponds to the $5$-grading $(L_i)_{-2 \le i \le 2}$ by
	\begin{align*}
		\psi(x) &\coloneq x, \\
		\psi((\nu, c, c', \rho)_\pm) &\coloneq (t\nu, c^\iota, (c')^\iotainv, t^{-1}\rho)_\pm, \\
		\psi(c') &\coloneq t(c')^\iotainv, \\
		\psi(\dd_{c,c'}) &\coloneq \dd_{c^\iota, (c')^\iotainv}, \quad \psi(\xi) \coloneq \xi, \quad \psi(\zeta) \coloneq \zeta, \\
		\psi(c) &\coloneq t^{-1}c^\iota, \\
		\psi(y) &\coloneq y
	\end{align*}
	for all $\nu, \rho \in k$, $c \in J$, $c' \in J'$. Note that $L(\iota,\iotainv,t) = L'(\iota,\iotainv, t)$ if $t=1$, that is, if $(\iota, \iotainv)$ is a homomorphism. Further, the two maps agree on $L_{0,0}$ for arbitrary $t$. The latter fact implies that the proof of well-definedness in \cref{le:reflections well-defined} is also valid for $L'(\iota, \iotainv, t)$.
\end{remark}

Before we define the automorphism $ \phihor $, $ \phiver $, $ \phisl $, $ \phibs $, we formalize what we mean by \enquote{reflections}.

\begin{definition}\label{def:refl abstr}
	Let $ \Phi $ be a root system, let $ \alpha \in \Phi $ and let
	\[ \bar{L} = \bigoplus_{\beta \in \Phi^0} \bar{L}_\beta \quad \text{and} \quad \bar{L}' = \bigoplus_{\beta \in \Phi^0} \bar{L}_\beta' \]
	be two $ \Phi $-graded Lie algebras. An \emph{$ \alpha $-reflection from $ \bar{L} $ to $ \bar{L}' $} is an isomorphism $ \varphi \colon \bar{L} \to \bar{L}' $ such that $ \varphi(\bar{L}_\beta) = \bar{L}'_{\beta^{\reflbr{\alpha}}} $ for all $ \beta \in \Phi^0 $ where $ \refl{\alpha} \colon \Phi^0 \to \Phi^0 $ is the reflection along $ \alpha^\perp $. The root $ \alpha $ will also sometimes be called the \emph{type of $ \varphi $}. If $ \bar{L}' = \bar{L} $, we will simply call $ \varphi $ a \emph{$ \alpha $-reflection on $ \bar{L} $}.
\end{definition}

The defining formulas for the following reflections on $ L $ have been obtained by applying the formulas in \cite[Lemma 2.8]{DMM25} to an \enquote{extremal element} with respect to the grading that should be reversed.

\begin{definition}\label{def:reflections}
	We define endomorphisms $\phihor$ and $\phisl$ of the $k$-module $L$ by
	\begin{align*}
		\phihor(x) &\coloneq y , \\
		\phihor((\nu, c, c', \rho)_-) &\coloneq (\nu, c, c', \rho)_+ , \\
		\phihor(c') &\coloneq c' , \\
		\phihor(\dd_{c,c'}) &\coloneq \dd_{c,c'} , \quad
		\phihor(\xi) \coloneq -\xi , \quad
		\phihor(\zeta) \coloneq \zeta - \xi , \\
		\phihor(c) &\coloneq c , \\
		\phihor((\nu, c, c', \rho)_+) &\coloneq -(\nu, c, c', \rho)_- , \\
		\phihor(y) &\coloneq x ,
	\end{align*}
	and
	\begin{align*}
		\phisl(x) &\coloneq (1,0,0,0)_- , \\
		\phisl((\nu, c, c', \rho)_-) &\coloneq -\nu x + (0,c,0,0)_- + c' - (\rho,0,0,0)_+ , \\
		\phisl(c') &\coloneq -(0,0,c',0)_- , \\
		\phisl(\dd_{c,c'}) &\coloneq \dd_{c,c'} + T(c,c') (\xi - \zeta) , \quad
		\phisl(\xi) \coloneq \zeta , \quad
		\phisl(\zeta) \coloneq \xi , \\
		\phisl(c) &\coloneq -(0,c,0,0)_+ , \\
		\phisl((\nu, c, c', \rho)_+) &\coloneq -(0,0,0,\nu)_- + c + (0,0,c',0)_+ - \rho y , \\
		\phisl(y) &\coloneq (0,0,0,1)_+ .
	\end{align*}
	for all $\nu, \rho \in k$, $c \in J$ and $c' \in J'$.
\end{definition}

\begin{remark}\label{rem:reflection bs}
	In a similar way, we can define a third map $\phiver$ by
	\begin{align*}
		\phiver(x) &\coloneq (0,0,0,1)_- , \\
		\phiver((\nu, c, c', \rho)_-) &\coloneq -\rho x + (0,0,c',0)_- + c + (0,0,0,\nu)_+ , \\
		\phiver(c') &\coloneq (0,0,c',0)_+ , \\
		\phiver(\dd_{c,c'}) &\coloneq \dd_{c,c'} - T(c,c') \zeta , \quad
		\phiver(\xi) \coloneq \xi - \zeta , \quad
		\phiver(\zeta) \coloneq - \zeta , \\
		\phiver(c) &\coloneq -(0,c,0,0)_- , \\
		\phiver((\nu, c, c', \rho)_+) &\coloneq (0,0,0,\rho)_- - c' + (0,c,0,0)_+ + \nu y , \\
		\phiver(y) &\coloneq (-1,0,0,0)_+ .
	\end{align*}
	We introduce it here for completeness, but we will never use it in our arguments in a meaningful way. For this reason, we will not verify in detail that it is indeed compatible with the Lie bracket, as we do for $\phihor$ and $\phisl$ in \cref{le:reflections auto}.
\end{remark}

We will also need the following reflection from $ L(J,J') $ to $ L(J',J) $ in \cref{def:phibs,le:exp J' psi conj}.

\begin{definition}\label{def:psibs}
	We define a map $ \psibs \coloneq \psibs^{(J,J')} \colon L(J, J') \to L(J', J)$ by setting
	\begin{align*}
		\psibs(x) &\coloneq (0,0,0,1)_+, \\
		\psibs((\nu, c, c', \rho)_-) &\coloneq -\nu y + (0, 0, -c, 0)_+ + c' + (0,0,0,\rho)_-, \\
		\psibs(c') &\coloneq \brackets[\big]{0, -c', 0, 0}_+, \\
		\psibs(\dd_{c,c'}) &\coloneq -\dd_{c',c} + T(c,c') (\zeta - \xi), \quad \psibs(\xi) \coloneq -\zeta, \quad \psibs(\zeta) \coloneq -\xi, \\
		\psibs(c) &\coloneq (0,0,-c, 0)_-, \\
		(\nu, c, c', \rho)_+ &\coloneq (\nu, 0, 0, 0)_+ + c + \brackets[\big]{0, -c', 0, 0}_- - \rho x, \\
		\psibs(y) &\coloneq (1,0,0,0)_-
	\end{align*}
	for all $\nu, \rho \in k$, $c \in J$, $c' \in J'$.
\end{definition}

\begin{remark}\label{rem:reflections on subspaces}
	It is immediate from \cref{def:psibs,def:reflections,rem:reflection bs} that
	\begin{align*}
		\phihor(L_\alpha) = L_{\alpha^{\reflbr{(2,1)}}}, \quad \phiver(L_\alpha) &= L_{\alpha^{\reflbr{(1,2)}}}, \quad \phisl(L_\alpha) = L_{\alpha^{\reflbr{(-1,1)}}}, \\
		\psibs^{(J,J')}\brackets[\big]{L(J,J')_\alpha} &= L(J',J)_{\alpha^{\reflbr{(1,1)}}}
	\end{align*}
	for all $ \alpha \in G_2^0 $. Thus to show that these maps are reflections, it only remains to show that they are isomorphisms of Lie algebras.
\end{remark}

\begin{remark}\label{rem:reflection grading}
	It follows from \cref{rem:reflections on subspaces} that
	\[ \phihor(L_i) = L_{-i}, \qquad \phiver(L_i') = L_{-i}', \qquad \phisl(L_i'') = L_{-i}'' \]
	for all $ i \in \{-2, \ldots, 2\} $. This means that the reflections $\phihor$, $\phiver$ and $\phisl$ correspond to the $ 5 $-gradings $ (L_i)_{-2 \le i \le 2} $, $ (L_i')_{-2 \le i \le 2} $ and $ (L_i'')_{-2 \le i \le 2} $
	 defined in \cref{rem:all gradings}, respectively. Similarly, for the $ 7 $-gradings $(L(J,J')^\backslasharrow_i)_{-3 \le i \le 3}$ of $ L(J,J') $ and $(L(J',J)^\backslasharrow_i)_{-3 \le i \le 3}$ of $ L(J',J) $ from \cref{rem:all gradings}, we have
	\[ \psibs^{(J,J')}\brackets[\big]{L(J,J')^\backslasharrow_i} = L(J',J)^\backslasharrow_{-i} \]
	for all $i \in \{-3, \ldots, 3\}$.
\end{remark}

\begin{remark}\label{rem:reflection parity}
	For all $ \varphi \in \{\phihor, \phiver, \phisl, \psibs\} $, the map $\varphi^2$ is the parity automorphism of the grading to which $ \varphi $ corresponds. In particular, $ \phihor, \phiver, \phisl $ have order~$4$ if $2 \neq 0$ in $k$ and order~$2$ if $2 = 0$ in $k$.
\end{remark}

We can now prove that all maps above are actually well-defined and homomorphisms of Lie algebras.

\begin{lemma}\label{le:reflections well-defined}
	\begin{enumerate}
		\item For any linear cubic norm pair $(J,J')$, the reflections $\phihor$, $\phisl$, $\phiver$ on $L(J,J')$ and the map $\psibs \colon L(J,J') \to L(J',J)$ are well-defined.
		
		\item For any invertible $ t \in k $ and any surjective $ t $-homotopy $(\iota, \iotainv) \colon (J,J') \to (J_2,J_2')$ of linear cubic norm pairs, the map $L(\iota,\iotainv, t) \colon L(J,J') \to L(J_2,J_2')$ is well-defined. 
	\end{enumerate}
\end{lemma}
\begin{proof}
	In this proof, all summations are understood to run over $a \in J$ and $a' \in J'$. Suppose that
	\[ \sum \lambda_{a,a'} \dd_{a,a'} + \lambda_\zeta \zeta + \lambda_\xi \xi = 0 \]
	for certain, finitely many scalars $\lambda_{a,a'}, \lambda_\zeta, \lambda_\xi \in k$. We have to show this sum evaluates to zero under the respective assignments in \cref{def:L functorial,def:psibs,def:reflections,rem:reflection bs}. We know from \cref{le:lin comb} that
	\begin{align}
		\sum \lambda_{a,a'} \dd_{a,a'} &= \lambda_\xi \cdot (2\zeta - \xi), \label{eq:refl wd lincomb} \\
		\sum \lambda_{a,a'} T(a,a') &= 3\lambda_\xi, \label{eq:refl wd sum Taa} \\
		\sum \lambda_{a,a'} D_{a,a'} &= 2\lambda_\xi \id_J, \qquad \sum \lambda_{a,a'} D_{a',a} = 2\lambda_\xi \id_{J'}. \label{eq:refl wd sum Daa}
	\end{align}
	Both sides of \eqref{eq:refl wd lincomb} are fixed by the assignments defining $\phihor$, so $\phihor$ is well-defined. Under the assignments defining $\phisl$, they evaluate to
	\begin{gather*}
		\sum \lambda_{a,a'} \dd_{a,a'} + \sum \lambda_{a,a'} T(a,a') (\xi - \zeta)  \quad \text{and} \\
		\lambda_\xi \cdot (2\xi - \zeta) = \lambda_\xi (2\zeta - \xi) + 3\lambda_\xi (\xi-\zeta).
	\end{gather*}
	By~\eqref{eq:refl wd sum Taa}, we infer that $\phisl$ is well-defined. Similarly, the two sides of~\eqref{eq:refl wd lincomb} are mapped to
	\begin{gather*}
		\sum \lambda_{a,a'} \dd_{a,a'} - \sum \lambda_{a,a'} T(a,a') \zeta = \sum \lambda_{a,a'} \dd_{a,a'} - 3\lambda_\xi \zeta \quad \text{and} \\
		\lambda_\xi \cdot (-2\zeta - \xi + \zeta) = \lambda_\xi \cdot (2\zeta-\xi) - 3\lambda_\xi \zeta
	\end{gather*}
	under the assignments defining $\phiver$, respectively. We infer that $\phiver$ is also well-defined.
	
	For the proof that $\psibs$ is well-defined, we have to distinguish between the elements $\xi,\zeta, \dd_{a,a'}$ of $L(J,J')$ and the elements $\xi', \zeta', \dd_{a',a}$ of $L(J',J)$. The two sides of~\eqref{eq:refl wd lincomb} are mapped to
	\begin{gather*}
		-\sum \lambda_{a,a'} \dd_{a', a} + \sum \lambda_{a,a'} T(a,a') (\zeta'-\xi') = -\sum \lambda_{a,a'} \dd_{a', a} + 3\lambda_\xi (\zeta'-\xi'), \\
		\lambda_\xi \cdot (-2\xi' + \zeta') = -\lambda_\xi \cdot (2\zeta'-\xi') + 3\lambda_\xi (\zeta'-\xi')
	\end{gather*}
	under the assignments defining $\psibs$. By \cref{pr:delta z}, \eqref{eq:refl wd sum Taa} and~\eqref{eq:refl wd sum Daa}, we have
	\[ \sum \lambda_{a,a'} [\dd_{a',a}, (\lambda, b', b, \mu)_\pm] = \lambda_\xi (-3\lambda, -b', b, 3\mu)_\pm \]
	for all $(\lambda, b', b, \mu)_\pm \in L(J',J)_{\pm 1}$ and $\dd_{a',a}(L(J',J)_{\pm 2}) = 0$, so it follows from \cref{rem:2zetaxi formula} that $\sum \lambda_{a,a'} \dd_{a',a} = \lambda_\xi (2\zeta' - \xi')$. We infer that $\psibs$ is well-defined.
	
	Now assume that $(\iota, \iotainv) \colon (J,J') \to (J_2,J_2')$ is a surjective $ t $-homotopy of linear cubic norm pairs for some invertible $ t \in k $. Similarly as in the previous paragraph, we have to distinguish between the elements $\xi,\zeta, \dd_{a,a'}$ of $L(J,J')$ and the elements $\xi_2, \zeta_2, \dd_{a^\iota, (a')^{\iotainv}}$ of $L(J_2,J_2')$. Now the two sides of~\eqref{eq:refl wd lincomb} are mapped to
	\begin{gather*}
		\sum \lambda_{a,a'} \dd_{a^\iota, (a')^{\iotainv}} \quad \text{and} \quad \lambda_\xi (2\zeta_2 - \xi_2)
	\end{gather*}
	under the assignments defining $L(\iota, \iotainv, t)$. It follows from \eqref{eq:refl wd sum Daa} and \cref{rem:CNP hom jordan} that
	\begin{align*}
		\brackets*{\sum \lambda_{a,a'} D_{a^\iota, (a')^\iotainv}}(b^\iota) &= \brackets*{\sum \lambda_{a,a'} D_{a,a'}(b)}^\iota = 2\lambda_\xi b^\iota
	\end{align*}
	for all $b \in J$. Since $\iota$ is surjective, we infer that $\sum \lambda_{a,a'} D_{a^\iota, (a')^{\iotainv}} = 2\lambda_\xi \id_{J_2}$. By a similar computation, $\sum \lambda_{a,a'} D_{(a')^\iotainv, a^\iota} = 2\lambda_\xi \id_{J_2'}$. Further, \eqref{eq:refl wd sum Taa} implies that $\sum \lambda_{a,a'} T(a^\iota, (a')^\iotainv) = 3\lambda_\xi$. Thus by similar arguments as in the previous paragraph, it follows from \cref{pr:delta z,rem:2zetaxi formula} that $\sum \lambda_{a,a'} \dd_{a^\iota, (a')^\iotainv} = \lambda_\xi (2\zeta_2 - \xi_2)$. Hence $L(\iota, \iotainv, t)$ is well-defined.
%
\end{proof}

\begin{lemma}\label{le:reflections auto}
	\begin{enumerate}
		\item \label{le:reflections auto:refl}For any linear cubic norm pair $(J,J')$, the maps $\phihor$, $\phiver$, $\phisl$ on $L(J,J')$ and the map $\psibs \colon L(J,J') \to L(J',J)$ are isomorphisms of Lie algebras. In particular, $\phihor$, $\phiver$ and $\phisl$ are reflections of type $ (-2,-1) $, $ (-1, -2) $ and $ (-1, 1) $, respectively.
		
		\item \label{le:reflections auto:homotopy}For any invertible $ t \in k $ and any surjective $ t $-homotopy $(\iota, \iotainv) \colon (J,J') \to (J_2,J_2')$ of linear cubic norm pairs, the map $L(\iota,\iotainv,t) \colon L(J,J') \to L(J_2,J_2')$ is a surjective homomorphism of Lie algebras. It is bijective if and only if $(\iota, \iotainv)$ is bijective.
	\end{enumerate}
\end{lemma}
\begin{proof}
	Let $\varphi \in \{\phihor, \phisl, \phiver, \psibs\}$ or $\varphi = L(\iota, \iotainv,t)$ for a surjective $ t $-homotopy $(\iota, \iotainv) \colon (J,J') \to (J_2,J_2')$ of linear cubic norm pairs and some invertible $ t \in k $. The assertion about the surjectivity or bijectivity of $\varphi$ are clear (see also \cref{rem:reflections on subspaces}), so it remains to show that it is also compatible with the Lie bracket. For this it suffices to show that $\varphi([a,b]) = [\varphi(a), \varphi(b)]$ for all $a,b$ in a generating set of the $k$-module $L$. By the anti-symmetry of the Lie bracket, we only have to check this for all $a<b$ where $<$ is an arbitrarily chosen total order on the generating set. The results of these computations for the \enquote{standard} generating set can be found in \cref{ta:phihor auto} for $\varphi = \phihor$, in \cref{ta:phisl auto} for $\varphi = \phisl$ and in \cref{ta:psibs auto} for $\varphi = \psibs$. The majority of computations are straightforward applications of the formulas in \cref{ta:Lie}, but we also have to use \cref{le:inner} to prove the relation $\varphi([\dd_{a,a'}, \dd_{e,e'}]) = [\varphi(\dd_{a,a'}), \varphi(\dd_{e,e'})]$ for $\varphi \in \{\phisl, \phibs\}$. A similar computation can be done for $\phiver$, but we omit the details (see \cref{rem:reflection bs}). For $\varphi = L(\iota, \iotainv,t)$, it is a straightforward verification that it is compatible with the Lie bracket.
	
	The final assertion in \cref{le:reflections auto:refl} follows from \cref{rem:reflection grading}.
\end{proof}

As announced, the functoriality of our construction has the following first application.

\begin{proposition}\label{pr:lie isotopy}
	The isomorphism type of $L(J,J')$ depends only on the isotopy type of $(J,J')$.
\end{proposition}
\begin{proof}
	Let $ (\iota, \iotainv) \colon (J,J') \to (J_2, J_2') $ be a $ t $-isotopy of linear cubic norm pairs for some invertible $ t \in k $. Then the map $ L(\iota, \iotainv, t) \colon L(J,J') \to L(J_2, J_2') $ is an isomorphism of Lie algebras by \cref{le:reflections auto}\cref{le:reflections auto:homotopy}.
\end{proof}

\begin{remark}\label{rem:reflections inv}
	We state explicit formulas for the inverses of $\phihor$ and $\phisl$ for future reference: We have
	\begin{align*}
		\phihor^{-1}(x) &= y , \\
		\phihor^{-1}((\nu, c, c', \rho)_-) &= -(\nu, c, c', \rho)_+ , \\
		\phihor^{-1}(c') &= c' , \\
		\phihor^{-1}(\dd_{c,c'}) &= \dd_{c,c'} , \quad
		\phihor^{-1}(\xi) = -\xi , \quad
		\phihor^{-1}(\zeta) = \zeta - \xi , \\
		\phihor^{-1}(c) &= c , \\
		\phihor^{-1}((\nu, c, c', \rho)_+) &= (\nu, c, c', \rho)_- , \\
		\phihor^{-1}(y) &= x,
	\end{align*}
	and
	\begin{align*}
		\phisl^{-1}(x) &= (-1,0,0,0)_- , \\
		\phisl^{-1}((\nu, c, c', \rho)_-) &= \nu x + (0,c,0,0)_- - c' - (\rho,0,0,0)_+ , \\
		\phisl^{-1}(c') &= (0,0,c',0)_- , \\
		\phisl^{-1}(\dd_{c,c'}) &= \dd_{c,c'} + T(c,c') (\xi - \zeta) , \quad
		\phisl^{-1}(\xi) = \zeta , \quad
		\phisl^{-1}(\zeta) = \xi , \\
		\phisl^{-1}(c) &= (0,c,0,0)_+ , \\
		\phisl^{-1}((\nu, c, c', \rho)_+) &= (0,0,0,-\nu)_- - c + (0,0,c',0)_+ + \rho y , \\
		\phisl^{-1}(y) &= (0,0,0,-1)_+ .
	\end{align*}
	for all $\nu, \rho \in k$, $c \in J$ and $c' \in J'$.
\end{remark}

We now define the final reflection $ \phibs $. Since it is a composition of maps defined earlier, we see immediately that it is a reflection.

\begin{definition}\label{def:phibs}
	Let $t \in k$ be invertible and assume that $(\iota, \iotainv) \colon (J, J') \to (J', J)$ is a $t$-isotopy of linear cubic norm pairs.
	Then we define a fourth reflection $\phibs^{(\iota, \iotainv,t)} \coloneq \psibs^{(J',J)} \circ L(\iota, \iotainv,t)$, which is a $ (-1,-1) $-reflection on $ L(J,J') $.
\end{definition}

\begin{remark}\label{rem:phibs formulas}
	Let $t \in k$ be invertible, let $(\iota, \iotainv) \colon (J, J') \to (J', J)$ be a $t$-isotopy of linear cubic norm pairs and put $ \phibs \coloneq \phibs^{(\iota, \iotainv,t)} $. Explicitly, we have
	\begin{align*}
		\phibs(x) &= (0,0,0,t^{-1})_+, \\
		\phibs((\nu, c, c', \rho)_-) &= -\nu t^{-1} y + (0, 0, -t^{-1}c^\iota, 0)_+ + (c')^\iotainv + (0,0,0,\rho)_-, \\
		\phibs(c') &= \brackets[\big]{0, -(c')^\iotainv, 0, 0}_+, \\
		\phibs(\dd_{c,c'}) &= -\dd_{(c')^\iotainv,c^\iota} + T(c,c') (\zeta - \xi), \quad \phibs(\xi) = -\zeta, \quad \phibs(\zeta) = -\xi, \\
		\phibs(c) &= (0,0,-c^\iota, 0)_-, \\
		\phibs(\nu, c, c', \rho)_+ &= (\nu, 0, 0, 0)_+ + c^\iota + \brackets[\big]{0, -t(c')^\iotainv, 0, 0}_- - t\rho x, \\
		\phibs(y) &= (t,0,0,0)_-
	\end{align*}
	for all $\nu, \rho \in k$, $c \in J$, $c' \in J'$. Now assume in addition that $ \iota^{-1} = \iotainv $. Then the inverse of $\phibs$ is given by
	\begin{align*}
		\phibs^{-1}(x) &= (0,0,0,-t^{-1})_+, \\
		\phibs^{-1}\brackets[\big]{(\nu,c,c',\rho)_-} &= t^{-1}\nu y + (0,0, -t^{-1}c^\iota,0)_+ - (c')^\iotainv + (0,0,0,\rho)_-, \\
		\phibs^{-1}(c') &= \brackets[\big]{0, (c')^\iotainv, 0, 0}_+, \\
		\phibs^{-1}(\dd_{c,c'}) &= -\dd_{(c')^\iotainv, c^\iota} + T(c,c')(\zeta - \xi), \quad \phibs^{-1}(\xi) = -\zeta, \quad \phibs^{-1}(\zeta) = -\xi, \\
		\phibs^{-1}(c) &= (0,0,c^\iota, 0)_-, \\
		\phibs^{-1}\brackets[\big]{(\nu, c, c',\rho)_+} &= (\nu, 0, 0, 0)_+ - c^\iota + \brackets[\big]{0, -t(c')^\iotainv, 0, 0}_- +t \rho x, \\
		\phibs^{-1}(y) &= (-t, 0, 0,0)_-
	\end{align*}
	for all $\nu, \rho \in k$, $c \in J$, $c' \in J'$. Further, $\phibs^2$ is the parity automorphism of the $7$-grading of $L$ defined in \cref{rem:reflection grading}, which coincides with the parity automorphism $\phisl^2$ of the $5$-grading $(L_i'')_{-2 \le i \le 2}$.
\end{remark}

\section{The ideal structure of \texorpdfstring{$ L $}{L}}\label{sec:simple}

In this section, $ (J,J') $ denotes an arbitrary cubic norm pair (or, more generally, a linear cubic norm pair) over some commutative ring $ k $ and $ L \coloneq L(J,J') $ denotes the corresponding Lie algebra. Our goal is to classify the ideals of $ L $ (\cref{pr:ideal classify,le:ideal ex}). Namely, we will show that modulo $L_{0,0}$, ideals of $L$ correspond bijectively to what we call ideals of the triple $(k,J,J')$. In particular, we will see that $ L $ is simple if and only if $k$ is a field and $T$ is non-degenerate (\cref{thm:L simple}). This section is independent of the following sections.

As a first step, we show that every ideal of $L$ is graded (\cref{pr:L ideal graded}).

\begin{notation}
	In this section, we denote by $ \pi_\alpha \colon L \to L_\alpha $ the natural projection map for all $ \alpha \in G_2^0 $. Further, we will write $ \alpha \sim \beta $ for $ \alpha, \beta \in G_2 $ if they lie in the same of the three orbits in $G_2$ described in \cref{rem:G2 orbits}.
\end{notation}

\begin{definition}
	An ideal $ I $ of $ L $ is called \emph{graded} if $ I = \bigoplus_{\alpha \in G_2^0} (I \cap L_\alpha) $, or equivalently, if $ \pi_\alpha(I) \subseteq I $ for all $ \alpha \in G_2^0 $.
\end{definition}

\begin{convention}
	In this section, we will always distinguish between $k$, $J$ and $J'$ even though they could coincide in certain special cases. With this convention, we have $a \sim \beta$ (for $\alpha, \beta \in G_2$) if and only if $L_\alpha$ and $L_\beta$ are (canonically) parametrized by the same of the three structures $k$, $J$, $J'$.
\end{convention}

The crucial ingredient in the proof of \cref{pr:L ideal graded} will be \cref{le:ad rootspace isom}, which says that the Lie bracket induces \enquote{natural} isomorphisms between root spaces of $L$. We first need to make this notion precise.

\begin{definition}\label{def:lie G2 natural rhom}
	Let $ \omega \colon G_2^0 \to \Z^2 $ be as in \cref{rem:G2 grid}. Then for all $ (i,j) \in \omega(G_2) $, we define $ P_{(i,j)} = P_{ij} \in \{k,J,J'\} $ to be the additive group to which $ L_{i,j} $ (as defined in \cref{def:L0 bar,def:dd and L0}) is naturally isomorphic (as can be read off from \cref{fig:grading}). Further, the \emph{root homomorphism with respect to $ (i,j) $} is the \enquote{canonical} isomorphism $ \rhomlieG{i,j} \colon P_{i,j} \to L_{i,j} $. For example,
	\begin{align*}
		\rhomlieG{-2,-1}(t) &= tx, & \rhomlieG{-1,-2}(t) &= (t,0,0,0)_-, & \rhomlieG{1,2}(t) &= (0,0,0,t)_+, & \rhomlieG{2,1}(t) &= ty, \\
		\rhomlieG{0,-1}(b') &= b', &\rhomlieG{-1,-1}(b) &= (0,b,0,0)_-, & \rhomlieG{1,1}(b') &= (0,0,b',0)_+, & \rhomlieG{0,1}(b) &= b
	\end{align*}
	for all $ t \in k $, $ b \in J $ and $ b' \in J' $. We will in this section also refer to the isomorphisms $ k \to k \rho \colon \lambda \to \lambda \rho $ for $ \rho \in \{\xi, \zeta\} $ as \enquote{canonical}.
\end{definition}

\begin{definition}
	Let $ \alpha, \beta \in G_2 $ such that $ \alpha \sim \beta $. An isomorphism $ \varphi \colon L_\alpha \to L_\beta $ of $ k $-modules is called \emph{natural} if it is compatible with the root homomorphisms $ \rhomlieG{\alpha} $ and $ \rhomlieG{\beta} $ from \cref{def:lie G2 natural rhom}, meaning that the diagram
	\[ \begin{tikzcd}
		L_\alpha \arrow[rr, "\varphi"] && L_\beta \\
		& P_\alpha = P_\beta \arrow[ul, "\rhomlieG{\alpha}"] \arrow[ur, "\rhomlieG{\beta}"'] &
	\end{tikzcd} \]
	commutes. Similarly, an isomorphism between $ L_\gamma $ for $ \gamma \in G_2 $ long and $ k \rho $ for $ \rho \in \{\xi, \zeta\} $ is called \emph{natural} if it is compatible with $ \rhomlieG{\gamma} $ and the \enquote{canonical} map $ k \to k \rho \colon \lambda \mapsto \lambda \rho $, and an isomorphism between $ k \xi $ and $ k \zeta $ is called \emph{natural} if it interchanges $ \xi $ and $ \zeta $.
\end{definition}

\begin{remark}\label{rem:natural auto}
	Note that there exists exactly one natural isomorphism from $ L_\alpha $ to $ L_\beta $, namely $\rhomlieG{\beta} \circ \rhomlieG{\alpha}^{-1}$, and it is the identity map if $ \alpha=\beta $. However, we will be interested in $ k $-module endomorphisms of $ L $ that restrict to a natural isomorphism on certain root subspaces, and there may be multiple ways to extend a natural isomorphism between two root spaces to such an endomorphism.
\end{remark}

In the following, we will use the convention from \cref{rem:G2 grid} to identify $ G_2^0 $ with a subset of $\Z^ 2$, so that we have root subspaces $ (L_\alpha)_{\alpha \in G_2^0} $ and corresponding parametrising spaces $ (P_\alpha)_{\alpha \in G_2^0} $.

\begin{lemma}\label{le:ad rootspace isom}
	Let $ \alpha, \beta \in G_2 $ such that $ \alpha \sim \beta $. Then there exist $ \gamma_1, \ldots, \gamma_n \in G_2 $ and $ a_1 \in L_{\gamma_1}, \ldots, a_n \in L_{\gamma_n} $ such that $ \ad(a_1) \circ \cdots \circ \ad(a_n) $ induces a natural isomorphism from $ L_\alpha $ to $ L_\beta $. Further, if $ \alpha $ is long and $ \rho \in \{\xi, \zeta\} $, then there exist $ b_1, \ldots, b_m \in L $ and $ c_1, \ldots, c_p \in L $ such that
	\[ \ad(b_1) \circ \cdots \circ \ad(b_m) \quad \text{and} \quad \ad(c_1) \circ \cdots \circ \ad(c_p) \]
	induce natural isomorphisms from $ L_\alpha $ to $ k\rho $ and from $ k\rho $ to $ L_\alpha $, respectively.
\end{lemma}
\begin{proof}
	In the following, an arrow labeled by an element $ a \in L $ stands for the $ k $-linear homomorphism $ \ad_a \colon L \to L \colon b \mapsto [a,b] $. On the root spaces (canonically) parametrized by $ J $, we have a chain of isomorphisms
	\[ \begin{tikzcd}[column sep=huge]
		L_{-1,-1} \arrow[r, "{(0,0,0,1)_+}"] & L_{0,1} \arrow[r, "{(-1,0,0,0)_+}"] & L_{1,0} \arrow[r, "-x"] & L_{-1,-1}
	\end{tikzcd} \]
	that are compatible with the root homomorphisms. On the root spaces (canonically) parametrized by $ J' $, we have a similar chain
	\[ \begin{tikzcd}[column sep=huge]
		L_{-1,0} \arrow[r, "{(1,0,0,0)_+}"] & L_{0,-1} \arrow[r, "{(1,0,0,0)_+}"] & L_{1,1} \arrow[r, "-x"] & L_{-1,0}. 
	\end{tikzcd} \]
	On the long root spaces, the isomorphisms
	\[ \begin{tikzcd}[column sep=huge, row sep=small]
		L_{-2,-1} \arrow[r, "{(1,0,0,0)_+}"] & L_{-1,-2} \arrow[r, "y"] & L_{1,-1} \arrow[r, "{(0,0,0,1)_+}"] & L_{2,-1} \arrow[dl, "{(0,0,0,-1)_-}"] \\
		& L_{-1, 1} \arrow[ul, "{(-1,0,0,0)_-}"] & L_{1,2} \arrow[l, "-x"] &
	\end{tikzcd} \]
	are compatible with the root homomorphisms. Finally, the isomorphisms
	\[ \begin{tikzcd}[column sep=huge, row sep = tiny]
		L_{-2,-1} \arrow[r, "-y"] & k \xi \arrow[r, "{(-1,0,0,0)_+}"] & L_{1,-1}, \\
		L_{-1,-2} \arrow[r, "{(0,0,0,-1)_+}"] & k \zeta \arrow[r, "-y"] & L_{2,1}
	\end{tikzcd} \]
	are compatible with the root homomorphisms and with the \enquote{canonical} isomorphisms $ \lambda \mapsto \lambda \xi $, $ \lambda \mapsto \lambda \zeta $. By composing a suitable sequence of the previous isomorphisms, the assertion follows.
\end{proof}

\begin{notation}
	For all $ \alpha, \beta \in G_2 $ such that $ \alpha \sim \beta $, we fix a $ k $-module endomorphism $ \varphi_{\alpha \to \beta} \colon L \to L $ that induces a natural isomorphism $ L_\alpha \to L_\beta $ and that is of the form $ \varphi_{\alpha \to \beta} = \ad(a_1) \circ \cdots \circ \ad(a_n) $ for suitable chosen root space elements $ a_1, \ldots, a_n \in L $. Note that $ \varphi_{\alpha \to \beta} $ leaves all ideals of $ L $ invariant and satisfies $ \varphi_{\alpha \to \beta}(L_\gamma) \subseteq L_{\gamma + \beta - \alpha} $ for all $ \gamma \in G_2^0 $. We also choose maps $ \varphi_{\alpha \to \rho} $, $ \varphi_{\rho \to \alpha} $ for $ \alpha \in G_2 $ long and $ \rho \in \{\xi, \zeta\} $ and maps $ \varphi_{\xi \to \zeta} $, $ \varphi_{\zeta \to \xi} $ in a similar way.
\end{notation}

\begin{remark}\label{rem:proj shift}
	Let $ \gamma \in G_2^0 $ and let $ \psi_\gamma $ be a $ k $-module endomorphism of $ L $ such that $ \psi_\gamma(L_\alpha) \subseteq L_{\alpha+\gamma} $ for all $ \alpha \in G_2^0 $. Then $ \pi_\alpha \circ \psi_\gamma = \psi_\gamma \circ \pi_{\alpha-\gamma} $ for all and $ \alpha \in G_2^0 $. In particular, $ \pi_\alpha \circ \varphi_{\beta \to \alpha} = \varphi_{\beta \to \alpha} \circ \pi_\beta $ and hence $ \varphi_{\alpha \to \beta} \circ \pi_\alpha \circ \varphi_{\beta \to \alpha} = \varphi_{\alpha \to \beta} \circ \varphi_{\beta \to \alpha} \pi_\beta = \pi_\beta $ for all $ \alpha, \beta \in G_2^0 $.
\end{remark}

\begin{proposition}\label{le:ad long rootspace proj}
	Let $ \alpha $ be a long root in $ G_2 $. Then there exist $ a_1, \ldots, a_n \in L $ such that $ \ad(a_1) \circ \cdots \circ \ad(a_n) = \pi_\alpha $.
\end{proposition}
\begin{proof}
	Since $ \gamma -2\alpha \notin G_2^0 $ for all $ \gamma \in G_2^0 \setminus \{\alpha\} $ (as can be seen in \cref{fig:grading}), $ \varphi_{\alpha \to -\alpha} $ maps all root spaces except for $ L_\alpha $ to zero. Hence $ \varphi_{-\alpha \to \alpha} \circ \varphi_{\alpha \to -\alpha} $ maps all root spaces except for $ L_\alpha $ to zero, and it is the identity map on $ L_\alpha $ (by \cref{rem:natural auto}). Thus $ \varphi_{-\alpha \to \alpha} \circ \varphi_{\alpha \to -\alpha} = \pi_\alpha $.
\end{proof}

It is not possible to realize the projections to short root spaces as in \cref{le:ad long rootspace proj}: We only obtain the following result because we cannot annihilate the $(0,0)$-part of $L$ in this way.

\begin{lemma}\label{le:ideal graded lem}
	Let $ I $ be an ideal of $ L $ and let $ b \in I $ with $ \pi_{(0,0)}(b) = 0 $. Then $ \pi_\alpha(b) \in I $ for all short roots $ \alpha \in G_2 $.
\end{lemma}
\begin{proof}
	For all long roots $ \delta \in G_2 $, we have $ \pi_\delta(b) \in I $ by \cref{le:ad long rootspace proj}, and hence we may assume that $ \pi_\delta(b) = 0 $.
	Let $ \alpha \in G_2 $ be short and let $ \beta, \gamma \in G_2 $ be the two short roots such that $ \alpha \sim \beta \sim \gamma $ and $ \alpha,\beta, \gamma $ are pairwise distinct.
	Then
	\[ \varphi \coloneq \varphi_{\gamma \to \alpha} \circ \varphi_{\beta \to \gamma} \circ \varphi_{\alpha \to \beta} \]
	is the identity map on $ L_{\alpha} $ and, because each factor shifts the grading, sends all other root spaces to zero except (possibly) for $ L_{0,0} $ and some long root spaces. Hence $ \pi_\alpha(b) $ is the only non-zero component of $ b $ on which $ \varphi $ is non-zero, so
	\begin{align*}
		\varphi(b) &= \varphi\brackets[\big]{\pi_{\alpha}(b)} = \pi_{\alpha}(b).
	\end{align*}
	Since $ \varphi(b) \in I $ by construction of $ \varphi $, the assertion follows.
\end{proof}

\begin{proposition}\label{pr:L ideal graded}
	Every ideal of $ L $ is graded.
\end{proposition}
\begin{proof}
	Let $ I $ be an ideal of $ L $ and let $ a \in I $. We have to show that $ \pi_\alpha(a) \in I $ for all $ \alpha \in G_2^0 $. For all long $ \alpha \in G_2 $, this holds by \cref{le:ad long rootspace proj}, so we may assume that $ \pi_\alpha(a) = 0 $ for such $ \alpha $.
	Put $ a_0 \coloneq \pi_{(0,0)}(a) \in L_{0,0} $ and write $ a_0 = a_0' + \lambda \xi + \mu \zeta $ where $ \lambda, \mu \in k $ and $ a_0' \in \langle \dd_{b,b'} \mid b \in J, b' \in J' \rangle $. (Note that this expression is, in general, not unique.) Since the root $ (2,1) $ is long and $ \ad_x(a) \in I $, we have $ \pi_{(2,1)}(\ad_x(a)) \in I $. Here, by \cref{rem:proj shift},
	\[ \pi_{(2,1)}\brackets[\big]{\ad_x(a)} = \ad_x\brackets[\big]{\pi_{(0,0)}(a)} = \ad_x(a_0' + \lambda \xi + \mu \zeta) = (2\lambda + \mu)x. \]
	Thus $ (2\lambda + \mu)x \in I $ and hence, by \cref{le:ad rootspace isom}, $ (2\lambda + \mu) \zeta \in I $. By subtracting this element from $ a $, we may assume without loss of generality that $ \mu = -2\lambda $, so that $ a_0 = a_0' + \lambda (\xi - 2\zeta) $. Thus $ \ad_y(a_0) = 0 $. This implies, by \cref{rem:proj shift}, that
	\begin{align*}
		\pi_{(0,0)}\brackets[\big]{\ad_{-x} \circ \ad_y(a)} &= \ad_{-x} \circ \ad_y \brackets[\big]{\pi_{(0,0)}(a)} = \ad_{-x}\brackets[\big]{\ad_y(a_0)} = 0.
	\end{align*}
	Hence $ b \coloneq \ad_{-x} \circ \ad_y(a) $ satisfies the requirements of \cref{le:ideal graded lem}. Thus for all short $ \alpha \in G_2 $, we have
	\begin{align*}
		\ad_{-x} \circ \ad_y\brackets[\big]{\pi_\alpha(a)} &= \pi_\alpha\brackets[\big]{\ad_{-x} \circ \ad_y(a)} = \pi_\alpha(b) \in I.
	\end{align*}
	Since $ \ad_{-x} \circ \ad_y $ is the identity map on $ L_1 $, we infer that $ \pi_\alpha(a) \in I $ for $ \alpha \in \{(1,0), (1,1)\} $. In other words, we have proven that $ \pi_\alpha(I) \subseteq I $ for $ \alpha \in \{(1,0), (1,1)\} $.
	
	Now let $ \beta \in G_2 $ be an arbitrary short root. Then $ \beta \sim \alpha $ for some $ \alpha \in \{(1,0), (1,1)\} $ and hence, by \cref{rem:proj shift},
	\[ \pi_\beta(I) = \varphi_{\alpha \to \beta} \circ \pi_\alpha \circ \varphi_{\beta \to \alpha}(I) \subseteq I. \]
	Altogether, we have shown that $ \pi_\gamma(I) \subseteq I $ for all $ \gamma \in G_2 $. Since
	\[ \pi_{(0,0)} = \id_L - \sum_{\gamma \in G_2} \pi_\gamma, \]
	it follows that $ \pi_{(0,0)}(I) \subseteq I $ as well. Therefore, $ I $ is graded.
\end{proof}

With \cref{pr:L ideal graded} proven, we can now begin to associate ideals of $L$ to substructures of $(k,J,J')$. The closure condition of ideals together with \cref{ta:Lie} motivates the following definition.

\begin{definition}\label{def:ideal}
	An \emph{ideal of $ (k,J,J') $} is a triple $ (\ell, U, U') $ with the following properties.
	\begin{enumerate}
		\item $ \ell $ is an ideal of the ring $ k $ and $ U $, $ U' $ are $ k $-submodules of $ J $, $ J' $, respectively.
		
		\item $ \ell J \subseteq U $ and $ \ell J' \subseteq U' $.
		
		\item $ T(u,b') \in \ell$ for all $ u \in U $, $ b' \in J' $ and $ T(b, u') \in \ell $ for all $ u' \in U' $, $ b \in J $.
		
		\item $ u \times b \in U' $ for all $ u \in U $, $ b \in J $ and $ u' \times b' \in U $ for all $ u' \in U' $, $ b' \in J' $.
	\end{enumerate}
\end{definition}

\begin{example}\label{ex:ideal ring CNP}
	Let $ \ell $ be an ideal of $ k $. Then $ (\ell, \ell J, \ell J') $ is an ideal of $ (k, J, J') $. In particular, $ (0,0,0) $ and $ (k, J, J') $ are ideals of $ (k,J,J') $, called the \emph{trivial ideals}.
\end{example}

\begin{example}\label{ex:CNP radical}
	Let
	\begin{align*}
		R &\coloneq \{u \in J \mid T(u,a') = 0 \text{ for all } a' \in J'\} \quad \text{and} \\
		R' &\coloneq \{u' \in J' \mid T(a,u') = 0 \text{ for all } a \in J\}
	\end{align*}
	denote the \emph{left radical of $ T $} and the \emph{right radical of $ T $}, respectively.
	Then for any ideal $ (\ell, U, U') $ of $ (k, J, J') $, it follows from \cref{le:CNP identities}\cref{le:CNP identities:7} that $ (\ell, U+R, U'+R') $ is an ideal as well. In particular, $ (\{0\}, R, R') $ is an ideal of $ (k, J, J') $. Further, if $ T=0 $, then $ (\ell, J, J') $ is an ideal of $ (k,J,J') $ for all ideals $ \ell $ of $ k $ because then $ R=J $ and $ R'=J' $.
\end{example}

\begin{remark}\label{rem:ideal D}
	Let $ (\ell, U, U') $ be an ideal of $ (k,J,J') $. Then by \cref{def:U-operators},
	\[ D_{b,b'}(u), D_{u,b'}(b), D_{b,u'}(c) \in U \quad \text{and} \quad D_{b',b}(u'), D_{u',b}(b'), D_{b',u}(c') \in U \]
	for all $ b,c \in J $, $ b',c' \in J' $, $ u \in U $ and $ u' \in U' $.
\end{remark}

\begin{remark}\label{rem:ideal quot}
	Let $ (\ell, U, U') $ be an ideal of $ (k,J,J') $. Then $ (J/U, J'/U') $ is a linear cubic norm pair over the ring $ k/\ell $ with induced natural maps $ \bar{T} $ and $ \bar{\times} $.\footnote{Note that, even if $ (J,J') $ is a proper cubic norm pair, then $ (J/U, J'/U') $ need not have the structure of a proper cubic norm pair because the definition of ideals of $ (k,J,J') $ does not refer to $ N $ and $ \sharp $ at all. For this reason, it might be more appropriate to refer to $ (\ell, U, U') $ as a \emph{linear ideal of $ (k,J,J') $}, but we will not do so.} Hence we have an associated Lie algebra $L(J/U, J'/U')$ over $k/\ell$. Note that $L(J/U, J'/U')$ may also be regarded as a Lie algebra over $k$ whereas $(J/U, J'/U')$ cannot be regarded as a linear cubic norm pair over $k$: It is in general not possible to lift the trace $\bar{T} \colon J/U \times J'/U' \to \ell$ to a map with image in $k$ such that the defining identities of linear cubic norm pairs are satisfied. Further, the triple
	\[ (\pi_k, \pi_J, \pi_{J'}) \colon (k, J, J') \to (k/\ell, J/U, J'/U') \]
	is a \emph{homomorphism of linear cubic norm pairs with base ring} in the sense that
	\[ \bar{T}\brackets[\big]{\pi_J(a), \pi_{J'}(a')} = \pi_k\brackets[\big]{T(a,a')} \]
	for all $a \in J$, $a' \in J'$ and $\pi_J$, $\pi_{J'}$ are compatible with $\times$, $\bar{\times}$ (as in \cref{def:linear CNP homotopy} with $ t=1 $). This triple induces a homomorphism
	\[ L(\pi_k, \pi_J, \pi_{J'}) \colon L(J,J') \to L(J/U, J'/U') \]
	of Lie algebras over $k$ that is defined just as in \cref{def:L functorial} (with $ (\iota, \iotainv, t) = (\pi_J, \pi_{J'}, 1) $), except that $(\nu, 0, 0, \rho)_\pm $ is mapped to $ (\pi_k(\nu), c, c', \pi_k(\rho))$ for $\nu, \rho \in k$, $ c \in J $, $ c' \in J' $. This homomorphism of Lie algebras is surjective by construction. In the following, when the ideal $(\ell, U, U')$ is clear from the context, we will denote $ L(\pi_k, \pi_J, \pi_{J'}) $ as well as the epimorphisms $\pi_k \colon k \to k/\ell$, $\pi_J \colon J \to J/U$ and $\pi_{J'} \colon J' \to J'/U'$ by $a \mapsto \bar{a}$.
\end{remark}

We would like to associate to each ideal $(\ell, U, U')$ of $(k, J, J')$ an ideal $L(\ell, U, U')$ of $L$. While $(\ell, U, U')$ canonically prescribes what $L(\ell, U, U') \cap L_\alpha$ for $\alpha \in G_2$ should be, the same is not clear for $\alpha = (0,0)$. Indeed, we will define a minimal and a maximal ideal corresponding to $(\ell, U, U')$ and show that every ideal of $ L $ is sandwiched between such a pair.

The following definition is motivated by \cref{le:lin comb}.

\begin{definition}
	Let $(\ell, U, U')$ be an ideal of $(k,J,J')$. The \emph{radical of $(\ell, U, U')$} is the $k$-submodule $\Rad(\ell, U, U')$ of $L_{0,0}$ consisting of all elements of the form
	\[ \sum \lambda_{a,a'} \dd_{a,a'} + \lambda_\xi \xi + \lambda_\zeta \zeta \]
	where the sum runs over finitely many $a \in J$, $a' \in J'$ and $\lambda_{a,a'}, \lambda \in k$, such that the following properties are satisfied.
	\begin{enumerate}
		\item $\lambda_\zeta + 2\lambda_\xi \in \ell$.
		\item $\sum \lambda_{a,a'} T(a,a') \equiv 3\lambda_\xi$ modulo $\ell$.
		
		\item $\sum \lambda_{a,a'} D_{a,a'} \equiv 2\lambda_\xi \id_J $ modulo $U$. That is, the image of $\sum \lambda_{a,a'} D_{a,a'} - 2\lambda_\xi \id_J$ is contained in $U$.
		
		\item $\sum \lambda_{a,a'} D_{a',a} \equiv 2\lambda_\xi \id_{J'}$ modulo $U'$. That is, the image of $\sum \lambda_{a,a'} D_{a',a} - 2\lambda_\xi \id_{J'}$ is contained in $U'$.
	\end{enumerate}
\end{definition}

\begin{definition}
	Let $ (\ell, U, U') $ be an ideal of $ (k, J, J') $. Then we define the following subsets of $ L $:
	\begin{align*}
		L(\ell, U, U')_{\tmin} &\coloneq \bigoplus_{\alpha \in G_2 \text{ long}} \rhomlieG{\alpha}(\ell) \oplus \bigoplus_{\alpha \sim (1,0)} \rhomlieG{\alpha}(U) \oplus \bigoplus_{\alpha \sim (1,1)} \rhomlieG{\alpha}(U') \\*
		&\qquad{} \oplus \brackets[\big]{\ell \xi + \ell \zeta + \langle \dd_{a,a'} \mid a \in U, a' \in J' \text{ or } a \in J, a' \in U' \rangle}, \\
		L(\ell, U, U')_{\tmax} &\coloneq \bigoplus_{\alpha \in G_2 \text{ long}} \rhomlieG{\alpha}(\ell) \oplus \bigoplus_{\alpha \sim (1,0)} \rhomlieG{\alpha}(U) \oplus \bigoplus_{\alpha \sim (1,1)} \rhomlieG{\alpha}(U') \\*
		&\qquad{}\oplus \Rad(\ell, U, U').
	\end{align*}
	Further, we put $L_i(\ell, U, U')_\sigma \coloneq L(\ell, U, U')_\sigma \cap L_i$ for all $i \in \{-2, \ldots, 2\}$ and $\sigma \in \{\tmin, \tmax\}$.
\end{definition}

\begin{remark}
	It is straightforward to check that $L(\ell, U, U')_{\tmin}$ is contained in $L(\ell, U, U')_{\tmax}$.
\end{remark}

Before we can turn to the classification of ideals of $ L $, we need two alternative characterizations of the radical.

\begin{lemma}\label{le:max is ker}
	Let $ (\ell, U, U') $ be an ideal of $ (k, J, J') $. Then the kernel of the epimorphism $\pi \colon L(J,J') \to L(J/U, J'/U')$ from \cref{rem:ideal quot} is precisely $L(\ell, U, U')_{\tmax}$. In particular, $L(\ell, U, U')_{\tmax}$ is an ideal of $L$ and $L(\{0\},\{0\},\{0\})_{\tmax} = 0$.
\end{lemma}
\begin{proof}
	Let $K$ denote the kernel of $\pi$. By \cref{pr:L ideal graded}, $K = \bigoplus_{\alpha \in G_2^0} \pi_\alpha(K)$. By the definition of $\pi$, we have $\pi_\alpha(K) = \pi_\alpha(L(\ell, U, U')_{\tmax})$ for all $\alpha \in G_2$. Now let $f \in L_{0,0}$ be arbitrary and write
	\[ f=\sum \lambda_{a,a'} \dd_{a,a'} + \lambda_\xi \xi + \lambda_\zeta \zeta \]
	where the sum runs over finitely many $a \in J$, $a' \in J'$ and $\lambda_{a,a'}, \lambda \in k$. Then
	\[ \bar{f} = \sum \overline{\lambda_{a,a'}} \dd_{\bar{a},\overline{a'}} + \overline{\lambda_\xi} \bar{\xi} + \overline{\lambda_\zeta} \bar{\zeta} \in L(J/U, J'/U'). \]
	By \cref{le:lin comb}, we have $\bar{f}=0$ if and only if $f \in \Rad(\ell,U,U')$. The assertion follows.
\end{proof}

\begin{lemma}\label{le:rad char}
	Let $(\ell, U, U')$ be an ideal of $(k,J,J')$ and let $\sigma \in \{\tmin, \tmax\}$. Then
	\begin{align*}
		\Rad(\ell, U, U') &= \{f \in L_{0,0} \mid [f,L] \subseteq L(\ell, U, U')_{\sigma}\} \\
		&= \{f \in L_{0,0} \mid [f,L_i] \subseteq L_i(\ell, U, U')_\sigma \text{ for } i \in \{-2,-1,1,1\}\}.
	\end{align*}
\end{lemma}
\begin{proof}
	Let
	\[ f = \sum \lambda_{a,a'} \dd_{a,a'} + \lambda_\xi \xi + \lambda_\zeta \zeta \]
	denote an arbitrary element of $L_{0,0}$. We have straightforward implications
	\[ \begin{tikzcd}[row sep=3ex]
		{[f,L] \subseteq L(\ell, U, U')_{\tmin}} \arrow[r, Rightarrow] \arrow[d, Rightarrow] & {[f,L_i] \subseteq L_i(\ell, U, U')_{\tmin} \text{ for } i \in \{-2,-1,1,1\}} \arrow[d, Leftrightarrow] \\
		 {[f,L] \subseteq L(\ell, U, U')_{\tmax}} \arrow[r, Rightarrow] & {[f,L_i] \subseteq L_i(\ell, U, U')_{\tmax} \text{ for } i \in \{-2,-1,1,1\}}
	\end{tikzcd} \]
	Further, by \cref{le:max is ker}, we have $f \in \Rad(\ell, U, U')$ if and only if $\bar{f} = 0$. By the definition of the 0-part of the Lie algebra, this is equivalent to the condition $[\bar{f}, L_i(J/U, J'/U')] = 0$ for all $i \in \{-2,-1,1,2\}$. Here
	\begin{align*}
		[\bar{f}, L_i(J/U, J'/U')] &= [\bar{f}, \overline{L_i}] = \overline{[f, L_i]}.
	\end{align*}
	Thus $f \in \Rad(\ell, U, U')$ if and only if $[f,L_i] \subseteq \ker(\pi) = L(\ell, U, U')_{\tmax}$ for all $i \in \{-2,-1,1,2\}$. Hence
	\[ \Rad(\ell, U, U') = \{f \in L_{0,0} \mid [f,L_i] \subseteq L_i(\ell, U, U')_{\tmax} \text{ for } i \in \{-2,-1,1,1\}\}. \]
	
	Now assume that $f \in \Rad(\ell, U, U')$. We must show that $[f,L] \subseteq L(\ell, U, U')_{\tmin}$. Since $L(\ell, U, U')_{\tmin} \cap L_\alpha = L(\ell, U, U')_{\tmax} \cap L_\alpha$ for all $\alpha \in G_2$, it remains to show that $[f,L_{0,0}] \subseteq L(\ell, U, U')_{\tmin}$. Let $c \in J$, $c' \in J'$ be arbitrary. Then, by the same computation as in \cref{cor:Z is sublie},
	\begin{align*}
		[f, \dd_{c,c'}] &= \dd\brackets[\big]{\sum \lambda_{a,a'} D_{a,a'}(c), c'} - \dd\brackets[\big]{c, \sum \lambda_{a,a'} D_{a',a}(c')}.
	\end{align*}
	Here, by the definition of the radical,
	\[ \sum \lambda_{a,a'} D_{a,a'}(c) \in 2c + U \quad \text{and} \quad \sum \lambda_{a,a'} D_{a',a}(c') \in 2c' + U'. \]
	Hence $[f, \dd_{c,c'}] \in \dd(U,c') + \dd(c, U') \subseteq L(\ell, U, U')_{\tmin}$. Since also $[f,\xi] = 0 = [f, \zeta]$, we infer that $[f, L_{0,0}] \subseteq L(\ell, U, U')_{\tmin}$.
\end{proof}

\begin{lemma}\label{le:ideal ex}
	Let $ (\ell, U, U') $ be an ideal of $ (k, J, J') $ and let $V$ be a $k$-submodule of $\Rad(\ell, U, U')$. Then $ L(\ell, U, U')_{\tmin} + V $ is an ideal of $ L $.
\end{lemma}
\begin{proof}
	By \cref{le:rad char}, it suffices to show that $L(\ell, U, U')_{\tmin}$ is an ideal of $L$. This follows from an inspection of \cref{ta:Lie}, using the properties of ideals in \cref{def:ideal,rem:ideal D}.
\end{proof}

\begin{proposition}\label{pr:ideal classify}
	Let $I$ be an ideal of $L$. Then there is a unique ideal $ (\ell, U, U') $ of $ (k, J, J') $ such that
	\[ L(\ell, U, U')_{\tmin} \subseteq I \subseteq L(\ell, U, U')_{\tmax}. \]
	Further, there exists a unique $k$-submodule $V$ of $L_{0,0}$ with
	\[ \brackets[\big]{\ell \xi + \ell \zeta + \langle \dd_{a,a'} \mid a \in U, a' \in J' \text{ or } a \in J, a' \in U' \rangle} \subseteq V \subseteq \Rad(\ell, U, U') \]
	such that $I = L(\ell, U, U')_{\tmin} + V$.
\end{proposition}
\begin{proof}
	Choose a long root $ \alpha \in G_2 $ and short roots $ \beta, \gamma \in G_2 $ with $ \beta \sim (1,0) $ and $ \gamma \sim (1,1) $. We define subgroups $ \ell $, $ U $, $ U' $ of $ k $, $ J $, $ J' $, respectively, by the property that
	\begin{equation}\label{eq:ideal param}
		\rhomlieG{\alpha}(\ell) = \pi_\alpha(I), \quad \rhomlieG{\beta}(U) = \pi_\beta(U), \quad \rhomlieG{\gamma}(U') = \pi_\gamma(U').
	\end{equation}
	By \cref{le:ad rootspace isom} and because $I$ is an ideal, $(\ell, U, U')$ does not depend on the choice of $\alpha$, $\beta$, $\gamma$. Hence
	\[ I \cap L_\alpha = L(\ell, U, U')_{\tmin} \cap L_\alpha = L(\ell, U, U')_{\tmax} \cap L_\alpha \]
	for all $\alpha \in G_2$. Further, $\ell \xi = \varphi_{\alpha \to \xi}(\pi_\alpha(I)) \subseteq I$ and $\ell \zeta = \varphi_{\alpha \to \zeta}(\pi_\alpha(I)) \subseteq I$. Now let $a \in U$ and $a' \in J'$. Then $\dd_{a,a'} = [a', a] \in [L,I] \subseteq I$. Similarly, if $a \in J$ and $a' \in U'$, then also $\dd_{a,a'} \in I$. We conclude that $L(\ell, U, U')_{\tmin} \subseteq I$. Since
	\begin{align*}
		[V,L_i] \subseteq [I \cap L_0, L_i] \subseteq I \cap L_i = L_i(\ell, U, U')
	\end{align*}
	for all $ i \in \{-2,-1,1,2\} $, it follows from \cref{le:rad char} that $V \subseteq L(\ell, U, U')_{\tmax}$ and hence
	\[ L(\ell, U, U')_{\tmin} \subseteq I \subseteq L(\ell, U, U')_{\tmax}. \]
	Note that $(\ell, U, U')$ is uniquely determined by this property because it must satisfy \cref{eq:ideal param}. Further, since $L(\ell, U, U')_{\tmin}$ and $L(\ell, U, U')_{\tmax}$ differ only in the $(0,0)$-component, there exists $V$ as in the assertion such that $I = L(\ell, U, U')_{\tmin} + V$, and it is uniquely determined by $V = \pi_{(0,0)}(I)$.
\end{proof}

\Cref{pr:ideal classify} has some straightforward consequences.

\begin{theorem}\label{thm:L simple}
	Let $(J,J')$ be a cubic norm pair (or, more generally, a linear cubic norm pair) over a commutative ring $k$.
	Then the following assertions are equivalent.
	\begin{enumerate}
		\item $(k,J,J')$ is simple in the sense that it has only the trivial ideals $(0,0,0)$ and $(k,J,J')$.
		\item $L$ is simple in the sense that it has only the trivial ideals $\{0\}$ and $L$.
		\item $k$ is simple (and hence a field) and $T$ is non-degenerate (in the sense that the left and right radicals defined in \cref{ex:CNP radical} are zero).
	\end{enumerate}
\end{theorem}
\begin{proof}
	First note that $L(0,0,0)_{\tmax} = 0$ by \cref{le:max is ker} and $L(k,J,J')_{\tmin} = L$. Thus it follows from \cref{pr:ideal classify} that if $(k,J,J')$ is simple, then $L$ is simple. Conversely, if $(k, J, J')$ has a non-trivial ideal $(\ell, U, U')$, then $L(\ell, U, U')_{\tmin}$ and $L(\ell, U, U')_{\tmax}$ are non-trivial ideals of $L$. This proves the equivalence of the first two statements.
	
	Assume that $k$ is simple (and hence, being commutative, a field) and $T$ is non-degenerate and let $(\ell, U, U')$ be an ideal of $(k,J,J')$. If $\ell=0$, then $T(u,b') = 0$ for all $u \in U$, $b' \in J'$, so $U = 0$ because $T$ is non-degenerate. Similarly, $U'=0$ in this case. If, however, $\ell = k$, then $U=J$ and $U'=J'$ because $\ell J \subseteq U$ and $\ell J' \subseteq U'$.
	
	Now assume that $(k,J,J')$ is simple. Since any ideal $\ell$ of $k$ induces an ideal $(\ell, \ell J, \ell J')$ of $(k, J, J')$ by \cref{ex:ideal ring CNP}, $k$ must be simple. Further, since the radicals $R$, $R'$ of $T$ induce an ideal $(0, R, R')$ of $(k, J, J')$ by \cref{ex:CNP radical}, $T$ must be non-degenerate.
\end{proof}

\begin{remark}
	In \cref{sec:CJMA}, we will see that cubic Jordan matrix algebras over composition algebras over fields provide examples of pairs $(k,J,J')$ which satisfy the properties in \cref{thm:L simple}. See \cref{rem:CJMA nondeg} for details. Further, we remark that ideals of cubic Jordan matrix algebras (but not of the corresponding triples $ (k,J,J) $) are classified in \cite[37.26]{GPR24}.
\end{remark}

\begin{lemma}
	The center of $L$ is trivial.
\end{lemma}
\begin{proof}
	Since the center $Z(L)$ is an ideal, there exists an ideal $(\ell, U, U')$ of $(k, J, J')$ such that
	\[ L(\ell, U, U')_{\tmin} \subseteq Z(L) \subseteq L(\ell, U, U')_{\tmax}. \]
	However, it follows from \cref{le:ad rootspace isom} that $Z(L) \cap L_\alpha = 0$ for all $\alpha \in G_2$. Thus $(\ell, U, U') = (0,0,0)$. Since $L(0,0,0)_{\tmax}=0$ by \cref{le:max is ker}, it follows that $Z(L) = 0$.
\end{proof}

\section{Exponential maps for the \texorpdfstring{$G_2$}{G2}-grading}\label{sec:G2 exp}

We are now ready to pass from root graded Lie algebras to root graded groups by \enquote{exponentiating} the appropriate adjoint maps in the Lie algebra. However, since $2$ and $3$ are not necessarily invertible in the base ring $k$, we cannot define the exponential maps by the usual formula $\exp(a) \coloneq \sum_{i=0}^4 \frac{1}{i!} \ad_a^i$ (where it is easy to see from the grading that $\ad_a^i=0$ for all $i \ge 5$). Instead, we will replace each term $\frac{1}{i!} \ad_a^i$ by an appropriate endomorphism $\theta^{[i]}_a$ of the Lie algebra, which coincides with $\frac{1}{i!} \ad_a^i$ if $2$ and $3$ happen to be invertible in $k$. The definition of $\theta^{[i]}_a$ involves the higher degree maps $N$ and $\sharp$, so from now on and for the rest of this paper, we have to assume that $(J,J')$ is a cubic norm pair in the proper sense and not merely a linear cubic norm pair.

Since the exponential maps are defined by distinct, explicit formulas for each root $\alpha \in G_2$, we have to prove many basic properties \enquote{by hand} that would, over base rings in which $2$ and $3$ are invertible, follow from the expression $\exp(a) = \sum_{i=0}^\infty \frac{1}{i!} \ad_a^i$. Most importantly, this concerns the property that exponential automorphisms are indeed compatible with the Lie bracket, which we prove in \cref{sec:leibniz} using what we call \enquote{higher Leibniz rules}. Still, some properties can be established in a more abstract way, using the framework of $\alpha$-parametrizations that we introduce and investigate in \cref{subsec:G2 exp:general}. In \cref{subsec:G2 exp:formulas}, we construct the specific $\alpha$-parametrizations in $L(J,J')$ that we will be interested in, and we verify that they satisfy certain commutator formulas in \cref{subsec:G2 exp:rgg}. Further, we prove by a long but straightforward computation in \cref{sec:weyl G2} that the reflections $\phihor$ and $\phibs$ from \cref{sec:reflections} are Weyl elements if $J$ contains an element with invertible norm, so we conclude that for every cubic norm structure $J$ we have a $G_2$-graded group in $\Aut(L(J,J))$.

\subsection{General properties of exponential maps}\label{subsec:G2 exp:general}

In this subsection, we denote by $\Phi$ an arbitrary (finite reduced crystallographic) root system and by $ \bar{L} = \bigoplus_{\beta \in \Phi^0} \bar{L}_\beta $ a $\Phi$-graded Lie algebra over some commutative base ring $k$.

\begin{definition}\label{def:param}
	Let $\alpha \in \Phi$.
	An \emph{$\alpha$-parametrization} is a map
	\[ \exp_\alpha \colon \bar{L}_\alpha \to \Aut(\bar{L}) \colon a \mapsto e_a \]
	of the form
	\[ e_a = \sum_{i \geq 0} \theta^{[i]}_a , \]
	where each $\theta^{[i]}_a$ is a $k$-module endomorphism of $\bar{L}$, satisfying:
	\begin{enumerate}[(a)]
		\item\label{item:param a} $\theta^{[i]}_a(\bar{L}_\beta) \leq \bar{L}_{\beta + i \alpha}$ for all $\beta \in \Phi^0$,
		\item\label{item:param b} $\theta^{[0]}_a = \id$ and $\theta^{[1]}_a = \ad_a$,
		\item\label{item:param c} $\theta^{[i]}_{\lambda a} = \lambda^i \theta^{[i]}_a$,
		\item\label{item:param sum} $e_\alpha(a + b) = e_\alpha(a) e_\alpha(b)$,
	\end{enumerate}
	for all $a,b \in L_\alpha$ and all $\lambda \in k$.
	The automorphisms in the image of $e_\alpha$ are called \emph{$\alpha$-exponential automorphisms}.
	By \cref{item:param sum}, the image $e_\alpha(\bar{L}_\alpha)$ forms an abelian group, which we will call a \emph{root group} (corresponding to the root $\alpha$).
\end{definition}

\begin{remark}
\begin{enumerate}
	\item Contrary to what our notation might suggest, an $\alpha$\dash parametrization is typically not unique, even when $k$ is, for example, a field of characteristic $0$. See also \cite[Theorem 3.13(v)]{DMM25}.
	\item If $ \exp_\alpha \colon \bar{L}_\alpha \to \Aut(\bar{L}) \colon a \mapsto e_a $ is an $ \alpha $-parametrization, then the $ k $-linear endomorphisms $\theta^{[i]}_a$ of $ \bar{L} $ satisfying $ e_a = \sum_{i \geq 0} \theta^{[i]}_a $ for all $ a \in \bar{L}_\alpha $ are uniquely determined by $ \exp_\alpha $.
	\item Let $ \exp_\alpha $ be an $ \alpha $-parametrization, let $ 1 \ne t \in k $ be invertible and let $ \mu_t \colon \bar{L}_\alpha \to \bar{L}_\alpha \colon a \mapsto ta $. Then $ \mathord{\exp_\alpha} \circ \mu_t $ satisfies all axioms of an $ \alpha $\dash parametrization except that $ \theta_a^{[1]} = \ad_a $.
\end{enumerate}
\end{remark}

We first show that parametrizations behave well with respect to reflections.

\begin{lemma}\label{le:conj by refl}
	Let $\bar{L}' = \bigoplus_{\beta \in \Phi^0} \bar{L}_\beta'$ be a second $ \Phi $-graded Lie algebra, let $ \alpha, \beta \in \Phi $, let $ \varphi \colon \bar{L} \to \bar{L}' $ be a $ \beta $-reflection (in the sense of \cref{def:refl abstr}) and let $ \exp_\alpha' \colon \bar{L}'_\alpha \to \Aut(\bar{L}') $ be an $ \alpha $-parametrization (in $ \bar{L}' $). Then the map
	\[ \exp_{\alpha^{\reflbr{\beta}}} \colon \bar{L}_{\alpha^{\reflbr{\beta}}} \to \Aut(\bar{L}) \colon a \mapsto \varphi^{-1} \circ \exp_\alpha'\brackets[\big]{\varphi(a)} \circ \varphi \]
	is an $ \alpha^{\reflbr{\beta}} $-parametrization (in $ \bar{L} $).
\end{lemma}
\begin{proof}
	The axioms in \cref{def:param} can be verified in a straightforward manner, using the linearity of $ \refl{\beta} $ for \cref{item:param a}.
\end{proof}

Next we show that parametrizations corresponding to adjacent roots commute, which allows us to omit a few cases in the computation of commutator relations in \cref{pr:G2 comm}.

\begin{definition}
	Let $ \Phi $ be a root system. A pair of roots $ \alpha, \beta \in \Phi $ is called \emph{adjacent} if they are non-proportional and the root interval $ \rootint{\alpha}{\beta} $ from \cref{def:rgg}\cref{def:rgg:comm} is empty.
\end{definition}

\begin{lemma}\label{le:adj lincomb}
	Let $ \Phi $ be a root system and let $ \alpha, \beta \in \Phi $ be adjacent. Then for all $ \gamma \in \Phi $ and all $ i,j \in \mathbb{N}_{\ge 2} $, we have $ \gamma + i\alpha + j\beta \not\in \Phi $.
\end{lemma}
\begin{proof}
	We normalize the root system $\Phi$ in Euclidean space such that the short roots have Euclidean length $1$.
	Assume, by contradiction, that $ i,j \in \mathbb{N}_{\ge 2} $ are such that $i\alpha + j\beta$ is the sum of two roots $\gamma,\delta \in \Phi$,
	so we have $i\alpha + j\beta = \gamma + \delta$.

	We first observe that the angle between the adjacent roots $\alpha$ and $\beta$ is at most~$90^\circ$.
	Indeed, $\alpha$ and $\beta$ are contained in a common root subsystem of rank $2$, so it suffices to examine each of the possibilities $A_1 \times A_1$, $A_2$, $B_2 = C_2$ and $G_2$ to see that a pair of roots at an angle in the interval $\rootint{90^\circ}{180^\circ}$ can never be adjacent.
	It now follows that the Euclidean length $\lVert i \alpha + j \beta \rVert$ takes its minimum for $i = j = 2$.
	In other words, it suffices to show that
	\begin{equation}\label{eq:adjroots}
		2 \lVert \alpha + \beta \rVert > \lVert \gamma + \delta \rVert \quad \text{ for all } \gamma,\delta \in \Phi .
	\end{equation}
	
	When $\Phi = G_2$ there are only $3$ possible configurations for adjacent roots (up to symmetry), namely either both $\alpha$ and $\beta$ are long and make an angle of $60^\circ$, or one is short and the other is long, making an angle of either $30^\circ$ or $90^\circ$.
	In any case, we have $\lVert \alpha + \beta \rVert > 2$. Since roots have length at most $\sqrt{3}$, we have $\lVert \gamma + \delta \rVert \leq 2 \sqrt{3}$, and \eqref{eq:adjroots} follows.

	Now assume $\Phi \neq G_2$, so that all roots have length at most $\sqrt{2}$. In particular,
	\begin{equation}\label{eq:adjroots2}
		\lVert \gamma + \delta \rVert \leq 2 \sqrt{2} .
	\end{equation}
	On the other hand, since $\alpha$ and $\beta$ make an angle of at most $90^\circ$, we have
	\begin{equation}\label{eq:adjroots3}
		2 \lVert \alpha + \beta \rVert \geq 2 \sqrt{\lVert \alpha \rVert^2 + \lVert \beta \rVert^2} .
	\end{equation}
	Assume now that \eqref{eq:adjroots} is false. Then $\sqrt{\lVert \alpha \rVert^2 + \lVert \beta \rVert^2} \leq \sqrt{2}$, from which we get $\lVert \alpha \rVert = \lVert \beta \rVert = 1$, i.e., both $\alpha$ and $\beta$ are short roots; moreover, both inequalities \eqref{eq:adjroots2} and \eqref{eq:adjroots3} are equalities. This implies that on the one hand, $\gamma = \delta$ is long, and on the other hand, $\alpha$ and $\beta$ make an angle of $90^\circ$.
	The fact that we have \eqref{eq:adjroots} with an equality rather than an inequality also implies that $i = j = 2$, so $2\alpha + 2\beta = \gamma + \delta$. However, since $\gamma = \delta$, we now see that $\alpha + \beta = \gamma$, contradicting the fact that $\alpha$ and $\beta$ are adjacent.
\end{proof}

\begin{lemma}\label{le:trivial comrels}
	Let $ \alpha, \beta \in \Phi $ be adjacent and let
	\[ \exp_\alpha \colon \bar{L}_\alpha \to \Aut(\bar{L}) \colon a \mapsto \sum_{i \ge 0} \theta^{[i,\alpha]}_a \quad \text{and} \quad \exp_\beta \colon \bar{L}_\beta \to \Aut(\bar{L}) \colon b \mapsto \sum_{i \ge 0} \theta^{[i,\beta]}_b \]
	be an $ \alpha $-parametrization and a $ \beta $-parametrization, respectively. Then for all $ x_\alpha \in \bar{L}_\alpha $ and $ x_\beta \in \bar{L}_\beta $, the automorphisms $ \exp_\alpha(x_\alpha) $ and $ \exp_\beta(x_\beta) $ commute.
\end{lemma}
\begin{proof}
	Let $ x_\alpha \in \bar{L}_\alpha $ and $ x_\beta \in \bar{L}_\beta $. For each $ i \in \mathbb{N} $ and $ \zeta \in \{\alpha, \beta\} $ and $y \coloneq x_\zeta$, put $ \delta^i_\zeta \coloneq \theta^{[i,\zeta]}_{y} $. Since
	\[ \exp_\alpha(x_\alpha) \circ \exp_\beta(x_\beta) = \brackets*{\sum_{i\ge 0} \delta^i_\alpha} \circ \brackets*{\sum_{j \ge 0} \delta^j_\beta} = \sum_{i \ge 0} \sum_{j \ge 0} \delta_\alpha^i \circ \delta_\beta^j, \]
	it suffices to show that $ \delta_\alpha^i $ and $ \delta_\beta^j $ commute for all $ i,j \ge 0 $. If $ i,j \ge 2 $, then it follows from \cref{le:adj lincomb} that $ \delta_\alpha^i \circ \delta_\beta^j $ and $ \delta_\beta^j \circ \delta_\alpha^i $ are both zero on $ L_\gamma $ for all $ \gamma \in \Phi $. If $ i = 0 $ or $ j=0 $, then either $ \delta_\alpha^i $ or $ \delta_\beta^j $ is the identity map. It remains to consider the case that $ i=1 $ or $ j=1 $, say $ i=1 $. Then $ \delta_\alpha^i = \ad_{x_\alpha} $. By \cref{le:leibniz endo converse}, the endomorphisms $ \delta_\beta^0, \delta_\beta^1, \ldots $ satisfy the higher Leibniz rules~\ref{eq:star l}. Hence for all $ \gamma \in \Phi $, all $ j \ge 0 $ and all $ x_\gamma \in L_\gamma $, we have
	\begin{align*}
		\delta_\beta^j \circ \delta_\alpha^i(x_\gamma) &= \delta_\beta^j([x_\alpha, x_\gamma]) = \sum_{\ell = 0}^j [\delta_\beta^\ell(x_\alpha), \delta_{\beta}^{j-\ell}(x_\gamma)].
	\end{align*}
	Since $ \delta_\beta^\ell(x_\alpha) \in L_{\alpha + \ell \beta} $ for all $ \ell \ge 0 $ and $ \alpha $, $ \beta $ are adjacent, all summands are zero except for $ \ell = 0 $. Hence $ \delta_\beta^j \circ \delta_\alpha^i(x_\gamma) = [x_\alpha, \delta_\beta^j(x_\gamma)] = \delta_\alpha^i \circ \delta_\beta^j(x_\gamma) $. We conclude that $ \delta_\alpha^i $ and $ \delta_\beta^j $ commute in all cases, and hence $ \exp_\alpha(x_\alpha) $ and $ \exp_\beta(x_\beta) $ commute.
\end{proof}

In \cref{sec:weyl G2}, we will show that the reflections $\phihor$ and $\phibs$ from \cref{def:reflections,def:phibs} are Weyl elements. For this we have to check that their conjugation action permutes the root groups accordingly. Since $\phihor^2$ and $\phibs^2$ are the parity automorphisms of certain gradings by \cref{rem:reflection parity,rem:phibs formulas}, the following result allows us to deduce from $U_\alpha^\varphi = U_{\sigma(\alpha)}$ that $U_{\sigma(\alpha)}^\varphi = U_\alpha^{\varphi^2} = U_\alpha$, where $\alpha \in G_2$, $\varphi \in \{\phihor, \phibs\}$ and $\sigma \colon G_2 \to G_2$ the reflection induced by $\varphi$. This cuts the amount of necessary computations by half. Explicitly, see \cref{le:exp y conj inv,le:exp L-1 conj inv,le:phibs short 1,le:phibs short 2,le:phibs short 3,le:phibs long 2,le:phibs long 3}.

\begin{lemma}\label{le:exp parity conj}
	Let $p \in \mathbb{N}$ and assume that $\Phi$ has a $(2p+1)$-grading, that is, a decomposition $\Phi = \bigsqcup_{i=-p}^p \Phi_i$ such that $(\Phi_i + \Phi_j) \cap \Phi \subseteq \Phi_{i+j}$ for all $i,j \in \Z$ (where $\Phi_i \coloneq \emptyset$ if $\lvert i \rvert > p$) and $\Phi_{-i} = -\Phi_i$ for all $i \in \mathbb{Z}$. Put $\bar{L}_i \coloneq \bigoplus_{\alpha \in \Phi_i} \bar{L}_\alpha$ for all $i \in \mathbb{Z}$, so that $(\bar{L}_i)_{-p \le i \le p}$ is a $(2p+1)$-grading of $\bar{L}$. Let $\exp_\alpha \colon L_\alpha \to \Aut(\bar{L})$ be an $\alpha$-parametrization for some $\alpha \in \Phi$ and let $\varphi \colon \bar{L} \to \bar{L}$ denote the parity automorphism of $L$ with respect to the $(2p+1)$-grading, so that $\varphi(x_i) = (-1)^i x_i$ for all $i \in \mathbb{Z}$ and $x_i \in \bar{L}_i$. Then for all $a \in \bar{L}_\alpha$, we have
	\[ \varphi^{-1} \circ \exp_\alpha(a) \circ \varphi = \exp_\alpha\brackets[\big]{\varphi(a)}. \]
\end{lemma}
\begin{proof}
	For all $i \in \mathbb{Z}$ and $\beta \in \Phi_i$, we put $\epsilon_\beta \coloneq (-1)^i$, and also $\epsilon_0 \coloneq 1$. Then for all $\beta,\gamma \in \Phi$ with $\beta+\gamma \in \Phi$, we have $\epsilon_{\beta+\gamma} = \epsilon_\beta \epsilon_\gamma$. Since root strings in $\Phi$ are unbroken, it follows that $\epsilon_{\beta + i\gamma} = \epsilon_\beta \epsilon_\gamma^i$ for all $\beta,\gamma \in \Phi^0$ and $i \in \mathbb{Z}$ with $\beta+i\gamma \in \Phi^0$. Now let $\beta \in \Phi^0$ and $b \in \bar{L}_\beta$ be arbitrary. Then
	\begin{align*}
		\varphi^{-1} \circ \exp_\alpha(a) \circ \varphi(b) &= \epsilon_\beta \sum_{i \ge 0} \epsilon_{\beta + i\alpha} \theta_{a}^{[i]}(b) \quad \text{and} \\
		\exp_\alpha\brackets[\big]{\varphi(a)}(b) &= \exp_\alpha(\epsilon_\alpha a)(b) = \sum_{i \ge 0} \theta_{\epsilon_\alpha a}^{[i]}(b) = \sum_{i \ge 0} \epsilon_\alpha^i \theta_{a}^{[i]}(b).
	\end{align*}
	Since $\epsilon_\beta\epsilon_{\beta+i\alpha} = \epsilon_\beta^2 \epsilon_\alpha^i = \epsilon_\alpha^i$ for all $i \in \mathbb{Z}$, the assertion follows.
\end{proof}

The following general result will allow us to establish Axiom~\ref{def:rgg}\cref{def:rgg:nondeg} of $\Phi$-graded groups for both $\Phi = G_2$ and $\Phi=F_4$. The argument is the same as in the context of Chevalley groups (see, for instance, Corollary~3 of Lemma~18 in \cite{Steinberg-ChevGroups})

\begin{proposition}\label{pr:G2 nondeg}
	For each $ \alpha \in \Phi $, let $ \exp_\alpha \colon \bar{L}_\alpha \to \Aut(\bar{L}) $ be an $ \alpha $-parametrisa\-tion and let $ U_\alpha $ denote the corresponding root group. Further, let $ \Pi $ be any positive system of roots in $ \Phi $ and denote by $ U_+ $ and $ U_- $ the groups generated by $ (U_\alpha)_{\alpha \in \Pi} $ and $ (U_\alpha)_{\alpha \in -\Pi} $, respectively. Then $ U_+ \cap U_- = 1 $. In particular, the root groups $ (U_\alpha)_{\alpha \in \Phi} $ satisfy Axiom~\ref{def:rgg}\cref{def:rgg:nondeg}.
\end{proposition}
\begin{proof}
	Denote by $ \rootht \colon \Phi \to \Z $ the height function with respect to $ \Pi $, which we extend to a function on $ \Phi^0 \coloneq \Phi \cup \{0\} $ by setting $ \rootht(0) \coloneq 0 $. Put $ \ell \coloneq \lvert \Phi^0 \rvert $ and choose an ordering $ (\alpha_1, \ldots, \alpha_{\ell}) $ of the elements of $ \Phi^0 $ such $ \rootht(\alpha_i) < \rootht(\alpha_j) $ implies $ i<j $ for all $ i,j \in \{1,\ldots, \ell\} $.
	For all $ \alpha \in \Pi $ and all $ \phi \in U_{\alpha} $, we have the following property because $ \exp_\alpha $ is an $ \alpha $-parametrization and $ \rootht(\alpha) > 0 $: There exists a $ k $-linear map $ \phi' \colon \bar{L} \to \bar{L} $ such that $ \phi = \id_{\bar{L}} + \phi' $ and $ \phi'(\bar{L}_{\alpha_m}) \subseteq \sum_{m < j \le \ell} \bar{L}_{\alpha_j} $ for all $ m \in \{1,\ldots, \ell\} $. Since this holds for all $ \alpha \in \Pi $, we infer that all $ \phi \in U_+ $ have the same property. By similar arguments, all $ \phi \in U_- $ can be written as $ \phi = \id_{\bar{L}} + \phi' $ for some $ \phi' \in \End(\bar{L}) $ with $ \phi'(\bar{L}_{\alpha_m}) \subseteq \sum_{1 \le j < m} \bar{L}_{\alpha_j} $ for all $ m \in \{1,\ldots, \ell\} $. Since the sum $ \bigoplus_{j=1}^{\ell} \bar{L}_{\alpha_j} $ is direct, it follows that $ U_+ \cap U_- = \{\id_{\bar{L}}\} $.
\end{proof}

\begin{remark}
	Let $ (\alpha_1, \ldots, \alpha_{\ell}) $ be as in the proof of \cref{pr:G2 nondeg}. Then any element of $ \Aut(L) $ has a \enquote{matrix} with respect to the direct sum decomposition $ L = \bigoplus_{j=1}^{\ell} L_{\alpha_j} $. The argument in the proof above says precisely that this matrix is lower triangular unipotent for all elements of $ U_+ $ and upper triangular unipotent for all elements of $ U_- $.
\end{remark}

\subsection{Definition of exponential maps in \texorpdfstring{$L$}{L}}\label{subsec:G2 exp:formulas}

For the rest of this section, we assume that $ (J,J') $ is a cubic norm pair over $ k $ and let $ L = L(J,J') $ denote the $ G_2 $-graded Lie algebra from \cref{sec:lie const}. Our next goal is to construct an $\alpha$-parametrization $L_\alpha \to \Aut(L) \colon a \mapsto e_a$ in $ L $ for each root $\alpha \in G_2$. We do so by giving explicit formulas involving the maps $N$ and $\sharp$, so our construction does not work for the more general linear cubic norm pairs introduced in \cref{subsec:linear CNP}. As for the reflections defined in \cref{sec:reflections}, the well-definedness of these maps is not evident and will be checked in \cref{le:exp well-defined} (see also \cref{rem:refl wd}).

\begin{remark}\label{rem:param domain}
	For the scope of this remark, we identify $ G_2^0 $ with a subset of $ \Z^2 $ as in \cref{rem:G2 grid}. Let $ \alpha \in G_2 $ and fix an isomorphism of $ k $-modules $ \vartheta \colon P_\alpha \cong L_\alpha $ (not necessarily the one from \cref{def:lie G2 natural rhom}). A map $ \exp_\alpha \colon P_\alpha \to \Aut(L) $ may be called an \emph{$ \alpha $-parametrization with respect to $ \vartheta $} if $ \mathord{\exp_\alpha} \circ \vartheta^{-1} $ is an $ \alpha $-parametrization. Note that properties \cref{item:param a,item:param c,item:param sum} in \cref{def:param} are independent of the choice of $ \vartheta $, but \cref{item:param b} is not. In our study of the $ G_2 $-grading of $ L $, we will always (implicitly) choose the \enquote{natural} isomorphisms from \cref{def:lie G2 natural rhom}, so the distinction between parametrizations as maps on $ P_\alpha $ and parametrizations as maps on $ L_\alpha $ is barely visible. When we refine the $ G_2 $-grading to an $ F_4 $-grading, however, it will be more natural to choose non-trivial isomorphisms between $ L_\beta $ (for $ \beta \in F_4 $) and the corresponding parametrising structure $ P_\beta $, and then the distinction between $ L_\beta $ and $ P_\beta $ will become crucial. For details, see \cref{subsec:chev}.
\end{remark}

\begin{definition}\label{def:exp xy}
	Let $\lambda \in k$. We define endomorphisms $e_{\lambda x}$ and $e_{\lambda y}$ of the $k$-module $L$ by setting
	\begin{align*}
		e_{\lambda x}(x) &\coloneq x, \\
		e_{\lambda x}\brackets[\big]{(\nu, c, c', \rho)_-} &\coloneq (\nu, c, c', \rho)_-, \\
		e_{\lambda x}(c') &\coloneq c', \\
		e_{\lambda x}(\dd_{c,c'}) &\coloneq \dd_{c,c'}, \quad
		e_{\lambda x}(\xi) \coloneq \xi + 2\lambda x, \quad
		e_{\lambda x}(\zeta) \coloneq \zeta + \lambda x, \\
		e_{\lambda x}(c) &\coloneq c, \\
		e_{\lambda x}\brackets[\big]{(\nu, c, c', \rho)_+} &\coloneq (\nu, c, c', \rho)_+ - \lambda (\nu, c, c', \rho)_-, \\
		e_{\lambda x}(y) &\coloneq y + \lambda \xi + \lambda^2 x
	\end{align*}
	and
	\begin{align*}
		e_{\lambda y}(x) &\coloneq x - \lambda \xi + \lambda^2 y, \\
		e_{\lambda y}\brackets[\big]{(\nu, c, c', \rho)_-} &\coloneq (\nu, c, c', \rho)_- + \lambda (\nu, c, c', \rho)_+, \\
		e_{\lambda y}(c') &\coloneq c', \\
		e_{\lambda y}(\dd_{c,c'}) &\coloneq \dd_{c,c'}, \quad
		e_{\lambda y}(\xi) \coloneq \xi - 2\lambda y, \quad
		e_{\lambda y}(\zeta) \coloneq \zeta - \lambda y, \\
		e_{\lambda y}(c) &\coloneq c, \\
		e_{\lambda y}\brackets[\big]{(\nu, c, c', \rho)_+} &\coloneq(\nu, c, c', \rho)_+, \\
		e_{\lambda y}(y) &\coloneq y
	\end{align*}
	for all $\nu, \rho \in k$, $c \in J$, $c' \in J'$.
\end{definition}

\begin{definition}\label{def:exp L1}
	Let $\lambda, \mu \in k$, $b \in J$, $b' \in J'$, put $a \coloneq (\lambda, b,b',\mu)_+ \in L_1$ and assume that at most one of the elements $\lambda,b,b',\mu$ is non-zero. We define an endomorphism $e_a$ of the $k$-module $L$ by setting
	\begin{align*}
		e_a(x) &\coloneq x + (\lambda, b, b', \mu)_- - b^\sharp - (b')^\sharp + \bigl(-N(b), 0, 0, N(b')\bigr)_+, \\
		e_a\brackets[\big]{(\nu, c, c', \rho)_-} &\coloneq (\nu, c, c', \rho)_- \\
		&\qquad {} + \lambda c' - c \times b + \nu b' \\
		&\qquad {} - \dd_{b,c'} - \dd_{c,b'} \\
		&\qquad {} + \brackets[\big]{\rho \lambda - T(b,c')}(\xi - \zeta) + \brackets[\big]{T(c,b') - \nu\mu} \zeta \\
		&\qquad {} + \rho b - c' \times b' + \mu c \\
		&\qquad {} + \bigl(-\rho \lambda^2 - T(b^\sharp, c), U_b(c') + \nu (b')^\sharp, \\
		&\hspace{3cm} -\rho b^\sharp - U_{b'}(c), T(c', (b')^\sharp) + \nu\mu^2\bigr)_+ \\
		&\qquad {} +\brackets[\big]{-\rho N(b) - \nu N(b')}y, \\
		e_a(c') &\coloneq c' + \brackets[\big]{T(b,c'), c' \times b', \mu c', 0}_+ - T((b')^\sharp, c') y, \\
		e_a(\dd_{c,c'}) &\coloneq \dd_{c,c'} + \bigl(\lambda T(c,c'), -D_{c,c'}(b) + T(c,c') b, \\
		& \hspace{4cm}D_{c',c}(b') - T(c,c')b', -\mu T(c,c')\bigr)_+, \\
		e_a(\xi) &\coloneq \xi - a, \\
		e_a(\zeta) &\coloneq \zeta + (\lambda, 0, -b', -2\mu)_+ , \\
		e_a(c) &\coloneq c - \brackets[\big]{0, \lambda c, c \times b, T(c,b')}_+- T(c, b^\sharp) y, \\
		e_a\brackets[\big]{(\nu, c, c', \rho)_+} &\coloneq (\nu, c, c', \rho)_+ + \brackets[\big]{T(b,c') - T(c,b') + \mu \nu - \lambda \rho} y, \\
		e_a(y) &\coloneq y
	\end{align*}
	for all $\nu, \rho \in k$, $c \in J$, $c' \in J'$.
\end{definition}

\begin{definition}\label{def:exp L-1}
	Let $\lambda, \mu \in k$, $b \in J$, $b' \in J'$, put $a \coloneq (\lambda, b,b',\mu)_- \in L_{-1}$ and assume that at most one of the elements $\lambda,b,b',\mu$ is non-zero. We define an endomorphism $e_a$ of the $k$-module $L$ by setting
	\begin{align*}
		e_a(x) &\coloneq x, \\
		e_a\brackets[\big]{(\nu, c, c', \rho)_-} &\coloneq (\nu, c, c', \rho)_- + \brackets[\big]{T(b,c') - T(c,b') + \mu \nu - \lambda \rho} x, \\
		e_a(c') &\coloneq c' + \brackets[\big]{T(b,c'), c' \times b', \mu c', 0}_- - T((b')^\sharp, c') x, \\
		e_a(\dd_{c,c'}) &\coloneq \dd_{c,c'} + \bigl(\lambda T(c,c'), -D_{c,c'}(b) + T(c,c')b, \\
		& \hspace{4cm} D_{c',c}(b') - T(c,c') b', -\mu T(c,c')\bigr)_-, \\
		e_a(\xi) &\coloneq \xi + (\lambda, b, b', \mu)_-, \\
		e_a(\zeta) &\coloneq \zeta + (2\lambda, b, 0, -\mu)_-, \\
		e_a(c) &\coloneq c - \brackets[\big]{0, \lambda c, c \times b, T(c,b')}_- - T(c, b^\sharp) x, \\
		e_a\brackets[\big]{(\nu, c, c', \rho)_+} &= (\nu, c, c', \rho)_+ \\
		&\qquad {} - \lambda c' + b \times c - \nu b' \\
		&\qquad {} + \dd_{b,c'} + \dd_{c,b'} \\
		&\qquad {} + \brackets[\big]{\lambda\rho - T(b,c')} \cdot \zeta + \brackets[\big]{T(c,b') - \mu\nu} \cdot (\xi-\zeta) \\
		&\qquad {} - \rho b + b' \times c' - \mu c \\
		&\qquad {} + \bigl( \rho\lambda^2 + T(c,b^\sharp), -U_b(c') - \nu (b')^\sharp, \\
		&\hspace{3cm} \rho b^\sharp + U_{b'}(c), -T((b')^\sharp , c') -\mu^2\nu \bigr)_- \\
		&\qquad {} + \brackets[\big]{\rho N(b) + \nu N(b')}x, \\
		e_a(y) &\coloneq y - (\lambda, b, b', \mu)_+ - b^\sharp - (b')^\sharp  + \brackets[\big]{-N(b), 0, 0, N(b')}_-
	\end{align*}
	for all $\nu, \rho \in k$, $c \in J$, $c' \in J'$.
\end{definition}

\begin{definition}\label{def:exp JJ'}
	Let $ b \in J \subseteq L_0 $, $b' \in J' \subseteq L_0$. We define endomorphisms $e_b$ and $e_{b'}$ of the $k$-module $L$ by setting
	\begin{align*}
		e_b(y) &\coloneq y, \\
		e_b\brackets[\big]{(\nu, c, c', \rho)_\pm} &\coloneq \brackets[\big]{\nu, c+\nu b, c' + b \times c + \nu b^\sharp, \rho + T(b,c') + T(c,b^\sharp) + \nu N(b)}_\pm, \\
		e_b(c') &\coloneq c' - \dd_{b,c'} + U_b(c'), \\
		e_b(\dd_{c,c'}) &\coloneq \dd_{c,c'} - D_{c,c'}(b), \quad e_b(\xi) \coloneq \xi, \quad e_b(\zeta) \coloneq \zeta - b, \\
		e_b(c) &\coloneq c, \\
		e_b(x) &\coloneq x,
	\end{align*}
	and
	\begin{align*}
		e_{b'}(y) &\coloneq y, \\
		e_{b'}\brackets[\big]{(\nu, c, c', \rho)_\pm} &\coloneq \bigl(\nu - T(c,b') + T((b')^\sharp, c') - \rho N(b'), \\
		& \qquad\qquad\qquad\qquad c-b' \times c' + \rho(b')^\sharp, c' - \rho b', \rho\bigr)_\pm, \\
		e_{b'}(c') &\coloneq c', \\
		e_{b'}(\dd_{c,c'}) &\coloneq \dd_{c,c'} + D_{c',c}(b'), \quad e_{b'}(\xi) \coloneq \xi, \quad e_{b'}(\zeta) \coloneq \zeta + b', \\
		e_{b'}(c) &\coloneq c + \dd_{c,b'} + U_{b'}(c), \\
		e_{b'}(x) &\coloneq x,
	\end{align*}
	for all $ \nu, \rho \in k $, $ c \in J $ and $ c' \in J' $.
\end{definition}

\begin{remark}\label{rem:param}
	The formulas in \cref{def:exp xy,def:exp L1,def:exp L-1,def:exp JJ'} have been obtained by applying the usual definition of an exponential map
	\[ e_a = \id + \ad_a + \tfrac{1}{2} \ad_a^2 + \tfrac{1}{6} \ad_a^3 + \tfrac{1}{24} \ad_a^4, \]
	simplifying the result using the facts $b \times b = 2 b^\sharp$, $D_{b,b'}(b) = 2 U_b(b')$ and $T(b, b \times b) = 6 N(b)$ for all $b \in J$, $b' \in J'$ and interpreting the resulting formula over the integers.
	In particular, it is obvious from this procedure that the properties \cref{item:param a,item:param b,item:param c} in \cref{def:param} are satisfied. To show that the maps $ a \mapsto e_a $ are parametrizations, it remains to show that property \cref{item:param sum} is satisfied (which implies that $ e_a $ is invertible with inverse $ e_{-a} $) and that, most importantly, all maps of the form $ e_a $ lie in $ \End(L) $ (that is, are compatible with the Lie bracket).
\end{remark}

\begin{remark}\label{rem:exp CNP notation}
	If necessary, we will write $e_a^{(J,J')}$ (where $a \in L(J,J')_{ij}$ for some $(i,j) \ne (0,0)$) to clarify that we consider the map $e_a$ defined on the Lie algebra $L(J,J')$. We will use a similar notation for the reflections defined in \cref{sec:reflections}. Note that if $(\iota, \iotainv) \colon (J_1, J_1') \to (J_2, J_2')$ is a surjective homomorphism of cubic norm pairs, then
	\[ L(\iota, \iotainv) \circ e_{a_1}^{(J_1, J_1')} = e_{a_2}^{(J_2, J_2')} \circ L(\iota, \iotainv) \]
	for all $a_1 \in L(J_1, J_1')_{ij}$ where $(i,j) \ne (0,0)$ and $a_2 \coloneq L(\iota, \iotainv)(a_1)$. In particular,
	\[ e_{a_1}^{(J_1, J_1)} = L(\iota, \iotainv)^{-1} \circ e_{a_2}^{(J_2, J_2')} \circ L(\iota, \iotainv) \]
	if $(\iota, \iotainv)$ is bijective.
\end{remark}

\begin{lemma}\label{le:exp well-defined}
	The maps $e_{(\lambda,b,b',\mu)_+}$, $e_{\lambda x}$, $e_{\lambda y}$, $e_b$ and $e_{b'}$ are well-defined for all $\lambda, \mu \in k$, $b \in J$, $b' \in J'$ such that at most one of $\lambda,b,b',\mu$ is non-zero, respectively.
\end{lemma}
\begin{proof}
	In this proof, all summations are understood to run over $a \in J$ and $a' \in J'$. We proceed as in the proof of \cref{le:reflections well-defined}. We have to show that if
	\[ \sum \lambda_{a,a'} \dd_{a,a'} = \lambda_\xi \cdot (2\zeta - \xi) \]
	for certain, finitely many scalars $\lambda_{a,a'}, \lambda_\zeta, \lambda_\xi \in k$, then both sides of this equation evaluate to the same term under the assignments defining $e_{(\lambda,b,b',\mu)_+}$ and $e_{\rho x}$. They are both fixed by the assignments defining $e_{\rho x}$, so $e_{\rho x}$ is well defined. Under the assignments defining $e_{(\lambda,b,b',\mu)_+}$, they evaluate to
	\begin{gather*}
		\sum \lambda_{a,a'} \dd_{a,a'} + \sum \lambda_{a,a'} t_{a,a'}  \quad \text{and} \\
		\lambda_\xi \cdot (2\zeta - \xi) + \lambda_\xi (3\lambda, b, -b', -3\mu)_+ ,
	\end{gather*}
	respectively, where
	\[ t_{a,a'} \coloneq \brackets[\big]{\lambda T(a,a'), -D_{a,a'}(b) + T(a,a') b, D_{a,a'}(b') - T(a,a')b', -\mu T(a,a')}. \]
	By \cref{le:lin comb}, we have
	\begin{align*}
		\sum \lambda_{a,a'} t_{a,a'} = \brackets[\big]{3\lambda\lambda_\xi, -2\lambda_\xi b + 3\lambda_\xi b, 2\lambda_\xi b' - 3\lambda_\xi b', -3\lambda_\xi \mu}.
	\end{align*}
	We conclude that also $ e_{(\lambda,b,b',\mu)_+}$ is well-defined. The well-definedness of the remaining maps can be verified in a similar way.
\end{proof}

\begin{remark}
	In fact, the well-definedness of the maps that we left out in the proof of \cref{le:exp well-defined} does not have to be verified explicitly: We will see in \cref{pr:exp conj aut summary} that all remaining maps are conjugate to maps for which we have verified the well-definedness in the proof of \cref{le:exp well-defined}, and hence they are well-defined themselves.
\end{remark}

\begin{remark}\label{rem:aut equality}
	Since $ [x,L_1] = L_{-1} $, $ [y,L_{-1}] = L_1 $, $ [L_i, L_i] = L_{2i} $ for $ i \in \{1,-1\} $ and $ [L_1, L_{-1}] = L_0 $, the sets  $\{x\} \cup L_1$ and $\{y\} \cup L_{-1}$ are generating sets of the Lie algebra $L$. For all $f,g \in \End(L)$, it follows that $f=g$ if and only if $f$ and $g$ agree on one of these sets.
\end{remark}

\begin{proposition}\label{pr:param}
	Identity $G_2$ with a subset of $\Z^2$ as in \cref{rem:G2 grid}, so that $L_\alpha$ is defined for all $\alpha \in G_2$. Then for all $ \alpha \in G_2 $, the map $ \exp_\alpha \colon L_\alpha \to \Aut(L) \colon a \mapsto e_a $ (with $ e_a $ as given by \cref{def:exp xy,def:exp L1,def:exp L-1,def:exp JJ'}) is a well-defined $ \alpha $-parametrization in $ L $.
\end{proposition}
\begin{proof}
	By \cref{rem:param}, we only have to show that for all $ \alpha \in G_2 $, the map $ \exp_\alpha $ is additive and has image in $ \Aut(L) $.
	In \cref{sec:leibniz}, we show that the maps $ e_{\rho x} $ for $ \rho \in k $ and $ e_{(0,b,0,0)_+} $ for $ b \in J $ are endomorphisms of the Lie algebra $ L $. We proceed to show that $ \exp_\alpha $ is additive for $ \alpha \in \{(-2,-1), (1,0)\} $. For all $ s,t \in k $, we have $ e_{(s+t)x} = e_{sx} \circ e_{tx} $ because both sides act in the same way on the generating set $ \{y\} \cup L_{-1} $ of $ L $. In particular, $ e_{st} $ is invertible for all $ s \in k $, so the image of $ \exp_{(-2,-1)} $ lies in $ \Aut(L) $. We conclude that the map $ L_{-2,-1} \to \Aut(L) $ is a $ (-2,-1) $-parametrization.
	
	Similarly, the identity $ e_{(0,b+c,0,0)_+} = e_{(0,b,0,0)_+} \circ e_{(0,c,0,0)_+} $ for $b,c \in J$ follows from the computations
	\begin{align*}
		(e_{(0,b,0,0)_+} \circ e_{(0,c,0,0)_+})(x) &=e_{(0,b,0,0)_+}\brackets[\big]{x + (0,c,0,0)_- - c^\sharp + (-N(c), 0, 0, 0)_+} \\
		&= x + (0,b,0,0)_- - b^\sharp + \brackets[\big]{-N(b), 0, 0, 0}_+ \\*
		&\qquad{} + (0,c,0,0)_- - b \times c + \brackets[\big]{-T(b^\sharp, c), 0, 0, 0}_+ \\*
		&\qquad{} - c^\sharp + \brackets[\big]{-T(b,c^\sharp), 0, 0, 0}_+ + \brackets[\big]{-N(c),0,0,0}_+ \\
		&= x + (0,b+c,0,0)_- - (b^\sharp + c^\sharp + b \times c) \\*
		&\qquad{}+ \brackets[\big]{-N(b) - N(c) - T(b^\sharp, c) - T(b, c^\sharp), 0, 0, 0}_+ \\
		&= x + (0,b+c,0,0)_- - (b+c)^\sharp \\*
		&\qquad{}+ \brackets[\big]{-N(b+c), 0, 0, 0}_+
	\end{align*}
	and
	\begin{align*}
		(e_{(0,b,0,0)_+} \circ e_{(0,c,0,0)_+})\brackets[\big]{(\lambda, a, a', \mu)_+} &= (\lambda, a, a', \mu) + T(c,a')y + T(b,a')y
	\end{align*}
	for $ (\lambda, a, a', \mu)_+ \in L_1 $. Hence the map $ L_{1,0} \to \Aut(L) $ is a $ (1,0) $-parametrization.
	For arbitrary $ \alpha \in G_2 $, we can now deduce from \cref{pr:exp conj aut summary} that $ \exp_\alpha $ is additive with image in $ \Aut(L) $, as desired.
\end{proof}

\begin{definition}\label{def:G2 rootgr}
	For all $ \alpha \in G_2 $, we will in the following always denote by
	\[ \exp_\alpha \colon L_\alpha \to \Aut(L) \colon a \mapsto e_a \]
	the $ \alpha $-parametrization from \cref{pr:param} and we denote by $ U_\alpha \coloneq \exp_\alpha(L_\alpha) $ the associated root group.
\end{definition}

\subsection{Constructing the \texorpdfstring{$G_2$}{G2}-graded group}\label{subsec:G2 exp:rgg}

We now show that the parametrizations previously defined satisfy some commutator relations. At first, we introduce a notation that is analogous to the one used in \cite[(16.8)]{TW02}.

\begin{notation}
	Identifying $G_2$ with a subset of $\Z^2$ as in \cref{rem:G2 grid}, we consider the following roots:
	\begin{align*}
		\alpha_1 &\coloneq \delta = (1,0), & \alpha_6 &\coloneq \gamma = (-2, -1), & \alpha_2 &\coloneq (1,-1) = \gamma + 3\delta, \\
		\alpha_3 &\coloneq (0,-1) = \gamma + 2\delta, & \alpha_4 &\coloneq (-1,-2) = 2\gamma +3\delta, & \alpha_5 &\coloneq (-1,-1) = \gamma+\delta.
	\end{align*}
	Note that $\{\alpha_1,\alpha_6\}$ is a system of simple roots with associated system $\{\alpha_1, \ldots, \alpha_6\}$ of positive roots. For all $i \in \{1, \ldots, 6\}$, we define a map
	\[ x_i \colon P_{\alpha_i} \to \Aut(L) \colon a \mapsto \exp_{\alpha_i}\brackets[\big]{\rhomlieG{\alpha_i}(a)} \]
	where $P_\alpha \in \{J,J',k\}$ is the space parametrising $L_\alpha$ and $\rhomlieG{\alpha} \colon P_\alpha \to L_\alpha$ is the \enquote{natural} isomorphism from \cref{def:lie G2 natural rhom}.
\end{notation}

\begin{proposition}\label{pr:G2 comm}
	For all $\lambda, \mu \in k$, $b, d \in J$ and $b' \in J'$, we have the following commutator relations between the positive root groups $U_\alpha$.
	\begin{align}
		&[x_1(b), x_3(b')] = x_2(T(b,b')), \label{G2 comm 1} \\
		&[x_3(b'), x_5(b)] = x_4(-T(b,b')), \label{G2 comm 2} \\
		&[x_2(\lambda), x_6(\mu)] = x_4(\lambda \mu), \label{G2 comm 3} \\
		&[x_1(b), x_5(d)] = x_2(T(d, b^\sharp)) \circ x_3(- b \times d) \circ x_4(T(b, d^\sharp)), \label{G2 comm 4} \\
		&[x_1(b), x_6(\lambda)] 	= x_2(-\lambda N(b)) \circ x_3(\lambda b^\sharp) \circ x_4(\lambda^2 N(b)) \circ x_5(\lambda b) . \label{G2 comm 5}
	\end{align}
	All other commutator relations are trivial.
\end{proposition}
\begin{proof}
	Observe first that the trivial commutator relations are all dealt with by \cref{le:trivial comrels}.
	For each of the remaining five cases, we will prove the equality using \cref{rem:aut equality}.
	
	To prove \eqref{G2 comm 1}, we have to show that
	\[ e_{(0,-b,0,0)_+} \circ e_{-b'} \circ e_{(0,b,0,0)_+} \circ e_{b'} = e_{(T(b,b'),0,0,0)_+} . \]
	When we apply both sides of this equation on $x$, we get the same result $x + \bigl( T(b, b'), 0, 0, 0 \bigr)_-$.
	By \cref{rem:aut equality}, it remains to show that both sides give the same result when applied on an element $(\nu, c, c', \rho)_+ \in L_1$.
	A short computation shows that we twice get the result $(\nu, c, c', \rho)_+ - \rho T(b, b') y$.
	
	To prove \eqref{G2 comm 2}, we have to show that
	\[ e_{-b'} \circ e_{(0,-b,0,0)_-} \circ e_{b'} \circ e_{(0,b,0,0)_-} = e_{(-T(b,b'),0,0,0)_-} . \]
	When we apply both sides of this equation on $y$, we get the same result $y + \bigl( T(b, b'), 0, 0, 0 \bigr)_+$.
	By \cref{rem:aut equality}, it remains to show that both sides give the same result when applied on an element $(\nu, c, c', \rho)_- \in L_{-1}$.
	A short computation shows that we twice get the result $(\nu, c, c', \rho)_- + \rho T(b, b') x$.

	To prove \eqref{G2 comm 3}, we have to show that
	\[ e_{(-\lambda,0,0,0)_+} \circ e_{-\mu x} \circ e_{(\lambda,0,0,0)_+} \circ e_{\mu x} = e_{(\lambda\mu,0,0,0)_-} . \]
	Applying both sides on $y$ gives the same result $y - (\lambda\mu, 0, 0, 0)_+$
	and applying both sides on $(\nu, c, c', \rho)_- \in L_{-1}$ gives the same result $(\nu, c, c', \rho)_- - \lambda\mu\rho x$.
	
	The remaining two commutator relations are considerably more difficult, but the method is similar.
	To prove \eqref{G2 comm 4}, we have to show that
	\begin{multline*}
		e_{(0,-b,0,0)_+} \circ e_{(0,-d,0,0)_-} \circ e_{(0,b,0,0)_+} \circ e_{(0,d,0,0)_-} \\*
		= e_{(T(d,b^\sharp),0,0,0)_+} \circ e_{-b \times d} \circ e_{(T(b,d^\sharp),0,0,0)_-} .
	\end{multline*}
	When we apply both sides of this equation on $y$, we get the same result $y - \bigl( T(b, d^\sharp), 0, 0, 0 \bigr)_+$.
	Again, by \cref{rem:aut equality}, it remains to show that both sides give the same result when applied on an element $(\nu, c, c', \rho)_- \in L_{-1}$. Directly computing the left hand side now turns out to be challenging, so instead, we verify that
	\begin{multline*}
		e_{(0,-d,0,0)_-} \circ e_{(0,b,0,0)_+} \circ e_{(0,d,0,0)_-}(\nu, c, c', \rho)_- \\
		= e_{(0,b,0,0)_+} \circ e_{(T(d,b^\sharp),0,0,0)_+} \circ e_{-b \times d} \circ e_{(T(b,d^\sharp),0,0,0)_-}(\nu, c, c', \rho)_- .
	\end{multline*}
	The left-hand side of this equation is equal to
	\begin{multline*}
		e_{(0,b,0,0)_+}(\nu, c, c', \rho)_- - \rho T(b, d^\sharp) x \\
		\begin{aligned}
			&+ \bigl( T(d, c \times b) + T(U_b(c'), d^\sharp) - \rho N(b) N(d) + T(d,c')T(d,b^\sharp) , \\
			&\hspace*{12ex} - D_{b, c'}(d) + T(b, c')d + T(d, c')b + \rho U_d(b^\sharp), \ \rho b \times d, \ 0 \bigr)_- \\
			&+ \bigl( -d \times U_b(c') + \rho N(b) d^\sharp - T(d, c') b^\sharp \bigr) + \rho \dd_{d, b^\sharp} - \rho T(d, b^\sharp) \zeta \\
			&+ \bigl( -T(d, c') N(b), \ - \rho N(b) d, \ 0, \ 0 \bigr)_+
		\end{aligned}
	\end{multline*}
	and the right-hand side is equal to
	\begin{multline*}
		e_{(0,b,0,0)_+}(\nu, c, c', \rho)_- - \rho T(b, d^\sharp) x \\
		\begin{aligned}
			&+ \bigl( - \rho T(b, d^\sharp) T(d, b^\sharp) + T(c, b \times d) + T((b \times d)^\sharp, c') + \rho N(b \times d) , \\
			&\hspace*{12ex} - \rho T(b, d^\sharp) b + (b \times d) \times c' + \rho (b \times d)^\sharp, \ \rho b \times d, \ 0 \bigr)_- \\
			&+ \bigl( T(d, b^\sharp) (c' + \rho b \times d) + \rho T(b, d^\sharp) b^\sharp - b \times (c' \times (b \times d)) - \rho (b \times d)^\sharp \times b \bigr) \\
			&\hspace*{12ex} + \rho T(d, b^\sharp) (\xi - \zeta) - \rho \dd_{b, b \times d} - \rho T(b, b \times d) (\xi - \zeta) \\
			&+ \bigl( T(d, b^\sharp) T(b, c' + \rho b \times d) - \rho T(d, b^\sharp)^2 + \rho T(b, d^\sharp) N(b) \\
			&\hspace*{12ex} - T((b \times d) \times c' + \rho (b \times d)^\sharp, b^\sharp), \ - \rho T(d, b^\sharp) b + \rho U_b(b \times d), \ 0, \ 0 \bigr)_+ .
		\end{aligned}
	\end{multline*}
	The equality of these two expressions now follows from \cref{def:U-operators}, identities \cref{le:CNP identities:7,le:CNP identities:10,le:CNP identities:16,le:CNP identities:26,le:CNP identities:29} in \cref{le:CNP identities}, \cref{le:triple}\cref{le:triple:norm r} and \cref{eq:new 1,eq:new 2,eq:new 3} in \cref{le:CNP new}.
	
	Finally, to prove \eqref{G2 comm 5}, we have to show that
	\begin{multline*}
		e_{(0,-b,0,0)_+} \circ e_{-\lambda x} \circ e_{(0,b,0,0)_+} \circ e_{\lambda x} \\*
		= e_{(-\lambda N(b),0,0,0)_+} \circ e_{\lambda b^\sharp} \circ e_{(\lambda^2 N(b),0,0,0)_-} \circ e_{(0, \lambda b, 0, 0)_-}.
	\end{multline*}
	When we apply both sides of this equation on $y$, we get the same result
	\[ y + \bigl( 2 \lambda^2 N(b), -\lambda b, 0, 0 \bigr)_+ - \lambda^2 b^\sharp + \bigl( -\lambda^3 N(b), 0, 0, 0 \bigr)_- , \]
	where we have used $T(b, b^\sharp) = 3 N(b)$, $N(\lambda b) = \lambda^3 N(b)$ and $(\lambda b)^\sharp = \lambda^2 b^\sharp$.
	Again, by \cref{rem:aut equality}, it remains to show that both sides also give the same result when applied on an element $(\nu, c, c', \rho)_- \in L_{-1}$.
	For the left-hand side, this gives
	\begin{multline*}
		(\nu, c, c', \rho)_- + \bigl( \lambda T(b, c') - \rho \lambda^2 N(b) \bigr) x \\
		\begin{aligned}
			&+ \bigl( -\lambda T(c, b^\sharp),\ - \lambda T(b, c') b + \rho \lambda^2 N(b) b + \lambda U_b(c'),\ - \rho \lambda b^\sharp,\ 0 \bigr)_- \\
			&+ \bigl( - \lambda T(b, c') b^\sharp + \rho \lambda^2 N(b) b^\sharp + \lambda U_b(c') \times b \bigr) \\
			&\hspace*{12ex} - \rho \lambda \dd_{b, b^\sharp} + \rho \lambda N(b) \xi - \rho \lambda T(b, b^\sharp) (\xi - \zeta) \\
			&+ \bigl( \lambda T(b, c') N(b) - \rho \lambda^2 N(b)^2 - \lambda T(U_b(c'), b^\sharp),\ \rho \lambda N(b) b - \rho \lambda U_b(b^\sharp) ,\ 0,\ 0 \bigr)_+
		\end{aligned}
	\end{multline*}
	while the right-hand side results in
	\begin{multline*}
		(\nu, c, c', \rho)_- + \bigl( \lambda T(b, c') - \rho \lambda^2 N(b) \bigr) x \\
		\begin{aligned}
			&+ \bigl( -\lambda T(c, b^\sharp),\ - \lambda b^\sharp \times c' + \rho \lambda^2 N(b) b,\ - \rho \lambda b^\sharp,\ 0 \bigr)_- \hspace*{20ex} \\
			&+ \bigl( -\lambda N(b) c' + \rho \lambda^2 N(b) b^\sharp \bigr) - \rho \lambda N(b) (\xi - \zeta) \\
			&+ \bigl( - \rho \lambda^2 N(b)^2,\ 0 ,\ 0,\ 0 \bigr)_+ .
		\end{aligned}
	\end{multline*}
	The equality of these two expressions now follows from \cref{def:U-operators}, \cref{le:triple}\cref{le:triple:aasharp} and \cref{eq:new 4,eq:new 5,eq:new 6} in \cref{le:CNP new}.
\end{proof}

\begin{theorem}\label{thm:G2 graded group}
	Let $ (J,J') $ be a cubic norm pair over the commutative ring $ k $. Then the family $ (U_\alpha)_{\alpha \in G_2} $ from \cref{def:G2 rootgr} satisfies Axioms~\ref{def:rgg}\cref{def:rgg:comm,def:rgg:nondeg}. If there exists $ p \in J $ such that $ N(p) $ is invertible (e.g., if $ (J,J') $ is a cubic norm structure), then it is a $ G_2 $-grading of the subgroup of $ \Aut(L(J,J')) $ that it generates.
\end{theorem}
\begin{proof}
	This follows from \cref{pr:G2 comm,pr:G2 nondeg,pr:G2 weyl}.
\end{proof}

\section{The Tits index \texorpdfstring{$F_4 \to G_2$}{F4→G2}}\label{sec:index}

In this section, we describe how to obtain $G_2$\dash gradings of groups and Lie algebras from $F_4$-gradings of the same objects (\cref{pr:lie index,pr:rgg index}). This purely combinatorial construction is similar to that of twisted Chevalley groups (see \cite[Chapter~13]{Carter72}) and also to the construction of $ H_\ell $-graded groups for $ \ell \in \{2,3,4\} $ in \cite[Section~4]{BW25}. It relies on a certain surjective map $F_4^0 \to G_2^0$ (\cref{le:index char}) which is closely related to the Tits index that realises the Weyl group of $G_2$ as a subgroup of the Weyl group of $ F_4 $. In the formal setup of \cite[20.1]{MPW15}, this Tits index is the tuple $ (F_4, \{\id\}, \{3,4\}) $, which is also encoded by \cref{fig:tits-index} (following the conventions of \cite{Tits65}, see also \cite[34.2]{MPW15}). We will informally refer to $ (F_4, \{\id\}, \{3,4\}) $ by the name \enquote{$ F_4 \to G_2 $}. While this Tits index and the associated embedding of Weyl groups are known, we are not aware of any reference that describes the associated map $ F_4^0 \to G_2^0 $ of root systems. 

In \cref{sec:lie F4,sec:refining aut}, we will reverse the construction in this section by showing that certain (but, of course, not all) $G_2$-gradings can be refined to obtain $F_4$-gradings.

\begin{figure}
	\centering\begin{tikzpicture}[scale=1.5]
		\draw[opacity=0.7, rounded corners] (0.85, 0.1) rectangle (1.15, -0.1);
		\draw[opacity=0.7, rounded corners] (1.85, 0.1) rectangle (2.15, -0.1);
		\draw (1,0) -- (2, 0) (3,0) -- (4,0);
		\draw (2, 0.03) -- (3, 0.03) (2, -0.03) -- (3, -0.03);
		\draw (2.4, 0.1) -- (2.7, 0) -- (2.4, -0.1);
		\foreach \x in {1, ..., 4}{
			\node[draw, fill, circle, inner sep=0pt, minimum size=1.5mm] at (\x, 0){};
			\node at (\x, -0.25){\scriptsize\x};
		}
	\end{tikzpicture}
	\caption{The Tits index $ (F_4, \{\id\}, \{3,4\}) $, which we refer to as \enquote{$ F_4 \to G_2 $}.}
	\label{fig:tits-index}
\end{figure}
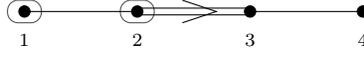

\begin{notation}
	Throughout this section, we denote by $\Delta_F= (f_1, f_2, f_3, f_4)$ a standard system of simple roots in $F_4$ (in the sense of \cref{def:F4 basis order}) and by $\Delta_G = (g_1, g_2)$ a system of simple roots in $G_2$ such that $g_1$ is short and $g_2$ is long. Further, we denote by $\Pi_F$, $\Pi_G$ the corresponding positive systems of roots.
\end{notation}

\begin{remark}\label{rem:F4 coeff list}
	For some of our arguments, it will be helpful to have a list of the roots in $ F_4 $ expressed as linear combinations of the vectors in $ \Delta_F $ at hand. Such a list can be read of from the first column of \cref{ta:existence parmap} on page~\pageref{ta:existence parmap}. Further,
	\[ G_2 = \{ \pm (a g_1 + b g_2) \mid (a,b) \in \{(0,1), (1,0), (1,1), (2,1), (3,1), (3,2)\}\}. \]
\end{remark}

\begin{remark}\label{rem:F4 inner product}
	If the roots of $ G_2 $ and $ F_4 $ are normalized so that short roots have length $ \sqrt{2} $, then the inner products are given as follows:
	\begin{gather*}
		f_1 \cdot f_2 = -2 = f_2 \cdot f_3, \quad f_3 \cdot f_4 = -1, \quad f_1 \cdot f_1 = 4 = f_2 \cdot f_2, \quad f_3 \cdot f_3 = 2 = f_4 \cdot f_4, \\
		g_1 \cdot g_2 = -3, \qquad g_1 \cdot g_1 = 2, \qquad g_2 \cdot g_2 = 6.
	\end{gather*}
\end{remark}

We begin with a characterization of the map $ \pi \colon F_4^0 \to G_2^0 $, which depends on the choices of $ \Delta_F $ and $ \Delta_G $.

\begin{lemma}\label{le:index char}
	\begin{enumerate}
		\item \label{le:index char:ex}There exists a unique map $ \pi \colon F_4^0 \to G_2^0 $ such that
		\[ \pi(f_1) = g_2 , \quad \pi(f_2) = g_1 , \quad \pi(f_3) = 0 = \pi(f_4) \]
		and such that $\pi$ is induced by a linear map of the vector spaces surrounding $F_4$ and $G_2$. In particular, for all $ \alpha, \beta \in F_4 $ with $ \alpha+\beta \in F_4 $, we have $ \pi(\alpha) + \pi(\beta) = \pi(\alpha+\beta) $ (and hence $ \pi(\alpha) + \pi(\beta) \in G_2^0 $).
		
		\item \label{le:index char:surj}The map $ \pi $ is surjective and maps $\Pi_F$ onto $\Pi_G \cup \{0\}$.

	\end{enumerate}
\end{lemma}
\begin{proof}
	The uniqueness of $ \pi $ follows from the assumption that it is induced by a linear map. To prove existence, consider the 4-dimensional Euclidean space $V$ surrounding $F_4$ and the subspaces
	\[ V_1 \coloneq \langle f_1, 3f_2 + 4f_3 + 2f_4 \rangle_{\mathbb{R}} \quad \text{and} \quad V_2 \coloneq \langle f_3, f_4 \rangle_{\mathbb{R}}. \]
	Then $V = V_1 \oplus V_2$ and, by \cref{rem:F4 inner product}, this decomposition is orthogonal. Denote by $\tilde{\pi} \colon V \to V_1$ the orthogonal projection onto $V_1$ and put $ \bar{g}_1 \coloneq \tilde{\pi}(f_2) = f_2 + \frac{4}{3} f_3 + \frac{2}{3} f_4 $, $ \bar{g}_2 \coloneq \tilde{\pi}(f_1) = f_1 $. Using \cref{rem:F4 coeff list}, we see that
	\begin{align*}
		\tilde{\pi}(F_4) = \{0_V\} \cup \{ \pm (a \bar{g}_1 + b \bar{g}_2) \mid (a,b) \in \{(0,1), (1,0), (1,1), (2,1), (3,1), (3,2)\}\}.
	\end{align*}
	Further, by \cref{rem:F4 inner product},
	\[ \bar{g}_1 \cdot \bar{g}_1 = \frac{4}{3}, \qquad \bar{g}_2 \cdot \bar{g}_2 = 4, \qquad \bar{g}_1 \cdot \bar{g}_2 = -3. \]
	We conclude that $ \tilde{\pi}(F_4) \setminus \{0_V\} $ is a root system of type $ G_2 $. Thus by identifying $ (\bar{g}_1, \bar{g}_2) $ with $ (g_1, g_2) $, we obtain the desired map $ \pi \colon F_4^0 \to G_2^0 $. We also see from our construction that $ \pi $ is surjective and $ \pi(\Pi_F) = \Pi_G \cup \{0\} $.
%
\end{proof}

The Tits index $ F_4 \to G_2 $ mentioned above yields an embedding $ u \colon \Weyl(G_2) \to \Weyl(F_4) $. The following \cref{le:index weyl} characterizes this embedding and describes its connection to $ \pi $.

\begin{remark}\label{rem:longest element}
	For any subset $ I $ of $ \{1,2,3,4\} $, we denote by $ w_I $ the longest element in the Coxeter system spanned by $ \{\reflbr{f_i} \mid i \in I\} $. Then
	\begin{align*}
		w_{\{1\}} &= \refl{f_1}, \qquad w_{\{3,4\}} = \refl{f_3} \refl{f_4} \refl{f_3}, \qquad w_{\{1,3,4\}} = w_{\{1\}} w_{\{3,4\}}, \\
		w_{\{2,3,4\}} &= \refl{f_2} (\refl{f_3} \refl{f_2} \refl{f_3}) (\refl{f_4} \refl{f_3} \refl{f_2} \refl{f_3} \refl{f_4}).
	\end{align*}
	The last assertion can be verified, for example, by calling
	\begin{verbatim}
		julia> W = weyl_group(:C, 3); f2,f3,f4 = W[3],W[2],W[1];
		julia> longest_element(W) == f2*(f3*f2*f3)*(f4*f3*f2*f3*f4)
	\end{verbatim}
	in the computer algebra system OSCAR \cite{OSCAR,OSCAR-book}. Note that
	\begin{align*}
		w_{\{2,3,4\}} &= \refl{f_2} \refl{f_2}^{\reflbr{f_3}} \refl{f_2}^{\reflbr{f_3} \reflbr{f_4}} = \refl{f_2} \refl{f_2+2f_3} \refl{f_2+2f_3 +2f_4}.
	\end{align*}
\end{remark}

\begin{lemma}\label{le:index weyl}
	The following hold, where $U \coloneq \langle \reflbr{f_3}, \reflbr{f_4} \rangle \le \Weyl(F_4)$.
	\begin{enumerate}
		\item \label{le:index weyl:hom}There exists a unique homomorphism $u \colon \Weyl(G_2) \to \Weyl(F_4)$ such that, in the notation of \cref{rem:longest element},
		\[ u(\refl{g_1}) = w_{\{2,3,4\}} w_{\{3,4\}}, \qquad u(\refl{g_2}) = w_{\{1\}} = \reflbr{f_1}. \]
		Further, $u$ is injective and its image centralizes $ w_{\{3,4\}} $.
		
		\item \label{le:index weyl:act}For all $w \in \Weyl(G_2)$, all $\tilde{w} \in u(w) U$ and all $\gamma \in F_4$, we have
		\[ \pi\brackets[\big]{\gamma^{\tilde{w}}} = \pi(\gamma)^w. \]
	\end{enumerate}
\end{lemma}
\begin{proof}
	Put $ S \coloneq \{w_{\{2,3,4\}} w_{\{3,4\}}, w_{\{1\}}\} $. By \cite[20.32]{MPW15}, $ (\langle S \rangle, S) $ is a Coxeter system. Is is a straightforward computation to show that the order of the product $ w_{\{2,3,4\}} w_{\{3,4\}} w_{\{1\}} $ is $ 6 $, so this Coxeter system is of type $ G_2 $. Further, $ u(\refl{g_1}) $ and $ u(\refl{g_2}) $ commute with $ w_{\{3,4\}} $ by \cite[20.11(i)]{MPW15}. Assertion~\ref{le:index weyl:hom} follows.
	
	It suffices to prove~\ref{le:index weyl:act} for the case that $ w \in \{\reflbr{g_1}, \reflbr{g_2}\} $ and $ \gamma \in \Delta_F $. Further, for all $ \beta \in F_4 $ and $ u \in U $, we have $ \beta^u - \beta \in \langle f_3, f_4 \rangle_{\mathbb{R}} $ and hence $ \pi(\beta^u) = \pi(\beta) $. Thus we may also assume that $ \tilde{w} = u(w) $. We conclude that there are only eight cases to consider, and each of them is covered by a short computation.
\end{proof}

We can now describe the fibers of $ \pi $, which will be crucial for our construction of $ G_2 $-gradings from $ F_4 $-gradings, and give an alternative description of $ u $ in terms of these fibers. For $ \beta \in G_2^0 $, we will always write $ \pi^{-1}(\beta) $ for $ \{\alpha \in F_4^0 \mid \pi(\alpha) = \beta\} $.

\begin{lemma}\label{le:index fibers}
	The following hold.
	\begin{enumerate}
		\item \label{le:index fibers:conj}For all $ \beta \in G_2 $ and $ w \in \Weyl(G_2) $, we have $ \pi^{-1}(\beta^w) = \pi^{-1}(\beta)^{u(w)} $.
		
		\item \label{le:index fibers:0}$ \pi^{-1}(0) = F_4^0 \cap \langle f_3, f_4 \rangle_\Z = \{\pm f_3, \pm f_4, \pm (f_3+f_3)\} \cong A_2^0 $.
		
		\item \label{le:index fibers:long}For every long root $ \beta \in G_2 $, the fiber $ \pi^{-1}(\beta) $ contains exactly one element, and this element is long in $ F_4 $. 
		
		\item \label{le:index fibers:short}For every short root $ \beta \in G_2 $, the fiber $ \pi^{-1}(\beta) $ contains exactly three long roots in $ F_4 $ and three short roots in $ F_4 $, and the subsystem $\langle \pi^{-1}(\beta) \rangle_{\mathbb{R}} \cap F_4$ spanned by this fiber is of type $C_3$. If we denote the three long roots by $ 2e_1 $, $ 2e_2 $, $ 2e_3 $, then $ (e_1, e_2, e_3) $ is an orthogonal basis of $ \langle \pi^{-1}(\beta) \rangle_{\mathbb{R}} $ and the three short roots in $ \pi^{-1}(\beta) $ are $ e_1+e_2 $, $ e_2+e_3 $, $ e_1+e_3 $.
		
		\item \label{le:index fibers:weyl embed}Let $\beta \in G_2$ and denote by $L(\beta)$ the set of long roots in $\pi^{-1}(\beta)$. Then
		\[ u(\reflbr{\beta}) = \begin{cases}
			\prod_{\alpha \in L(\beta)} \refl{\alpha} & \text{if } \beta \text{ is long}, \\
			(\prod_{\alpha \in L(\beta)} \refl{\alpha}) \refl{f_3} \refl{f_4} \refl{f_3} & \text{if } \beta \text{ is short}
		\end{cases} \]
		where the product over $ L(\beta) $ may be taken in any order.
	\end{enumerate}
\end{lemma}
\begin{proof}
	Assertions~\ref{le:index fibers:conj} and~\ref{le:index fibers:0} follow from \cref{le:index weyl}\cref{le:index weyl:act} and \cref{le:index char}\cref{le:index char:ex}, respectively. Using \cref{rem:F4 coeff list}, we can compute that
	\begin{align*}
		\pi^{-1}(g_1) &= \{f_2+f_3, f_2+f_3+f_4, f_2+2f_3+f_4, f_2, f_2+2f_3, f_2+2f_3+2f_4\}, \\
		\pi^{-1}(g_2) &= \{f_1\}.
	\end{align*}
	By a short computation, the long roots in $ \pi^{-1}(g_1) $ are $ f_2, f_2+2f_3, f_2+2f_3+2f_4 $ and their pairwise sums are precisely the remaining roots in $ \pi^{-1}(g_1) $ multiplied by $ 2 $. We infer that assertions~\ref{le:index fibers:long} and~\ref{le:index fibers:short} hold for $ \beta \in \{g_1, g_2\} $. By~\ref{le:index fibers:conj}, it follows that they hold for all $ \beta \in G_2 $.
	
	Assertion~\ref{le:index fibers:weyl embed} is satisfied for the long root $ \beta = g_2 $ because $ L(g_2) = \{f_1\} $. Since $ L(g_1) = \{f_2, f_2+2f_3, f_2+2f_3+2f_4\} $, it follows from \cref{rem:longest element} that $ w_{\{2,3,4\}} = \prod_{\alpha \in L(g_1)} \refl{\alpha} $ and hence~\ref{le:index fibers:weyl embed} holds for $ \beta = g_1 $ as well. Since all roots in $ G_2 $ are conjugate to $ g_1 $ or $ g_2 $ and since the image of $ u $ centralizes $ w_{\{3,4\}} = \refl{f_3} \refl{f_4} \refl{f_3} $ by \cref{le:index weyl}\cref{le:index weyl:hom}, the general assertion follows by an application of~\ref{le:index fibers:conj}.
\end{proof}

\begin{remark}\label{rem:3-graded rs}
	A \emph{$ 3 $-grading} of a root system $ \Phi $ is a decomposition $ \Phi = \Phi_{-1} \sqcup \Phi_0 \sqcup \Phi_1 $ such that $ (\Phi_i + \Phi_j) \cap \Phi \subseteq \Phi_{i+j} $ for all $ i,j \in \{-1,0,1\} $ and some other properties are satisfied (see \cite[14.1]{LoosNeher} for details). The root system $ C_3 $ has the following $ 3 $-grading \cite[14.5(c)]{LoosNeher}:
	\begin{align*}
		(C_3)_{\pm 1} &= \{ \pm 2e_i \mid i \in \{1,2,3\}\} \cup \{ \pm(e_i + e_j) \mid i \ne j \in \{1,2,3\}\}, \\
		(C_3)_0 &= \{e_i - e_j \mid i \ne j \in \{1,2,3\}\}.
	\end{align*}
	Note that the set $ \pi^{-1}(\beta) $ in \cref{le:index fibers}\cref{le:index fibers:short} is precisely the set $ (C_3)_1 $ of $ 1 $-graded roots while $ \pi^{-1}(-\beta) = -\pi^{-1}(\beta) $ is the set $ (C_3)_{-1} $ of $ (-1) $-graded roots.
\end{remark}

\begin{remark}\label{rem:index U}
	It follows from \cref{le:index fibers}\cref{le:index fibers:weyl embed} that $\prod_{\alpha \in L(\beta)} \refl{\alpha} \in u(\refl{\beta}) U$ for all $ \beta \in G_2 $ where $ U \coloneq \langle \refl{f_3}, \refl{f_4} \rangle $. Hence we may always choose $\tilde{w} = \prod_{\alpha \in L(\beta)} \refl{\alpha}$ for $w = \refl{\beta}$ in \cref{le:index weyl}\cref{le:index weyl:act}.
\end{remark}

We can now apply our observations to construct $G_2$-gradings. For Lie algebras, the distributive law makes this fairly straightforward.

\begin{proposition}\label{pr:lie index}
	Let $ L = \bigoplus_{\alpha \in F_4^0} L_\alpha $ be an $ F_4 $-graded Lie algebra over a commutative ring $ k $. For all $ \beta \in G_2^0 $, we define
	\[ L_\beta' \coloneq \langle L_\alpha \mid \alpha \in \pi^{-1}(\{\beta\}) \rangle_k. \]
	Then $ (L_\beta')_{\beta \in G_2^0} $ is a $ G_2 $-grading of $ L $.
\end{proposition}
\begin{proof}
	Since $ L_\beta' = \bigoplus_{\alpha \in \pi^{-1}(\beta)} L_\alpha $
	for all $ \beta \in G_2^0 $, we have $ L = \bigoplus_{\beta \in G_2^0} L_\beta' $.
	Now let $ \beta, \bar{\beta} \in G_2^0 $ and let $ x_\beta \in L_\beta' $, $ y_{\bar{\beta}} \in L_{\bar{\beta}}' $. Write
	\[ x_\beta = \sum_{\alpha \in \pi^{-1}(\beta)} x_\alpha \quad \text{and} \quad y_{\bar{\beta}} = \sum_{\bar{\alpha} \in \pi^{-1}(\bar{\beta})} y_{\bar{\alpha}} \]
	where $ x_\alpha \in L_\alpha $ and $ y_{\bar{\alpha}} \in L_{\bar{\alpha}} $. Thus
	\[ [x_\beta, y_{\bar{\beta}}] = \sum_{\alpha \in \pi^{-1}(\beta)} \sum_{\bar{\alpha} \in \pi^{-1}(\bar{\beta})} [x_\alpha, y_{\bar{\alpha}}]. \]
	For all $ \alpha \in \pi^{-1}(\beta) $, $ \bar{\alpha} \in \pi^{-1}(\bar{\beta}) $, there are two possibilities:
	\begin{enumerate}
		\item If $ \alpha + \bar{\alpha} \notin F_4 $, then $ [x_\alpha, y_{\bar{\alpha}}] = 0 $.
		
		\item If $ \alpha + \bar{\alpha} \in F_4 $, then $ \pi(\alpha+\bar{\alpha}) = \pi(\alpha) + \pi(\bar{\alpha}) = \beta + \bar{\beta} $ by \cref{le:index char}\cref{le:index char:ex}. Hence $ \beta + \bar{\beta} \in G_2^0 $ and $ \alpha + \bar{\alpha} \in \pi^{-1}(\beta + \bar{\beta}) $ in this case, which implies that $ [x_\alpha, y_{\bar{\alpha}}] \in L_{\alpha+\bar{\alpha}} \le L'_{\beta + \bar{\beta}} $.
	\end{enumerate}
	We infer that $ [x_\beta, y_{\bar{\beta}}] = 0 $ if $ \beta + \bar{\beta} \notin G_2^0 $ and $ [x_\beta, y_{\bar{\beta}}] \in L'_{\beta + \bar{\beta}} $ otherwise. It follows that $ (L_\beta')_{\beta \in G_2^0} $ is a $ G_2 $-grading of $ L $.
\end{proof}

The proof of $ G_2 $-commutator relations in $ F_4 $-graded groups is more difficult due to the non-linear nature of groups.
We can use the criterion \cite[2.1.29]{WiedemannPhD} (which is an application of \cite[3.9]{LoosNeher}), but for this we first need to establish \cref{le:index nonprop}.

\begin{lemma}\label{le:index possys}
	Let $\Delta_G' = (g_1', g_2')$ be a system of simple roots in $G_2$ such that $g_1$ is short and let $\Pi_G'$ be the corresponding positive system in $G_2$. Then there exists a standard system $\Delta_F' = (f_1', f_2', f_3', f_4')$ of simple roots in $F_4$ with corresponding positive system $\Pi_F'$ such that $\pi(\Pi_F') = \Pi_G' \cup \{0\}$ and $\pi(\Delta_F') = (g_2', g_1', 0, 0)$.
\end{lemma}
\begin{proof}
	If $ \Delta_G' = \Delta_G $, then we can choose $ \Delta_F \coloneq \Delta_F' $ by \cref{le:index char}. In the general case, there exists $ w \in \Weyl(G_2) $ such that $ \Delta_G' = \Delta_G^{w} $. Put $ \Delta_F' \coloneq \Delta_F^{u(w)} $. Then it follows from \cref{le:index weyl}\cref{le:index weyl:act} that
	\begin{align*}
		\pi(\Delta_F') &= \pi\brackets[\big]{\Delta_F^{u(w)}} = \pi(\Delta_F)^{w} = (g_2, g_1, 0, 0)^{w} = (g_2', g_1', 0, 0).
	\end{align*}
	A similar computation shows that $ \pi(\Pi_F') = \Pi_G' \cup \{0\} $.
\end{proof}

\begin{lemma}\label{le:index nonprop}
	For all non-proportional $ \alpha,\beta \in G_2 $, there exists a positive system $ \Pi_F' $ in $ F_4 $ which contains $ \pi^{-1}(\{\alpha, \beta\}) $.
\end{lemma}
\begin{proof}
	Choose a system $ \Delta_G'= \{g_1', g_2'\} $ of simple roots in $ G_2 $ such that $ g_1 $ is short and such that $ \alpha,\beta $ lie in the positive system $ \Pi_G $ corresponding to $ \Delta_G' $. Let $\Delta_F' = (f_1, f_2, f_3, f_4)$ be a standard system of simple roots in $F_4$ chosen as in \cref{le:index possys}. We show that the corresponding positive system $\Pi_F'$ contains $ \pi^{-1}(\{\alpha',\beta'\}) $. It suffices to show that $ \pi^{-1}(\{\alpha,\beta\}) $ contains no root in $ F_4 \setminus \Pi_F' = -\Pi_F' $. This is satisfied because $ \pi(-\Pi_F') = -\Pi_G' \cup \{0\} $, a set which contains neither $ \alpha $ nor $ \beta $.
\end{proof}

\begin{proposition}\label{pr:rgg index}
	Let $ G $ be a group with an $ F_4 $-grading $ (U_\alpha)_{\alpha \in F_4} $. For all $ \beta \in G_2^0 $, we define
	\[ U_\beta' \coloneq \langle U_\alpha \mid \alpha \in \pi^{-1}(\beta) \rangle. \]
	Then $ (G, (U_\beta')_{\beta \in G_2}) $ satisfies Axioms~\ref{def:rgg}\cref{def:rgg:gen,def:rgg:comm,def:rgg:weyl}, and hence it is a $ G_2 $-graded group if it satisfies Axiom~\ref{def:rgg}\cref{def:rgg:nondeg}.
\end{proposition}
\begin{proof}
	For Axiom~\ref{def:rgg}\cref{def:rgg:gen}, we show that the group generated by $ (U_\alpha)_{\alpha \in \pi^{-1}(0)}$ is contained in the group generated by $(U_\beta')_{\beta \in G_2}$. Recall from \cref{le:index fibers}\cref{le:index fibers:0} that
	\[ \pi^{-1}(0) = \{\pm f_3, \pm f_4, \pm (f_3+f_3)\}. \]
	Observe that $f_3 = \delta_1 + \delta_2$ and $f_4 = \delta_3 + \delta_4$ where
	\begin{gather*}
		\delta_1 \coloneq f_2 + 2f_3 + f_4, \qquad \delta_2 \coloneq -(f_2 + f_3 + f_4), \\
		\delta_3 \coloneq f_1 + f_2 + f_3 + f_4, \qquad \delta_4 \coloneq -(f_1 + f_2 + f_3).
	\end{gather*}
	Here $\{\delta_1, \delta_2\}$ and $\{\delta_3, \delta_4\}$ span subsystems of type $A_2$. By \cite[5.4.7]{WiedemannPhD}, which is an elementary observation in the study of $A_2$-graded groups, it follows that $U_{\varepsilon f_3} = [U_{\varepsilon \delta_1}, U_{\varepsilon \delta_2}]$, $U_{\varepsilon f_4} = [U_{\varepsilon \delta_3}, U_{\varepsilon \delta_4}]$ and $U_{\varepsilon (f_3+f_4)} = [U_{\varepsilon f_3}, U_{\varepsilon f_4}]$ for all $\varepsilon \in \{\pm 1\}$.
	Hence the root groups $ (U_\alpha)_{\alpha \in \pi^{-1}(0)}$ are contained in the group generated by $ (U_{\delta_i})_{i \in \{1,2,3,4\}} $.
	Since $\delta_1, \delta_2, \delta_3, \delta_4$ lie in $\pi^{-1}(G_2)$, we infer that $G$ is generated by $(U_\beta')_{\beta \in G_2}$.
	
	For the existence of Weyl elements, let $\beta \in G_2$ be arbitrary and let $\alpha_1, \ldots, \alpha_m$ denote the pairwise distinct long roots in $\pi^{-1}(\beta)$. For each $i \in \{1, \ldots, m\}$, choose an $\alpha_i$-Weyl element $w_i = a_{-\alpha_i} b_{\alpha_i} c_{-\alpha_i}$ where $a_{-\alpha_i}, c_{-\alpha_i} \in U_{-\alpha_i}$ and $b_{\alpha_i} \in U_{\alpha_i}$. By \cref{le:index fibers}\cref{le:index fibers:short,le:index fibers:long}, we have $\rootint{\alpha_i}{\alpha_j} = \emptyset = \rootint{\alpha_i}{-\alpha_j}$ and hence $[U_{\alpha_i}, U_{\pm \alpha_j}]=1$ for all distinct $i,j \in \{1, \ldots, m\}$. Together with \cref{rem:rgg abelian}, this implies that $w_\beta \coloneq \prod_{i=1}^m w_{i}$ may be written as
	\[ w_\beta = \prod_{i=1}^m a_{-\alpha_i} \prod_{i=1}^m b_{\alpha_i} \prod_{i=1}^m c_{-\alpha_i} \in U_{-\beta}' U_\beta' U_{-\beta}'. \]
	Further, by \cref{le:index weyl}\cref{le:index weyl:act}, \cref{rem:index U} and because $ \refl{\alpha_1},\ldots, \refl{\alpha_m} $ commute pairwise, we have
	\[ U_\gamma^{w_\beta} = U_{\gamma^{\reflbr{\alpha_1} \cdots \reflbr{\alpha_m}}} \subseteq U_{\pi(\gamma)^{\reflbr{\beta}}}' \quad \text{and} \quad U_\gamma^{w_\beta^{-1}} = U_{\gamma^{\reflbr{\alpha_m} \cdots \reflbr{\alpha_1}}} \subseteq U_{\pi(\gamma)^{\reflbr{\beta}}} \]
	for all $\gamma \in F_4$. This implies that $(U_\delta')^{w_\beta} \subseteq U_{\delta^{\reflbr{\beta}}}' \subseteq (U_\delta')^{w_\beta}$ for all $\delta \in G_2$. We conclude that $w_\beta$ is a $\beta$-Weyl element.
	
	The commutator relations with respect to $ (U_\beta')_{\beta \in G_2} $ follow from \cite[2.1.29]{WiedemannPhD}, whose assumptions are satisfied by \cref{le:index nonprop}. (More precisely, \emph{loc.\@ cit.} is formulated only for maps $ \pi \colon \Phi \to \Phi' $ between root systems, but the same proof remains valid if the image of $ \pi $ lies in $ \Phi' \cup \{0\} $.)
\end{proof}

\section{Cubic Jordan matrix algebras}\label{sec:CJMA}

Now and for the rest of this paper, we move on to $F_4$-gradings. We do so by specializing to a specific kind of cubic norm pairs (in fact, cubic norm \emph{structures}), which are known as \emph{cubic Jordan matrix algebras}. In this section, we briefly introduce these objects and the so-called \emph{multiplicative conic alternative algebras} over which they are defined, and collect all the properties that we will need later. We closely follow the exposition in \cite[Sections~36, 37]{GPR24}.

\begin{convention}
	For a commutative ring $k$, all $k$-algebras are understood to be unital but not necessarily associative or commutative. We recall that a $k$-algebra $A$ is called \emph{alternative} if $(aa)b = a(ab)$ and $(ab)b = a(bb)$ for all $a,b \in A$, which by Artin's Theorem implies that the subalgebra generated by any two elements (and their inverses, if they exist) is associative (see, for example, \cite[14.5]{GPR24}).
\end{convention}

\subsection{Notation and setup}

\begin{notation}\label{not:conic}
	From now on (and for the rest of this paper), we denote by $k$ a commutative ring and by $(C, \connorm)$ a multiplicative conic alternative $k$-algebra in the sense of \cite[16.1,~17.1]{GPR24}. This means that $C$ is an alternative $k$-algebra equipped with a quadratic form $\connorm \colon C \to k$, called the \emph{norm of $ C $}, such that $\connorm(1_C) = 1_k$, $\connorm(ab) = \connorm(a) \connorm(b)$ for all $a,b \in C$ and
	\[ a^2 - \connorm(1_C,a)a + \connorm(a)1_C = 0_C \]
	for all $a \in C$. Here $\connorm(a,b) \coloneq \connorm(a+b) - \connorm(a) - \connorm(b)$ denotes the linearization of $\connorm$ for $a,b \in C$. Further, we denote by
	\begin{align*}
		\contr &\colon C \to k \colon a \mapsto \connorm(1_C,a), & \contr &\colon C \times C \to k \colon (a,b) \mapsto \contr(ab), \\
		\conconj{\:\cdot\:} &\colon C \to C \colon a \mapsto \conconj{a} \coloneq \contr(a)1_C - a
	\end{align*}
	the (linear) trace, the bilinear trace and the conjugation of $ C $, respectively.
	On a purely typographical level, we point out that the expression $ \conconj{ac} $ for $ a,c \in C $ will always stand for $ \conconj{(ac)} $ and not for $ \conconj{a} \conconj{c} $.
	We will usually refer to $ C $ as a multiplicative conic alternative algebra, dropping $ \connorm $ from the notation.
\end{notation}

\begin{example}
	Composition algebras over rings (in the sense of \cite[19.5]{GPR24}), such as octonion or quaternion algebras, are the most prominent examples of multiplicative conic alternative algebras. Further, by \cite[16.4]{GPR24} any $k$-algebra that is projective of rank~$2$ as a $k$-module is a multiplicative conic algebra that is commutative and associative. This provides examples of multiplicative conic alternative algebras that are not composition algebras, such as the dual numbers $k[\epsilon] \coloneq k[t]/(t^2)$. The crucial difference between composition algebras and multiplicative conic alternative algebras is that the latter need not be finitely generated projective as a $k$-module and that their norm need not satisfy any \enquote{regularity assumption} (such as non-degeneracy). For a detailed characterization of composition algebras in terms of conic alternative algebras, see \cite[19.13]{GPR24}. For a discussion of the subtleties involving different \enquote{regularity assumptions} that coincide for finite-dimensional vector spaces over fields of characteristic not~$2$ but not over base rings, see \cite[11.9, 11.11]{GPR24}.
\end{example}

\begin{remark}\label{rem:conic identities}
	For all $ a,b,c \in C $, the following identities hold.
	\begin{enumerate}
		\item \label{rem:conic identities:invo}The conjugation is an involution: $ \conconj{a+b} = \conconj{a} + \conconj{b} $, $ \conconj{ab} = \conconj{b} \conconj{a} $ and $ \conconj{\conconj{a}} = a $. \cite[17.2]{GPR24}
		\item \label{rem:conic identities:scalar}The conjugation involution is scalar: $ a \conconj{a} = \conconj{a}a = \connorm(a) 1_C$ and $ a + \conconj{a} = \contr(a)1_C $. \cite[(16.5.6)]{GPR24}
		\item \label{rem:conic identities:invar}$ \connorm $ and $ \contr $ are invariant under conjugation: $ \connorm(\conconj{a}) = \connorm(a) $ and $ \contr(\conconj{a}) = \contr(a) $. \cite[(16.5.7)]{GPR24}
		\item \label{rem:conic identities:kirmse} \emph{Kirmse's identities}: $ a(\conconj{a}b) = \connorm(a)b = (b \conconj{a})a $ and, thus, the linearization $ a (\conconj{c}b) + \conconj{a}(cb) = \connorm(a,c)b = (b \conconj{a})c + (b \conconj{c})a $. \cite[(17.4.1)]{GPR24}
		\item \label{rem:conic identities:tr-square}$ \contr(a^2) = \contr(a)^2 - 2\connorm(a) $. \cite[(16.5.8)]{GPR24}
		\item \label{rem:conic identities:norm-lin}$ \connorm(a, \conconj{b}) = \contr(a) \contr(b) - \connorm(a,b) $. \cite[(16.5.10)]{GPR24}
		\item \label{rem:conic identities:norm-assoc}$ C $ is norm-associative: $ \connorm(a,ba) = \contr(b) \connorm(a) $, $ \connorm(a,ab) = \contr(b) \connorm(a) $, $ \connorm(ab,c) = \connorm(a,c \conconj{b}) = \connorm(b, \conconj{a}c) $, $ \contr(a,b) = \connorm(a, \conconj{b}) $. \cite[16.12, 17.2]{GPR24} 
		\item \label{rem:conic identities:tr-assoc}$ \contr $ is an associative linear form (in the sense of \cite[7.9]{GPR24}): $ \contr(ab) = \contr(ba) $ and $ \contr((ab)c) = \contr(a(bc)) \eqcolon \contr(abc) $. \cite[(16.13.1)]{GPR24}
	\end{enumerate}
\end{remark}

\begin{notation}[{\cite[36.1]{GPR24}}]
	We will often say that \enquote{a sum is taken over all cyclic permutations $(i\,j\,l)$ of $(1\,2\,3)$} to mean that an expression $ \sum x_{ijl} $ stands for $ x_{123} + x_{231} + x_{312} $.
\end{notation}

\begin{notation}\label{not:CJMA def}
	We fix a (not necessarily invertible) diagonal matrix $\Gamma =(\gamma_1, \gamma_2, \gamma_3)$ in $\Mat_3(k)$ and we denote by $ J $ the \emph{cubic Jordan matrix algebra} $ \Her_3(C, \Gamma) $, as in \cite[36.4]{GPR24}. For the reader's convenience, we spell out what this means. As a $k$-module, $ J $ is isomorphic to $ k^3 \oplus C^3 $ and consists of all formal expressions
	\[ \sum \bigl( \xi_i e_{ii} + u_i [jl] \bigr) \]
	with $\xi_i \in k$, $u_i \in C$, where the sum is taken over all cyclic permutations $(i\,j\,l)$ of $(1\,2\,3)$.
	We will write
	\begin{gather*}
		e_i \coloneq e_{ii} \in J, \quad \cubel{u}{lj} \coloneq \cubel{\conconj{u}}{jl}, \quad 1_J \coloneq e_1+e_2+e_3, \\
		J_{ii} \coloneq k e_i, \quad J_{jl} \coloneq J_{lj} \coloneq C[jl]
	\end{gather*}
	for any cyclic permutation $(i\,j\,l)$ of $(1\,2\,3)$ and all $ \xi \in k $, $ u \in C $. Thus we have a decomposition
	\begin{align*}
		J &= J_{11} \oplus J_{22} \oplus J_{33} \oplus J_{23} \oplus J_{31} \oplus J_{12}.
	\end{align*}
	For arbitrary elements
	\[ x = \sum \brackets[\big]{\xi_i e_i + \cubel{u_i}{jl}} \quad \text{and} \quad y = \sum \brackets[\big]{\eta_i e_i + \cubel{v_i}{jl}} \]
	of $ J $, we define
	\begin{align*}
		x^\sharp &\coloneq \sum \brackets[\Big]{\brackets[\big]{\xi_j \xi_l - \gamma_j \gamma_l \connorm(u_i)} e_i + \cubel{(-\xi_i u_i + \gamma_i \conconj{(u_j u_l)})}{jl}}, \\
		N(x) &\coloneq \xi_1 \xi_2 \xi_3 - \sum \gamma_j \gamma_l \xi_i \connorm(u_i) + \gamma_1 \gamma_2 \gamma_3 \contr(u_1 u_2 u_3), \\
		x \times y &\coloneq \sum \Bigl(\brackets[\big]{\xi_j \eta_l + \eta_j \xi_l - \gamma_j \gamma_l \connorm(u_i, v_i)} e_i \\
		&\qquad \qquad + \cubel{\brackets[\big]{-\xi_i v_i - \eta_i u_i + \gamma_i \conconj{(u_j v_l + v_j u_l)}}}{jl}\Bigr), \\
		T(x,y) &\coloneq \sum \brackets[\big]{\xi_i \eta_i + \gamma_j \gamma_l \connorm(u_i, v_i)}
	\end{align*}
	where all sums are taken over all cyclic permutations $(i\,j\,l)$ of $(1\,2\,3)$. This defines maps $ \sharp $, $ N $, $ \times $, $ T $ that, by \cite[36.5]{GPR24}, equip $ J $ with the structure of a cubic norm structure with base point $ 1_J $.
\end{notation}

\begin{notation}\label{not:CJMA pair}
	Being a cubic norm structure, $ J $ induces a cubic norm pair that we denote by $ (J,J') $. Thus $J'$ is an isomorphic copy of $J$ and we let
	\[ \iota \colon J \to J' \colon a \mapsto a^\iota \]
	be a fixed isomorphism between $J$ and $J'$, with inverse
	\[ \iotainv \colon J' \to J \colon a \mapsto a^\iotainv. \]
	Hence $ (\iota, \iotainv) \colon (J,J') \to (J', J) $ is an involution of $ (J,J') $. Further, we denote by $ e_{ii}' = e_i' $, $ 1_{J'} $, $ J_{ii}' $ and $ J_{jl}' $ the images of the corresponding elements and subspaces of $ J $ under $ \iota $.
\end{notation}

\begin{remark}\label{rem:gamma independent}
	Assume that $ \gamma_1, \gamma_2, \gamma_3 $ are invertible, put $ p \coloneq \sum_{i=1}^3 \gamma_i e_i \in J $ and denote by $ \Her_3(C) $ the cubic Jordan matrix algebra defined with respect to the diagonal matrix $ \bar{\Gamma} = (1_k, 1_k, 1_k) $. Then by \cite[37.23]{GPR24}, the $ p $-isotope of $ J $ (in the sense of \cref{rem:isotopes}) is isomorphic to $ \Her_3(C) $. In particular, if we restrict to the case that $ \Gamma $ is invertible, then the isotopy type of $ (J,J') $ (and hence, by \cref{pr:lie isotopy}, the isomorphism type of the Lie algebra $ L(J,J') $) does not depend on the choice of~$ \Gamma $.
\end{remark}

\begin{remark}\label{rem:CJMA nondeg}
	It follows from the formula defining $T$ that $T$ is non-degenerate (in the sense that the radical defined in \cref{ex:CNP radical} is zero, i.e., $T(a,J) = \{0\}$ implies $a=0$) if and only if the linearization of $\connorm$ is non-degenerate (in a similar sense). Thus if $k$ is a field and the linearization of $n$ is non-degenerate, then $L(J,J')$ is simple by \cref{thm:L simple}. In particular, these conditions are satisfied if $k$ is field and $C$ is a composition algebra in the sense of \cite[19.5]{GPR24}. Under these assumptions, $J$ is a Freudenthal algebra by \cite[39.9(d)]{GPR24}. If, in addition, $C$ is an octonion algebra, then $J$ is an Albert algebra by \cite[39.19(a)]{GPR24}.
\end{remark}

\begin{remark}
	By \cite[37.8, 37.9]{GPR24}, the triple $ (e_1, e_2, e_3) $ is an elementary frame in the sense of \cite[37.5]{GPR24}, and it is connected if and only if $ \Gamma $ is invertible. With this setup, we may apply the results of \cite[Section~37]{GPR24} to $ J $ (with the additional assumption that $ \Gamma $ is invertible if a co-ordinate system is required).
\end{remark}


\subsection{Identities in cubic Jordan matrix algebras}

We begin by collecting some results which break down the long formulas in \cref{not:CJMA def} into more succinct identities. For the next lemma, recall the Jordan pair structure on $ (J,J') $ from \cref{def:U-operators}.

\begin{lemma}[Multiple Peirce decomposition]\label{le:peirce}
	The following assertions hold for all $ i,j,l,m,p,q \in \{1,2,3\} $.
	\begin{enumerate}
		\item $ \{ J_{ij}, J_{jl}', J_{lm} \} \subseteq J_{im} $ and $ U_{J_{ij}}(J_{ij}') \subseteq J_{ij} $.
		
		\item $\{ J_{ij}, J_{lm}', J_{pq} \} = 0$ and $ U_{J_{ij}}(J_{lm}') = 0 $ if the indices $i,j,l,m,p,q$ (respectively, $ i,j,l,m $) cannot be pairwise reordered in the above form. By this we mean that, in the sense of \cite[32.14]{GPR24}, the Peirce triple $ (ij,lm,pq) $ (respectively, $ (ij,lm,ij) $) is not connected.
	\end{enumerate}
	Similar statements hold if the roles of $ J $ and $ J' $ are interchanged.
\end{lemma}
\begin{proof}
	By \cite[37.8]{GPR24}, the spaces $ J_{ij} $ for $ i,j \in \{1,2,3\} $ are precisely the Peirce components of the cubic norm structure $ J $. Hence the assertion follows from the Peirce decomposition in Jordan algebras, see \cite[32.15]{GPR24}.
\end{proof}

\begin{lemma}\label{le:T formulas}
	Let $ a \in J $ and write
	\[ a = \xi_1 e_1 + \xi_2 e_2 + \xi_3 e_3 + \cubel{c_1}{23} + \cubel{c_2}{31} + \cubel{c_3}{12} \]
	with each $ \xi_i \in k $ and each $ c_{i} \in C $. Let $(i\, j\, l)$ be any (not necessarily cyclic) permutation of~$(1\, 2\, 3)$. Then the following formulas (and their analogues for $ T' $ in place of $ T $) hold.
	\begin{enumerate}
		\item \label{le:T formulas:ei}$ T(a,e_i') = \xi_i $.
		\item \label{le:T formulas:dij}$ T(a, \cubel{d_l}{ij}^\iota) = T(\cubel{c_l}{ij}, \cubel{d_l}{ij}^\iota) = \gamma_i \gamma_j \contr(\conconj{c_l} d_l) $ for all $ d_l \in C $.
		\item \label{le:T formulas:ortho}The Peirce decomposition of $ (J,J') $ is orthogonal with respect to $ T $ in the sense that $ T(b_{pq}, b'_{rs}) = 0 $ for all $ b_{pq} \in J_{pq} $, $ b'_{rs} \in J_{rs}' $ with $ \{p,q\} \ne \{r,s\} $.
	\end{enumerate}
\end{lemma}
\begin{proof}
	These are special cases of the formulas defining and $ T $ in \cref{not:CJMA def}. Note that we may assume that $ (i\, j\, l) $ is a cyclic permutation, the general case being a consequence of this case (because $ \contr(\conconj{c_l} d_l) = \contr(c_l \conconj{d_l}) $ by \cref{rem:conic identities}\cref{rem:conic identities:invar,rem:conic identities:tr-assoc}). Alternatively, we may cite \cite[(37.13.5)]{GPR24} for \cref{le:T formulas:ei} and the first equation in \cref{le:T formulas:dij} and \cite[37.4(c)]{GPR24} for \cref{le:T formulas:ortho}.
\end{proof}

\begin{lemma}\label{le:times formulas}
	Let $ a \in J $ and write
	\[ a = \xi_1 e_1 + \xi_2 e_2 + \xi_3 e_3 + \cubel{c_1}{23} + \cubel{c_2}{31} + \cubel{c_3}{12} \]
	with each $ \xi_i \in k $ and each $ c_{i} \in C $. Put $ c_{ij} \coloneq \cubel{c_l}{ij} $ for all cyclic permutations $(i\, j\, l)$ of~$(1\, 2\, 3)$. Let $(i\, j\, l)$ and $ (i'\, j'\, l') $ be permutations of~$(1\, 2\, 3)$ such that $ (i\,j\,l) $ is cyclic. Then the following formulas (and their analogues for $ T', \mathord{\times}' $ in place of $ T, \times $) hold.
	\begin{enumerate}
		\item \label{le:times formulas:ei}$ a \times e_i = \xi_l e_j' + \xi_j e_l' - \cubel{c_i}{jl}^\iota $. In particular, $ e_{i'} \times e_{j'} = e_{l'}' $ and $ \cubel{c_{i'}}{j'l'} \times e_{i'} = -\cubel{c_{i'}}{j'l'}^\iota $.
		\item \label{le:times formulas:ijl}$ \cubel{c_{l'}}{i'j'} \times \cubel{d_{i'}}{j'l'} = \cubel{\gamma_{j'} c_{l'} d_{i'}}{il} $ for all $ d_{i'} \in C $.
		\item \label{le:times formulas:dij}$ a \times d_{ij} = -T(c_{ij}, d_{ij}) e_l' - \xi_l d_{ij}^\iota + c_{li} \times d_{ij} + d_{ij} \times c_{jl} $ for all $ d_{ij} \in J_{ij} $.
	\end{enumerate}
\end{lemma}
\begin{proof}
	These are special cases of the formulas defining $ \times $ and $ T $ in \cref{not:CJMA def}. Again, the general case follows from the cyclic one because
	\begin{align*}
		e_l \times e_j &= e_j \times e_l = e_i, \\
		\cubel{c_l}{ji} \times e_l &= \cubel{\conconj{c_l}}{ij} \times e_l = -\cubel{\conconj{c_l}}{ij}^\iota = -\cubel{c_l}{ji}^\iota \quad \text{and} \\
		\cubel{c_i}{lj} \times \cubel{d_l}{ji} &= \cubel{\conconj{d_l}}{ij} \times \cubel{\conconj{c_i}}{jl} = \cubel{\gamma_j \conconj{d_l} \conconj{c_i}}{il} = \cubel{\gamma_j c_i d_l}{li}.
	\end{align*}
	Alternatively, we may cite \cite[(37.13.4)]{GPR24} for \cref{le:times formulas:ei,le:times formulas:dij}.
\end{proof}

\begin{lemma}\label{le:times blocks}
	For all pairwise distinct $ i,j,l \in \{1,2,3\} $, the following hold:
	\begin{enumerate}
		\item \label{le:times blocks:ijl}$ J_{ij} \times J_{jl} \subseteq J_{il}' $.
		
		\item $ J_{ij} \times J_{ij} \subseteq J_{ll}' $ and $ J_{ij} \times J_{ll} \subseteq J_{ij}' $.
		
		\item $ J_{ij} \times J_{jj} = 0 = J_{jj} \times J_{jj} $.
	\end{enumerate}
\end{lemma}
\begin{proof}
	If $ (i\, j \, l) $ is a cyclic permutation of $ (1 \, 2 \, 3) $, then by \cref{le:times formulas}, all assertions hold by \cref{le:times formulas}. The general case follows because $ \times $ is symmetric and $ J_{ij} = J_{ji} $. Further, \cref{le:times blocks:ijl} is \cite[(37.13.1)]{GPR24}.
\end{proof}

We can also compute explicit formulas for the operators $D_{a,a'} = D(a,a')$ from \cref{def:U-operators}, which together with \cref{le:peirce} characterize these operators. In practice, we will only need a few of these formulas, but we state them all for completeness. We can already see at this point that for many (though not all) possible choices of $i,j,p,q$, the operator $D(\cubel{a}{ij}, \cubel{b}{pq})$ depends only on the product $ab$. This will later lead to the similar statements in \cref{le:d-iji shift,le:dd rel} about $\dd(\cubel{a}{ij}, \cubel{b}{pq})$. However, only \cref{le:d-iji shift} will be proven by using the explicit formulas in \cref{le:D-formulas}, while we derive \cref{le:dd rel} from \cref{le:triple} by specializing to cubic Jordan matrix algebras.

\begin{lemma}\label{le:D-formulas}
	Let $i,j,l \in \{1,2,3\}$ be pairwise distinct, let $a_1, a_2, a_3 \in C$ and let $s_1, s_2, s_3 \in k$. Then the following hold:
	\begin{enumerate}
		\item\label{le:D-formulas:jli} The operator $D(\cubel{a_1}{jl}, \cubel{a_2}{li}^\iota)$ satisfies
		\begin{align*}
			\braces{\cubel{a_1}{jl}, \cubel{a_2}{li}^\iota, \cubel{a_3}{ij}} &= \gamma_i \gamma_j \gamma_l \contr(a_1 a_2 a_3) e_j, \\
			\braces{\cubel{a_1}{jl}, \cubel{a_2}{li}^\iota, \cubel{a_3}{il}} &= \cubel{\gamma_i \gamma_l (a_1 a_2) a_3}{jl}, \\
			\braces{\cubel{a_1}{jl}, \cubel{a_2}{li}^\iota, s_3 e_i} &= \cubel{\gamma_l s_3 a_1 a_2}{ji}.
		\end{align*}
		
		\item \label{le:D-formulas:jji}The operator $ D(s_1 e_j, \cubel{a_2}{ji}^\iota) $ satisfies
		\begin{align*}
			\braces{s_1 e_j, \cubel{a_2}{ji}^\iota, \cubel{a_3}{ij}} &= \gamma_i \gamma_j s_1 \contr(a_2 a_3) e_j, \\
			\braces{s_1 e_j, \cubel{a_2}{ji}^\iota, \cubel{a_3}{il}} &= \cubel{\gamma_i s_1 a_2 a_3}{jl} \\
			\braces{s_1 e_j, \cubel{a_2}{ji}^\iota, s_3 e_i} &= \cubel{s_1 s_3 a_2}{ji}.
		\end{align*}
		
		\item \label{le:D-formulas:jii}The operator $ D(\cubel{a_1}{ji}, s_2 e_i') $ satisfies
		\begin{align*}
			\braces{\cubel{a_1}{ji}, s_2 e_i', \cubel{a_3}{ij}} &= \gamma_i \gamma_j s_2 \contr(a_1 a_3) e_j, \\
			\braces{\cubel{a_1}{ji}, s_2 e_i', \cubel{a_3}{il}} &= \cubel{\gamma_i s_2 a_1 a_3}{jl}, \\
			\braces{\cubel{a_1}{ji}, s_2 e_i', s_3 e_i} &= \cubel{s_2 s_3 a_1}{ji}.
		\end{align*}
		
%
%
		
		\item \label{le:D-formulas:iii}The operator $D(s_1 e_i, s_2 e_i)$ satisfies
		\begin{align*}
			\braces{s_1 e_i, s_2 e_i', s_3 e_i} &= 2s_1 s_2 s_3 e_i, \\
			\braces{s_1 e_i, s_2 e_i', \cubel{a_3}{ij}} &= \cubel{s_1 s_2 a_3}{ij}.
		\end{align*}
		
		\item \label{le:D-formulas:iji}The operator $D(\cubel{a_1}{ij}, \cubel{a_2}{ji})$ satisfies
		\begin{align*}
			\braces{\cubel{a_1}{ij}, \cubel{a_2}{ji}^\iota, s_3 e_i} &= \gamma_i \gamma_j s_3 \contr(a_1 a_2) e_i, \\
			\braces{\cubel{a_1}{ij}, \cubel{a_2}{ji}^\iota, \cubel{a_3}{ij}} &= \cubel{\gamma_i \gamma_j \brackets[\big]{\contr(a_2 a_3) a_1 + \contr(a_1 a_2) a_3 - \contr(a_1 \conconj{a_3}) \conconj{a_2}}}{ij}, \\
			\braces{\cubel{a_1}{ij}, \cubel{a_2}{ji}^\iota, \cubel{a_3}{il}} &= \cubel{\gamma_i \gamma_j a_1(a_2 a_3)}{il}.
		\end{align*}
		
	\end{enumerate}
	Further, the same statements hold if all elements $ b \in J $ and $ c \in J' $ are replaced by $ b^\iota \in J' $ and $ c^\iotainv \in J $, respectively. 
\end{lemma}
\begin{proof}
	This is a straightforward computation, see \cref{subsec:computer} for details. Note that, since $\braces{a,b',c} = \braces{c,b',a}$ for $a,c \in J$, $b' \in J'$ and $\cubel{a}{ij} = \cubel{\conconj{a}}{ji}$ for $a \in C$ and $i \ne j$, some of the formulas are redundant. For the second formula in \cref{le:D-formulas:jli}, we first prove by a straightforward computation that
	\[ \braces{\cubel{a_1}{jl}, \cubel{a_2}{li}^\iota, \cubel{a_3}{il}} = \gamma_i \gamma_l \brackets[\big]{\contr(a_2 a_3) a_1 - (a_1 \conconj{a_3}) \conconj{a_2}}, \]
	and then the desired formula follows by an application of \cref{rem:conic identities}\cref{rem:conic identities:kirmse}.
\end{proof}

\subsection{Computations in cubic Jordan matrix algebras}\label{subsec:computer}

The computations in $J$ that are required in the proof of \cref{le:D-formulas}, and similar computations in $L$ and $\Aut(L)$ that we will need later, boil down to the following two steps: First, we unravel the definition of all involved operators (such as $D$, $U$, $N$, $\sharp$, $\times$, $T$ and the Lie bracket) to see that the desired assertion is equivalent to a set of equations in $C$ and $k$. Second, we show that these equations follow from known identities on multiplicative conic alternative algebras, such as the ones in \cref{rem:conic identities}. While this procedure is straightforward in principle, the required computations can be rather long. In particular, the verification that certain elements in $\Aut(L)$ are Weyl elements (\cref{pr:F4 weyl}) seems too extensive to be manageable by hand. For this reason, we have performed these computations with the computer algebra system GAP \cite{GAP4}. Our code is available at \cite{Code}. In the following, we provide a few details on how these computations are performed.

The key problem consists in working with an \emph{arbitrary} multiplicative conic alternative algebra. In other words, we have to work in a \emph{free} multiplicative conic algebra, whose generators we may think of as indeterminates. Say that we want to prove identities involving at most $ p $ elements in the conic algebra and at most $ q $ elements in the commutative ring. We introduce corresponding indeterminates (that is, abstract symbols) $ a_1, \ldots, a_p $ and $ t_1, \ldots, t_q $. For each $ i \in \{1, \ldots, p\} $, we also introduce abstract symbols $ \conconj{a_i} $ and $ \connorm(a_i) = \connorm(\conconj{a_i}) $, which do not (yet) have any relation to the symbol $ a_i $. Similarly, we define an abstract symbol $ \contr(a) $ for each $ a $ in the free magma $ A^* $ generated by the symbols in $A \coloneq \{ a_1, \ldots, a_p, \conconj{a_1}, \ldots, \conconj{a_p} \}$.

Using all the symbols above, we define $\tilde{k}$ to be the polynomial ring over $ \Z $ in indeterminates $ t_1, \ldots, t_q, \connorm(a_1), \ldots, \connorm(a_p) $ and $ \contr(a) $ for $ a \in A $. Similarly, we define $ \tilde{C} $ to be the free nonassociative algebra over $ \tilde{k} $ generated by the symbols $ a_1, \ldots, a_p, \conconj{a_1}, \ldots, \conconj{a_p} $. We can now give meaning to the symbol names introduced earlier: Define $ \contr \colon \tilde{C} \to \tilde{k} $ to be the unique $ \tilde{k} $-linear map which sends $ a \in A^* \subseteq \tilde{C} $ to $ \contr(a) \in \tilde{k} $ and define $ \conconj{\:\cdot\:} \colon \tilde{C} \to \tilde{C} $ to be the unique $ k $-linear map which sends $ a_i $ to $ \conconj{a_i} $ and $ \conconj{a_i} $ to $ a_i $ for all $ i \in \{1, \ldots, p\} $ and such that $ \conconj{ab} = \conconj{b} \conconj{a} $ for all $ a,b \in A^* \subseteq \tilde{C} $. We also define $ \connorm \colon A^* \to \tilde{k} $ to be the unique homomorphism of magmas that maps $a \in A$ to the abstract symbol $\connorm(a)$. Choosing an arbitrary total order $<$ on $A^*$ (for example, the lexicographic order), we can extend $ \connorm $ to a map $ \connorm \colon \tilde{C} \to \tilde{k} $ (which depends on the choice of $<$) by putting
\begin{equation}\label{eq:free-norm-def}
	\connorm\Bigl( \sum_{a \in A^*} \lambda_a a \Bigr) \coloneq \sum_{a \in A^*} \lambda_a^2 \connorm(a) + \sum_{a<b \in A^*} \lambda_a \lambda_b \contr(a \conconj{b})
\end{equation}
for any family $ (\lambda_a)_{a \in A^*} $ of scalars in $ \tilde{k} $ of which only finitely many are non-zero. In other words, $ \connorm $ is the unique $ \tilde{k} $-quadratic map $ \tilde{C} \to \tilde{k} $ mapping $ a_i $ and $ \conconj{a_i} $ to $ \connorm(a_i)=\connorm(\conconj{a_i}) $ for all $ i \in \{1, \ldots, p\} $ and such that $\connorm(a+b) = \connorm(a) + \connorm(b) + \contr(a \conconj{b})$ for all $a<b \in A^*$. The objects and maps $ \tilde{k} $, $ \tilde{C} $, $ \contr $, $ \connorm $, $ \conconj{\:\cdot\:} $ are straightforward to implement in GAP. Hence we can also implement the maps $ N $, $ T $, $ \times $ and $ \sharp $ defined on $ \tilde{J} \coloneq \tilde{C}^3 \oplus \tilde{k}^3 $.

By construction, $\tilde{k}$ and $\tilde{C}$ satisfy the following universal property: For arbitrary $ b_1, \ldots, b_p \in C $ and $ s_1, \ldots, s_q \in k $, there exist unique homomorphisms $ \pi_k \colon \tilde{k} \to k $ and $ \pi_C \colon \tilde{C} \to C $ such that $ \pi_k(t_i) = s_i $ for all $ i \in \{1, \ldots, q\} $, $ \pi_C(a_j) = b_j $ for all $ j \in \{1, \ldots, p\} $ and 
\begin{align*}
	\pi_k(\contr(c)) = \contr_C\brackets[\big]{\pi_C(c)}, \quad \pi_k(\connorm(c)) = \connorm_C\brackets[\big]{\pi_C(c)}, \quad \pi_C(\conconj{c}) = \conconj{\pi_C(c)}
\end{align*}
for all $ c \in \tilde{C} $ where $\connorm_C$, $\contr_C$ denote the norm and trace on $C$. In particular, while $\connorm \colon \tilde{C} \to \tilde{k}$ depends on the choice of the total order $<$ on $A^*$, the map $ \pi_k \circ \connorm \colon \tilde{C} \to k $ is independent of this choice because it equals $\connorm_C \circ \pi_C$.
We say that a \emph{simplification map} is a map $ \alpha \colon \tilde{C} \to \tilde{C} $ (or $ \alpha \colon \tilde{k} \to \tilde{k} $) such that for all possible choices of $ k $, $ C $, $ b_1, \ldots, b_m \in C $ and $ s_1, \ldots, s_n \in k $, we have $ \alpha(c) -c \in \ker(\pi_C) $ for all $ c \in \tilde{C} $ (or $ \alpha(s) - s \in \ker(\pi_k) $ for all $ s \in \tilde{k} $). For example, for any total order $ \prec $ on $ A = \{a_1, \ldots, a_p, \conconj{a_1}, \ldots, \conconj{a_p}\} $, the map which replaces all occurrences of $ \contr(b c) $ for $ c \prec b \in A $ by $ \contr(cb) $ is a simplification map by \cref{rem:conic identities}\cref{rem:conic identities:tr-assoc}. In order to prove identities $ \ell = r $ such as those in \cref{le:D-formulas}, we compute both sides $ \ell $, $ r $ \enquote{in $ \tilde{J} $} (using one of the indeterminates $ a_1, \ldots, a_p $ for each variable in $ C $ that occurs, and similarly using $ t_1, \ldots, t_q $ for variables in $ k $) and apply a certain (fixed) sequence of simplification maps to the difference $ \ell - r $. If the result is $ 0 $, then we have proven that the identity holds for arbitrary $ C $ and $ k $. If not, we obtain no information. The main challenge now lies in finding and implementing simplification maps that are computationally feasible but still \enquote{sufficiently powerful} to cover all identities that we want to prove. The correctness of these simplification maps relies on the identities in \cref{rem:conic identities}.

From this point, the implementation of a suitable object $ \tilde{L} $ is straightforward except for the zero component $ \tilde{L}_{0,0} $. Again, we use (linear combinations of) abstract symbols $ \zeta $, $ \xi $ and $ \dd_{a,a'} $ for $ a,a' \in \tilde{J} $ to represent elements of $ \tilde{L}_{0,0} $ and suitable \enquote{simplification maps} to (try to) prove equality of terms. These simplification maps rely on the identities in \cref{le:dd rel,le:d-iji shift,le:dd sum}. Finally, we may define exponential automorphisms by the formulas in \cref{sec:G2 exp} and prove equality of products of such automorphisms by evaluating them on $ \tilde{L}_{-2} \cup \tilde{L}_1 $, whose image in $ L $ is a generating set of the Lie algebra.

In a few cases, the simplification maps do not reduce an element of $ \tilde{C} $ to $ 0 $ but instead to a rather short expression. In such cases, we simply prove by hand that the \enquote{simplified} expression lies in $ \ker(\pi_C) $ for all possible choices of $ C $ (instead of implementing a potentially complicated simplification map which would prove this automatically).

Note that the strategy described in this section applies only to cubic Jordan matrix algebras and not to arbitrary cubic norm pairs. While a similar strategy could certainly work for cubic norm pairs, it would be more complicated to implement due to the higher number of maps that are part of the structure of a cubic norm pair. Moreover, the fact that the norm $ N $ is cubic and not merely quadratic would complicate this endeavor even further. At the same time, there is less need for computer assistance because the number of root spaces in $ G_2 $ is more manageable.

We have also considered an approach using \enquote{normal forms} of elements of $ \tilde{k} $, $ \tilde{C} $ and $ \tilde{L}_{0,0} $ instead of using simplification maps. To some extent, the simplification maps that we use transform elements \enquote{of short (product) length} to something close to a normal form, but the task of finding a normal form for arbitrary elements seems much more challenging.

\section{Refining the \texorpdfstring{$ G_2 $}{G2}-grading of the Lie algebra}\label{sec:lie F4}

We keep up the notation from \labelcref{not:conic,not:CJMA def,not:CJMA pair} and denote by $ L \coloneq L(J,J') $ the Lie algebra from \cref{sec:lie const} with respect to the cubic norm pair $(J,J')$ coming from a cubic Jordan matrix algebra.
Our goal is to further refine the $G_2$-grading of $ L $ from \cref{pr:G2-graded} to an $F_4$-grading.

\subsection{Construction of the refinement}

In order to construct this refinement of the grading, we will first construct three new gradings.
They arise from the \enquote{hyperbolic pairs} $(e_i, e'_i)$, where the $e_i$ and $e'_i$ are as in \cref{sec:CJMA}.
When $k$ is a field, we can indeed directly follow the approach from \cite{DMM25} and construct the gradings from these hyperbolic pairs,
but in our setting, we will reverse the process by first describing the different pieces of the grading (\cref{def:3 new gradings}), and then verifying that this does indeed define a $5$-grading (\cref{pr:igrading}).

%
%

The hardest part to control is the ``middle'' subspace $Z$ (see \cref{def:dd and L0}). We rely on the multiple Peirce decomposition to introduce certain subspaces $Z_{i \to j}$ of~$Z$ and we examine how these subspaces interact with various components in \cref{le:Zitoj bracket Jpq,le:Zitoj bracket L1,le:Zitoj bracket}. In \cref{pr:Zitoj-param}, we will shows that $ Z_{i \to j} $ is isomorphic to the conic algebra $ C $.

\begin{definition}\label{def:Zijlm}
	\begin{enumerate}
		\item 
			For each subspace $A \leq J$ and each subspace $A' \leq J'$, we write
			\[ Z(A,A') := \left\langle \dd_{a,a'} \mid a \in A, a' \in A' \right\rangle \leq Z. \]
		\item
			Using the Peirce subspaces $J_{ij}$ from \cref{le:peirce}, we write
			\[ Z_{ij,lm} := Z(J_{ij}, J_{lm}') , \]
			for all $i,j,l,m \in \{ 1,2,3 \}$. (We allow $i=j$ or $l=m$.)
		\item
			Finally, for all $i,j \in \{ 1,2,3 \}$, not necessarily distinct, we write
			\[ Z_{i \to j} := Z_{i1,1j} + Z_{i2, 2j} + Z_{i3, 3j} . \]
	\end{enumerate}
\end{definition}

\begin{remark}\label{rem:Zijlm zero}
	Note that by \cref{le:peirce} and \cref{pr:Jd}, $Z_{ij,lm} = 0$ if (and only if) $\{ i,j \} \cap \{ l,m \} = \emptyset$.
\end{remark}


The following notation will be useful to avoid some case distinctions.
\begin{definition}\label{def:Z logical}
	Let $ I $ be a subspace of $J$, $J'$ or $Z$ and let $P$ be a logical statement. We define $ I^P \coloneq I $ if $ P $ is true and $ I^P \coloneq 0 $ if $ P $ is false.
\end{definition}

\begin{lemma}\label{le:less technical}
	Let $i,j,l,p,q \in \{ 1,2,3 \}$ be arbitrary.
	Let $a \in J_{il}$ and $a' \in J_{lj}'$.
	Then:
	\begin{enumerate}
		\item\label{le:less technical:il lj}
			If $i \neq j$, then
			\[ D_{a,a'}(J_{pq}) \subseteq J_{iq}^{p=j} + J_{ip}^{q=j}. \]
		\item\label{le:less technical:il li}
			If $i = j$, then
			\[ D_{a,a'}(J_{pq}) \subseteq J_{pq}^{\{p,q\} \cap \{i,l\} \neq \emptyset} . \]
	\end{enumerate}
	The same assertions hold if we interchange the roles of $ J $ and $ J' $.
\end{lemma}
\begin{proof}
	This is an immediate corollary of \cref{le:peirce}.
\end{proof}

\begin{lemma}\label{le:Zitoj bracket Jpq}
	Let $ i,j,p,q \in \{1,2,3\} $ with $ i \ne j $. Then the following hold:
	\begin{enumerate}
		\item \label{le:Zitoj bracket Jpq:ij}$ [Z_{i \to j}, J_{pq}] \subseteq J_{iq}^{p=j} + J_{ip}^{q=j} $ and $ [Z_{i \to j}, J_{pq}'] \subseteq (J_{jq}')^{p=i} + (J_{jp}')^{q=i} $.
		
		\item \label{le:Zitoj bracket Jpq:ii}$ [Z_{i \to i}, J_{pq}] \subseteq J_{pq} $ and $ [Z_{i \to i}, J_{pq}'] \subseteq J_{pq}' $.
	\end{enumerate}
\end{lemma}
\begin{proof}
	Let $ i,j,p,q,l \in \{1,2,3\} $. (We do not assume that $ i \ne j $.) Let $ a \in J_{il} $, $ a' \in J_{lj} $, $ c_{pq} \in J_{pq} $ and $ c_{pq}' \in J_{pq}' $. By \cref{pr:Jd},
	\[ [\dd_{a,a'}, c_{pq}] = D_{a,a'}(c_{pq}) \quad \text{and} \quad [\dd_{a,a'}, c_{pq}'] = -D_{a',a}(c_{pq}'). \]
	The assertion now follows from \cref{le:less technical}.
\end{proof}

\begin{lemma}\label{le:Zitoj bracket L1}
	Let $ i,j,p,q \in \{1,2,3\} $ with $ i \ne j $. Then the actions of $ Z_{i \to j} $ and $ Z_{i \to i} $ on $ L_{\pm 1} $ respect the following formulas:
	\begin{enumerate}
		\item \label{le:Zitoj bracket L1:ij}$ [Z_{i \to j}, (k, J_{pq}, J_{pq}', k)_\pm] \subseteq (0, J_{iq}^{p=j} + J_{ip}^{q=j}, (J_{jq}')^{p=i} + (J_{jp}')^{q=i}, 0)_\pm $.
		
		\item \label{le:Zitoj bracket L1:ii}$ [Z_{i \to i}, (0, J_{pq}, J_{pq}', 0)_\pm] \subseteq (0, J_{pq}, J_{pq}', 0)_\pm $.
	\end{enumerate}
\end{lemma}
\begin{proof}
	This follows from \cref{pr:delta z,le:Zitoj bracket Jpq}. For the first assertion, we also use that for all $ a \in J_{il} $, $ a' \in J_{lj} $ with $ l \in \{1,2,3\} $ (i.e., for the set of $ (a,a') $ for which the corresponding $ \dd_{a,a'} $ generate $ Z_{i \to j} $), we have $ T(a,a') = 0 $ by \cref{le:T formulas}\cref{le:T formulas:ortho}.
\end{proof}

\begin{lemma}\label{le:Zitoj bracket}
	Let $i,j,l \in \{ 1,2,3 \}$ be arbitrary.
	\begin{enumerate}
		\item \label{le:Zitoj bracket:subalg}$ Z_{i \to j} $ is a Lie subalgebra of $ Z $, which is abelian if $ i \ne j $.
		
		\item $ [Z_{i \to j}, Z_{j \to i}] \subseteq Z_{i \to i} + Z_{j \to j} $.
		
		\item The Lie algebra $ Z_{i \to i} $ normalizes $ Z_{j \to l} $, i.e., $ [Z_{i \to i}, Z_{j \to l}] \subseteq Z_{j \to l} $.
		
		\item $ [Z_{i \to j}, Z_{j \to l}] \subseteq Z_{i \to l} $ when $ i,j,l $ are pairwise distinct.
		
		\item $ [Z_{i \to j}, Z_{i \to l}] = 0 $ and $ [Z_{i \to j}, Z_{l \to j}] = 0 $ when $ i,j,l $ are pairwise distinct.
	\end{enumerate}
\end{lemma}
\begin{proof}
	Recall from \cref{cor:Z is sublie} that
	\begin{equation}\label{eq:d bracket rep}
		[\dd_{a,a'}, \dd_{b,b'}] = \dd_{D_{a,a'}(b), b'} - \dd_{b, D_{a',a}(b')} .
	\end{equation}
	Throughout this proof, we will assume that $i,j,l,p,q \in \{ 1,2,3 \}$ are arbitrary.
	We first observe that by \cref{le:less technical}\cref{le:less technical:il li} and \cref{eq:d bracket rep},
	$Z_{i \to i}$ normalizes each $Z_{jl,pq}$, so in particular,
	\begin{equation}\label{eq:Zii}
		[Z_{i \to i}, Z_{p \to q}] \subseteq Z_{p \to q} .
	\end{equation}
	\begin{enumerate}
		\item
			Assume first that $i = j$.
			Then by \cref{eq:Zii} with $p=q=i$, $Z_{i \to i}$ is indeed a Lie subalgebra of $Z$.
		
			Assume next that $i \neq j$.
			We have to show that, for all $a \in J_{ip}$, $a' \in J_{pj}'$, $b \in J_{iq}$, $b' \in J_{qj}'$, we have $[\dd_{a,a'}, \dd_{b,b'}] = 0$.
			By \cref{le:less technical}\cref{le:less technical:il lj}, we have
			\[ D_{a,a'}(b) \in J_{ii}^{q=j} \quad \text{and} \quad D_{a',a}(b') \in {J_{jj}'}^{q=i} . \]
			By \cref{le:peirce}, however, we have $Z_{ii,jj} = 0$, which implies that $\dd_{D_{a,a'}(b), b'}$ and $\dd_{b, D_{a',a}(b')}$ are trivial also when $q=j$ or $q=i$.
		\item
			We may assume that $i \neq j$.
			Let $a \in J_{ip}$, $a' \in J_{pj}'$, $b \in J_{jq}$, $b' \in J_{qi}'$.
			Then by \cref{le:less technical}\cref{le:less technical:il lj} and \cref{eq:d bracket rep}, we have
			\[ [\dd_{a,a'}, \dd_{b,b'}] \in Z_{iq,qi} + Z_{jq,qj} \subseteq Z_{i \to i} + Z_{j \to j} . \]
		\item
			This is \cref{eq:Zii}.
		\item
			Let $a \in J_{ip}$, $a' \in J_{pj}'$, $b \in J_{jq}$, $b' \in J_{ql}'$.
			Then by \cref{le:less technical}\cref{le:less technical:il lj} and \cref{eq:d bracket rep}, we have
			\[ [\dd_{a,a'}, \dd_{b,b'}] \in Z_{iq,ql} + Z_{qj,jl}^{q=i} \subseteq Z_{i \to l} . \]
		\item
			Let $a \in J_{ip}$, $a' \in J_{pj}'$, $b \in J_{iq}$, $b' \in J_{ql}'$.
			Then by \cref{le:less technical}\cref{le:less technical:il lj} and \cref{eq:d bracket rep}, we have
			\[ [\dd_{a,a'}, \dd_{b,b'}] \in Z_{ii,ql}^{q=j} + Z_{iq,jl}^{q=i} \subseteq Z_{ii,jl} = 0 . \]
			The proof of the other identity is completely similar.
		\qedhere
	\end{enumerate}
\end{proof}

Our next goal is to show that each of the spaces $Z_{i \to j}$ is isomorphic to the conic algebra $C$ for $i \ne j$ (\cref{pr:Zitoj-param}). Clearly, we have a surjective homomorphism $C \to Z_{ii,ij} \colon a \mapsto \dd_{e_i,\cubel{a}{ij}^\iota}$ (using that $i \ne j$), so the main property to show is that $Z_{ii,ij} = Z_{i \to j}$.

\begin{lemma}\label{le:Djllj indep}
	Let $ i,j \in \{1,2,3\} $ be distinct and let $a \in C$. Then
	\[ D(\cubel{a}{ij}, \cubel{1_C}{ji}^\iota) = D(\cubel{1_C}{ij}, \cubel{a}{ji}^\iota) \quad \text{and} \quad D(\cubel{a}{ij}^\iota, \cubel{1_C}{ji}) = D(\cubel{1_C}{ij}^\iota, \cubel{a}{ji}). \]
\end{lemma}
\begin{proof}
	We only prove the first statement, the proof of the second one being similar. It is immediate from \cref{le:peirce} and \cref{le:D-formulas}\cref{le:D-formulas:iji} that both operators act in the same way on all components of $J$ except possibly on $J_{ij}$. For the action on $J_{ij}$, we compute for arbitrary $b \in C$:
	\begin{align*}
		\braces{\cubel{a}{ij}, \cubel{1_C}{ji}^\iota, \cubel{b}{ij}} &= \cubel{\gamma_i \gamma_j \brackets[\big]{\contr(b)a + \contr(a)b - \contr(a \conconj{b})1_C}}{ij}, \\
		\braces{\cubel{1_C}{ij}, \cubel{a}{ji}^\iota, \cubel{b}{ij}} &= \cubel{\gamma_i \gamma_j \brackets[\big]{\contr(ab)1_C + \contr(a)b - \contr(\conconj{b}) \conconj{a}}}{ij}.
	\end{align*}
	Recall from \cref{rem:conic identities}\cref{rem:conic identities:invar,rem:conic identities:norm-lin,rem:conic identities:norm-assoc} that $\contr(\conconj{b}) = \contr(b)$ and
	\[ \contr(a \conconj{b}) = \connorm(a,b) = \contr(a)\contr(\conconj{b}) - \connorm(a,\conconj{b}) = \contr(a) \contr(b) - \contr(ab). \]
	Hence the difference between the two expressions above is
	\begin{align*}
		\contr(b)a - \contr(a) \contr(b) 1_C +\contr(b) \conconj{a},
	\end{align*}
	which is zero because $\contr(a)1_C = a + \conconj{a}$.
\end{proof}

\begin{lemma}\label{le:d-iji shift}
	Let $i,j \in \{1,2,3\}$ be distinct and let $a \in C$. Then $\dd(\cubel{a}{ij}, \cubel{1_C}{ji}^\iota) = \dd(\cubel{1_C}{ij}, \cubel{a}{ji}^\iota)$.
\end{lemma}
\begin{proof}
	Since $T(\cubel{a}{ij}, \cubel{1_C}{ji}^\iota) = \gamma_i \gamma_j \contr(a) = T(\cubel{1_C}{ij}, \cubel{a}{ji}^\iota)$, this follows from \cref{pr:delta z,le:Djllj indep}.
\end{proof}

\begin{lemma}\label{le:dd rel}
	Let $ i,j,l \in \{1,2,3\} $ be pairwise distinct and let $a,b \in C$. Then the following hold.
	\begin{enumerate}
		\item \label{le:dd rel:sum}$\dd_{e_1,e_1'} + \dd_{e_2,e_2'} + \dd_{e_3,e_3'} = 2\zeta - \xi$.
		\item \label{le:dd rel:jji}$\dd(e_j, \cubel{a}{ji}^\iota) = \dd(\cubel{a}{ji}, e_i')$.
		\item \label{le:dd rel:jli}$\begin{aligned}[t]
			\dd(\cubel{a}{jl}, \cubel{b}{li}^\iota) &= \dd(\cubel{1_C}{jl}, \cubel{ab}{li}^\iota) = \dd(\cubel{ab}{jl}, \cubel{1_C}{li}^\iota) \\
			&= \gamma_l \dd(e_j, \cubel{ab}{ji}^\iota) = \gamma_l \dd(\cubel{ab}{ji}, e_i').
		\end{aligned}$
		\item \label{le:dd rel:iji norm}$\dd(\cubel{a}{ij}, \cubel{\conconj{a}}{ji}^\iota) = \gamma_i \gamma_j \connorm(a) \brackets[\big]{\dd_{e_i, e_i'} + \dd_{e_j,e_j'}}$.
		
		\item \label{le:dd rel:iji trace}$\dd(\cubel{a}{ij}, \cubel{\conconj{b}}{ji}^\iota) + \dd(\cubel{b}{ij}, \cubel{\conconj{a}}{ji}^\iota) = \gamma_i \gamma_j \contr(a \conconj{b}) \brackets[\big]{\dd_{e_i, e_i'} + \dd_{e_j,e_j'}}$.
	\end{enumerate}
\end{lemma}
\begin{proof}
	Recall from \cref{le:triple}\cref{le:triple:r} that	for all $p', q', r' \in J'$, we have
	\begin{align*}
		\dd_{p' \times q', r'} + \dd_{q' \times r',p'} + \dd_{r' \times p', q'} = T(r', p' \times q') \cdot (2\zeta - \xi).
	\end{align*}
	We will denote this term by $ \varphi(p', q', r') $ and prove the assertions by specializing to well-chosen values of $ p' $, $ q' $, $ r' $.
	We will also use \cref{le:times formulas,le:T formulas} to compute terms such as $p' \times \tilde{b'}$ or $T(r', p' \times q')$. For $p' \coloneq e_1'$, $q' \coloneq e_2'$, and $r' \coloneq e_3'$, we obtain
	\begin{align*}
		\varphi(p', q', r') &= \dd(e_1' \times e_2', e_3') + \dd(e_2' \times e_3', e_1') + \dd(e_3' \times e_1', e_2') \\
		&= \dd_{e_1,e_1'} + \dd_{e_2,e_2'} + \dd_{e_3,e_3'} \quad \text{and} \\
		\varphi(p', q', r') &= T(e_3', e_1' \times e_2') (2\zeta-\xi) = T(e_3', e_3') (2\zeta-\xi) = 2\zeta-\xi.
	\end{align*}
	This implies \cref{le:dd rel:sum}. For $p' \coloneq e_l'$, $q' \coloneq e_i'$ and $r' \coloneq \cubel{a}{ji}^\iota$, we obtain
	\begin{align*}
		\varphi(p', q', r') &= \dd(e_l' \times e_i', \cubel{a}{ji}^\iota) + \dd(e_i' \times \cubel{a}{ji}^\iota, e_l') + \dd(\cubel{a}{ji}^\iota \times e_l', e_i') \\
		&= \dd(e_j, \cubel{a}{ji}^\iota) + \dd(0, e_l') - \dd(\cubel{a}{ji}^\iota, e_i') \quad \text{and} \\
		\varphi(p', q', r') &= T(\cubel{a}{ji}^\iota, e_l' \times e_i') (2\zeta - \xi) = T(\cubel{a}{ji}^\iota, e_j) (2\zeta - \xi) = 0.
	\end{align*}
	Identity \cref{le:dd rel:jji} follows, which also implies the final equality sign in \cref{le:dd rel:jli}. For $p' \coloneq e_i'$, $q' \coloneq \cubel{1}{jl}^\iota$, $r' \coloneq \cubel{ab}{li}^\iota$, we obtain
	\begin{align*}
		\varphi(p', q', r') &= \dd(e_i' \times \cubel{1}{jl}^\iota, \cubel{ab}{li}^\iota) + \dd(\cubel{1}{jl}^\iota \times \cubel{ab}{li}^\iota, e_i') + \dd(\cubel{ab}{li}^\iota \times e_i', \cubel{1}{jl}^\iota) \\
		&= -\dd(\cubel{1}{jl}, \cubel{ab}{li}^\iota) + \dd(\gamma_l \cubel{ab}{ji}, e_i') + \dd(0, \cubel{1}{jl}^\iota) \quad \text{and} \\
		\varphi(p', q', r') &= T(\cubel{ab}{li}^\iota, e_i' \times \cubel{1}{jl}^\iota)(2\zeta - \xi) = T(\cubel{ab}{li}^\iota, -\cubel{1}{jl})(2\zeta - \xi) = 0.
	\end{align*}
	It follows that $\dd(\cubel{1}{jl}, \cubel{ab}{li}^\iota) = \gamma_l \dd(\cubel{ab}{ji}, e_i')$. For $p' \coloneq e_i'$, $q' \coloneq \cubel{ab}{jl}^\iota$, $r' \coloneq \cubel{1}{li}^\iota$, we obtain
	\begin{align*}
		\varphi(p', q', r') &= \dd(e_i' \times \cubel{ab}{jl}^\iota, \cubel{1}{li}^\iota) + \dd(\cubel{ab}{jl}^\iota \times \cubel{1}{li}^\iota, e_i') + \dd(\cubel{1}{li}^\iota \times e_i', \cubel{ab}{jl}^\iota) \\
		&= -\dd(\cubel{ab}{jl}^\iota, \cubel{1}{li}) + \dd(\gamma_l \cubel{ab}{ji}, e_i') + \dd(0, \cubel{ab}{jl}) \quad \text{and} \\
		\varphi(p', q', r') &= T(\cubel{1}{li}^\iota, e_i' \times \cubel{ab}{jl}) (2\zeta - \xi) = T(\cubel{1}{li}^\iota, -\cubel{ab}{jl}^\iota) (2\zeta - \xi) = 0.
	\end{align*}
	It follows that $\dd(\cubel{ab}{jl}, \cubel{1}{li}) = \gamma_l \dd(\cubel{ab}{ji}, e_i')$. For $p \coloneq \cubel{a}{jl}$, $q \coloneq \cubel{b}{li}$, $r \coloneq e_j \in J$, we have
	\begin{align*}
		\dd_{p, q \times r} + \dd_{q, r \times p} + \dd_{r, p \times q} 
		&= -\dd(\cubel{a}{jl}, \cubel{b}{li}^\iota) + \dd(\cubel{b}{li}, 0) + \dd(e_j, \gamma_l \cubel{ab}{ji}), \\
		T(r, p \times q) (2\zeta - \xi) &= T(e_j, \cubel{a}{jl} \times \cubel{b}{li}) (2\zeta - \xi) = T(e_j, \gamma_l \cubel{ab}{ji}) (2\zeta - \xi) = 0.
	\end{align*}
	Thus it follows from \cref{le:triple}\cref{le:triple:l} that $\dd(\cubel{a}{jl}, \cubel{b}{li}^\iota) = \gamma_l \dd(e_j, \cubel{ab}{ji})$, which finishes the proof of \cref{le:dd rel:jli}.
	
	For \cref{le:dd rel:iji norm}, we put $p \coloneq e_l$ and $q \coloneq \cubel{a}{ij}$. Then
	\begin{align*}
		\dd_{p, q^\sharp} + \dd_{q, p \times q} &= \dd(e_l, -\gamma_i \gamma_j \connorm(a) e_l') + \dd(\cubel{a}{ij}, -\cubel{a}{ij}^\iota), \\
		T(p, q^\sharp) (2\zeta - \xi) &= T(e_l, -\gamma_i \gamma_j \connorm(a) e_l) (2\zeta - \xi) = -\gamma_i \gamma_j \connorm(a) (2\zeta - \xi) \\
		&= -\gamma_i \gamma_j \connorm(a) (\dd_{e_1,e_1} + \dd_{e_2,e_2} + \dd_{e_3,e_3})
	\end{align*}
	by \cref{le:dd rel:sum}. Now \cref{le:dd rel:iji norm} follows from \cref{le:triple}\cref{le:triple:norm r}. Finally, \cref{le:dd rel:iji trace} is the linearization of \cref{le:dd rel:iji norm}.
\end{proof}

\begin{remark}
	The identities in \cref{le:dd rel} also hold with $D$ in place of $\dd$ (and then as identities in $J$ or $J'$), except that~\cref{le:dd rel:sum} should be read as
	\[ \sum_{i=1}^3 D(e_i, e_i') = 2\id_{J} \quad \text{and} \quad \sum_{i=1}^3 D(e_i', e_i) = 2\id_{J'}. \]
\end{remark}

\begin{remark}
	Let $i,j,l \in \{1,2,3\}$ be pairwise distinct. We want to show that we do \emph{not} have $\dd(\cubel{a}{il}, \cubel{b}{li}^\iota) = \dd(\cubel{1}{il}, \cubel{ab}{li}^\iota)$ in general. Suppose for a contradiction that $\dd(\cubel{a}{il}, \cubel{b}{li}) = \dd(\cubel{1}{il}, \cubel{ab}{li}^\iota)$ for all $a,b \in C$, and let $a,b,c \in C$ be arbitrary. Since
	\begin{align*}
		T(\cubel{a}{il}, \cubel{b}{li}) = \gamma_i \gamma_j \contr(ab) = T(\cubel{ab}{il}, \cubel{1}{li})
	\end{align*}
	by \cref{le:T formulas}\cref{le:T formulas:dij}, it follows from \cref{pr:delta z} that
	\[ D(\cubel{a}{il}, \cubel{b}{li}) = D(\cubel{1}{il}, \cubel{ab}{li}). \]
	Thus
	\begin{align*}
		\braces{\cubel{a}{jl}, \cubel{b}{li}^\iota, \cubel{c}{il}} &= \braces{\cubel{c}{il}, \cubel{b}{li}^\iota, \cubel{a}{jl}} = \braces{\cubel{1}{il}, \cubel{cb}{li}^\iota, \cubel{a}{jl}} \\
		&= \braces{\cubel{a}{jl}, \cubel{cb}{li}^\iota, \cubel{1}{il}} = \cubel{\gamma_i \gamma_l a(cb)}{jl}.
	\end{align*}
	by \cref{le:D-formulas}\cref{le:D-formulas:jli}. On the other hand,
	\begin{align*}
		\braces{\cubel{a}{jl}, \cubel{b}{li}^\iota, \cubel{c}{il}} = \cubel{\gamma_i \gamma_l (ab)c}{jl}
	\end{align*}
	Assuming additionally that $\gamma_i$, $\gamma_j$ are invertible, it follows that $(ab)c = a(cb)$ for all $a,b,c \in C$. In other words, $C$ is commutative and associative, which is not the case for arbitrary multiplicative conic alternative algebras.
\end{remark}

\begin{lemma}\label{le:dd sum}
	Let $a_1,a_2,a_3 \in C$. Then
	\[ \sum \gamma_i \dd_{\cubel{a_i}{jl}, \cubel{\conconj{a_l} \conconj{a_j}}{jl}} = \gamma_1 \gamma_2 \gamma_3 \contr(a_1a_2a_3) \cdot (2\zeta - \xi) \]
	where the sum runs over all cyclic permutations $(i\,j\,l)$ of $(1\,2\,3)$.
\end{lemma}
\begin{proof}
	By \cref{le:triple}, we have
	\[ \sum \dd_{\cubel{a_i}{jl}, \cubel{a_j}{li} \times \cubel{a_l}{ij}} = T(\cubel{a_2}{31}, \cubel{a_3}{12} \times \cubel{a_1}{23}) \cdot (2\zeta - \xi). \]
	Recall from \cref{le:times blocks}\cref{le:times blocks:ijl} that
	\begin{align*}
		\cubel{a_i}{jl} \times \cubel{a_j}{li} = \cubel{\gamma_l \conconj{a_j} \conconj{a_i}}{ij}
	\end{align*}
	for all cyclic permutations $(i\,j\,l)$ of $(1\,2\,3)$. Hence
	\[ \sum \dd_{\cubel{a_i}{jl}, \cubel{a_j}{li} \times \cubel{a_l}{ij}} = \sum \gamma_i \dd_{\cubel{a_i}{jl}, \cubel{\conconj{a_l} \conconj{a_j}}{jl}}. \]
	Further,
	\begin{align*}
		T(\cubel{a_2}{31}, \cubel{a_3}{12} \times \cubel{a_1}{23}) &= T(\cubel{a_2}{31}, \gamma_2 \cubel{\conconj{a_1} \conconj{a_3}}{31}) = \gamma_1 \gamma_2 \gamma_3 \contr(a_2 a_3 a_1) \\
		&= \gamma_1 \gamma_2 \gamma_3 \contr(a_1 a_2 a_3)
	\end{align*}
	by \cref{le:times blocks}\cref{le:times blocks:ijl} and \cref{le:T formulas}\cref{le:T formulas:dij}. The assertion follows.
\end{proof}

\begin{proposition}\label{pr:Zitoj-param}
	Let $ i \ne j \in \{1,2,3\} $. Then the map
	\[ \theta \colon C \to Z_{i \to j}, a \mapsto \dd_{e_i, \cubel{a}{ij}^\iota} \]
	is an isomorphism of $ k $-modules.
\end{proposition}
\begin{proof}
	The map is clearly a homomorphism and surjective by \cref{le:dd rel}\cref{le:dd rel:jji,le:dd rel:jli}.
	Now let $ a \in C $ with $ \theta(a) = 0 $. Then by \cref{pr:delta z},
	\begin{align*}
		0 &=[\theta(a), (0, e_j, 0, 0)_+] = \brackets[\big]{0, D_{e_i, \cubel{a}{ij}^\iota}(e_j), 0, 0}_+.
	\end{align*}
	Since $D_{e_i, \cubel{a}{ij}^\iota}(e_j) = \cubel{a}{ij}$ by \cref{le:D-formulas}\cref{le:D-formulas:jji}, we infer that $a=0$. Thus $\theta$ is injective.
\end{proof}

We are now ready to define three new gradings on our Lie algebra $L$.
The verification that these are, in fact, gradings, will occupy us until \cref{pr:igrading}.

\begin{definition}\label{def:3 new gradings}
	Let $i \in \{ 1,2,3 \}$ and let $(i\, j\, l)$ be the unique cyclic permutation of~$(1\, 2\, 3)$. We will define a decomposition
	\[ L = L^{(i)}_{-2} \oplus L^{(i)}_{-1} \oplus L^{(i)}_{0} \oplus L^{(i)}_{1} \oplus L^{(i)}_{2} \]
	with $L^{(i)}_{-2} = k e'_i$ and $L^{(i)}_{2} = k e_i$.
	We set
	\begin{align*}
		L^{(i)}_{-1} &:=
		\begin{multlined}[t][.8\textwidth]
			0 \ \oplus\ (k,\ k e_j + k e_l + J_{jl},\ k e_i',\ 0)_- \\
				\oplus\ \Bigl( (J_{ij}' + J_{li}') \oplus Z^{(i)}_{-1} \oplus 0 \Bigr) \\
				\oplus\ (k,\ k e_j + k e_l + J_{jl},\ k e_i',\ 0)_+ \ \oplus\ 0 \,,
		\end{multlined} \\
		L^{(i)}_{0} &:=
		\begin{multlined}[t][.8\textwidth]
			kx \ \oplus\ (0,\ J_{ij} + J_{li},\ J_{ij}' + J_{li}',\ 0)_- \\
				\oplus\ \Bigl( (k e_j' + k e_l' + J_{jl}') \oplus Z^{(i)}_{0} \oplus (k e_j + k e_l + J_{jl}) \Bigr) \\
				\oplus\ (0,\ J_{ij} + J_{li},\ J_{ij}' + J_{li}',\ 0)_+ \ \oplus\ ky \,,
		\end{multlined} \\
		L^{(i)}_{1} &:=
		\begin{multlined}[t][.8\textwidth]
			 0 \ \oplus\ (0,\ k e_i,\ k e_j' + k e_l' + J_{jl}',\ k)_- \\
				\oplus\ \Bigl( 0 \oplus Z^{(i)}_{1} \oplus (J_{ij} + J_{li}) \Bigr) \\
				\oplus\ (0,\ k e_i,\ k e_j' + k e_l' + J_{jl}',\ k)_+ \ \oplus\ 0 \,,
		\end{multlined}
	\intertext{where}
		Z^{(i)}_{-1} &:= Z_{j \to i} + Z_{l \to i} \, , \\
		Z^{(i)}_{0} &:= Z_{i \to i} + Z_{j \to j} + Z_{l \to l} + Z_{j \to l} + Z_{l \to j} + k \xi + k \zeta , \\
		Z^{(i)}_{1} &:= Z_{i \to j} + Z_{i \to l} .
	\end{align*}	
\end{definition}

\subsection{Verification of the grading property}

\begin{proposition}\label{pr:Z-decomp}
	Let $ i \in \{1,2,3\} $. Then:
	\begin{enumerate}
		\item\label{item:Z:1} We have $ Z = Z_{-1}^{(i)} \oplus Z_0^{(i)} \oplus Z_1^{(i)} $ as $ k $\dash modules.
		\item\label{item:Z:2} The decomposition from \cref{item:Z:1} defines a $3$-grading on $Z$.
	\end{enumerate}
\end{proposition}
\begin{proof}
	Let $(i\, j\, l)$ denote the unique cyclic permutation of~$(1\, 2\, 3)$ starting from $ i $. By \cref{def:dd and L0}, we have $ Z = Z_{-1}^{(i)} + Z_0^{(i)} + Z_1^{(i)} $. Now let $ d_{-1} \in Z_{-1}^{(i)} $, $ d_0 \in Z_0^{(i)} $ and $ d_1 \in Z_1^{(i)} $ such that $ d_{-1} + d_0 + d_1 = 0 $; we have to show that $ d_{-1} = d_0 = d_1 = 0 $. Recall from \cref{def:L0 bar} that an element of $ M $ is zero if and only if it acts trivially on $ \bigoplus_{m \ne 0} L_m $. By \cref{pr:delta z}, $ d_{-1} $ and $ d_1 $ act trivially on $ L_{\pm 2} $, so $ d_0 = -d_{-1} - d_1 $ does the same. Similarly, $ d_{-1} $ and $ d_1 $ act trivially on the subspace $ (k, 0, 0, k)_\pm $ of $ L_{\pm 1} $ by \cref{le:Zitoj bracket L1}\cref{le:Zitoj bracket L1:ij}, so the same holds for $ d_0 $.
	
	It remains to show that for all $ p,q \in \{1,2,3\} $, $ d_{-1} $, $ d_1 $ and $ d_0 $ act trivially on the subspaces $ (0, J_{pq}, J_{pq}', 0)_\pm $ of $ L_{\pm 1} $. It follows from \cref{le:Zitoj bracket L1} and \cref{def:L0 bar}\cref{item:zeta and xi} that $ d_{-1} $, $ d_1 $ and $ d_0 $ map $ (0, J_{ii}, J_{ii}', 0)_\pm $ to
	\[ (0, J_{ij} + J_{il}, 0, 0)_\pm, \qquad (0, 0, J_{ij}' + J_{il}', 0)_\pm \quad \text{and} \quad (0, J_{ii}, J_{ii}', 0)_\pm, \]
	respectively. Since the sum of these three spaces is direct and $ d_{-1} + d_0 + d_1 = 0 $, we deduce that $ d_{-1} $, $ d_1 $ and $ d_0 $ all act trivially on $ (0, J_{ii}, J_{ii}', 0)_\pm $.
	
	Similarly, $ d_{-1} $, $ d_1 $ and $ d_0 $ map the spaces $ (0, J_{pp}, J_{pp}', 0)_\pm $, $ (0, J_{ip}, J_{ip}', 0)_\pm $ and $ (0, J_{pq}, J_{pq}', 0)_\pm $ (for any $ (p,q) \in \{ (j,l), (l,j)\} $) to
	\begin{align*}
		& (0, 0, J_{pi}', 0)_\pm, && (0, J_{pi}, 0, 0)_\pm, && (0, J_{pp} + J_{pq}, J_{pp}' + J_{pq}', 0)_\pm, \\
		& (0, J_{pp} + J_{pq}, J_{ii}', 0)_\pm, && (0, J_{ii}, J_{pp}' + J_{pq}', 0)_\pm, && (0, J_{ip} + J_{iq}, J_{ip}' + J_{iq}', 0)_\pm, \\
		& (0, 0, J_{iq}' + J_{pi}', 0), && (0, J_{iq} + J_{pi}, 0, 0), && (0, J_{pq}, J_{pq}', 0)_\pm,
	\end{align*}
	respectively. Again, the three subspaces in each row form a direct sum, and we conclude that $ d_{-1} $, $ d_1 $ and $ d_0 $ act trivially on each of the subspaces of the form $ (0, J_{pq}, J_{pq}', 0)_\pm $ with $p,q \in \{ 1,2,3 \}$.
	It follows that indeed $ d_{-1} = d_0 = d_1 = 0 $, and this proves \cref{item:Z:1}.
	Now \cref{item:Z:2} follows immediately from \cref{le:Zitoj bracket,rem:xi zeta formulas}.
\end{proof}

\begin{lemma}\label{le:ad ei}
	Let $i \in \{ 1,2,3 \}$ and let $(i\, j\, l)$ be the unique cyclic permutation of~$(1\, 2\, 3)$ starting from $ i $.
	Consider the $k$-module homomorphisms
	\[ \theta^{(i)}_\pm \colon L_{\pm 1} \to L_{\pm 1} \colon z \mapsto [e_i, z] . \]
	Let $z = (\lambda, b, b', \mu)_\pm \in L_{\pm 1}$ and write
	\begin{align*}
		b &= \xi_1 e_1 + \xi_2 e_2 + \xi_3 e_3 + c_{23} + c_{31} + c_{12} \in J, \\
		b' &= \xi_1' e_1' + \xi_2' e_2' + \xi_3' e_3' + c_{23}' + c_{31}' + c_{12}' \in J',
	\end{align*}
	with each $\xi_i, \xi_i' \in k$ and each $c_{jl} \in J_{jl}$, $c_{jl}' \in J_{jl}'$.
	Then
	\[ [e_i, z] = \bigl( 0,\ \lambda e_i,\ \xi_l e_j' + \xi_j e_l' - c_{jl}^\iota,\ \xi_i' \bigr)_\pm. \]
	In particular, the image and kernel of $\theta^{(i)}_\pm$ are given by
	\begin{align*}
		\im \theta^{(i)}_\pm &= \bigl( 0,\ k e_i,\ k e_j' + k e_l' + J_{jl}',\ k \bigr)_\pm \leq L_1^{(i)}, \\
		\ker \theta^{(i)}_\pm &= \bigl( 0,\ k e_i + J_{ij} + J_{li}, \ k e_j' + k e_l' + J_{ij}' + J_{jl}' + J_{li}',\ k \bigr)_\pm \leq L_0^{(i)} + L_1^{(i)} ,
	\end{align*}
	respectively.
\end{lemma}
\begin{proof}
	By \cref{le:T formulas}, we have
	\[ e_i \times b = \xi_l e_j' + \xi_j e_l' - c_{jl}^\iota \quad \text{and} \quad T(e_i, b') = \xi_i' . \]
	Therefore, by \cref{def:Li as J-module}, we have
	\begin{align*}
		[e_i, z]
		&= \bigl( 0,\ \lambda e_i,\ e_i \times b,\ T(e_i, b') \bigr)_\pm \\
		&= \bigl( 0,\ \lambda e_i,\ \xi_l e_j' + \xi_j e_l' - c_{jl}^\iota,\ \xi_i' \bigr)_\pm ,
	\end{align*}
	and the result follows.
\end{proof}

We now arrive at a sequence of lemmas dealing with some of the more involved cases required to check that \cref{def:3 new gradings} gives rise to $5$-gradings of $L$. We advise the reader to have a copy of \cref{def:3 new gradings} ready, to be able to follow our computations.
The different cases and subcases in our proofs correspond to the different $L_i$-components for each of the three subspaces occuring in \cref{def:3 new gradings}.

\begin{lemma}\label{lem:igrading:-11}
	Let $ i \in \{1,2,3\} $. Then $ [L_{-1}^{(i)}, L_1^{(i)}] \subseteq L_0^{(i)} $.
\end{lemma}
\begin{proof}
	Let $(i\, j\, l)$ denote the unique cyclic permutation of~$(1\, 2\, 3)$ starting from $ i $. Let $ z_1 \in L_{-1}^{(i)} $ and $ z_2 \in L_1^{(i)} $. We will write
	\[ z_1 = (\lambda, b, \lambda_i' e_i', 0)_\pm \quad \text{or} \quad z_2 = (0, \mu_i e_i, c', \mu)_\pm \]
	if $ z_1 $ or $ z_2 $, respectively, lies in $ L_{\pm 1} $, where $ \lambda, \lambda_i', \mu_i, \mu \in k $, $ b \in ke_j + ke_l + J_{jl} $ and $ c' \in ke_j' + ke_l' + J_{jl}' $.
	
	First, we consider the case that $ z_1 \in L_{\pm 1} $. As the first subcase, we assume that $ z_2 \in L_{\pm 1} $ as well. If $ z_1 $ and $ z_2 $ are both in $ L_1 $ or both in $ L_{-1} $, then $ [z_1, z_2] \in k y \subseteq L_0^{(i)} $ or $ [z_1,z_2] \in kx \subseteq L_0^{(i)} $, respectively. If $ z_1 \in L_{-1} $ and $ z_2 \in L_1 $, then by \cref{def:L bracket},
	\begin{multline*}
		[z_1, z_2] = \brackets[\big]{-\lambda c' + b \times \mu_i e_i} \\
		\begin{aligned}
			& + \brackets[\big]{(\lambda \mu - T(b,c'))\zeta + T(\mu_i e_i, \lambda_i' e_i')(\xi - \zeta) + \dd_{b,c'} + \dd_{\mu_i e_i, \lambda_i' e_i'}} \\
			& + \brackets[\big]{- \mu b + \lambda_i' e_i' \times c'} .
		\end{aligned}
	\end{multline*}
	Now
	\begin{align*}
		b \times \mu_i e_i \in ke_l' + ke_j' + J_{jl}' \quad &\text{and} \quad \lambda_i' e_i' \times c' \in ke_l + ke_j + J_{jl}, \\
		\dd_{b,c'} \in Z_{j \to j} + Z_{j \to l} + Z_{l \to l} + Z_{l \to j} \quad &\text{and} \quad \dd_{\mu_i e_i, \lambda_i' e_i'} \in Z_{i \to i},
	\end{align*}
	by \cref{le:times formulas}\cref{le:times formulas:ei} and \cref{def:Zijlm}, respectively. Hence $ [z_1, z_2] \in L_0^{(i)} $ in this subcase. The subcase $ z_1 \in L_{1} $, $ z_2 \in L_{-1} $ is completely similar.
	
	We now turn to the subcase $ z_2 \in L_0 $. If $ z_2 \in J_{ij} + J_{li} $, then by \cref{def:Li as J-module},
	\begin{align*}
		[z_1, z_2] &= -[z_2, z_1] = -\brackets[\big]{0, \lambda z_2, z_2 \times b, T(z_2, \lambda_i' e_i')}
	\end{align*}
	where $ z_2 \times b \in J_{ij}' + J_{il}' $ by \cref{le:times blocks} and $ T(z_2, \lambda_i' e_i') = 0 $ by \cref{le:T formulas}\cref{le:T formulas:ortho}. Hence $ [z_1, z_2] \in L_0^{(i)} $ in this subcase. In the final subcase $ z_2 \in Z_1^{(i)} $, we have
	\[ [z_1, z_2] \in (0,J_{ij} + J_{il}, J_{ij}' + J_{il}', 0)_\pm \subseteq L_0^{(i)} \]
	by \cref{le:Zitoj bracket L1}\cref{le:Zitoj bracket L1:ij}.
	
	Second, we consider the case that $ z_1 \in J_{ij}' + J_{li}' $. If $ z_2 \in L_{\pm 1} $, then \cref{def:Li as J-module} yields
	\begin{align*}
		[z_1, z_2] &= -\brackets[\big]{T(\mu_i e_i, z_1), z_1 \times c', \mu z_1, 0}_\pm
	\end{align*}
	where $ T(\mu_i e_i, z_1) = 0 $ by \cref{le:T formulas}\cref{le:T formulas:ortho} and $ z_1 \times c' \in J_{ij} + J_{li} $ by \cref{le:times blocks}. If $ z_2 \in J_{ij} + J_{li} $, then
	\begin{align*}
		[z_1, z_2] &= \dd_{z_2, z_1} \in Z_{i \to i} + Z_{j \to l} + Z_{l \to j} \subseteq L_0^{(i)}.
	\end{align*}
	If $ z_2 \in Z_1^{(i)} $, then by \cref{le:Zitoj bracket Jpq}\cref{le:Zitoj bracket Jpq:ij}, $ [z_1, z_2] \in J_{jj}' + J_{ll}' + J_{jl}' \subseteq L_0^{(i)} $, which finishes the case $ z_1 \in J_{ij}' + J_{li}' $.
	
	Finally, we assume that $ z_1 \in Z_{-1}^{(i)} $. If $ z_2 \in L_{\pm 1} $, then we infer from \cref{le:Zitoj bracket L1}\cref{le:Zitoj bracket L1:ij} that
	\[ [z_1, z_2] \in (0, J_{ji} + J_{li}, J_{ji}' + J_{li}', 0)_\pm \subseteq L_0^{(i)}. \]
	If $ z_2 \in J_{ij} + J_{li} $, then $ [z_1, z_2] \in J_{jj} + J_{ll} + J_{lj} $ by \cref{le:Zitoj bracket Jpq}\cref{le:Zitoj bracket Jpq:ij}.
	If finally $ z_2 \in Z_1^{(i)} $, then
	\[ [z_1, z_2] \in Z_{i \to i} + Z_{j \to j} + Z_{j \to l} + Z_{l \to j} + Z_{l \to l} \subseteq L_0^{(i)} \]
	by \cref{le:Zitoj bracket} (as we already observed in \cref{pr:Z-decomp}\cref{item:Z:2}).
\end{proof}

\begin{lemma}\label{lem:igrading:01}
	Let $ i \in \{1,2,3\} $. Then $ [L_{0}^{(i)}, L_1^{(i)}] \subseteq L_1^{(i)} $.
\end{lemma}
\begin{proof}
	Let $(i\, j\, l)$ denote the unique cyclic permutation of~$(1\, 2\, 3)$ starting from $ i $. Let $ z_1 \in L_{0}^{(i)} $ and $ z_2 \in L_1^{(i)} $. We will write
	\[ z_1 = (0, b_{ij} + b_{li}, b_{ij}' + b_{li}', 0)_\pm \quad \text{or} \quad z_2 = (0, \mu_i e_i, c', \mu)_\pm \]
	if $ z_1 $ or $ z_2 $, respectively, lies in $ L_{\pm 1} $, where $ \mu_i, \mu \in k $, $ b_{ij} \in J_{ij} $, $ b_{li} \in J_{li} $, $ b_{ij}' \in J_{ij}' $, $ b_{li}' \in J_{li}' $ and $ c' \in ke_j' + ke_l' + J_{jl}' $.
	
	First assume that $ z_1 \in \{x,y\} $. If $ z_2 = (0, \mu_i e_i, c', \mu)_\pm \in L_{\pm 1} $, then
	\[ [z_1, z_2] \in \{ 0, (0, \mu_i e_i, c', \mu)_\mp, -(0, \mu_i e_i, c', \mu)_\mp\} \subseteq L_1^{(i)} \]
	by \cref{def:L bracket}. If $ z_2 \in J_{ij} + J_{li} $ or $ z_2 \in Z_1^{(i)} $, we have $ [z_1, z_2] = 0 $ by \cref{def:Li as J-module} or \cref{pr:delta z}, respectively.
	
	Second, we consider the case $ z_1 \in L_{\pm 1} $. We begin with the subcase $ z_2 \in L_{\pm 1} $. If $ z_1 $ and $ z_2 $ both lie in $ L_1 $, then by \cref{def:L bracket} and \cref{le:T formulas}\cref{le:T formulas:ortho},
	\[ [z_1, z_2] = \brackets[\big]{T(b_{ij} + b_{li}, c') - T(\mu_i e_i, b_{ij}' + b_{li}')}x = (0-0)x = 0. \]
	The same holds if $ z_1, z_2 $ both lie in $ L_{-1} $. If $ z_1 \in L_{-1} $ and $ z_2 \in L_{1} $, then \cref{def:L bracket} gives us
	\begin{multline*}
		[z_1, z_2] = (b_{ij} + b_{li}) \times \mu_i e_i \\
		\begin{aligned}
			&+ \brackets[\big]{-T(b_{ij} + b_{li}, c')\zeta + T(\mu_i e_i, b_{ij}' + b_{li}')(\xi - \zeta) + \dd_{b_{ij} + b_{li}, c'} + \dd_{\mu_i e_i, b_{ij}' + b_{li}'}} \\
			&+ \brackets[\big]{-\mu (b_{ij} + b_{li}) + (b_{ij}' + b_{li}') \times c'}.
		\end{aligned}
	\end{multline*}
	Now
	\begin{align*}
		(b_{ij} + b_{li}) \times \mu_i e_i = 0 \quad &\text{and} \quad (b_{ij}' + b_{li}') \times c' \in J_{ij} + J_{il}, \\
		T(b_{ij} + b_{li}, c') = 0 \quad &\text{and} \quad T(\mu_i e_i, b_{ij}' + b_{li}') = 0, \\
		\dd_{b_{ij} + b_{li}, c'} \in Z_{i \to j} + Z_{i \to l} \quad &\text{and} \quad \dd_{\mu_i e_i, b_{ij}' + b_{li}'} \in Z_{i \to j} + Z_{i \to l},
	\end{align*}
	by \cref{le:times blocks}, \cref{le:T formulas}\cref{le:T formulas:ortho} and \cref{def:Zijlm}, respectively, so $ [z_1, z_2] \in L_1^{(i)} $.
	The case $z_1 \in L_1$ and $z_2 \in L_{-1}$ is similar.
	
	We continue with the subcase $ z_2 \in J_{ij} + J_{li} $. By \cref{def:Li as J-module},
	\begin{align*}
		[z_1, z_2] &= -[z_2, (0, b_{ij} + b_{li}, b_{ij}' + b_{li}', 0)] = -\brackets[\big]{0, 0, z_2 \times (b_{ij} + b_{li}), T(z_2, b_{ij}' + b_{li}')}_\pm,
	\end{align*}
	which is an element of $ L_1^{(i)} $ because $ z_2 \times (b_{ij} + b_{li}) \in J_{ll}' + J_{jl}' + J_{jj}' $ by \cref{le:times blocks}.
	In the last subcase $ z_2 \in Z_1^{(i)} $, we have
	\[ [z_1, z_2] \in \brackets{0, J_{ii}, J_{jj}' + J_{ll}' + J_{jl}', 0}_\pm \subseteq L_1^{(i)} \]
	by \cref{le:Zitoj bracket L1}\cref{le:Zitoj bracket L1:ij}.
	
	Third, we turn to the case $ z_1 \in ke_j + ke_l + J_{jl} $, and we simultaneously consider the opposite case of an element $ z_1' \in ke_j' + ke_l' + J_{jl} $. If $ z_2 \in L_{\pm 1} $, then \cref{def:Li as J-module} yields
	\begin{align*}
		[z_1, z_2] &= \brackets[\big]{0, 0, z_1 \times \mu_i e_i, T(z_1, c')}_\pm, \quad [z_1', z_2] = -\brackets[\big]{T(\mu_i e_i, z_1'), z_1' \times c', \mu z_1', 0}_\pm
	\end{align*}
	where $ z_1 \times \mu_i e_i \in J_{ll}' + J_{jj}' + J_{jl}' $ and $ z_1' \times c' \in J_{ii} $ by \cref{le:times blocks} and $ T(\mu_i e_i, z_1') = 0 $ by \cref{le:T formulas}\cref{le:T formulas:ortho}. Hence $ [z_1, z_2], [z_1', z_2] \in L_1^{(i)} $ in this subcase. If $ z_2 \in J_{ij} + J_{li} $, then $ [z_1, z_2] = 0 $ because $ J $ is abelian and $ [z_1', z_2] = \dd_{z_2, z_1'} \in Z_{i \to j} + Z_{i \to l} \in L_1^{(i)} $. If $ z_2 \in Z_1^{(i)} $, then it follows from \cref{le:Zitoj bracket Jpq} that $ [z_1, z_2] \in J_{ij} + J_{il} $ and $ [z_1', z_2] = 0 $.
	
	Finally, we assume that $ z_1 \in Z_0^{(i)} $. By \cref{def:L0 bar}\cref{item:zeta and xi} and \cref{rem:xi zeta formulas}, the assertion is satisfied for $ z_1 \in \{\xi, \zeta\} $, so we may assume that
	\[ z_1 \in Z_{i \to i} + Z_{j \to j} + Z_{l \to l} + Z_{j \to l} + Z_{l \to j}. \]
	If $ z_2 \in L_{\pm 1} $, then
	\[ [z_1, z_2] \in (0, ke_i, ke_j' + ke_l' + J_{jl}', k)_\pm \subseteq L_1^{(i)} \]
	by \cref{le:Zitoj bracket L1}. If $ z_2 \in J_{ij} + J_{li} $, then $ [z_1, z_2] \in J_{ij} + J_{li} $ by \cref{le:Zitoj bracket Jpq}. In the final subcase $ z_2 \in Z_1^{(i)} $, we have $ [z_1, z_2] \in Z_1^{(i)} $ by \cref{le:Zitoj bracket} (as we already observed in \cref{pr:Z-decomp}\cref{item:Z:2}).
\end{proof}

\begin{lemma}\label{lem:igrading:00}
	Let $ i \in \{1,2,3\} $. Then $ [L_{0}^{(i)}, L_0^{(i)}] \subseteq L_0^{(i)} $.
\end{lemma}
\begin{proof}
	Let $(i\, j\, l)$ denote the unique cyclic permutation of~$(1\, 2\, 3)$ starting from $ i $. Let $ z_1, z_2 \in L_0^{(i)} $. We will write
	\[ z_1 = (0, b, b', 0)_\pm \quad \text{or} \quad z_2 = (0, c, c', 0)_\pm \]
	if $ z_1 $ or $ z_2 $, respectively, lies in $ L_{\pm 1} $, where $ b,c \in J_{ij} + J_{li} $ and $ b', c' \in J_{ij}' + J_{li}' $.
	
	First, we consider the case $ z_1 \in \{x,y\} $. We have $ [x,x] = 0 = [y,y] $ and $ [x,y] = \xi $ by \cref{def:L bracket}.
	By \cref{def:L0 bar}\cref{item:zeta and xi} and \cref{def:Li as J-module}, we have $ [x,L_0] \leq kx $ and $ [y,L_0] \leq ky $.
	This already covers the subcases $ z_2 \in \{x,y\} $ and $ z_2 \in L_0 $. For $ z_2 \in L_1 $, we have $ [y,z_2] = 0 $ and $ [x,z_2] = -(0, c, c', 0)_- \in L_0^{(i)} $ by \cref{def:L bracket}. The subcase $ z_2 \in L_{-1} $ is completely similar.
	
	Second, we assume that $ z_1 \in L_{\pm 1} $. We begin with the subcase $ z_2 \in L_{\pm 1} $. \Cref{def:L bracket} yields that $ [z_1, z_2] $ lies in $ kx $ or $ ky $ if $ z_1 $, $ z_2 $ are both in $ L_1 $ or both in $ L_{-1} $, and that 
	\begin{align*}
		[z_1, z_2] &= b \times c + \brackets[\big]{-T(b,c') \zeta + T(c,b') (\xi - \zeta) + \dd_{b,c'} + \dd_{c,b'}} + b' \times c'
	\end{align*}
	if $ z_1 \in L_{-1} $ and $ z_2 \in L_1 $. Now $ b \times c \in J_{jj}' + J_{ll}' + J_{jl}' $ and $ b' \times c' \in J_{jj} + J_{ll} + J_{jl} $ by \cref{le:times blocks} and $ \dd_{b,c'}, \dd_{c,b'} \in Z_{i \to i} + Z_{j \to l} + Z_{l \to j} $ by \cref{def:Zijlm}, so $ [z_1, z_2] \in L_0^{(i)} $. 
	
	We now turn to the subcase $ z_2 \in ke_j + ke_l + J_{jl} $ and we simultaneously consider the subcase $ z_2' \in ke_j' + ke_l' + J_{jl}' $. By \cref{def:Li as J-module},
	\begin{align*}
		[z_1, z_2] &= -\brackets[\big]{0, 0, z_2 \times b, T(z_2, b')}_\pm, \quad [z_1, z_2'] = \brackets[\big]{T(b,z_2'), z_2' \times b', 0, 0}_\pm
	\end{align*}
	where $ T(z_2, b') = 0 $, $ T(b,z_2') = 0 $ by \cref{le:T formulas} and $ z_2 \times b \in J_{ij}' + J_{il}' $, $ z_2' \times b' \in J_{ij} + J_{il} $ by \cref{le:times blocks}. Hence $ [z_1, z_2] \in L_0^{(i)} $. Finally, if $ z_2 \in Z_0^{(i)} $, then $ \ad_{z_2} $ leaves $ (0, J_{ij} + J_{li}, J_{ij}' + J_{li}', 0)_\pm $ invariant by \cref{def:L0 bar,le:Zitoj bracket L1}, so that $ [z_1, z_2] \in L_0^{(i)} $.
	
	Third, we consider the cases $ z_1 \in k e_j + ke_l + J_{jl} $ and $ z_1' \in ke_j' + ke_l' + J_{jl}' $. For $ z_2 \in ke_j + ke_l + J_{jl} $, we have $ [z_1, z_2] = 0 $ because $ J $ is abelian and
	\[ [z_1', z_2] = \dd_{z_2, z_1'} \in Z_{j \to j} + Z_{j \to l} + Z_{l \to j} + Z_{l \to l} \subseteq L_0^{(i)}. \]
	If $ z_2 \in Z_0^{(i)} $, then $ \ad_{z_2} $ leaves $ ke_j + ke_l + J_{jl} $ and $ ke_j' + ke_l' + J_{jl}' $ invariant by \cref{rem:xi zeta formulas,le:Zitoj bracket Jpq}, which implies $ [z_1, z_2] \in L_0^{(i)} $.
	
	Fourth, in the case $ z_1 \in \{\xi, \zeta\} $ and $ z_2 \in Z_0^{(i)} $, we have $ [z_1, z_2] = 0 $ by \cref{rem:xi zeta formulas}. Finally, $ Z_{i \to i} + Z_{j \to j} + Z_{l \to l} + Z_{j \to l} + Z_{l \to j} $ is a Lie subalgebra of $ L $ by \cref{le:Zitoj bracket}, so $ [z_1, z_2] \in L_0^{(i)} $ in each possible case.
\end{proof}

\begin{proposition}\label{pr:igrading}
	For each $i \in \{ 1,2,3 \}$, the decomposition $L = \bigoplus_{p=-2}^{2} L^{(i)}_p$ defines a $5$-grading of $L$.
\end{proposition}
\begin{proof}
	Denote the cyclic permutation of $ (1\, 2\, 3) $ starting with $ i $ by $ (i\, j \, l) $.
	We have to check that $[L^{(i)}_p, L^{(i)}_q] \subseteq L^{(i)}_{p+q}$ for all possible pairs $(p,q)$ with $-2 \leq p \leq q \leq 2$. This will, of course, require a case distinction. There is some obvious symmetry in the definitions, and the cases $ (-1,1) $, $ (0,1) $ and $ (0,0) $ are already covered by \cref{lem:igrading:-11,lem:igrading:01,lem:igrading:00}, leaving us with the following $6$ cases for $(p,q)$ to check:
	\begin{align*}
		&&&(2, 2), && (1, 2), && (0, 2), && (-1, 2), && (-2, 2) && (1,1).
	\end{align*}
	\begin{description}\setlength{\itemsep}{1.5ex}
        \item[Case $(2,2)$]
        	This case is trivial because $[e_i, e_i] = 0$.
        \item[Case $(1,2)$]
        	Let $ z \in L^{(i)}_1 $.
        	We have to check that $ [z,e_i] = 0 $.
        	If $ z \in L_{\pm 1} $, this holds by \cref{le:ad ei}.
        	If $ z \in J_{ij} + J_{li} $, it holds because $[J, J] = 0$.
        	Finally, if $ z \in Z_{i \to j} + Z_{i \to l} $, it holds by \cref{le:Zitoj bracket Jpq}\cref{le:Zitoj bracket Jpq:ij} with $p=q=i$.
        \item[Case $(0,2)$]
        	Let $ z \in L^{(i)}_{0} $.
        	In most cases, we have $ [z,e_i] = 0 $:
        	By \cref{def:Li as J-module} if $ z \in kx + ky $;
        	by \cref{le:ad ei} if $ z \in L_{\pm 1} $;
        	by \cref{rem:Zijlm zero} if $ z \in ke_j' + ke_l' + J_{jl}' $;
        	by \cref{rem:xi zeta formulas} if $ z = \xi $;
        	and finally, because of $ [J,J] = 0 $ if $ z \in ke_j + ke_l + J_{jl} $.
        	If $ z \in Z_{i \to i} + Z_{j \to j} + Z_{l \to l} + Z_{j \to l} + Z_{l \to j} $, then $[z,e_i] \in ke_i$ by \cref{le:Zitoj bracket Jpq}.
        	Finally, $ [\zeta, e_i] = e_i $ by \cref{rem:xi zeta formulas}, so $ [z,e_i] \in ke_i = L^{(i)}_2 $ in all cases.
        \item[Case $(-1,2)$]
        	Let $ z \in L^{(i)}_{-1} $.
        	If $ z \in L_{\pm 1} $, then $[z,e_i] \in L_1^{(i)}$ by \cref{le:ad ei}.
        	If $ z \in J_{ij}' + J_{li}' $, then
        	\[ [z,e_i] = \dd_{e_i, z} \in Z_{ii,ij} + Z_{ii,il} \subseteq Z_{i \to j} + Z_{i \to l} = Z_1^{(i)}. \]
        	Finally, in the case $ z \in Z^{(i)}_{-1} = Z_{j \to i} + Z_{l \to i} $, we have $ [z,e_i] \in J_{ij} + J_{li} \subseteq L^{(i)}_{1} $ by \cref{le:Zitoj bracket Jpq}\cref{le:Zitoj bracket Jpq:ij}.
        \item[Case $(-2,2)$]
        	This follows from the fact that $[e_i', e_i] = \dd_{e_i, e_i'} \in Z_{i \to i} \leq L^{(i)}_0$.
        	
        \item[Case $ (1,1) $]
        	Let $ z_1, z_2 \in L_1^{(i)} $. We will write
        	\[ z_1 = (0, \lambda_i e_i, b', \lambda)_\pm \quad \text{or} \quad z_2 = (0, \mu_i e_i, c', \mu)_\pm \]
        	if $ z_1 $ or $ z_2 $, respectively, lies in $ L_{\pm 1} $, where $ \lambda_i, \lambda, \mu_i, \mu \in k $ and $ b', c' \in ke_j' + ke_l' + J_{jl}' $. In the case $ z_1, z_2 \in L_1 $, we have
        	\[ [z_1, z_2] = \brackets[\big]{T(\lambda_i e_i, c') - T(\mu_i e_i, b')} x = 0 \]
        	by \cref{def:L bracket} and \cref{le:T formulas}\cref{le:T formulas:ortho}. The case $ z_1, z_2 \in L_{-1} $ is completely similar. If $ z_1 \in J_{ij} + J_{li} $ and $ z_2 \in L_{\pm 1} $, then
        	\[ [z_1, z_2] = \brackets[\big]{0, 0, z_1 \times \mu_i e_i, T(z_1, c')}_\pm = 0 \]
        	by \cref{def:Li as J-module}, \cref{le:times blocks} and \cref{le:T formulas}\cref{le:T formulas:ortho}. In the case $ z_1 \in Z_1^{(i)} $, $ z_2 \in L_{\pm 1} $, we have $ [z_1, z_2] = 0 $ by \cref{le:Zitoj bracket L1}\cref{le:Zitoj bracket Jpq:ij}. Finally, assume that $ z_1, z_2 \in L_0 $. Since $ J $ and $ Z_1^{(i)} $ are abelian, the latter by \cref{le:Zitoj bracket}\cref{le:Zitoj bracket:subalg}, we may assume that $ z_1 \in Z_1^{(i)} $ and $ z_2 \in J_{ij} + J_{li} $. Then $ [z_1, z_2] \in J_{ii} = L_2^{(i)} $ by \cref{le:Zitoj bracket Jpq}\cref{le:Zitoj bracket Jpq:ij}.
        	\qedhere
	\end{description}
\end{proof}

We now obtain our promised $F_4$-grading.

\begin{remark}\label{rem:F4 grid}
	Let $ e_1, e_2, e_3, e_4 $ denote the standard basis of $ \Z^4 $ and put
	\begin{align*}
		\Phi &\coloneq \{\pm 2e_i \mid i \in \{1,2,3,4\}\} \sqcup \{\epsilon_1 e_i + \epsilon_2 e_j \mid i \ne j \in \{1,2,3,4\}, \epsilon_1, \epsilon_2 \in \{\pm 1\}\} \\
		&\qquad{}\sqcup \{\epsilon_1 e_1 + \epsilon_2 e_2 + \epsilon_3 e_3 + \epsilon_4 e_4 \mid \epsilon_1, \ldots, \epsilon_4 \in \{\pm 1\}\}
	\end{align*}
	and $ \Phi^0 \coloneq \Phi \cup \{(0,0,0,0)\} $. Further, let $ (\delta_1, \ldots, \delta_4) $ be a standard system of simple roots in $ F_4 $ (in the sense of \cref{def:F4 basis order}). Then there exists a unique group homomorphism $ \omega \colon \Z F_4 \to \Z^4 $ mapping
	\[ \delta_1 \mapsto (1,1,-1,-1), \quad \delta_2 \mapsto (-2,0,0,0), \quad \delta_3 \mapsto (1, -1, 0, 0), \quad \delta_4 \mapsto (0,1,1,0). \]
	A straightforward computation shows that $ \omega(\Phi) = F_4 $. As we will now see in \cref{thm:lie F4 grading}, the set $ \Phi $ is the natural description of the root system $ F_4 $ in our context. Notice that this description is dual to the one in \cite[VI.4.9]{BourbakiLie46}, but since $ F_4 $ is a self-dual root system, this difference is irrelevant. We will occasionally identify $ F_4^0 $ with the set $ \Phi^0 \subseteq \Z^4 $ to simplify notation, similarly as we did for $ G_2 $ (see \cref{rem:G2 grid}).
\end{remark}

\begin{theorem}\label{thm:lie F4 grading}
	Let $ k $ be a commutative ring, let $C$ be a multiplicative conic alternative algebra, let $\Gamma \in k^ 3$ and let $ (J,J') $ be the cubic norm pair induced by the cubic Jordan matrix algebra $\Her_3(C, \Gamma)$. Let $ L = L(J,J') $ be the Lie algebra constructed in \cref{sec:lie const}, let $ \Phi,\Phi^0 \subseteq \Z^4 $ be as in \cref{rem:F4 grid} and define $ L_p^{(i)} $ for all $ i \in \{1,2,3\} $ and $ p \in \{-2,-1,0,1,2\} $ as in \cref{def:3 new gradings}.
	For all $p,i,j,l \in \{ -2,-1,0,1,2 \}$ (in particular, for all $ (p,i,j,l) \in \Phi^0 $), we put
	\[ L_{pijl} \coloneq L_p \cap L_i^{(1)} \cap L_j^{(2)} \cap L_l^{(3)}. \]
	Then $(L_\alpha)_{\alpha \in \Phi^0}$ is an $F_4$-grading of $ L $ with
	\[ L_\alpha \cong \begin{cases}
		(k,+) & \text{if } \alpha \text{ is long,} \\
		(C,+) & \text{if } \alpha \text{ is short}
	\end{cases} \]
	for all $\alpha \in F_4$. Further, if $ J \ne 0 \ne J' $, then
	\[ \Phi^0 = \{ (p,i,j,l) \in \{ -2,-1,0,1,2 \}^4 \mid L_{pijl} \neq 0 \}. \]
\end{theorem}
\begin{proof}
	By examining the gradings introduced in \cref{def:3 new gradings}, it is not difficult to compute their intersections.
	We have summarized the result in \cref{fig:F4 grading}, where we have written $\overline{2} := -2$ and $\overline{1} := -1$ to save some space, and where we have written $J$ and $J'$ in their ``matrix form''
	\[ J = \begin{pmatrix}
		k e_1 & J_{12} & J_{13} \\ & k e_2 & J_{23} \\ && k e_3
	\end{pmatrix} , \qquad
	J' = \begin{pmatrix}
		k e_1' & J_{12}' & J_{13}' \\ & k e_2' & J_{23}' \\ && k e_3'
	\end{pmatrix} .	\]
	The middle component $Z$ decomposes into $7$ parts: there are $6$ components of the form $Z_{i \to j} $, which we denote as $ Z_{0p_1p_2p_3}$ with $p_i = 1$, $p_j = -1$ and the remaining $p_l = 0$;
	the seventh component is $Z_{0000} = Z_{1 \to 1} + Z_{2 \to 2} + Z_{3 \to 3} + k\xi + k\zeta$.
	
	It is now a simple observation that $ \Phi^0 $ coincides with the set of index tuples $ (p,i,j,l) $ for which $ L_{pijl} $ is \enquote{generically non-trivial}, that is, non-zero except possibly in the degenerate case that $ J=0 $ or $ J'=0 $.
	
	Since each of the four gradings that we have used for the labelling is indeed a grading (by \cref{pr:L is Lie algebra,pr:igrading}), this automatically implies that the resulting decomposition of the whole Lie algebra $L$ is a $ \Phi $-grading and hence an $F_4$-grading.
\end{proof}
\begin{figure}[ht!]
\[
\scalebox{.7}{%
	\begin{tikzpicture}[x=36mm, y=25mm, 	label distance=-3pt]
		\node[mysquare=9mm] at (0,3) (N03) {$\bar{2}000$};
		\node[mysquare=9mm] at (1,1) (N11) {$\bar{1}111$};
		\node[mysquare=16mm] at (1,2) (N12) {$\begin{matrix} \bar{1}\bar{1}11 & \bar{1}001 & \bar{1}010 \\[.5ex] & \bar{1}1\bar{1}1 & \bar{1}100 \\[.5ex] && \bar{1}11\bar{1} \end{matrix}$};
		\node[mysquare=16mm] at (1,3) (N13) {$\begin{matrix} \bar{1}1\bar{1}\bar{1} & \bar{1}00\bar{1} & \bar{1}0\bar{1}0 \\[.5ex] & \bar{1}\bar{1}1\bar{1} & \bar{1}\bar{1}00 \\[.5ex] && \bar{1}\bar{1}\bar{1}1 \end{matrix}$};
		\node[mysquare=9mm] at (1,4) (N14) {$\bar{1}\bar{1}\bar{1}\bar{1}$};
		\node[mysquare=16mm] at (2,1) (N21) {$\begin{matrix} 0200 & 0110 & 0101 \\[.5ex] & 0020 & 0011 \\[.5ex] && 0002 \end{matrix}$};
		\node[mysquare=16mm] at (2,2) (N22) {$\begin{matrix} & 01\bar{1}0 & 010\bar{1} \\[.5ex] 0\bar{1}10 & \boxed{\mathbf{0000}} & 001\bar{1} \\[.5ex] 0\bar{1}01 & 00\bar{1}1 & \end{matrix}$};
		\node[mysquare=16mm] at (2,3) (N23) {$\begin{matrix} 0\bar{2}00 & 0\bar{1}\bar{1}0 & 0\bar{1}0\bar{1} \\[.5ex] & 00\bar{2}0 & 00\bar{1}\bar{1} \\[.5ex] && 000\bar{2} \end{matrix}$};
		\node[mysquare=9mm] at (3,0) (N30) {$1111$};
		\node[mysquare=16mm] at (3,1) (N31) {$\begin{matrix} 1\bar{1}11 & 1001 & 1010 \\[.5ex] & 11\bar{1}1 & 1100 \\[.5ex] && 111\bar{1} \end{matrix}$};
		\node[mysquare=16mm] at (3,2) (N32) {$\begin{matrix} 11\bar{1}\bar{1} & 100\bar{1} & 10\bar{1}0 \\[.5ex] & 1\bar{1}1\bar{1} & 1\bar{1}00 \\[.5ex] && 1\bar{1}\bar{1}1 \end{matrix}$};
		\node[mysquare=9mm] at (3,3) (N33) {$1\bar{1}\bar{1}\bar{1}$};
		\node[mysquare=9mm] at (4,1) (N41) {$2000$};
		\path[ugentred]
			(0,4.5) node (C0) {}
			(1,4.5) node (C1) {}
			(2,4.5) node (C2) {}
			(3,4.5) node (C3) {}
			(4,4.5) node (C4) {};
		\path[ugentblue]
			(4.5,0) node (R0) {}
			(4.5,1) node (R1) {}
			(4.5,2) node (R2) {}
			(4.5,3) node (R3) {}
			(4.5,4) node (R4) {};
		\draw[myedge,ugentblue]
			(N11) -- (N21) -- (N31) -- (N41)
			(N12) -- (N22) -- (N32)
			(N03) -- (N13) -- (N23) -- (N33);
		\draw[myedge,ugentblue,opacity=.15]
			(-.25,0) -- (N30) -- (R0)
			(-.25,1) -- (N11) (N41) -- (R1)
			(-.25,2) -- (N12) (N32) -- (R2)
			(-.25,3) -- (N03) (N33) -- (R3)
			(-.25,4) -- (N14) -- (R4);
		\draw[myedge,ugentred]
			(N11) -- (N12) -- (N13) -- (N14)
			(N21) -- (N22) -- (N23)
			(N30) -- (N31) -- (N32) -- (N33);
		\draw[myedge,ugentred,opacity=.15]
			(0,-.25) -- (N03) -- (C0)
			(1,-.25) -- (N11) (N14) -- (C1)
			(2,-.25) -- (N21) (N23) -- (C2)
			(3,-.25) -- (N30) (N33) -- (C3)
			(4,-.25) -- (N41) -- (C4);
		\draw[myedge,ugentyellow]
			(N03) --  (N12) -- (N21) -- (N30)
			(N13) -- (N22) -- (N31)
			(N14) --  (N23) -- (N32) -- (N41);
	\end{tikzpicture}
}
\]
\caption{Intersecting the four gradings}\label{fig:F4 grading}
\end{figure}
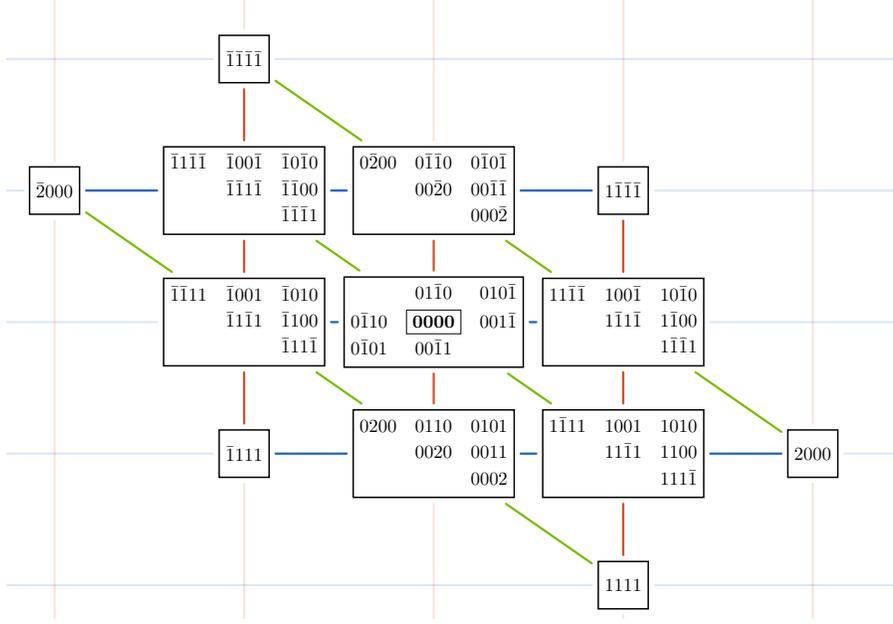

\begin{remark}\label{rem:pi in grid}
	As in \cref{rem:G2 grid,rem:F4 grid}, we identify $ G_2^0 $ and $ F_4^0 $ with subsets of $ \Z^2 $ and $ \Z^4 $, respectively, and we consider the systems of simple roots
	\begin{gather*}
		g_1 = (1,0), \qquad g_2 = (-2,-1) \in G_2, \\
		f_1 = (1,1,-1,-1), \quad f_2 = (-2,0,0,0), \quad f_3 = (1,-1,0,0), \quad f_4 = (0,1,1,0) \in F_4.
	\end{gather*}
	Then the map $\rootproj \colon F_4^0 \to G_2^0$ from \cref{le:index char} (defined with respect to the systems $ (f_1, f_2, f_3, f_4) $ and $ (g_1, g_2) $ of simple roots) is described by \cref{fig:F4 grading}: For any $ \alpha \in F_4^0 $, $ \pi(\alpha) $ is given by the coordinates in the $ \Z^2 $-grid of the rectangular block in which $ \alpha $ is contained. This implies that the $ G_2 $-graded constructed from the $ F_4 $-grading $ (L_\alpha)_{\alpha \in F_4^0} $ as in \cref{pr:lie index} is precisely the $ G_2 $-grading $ (L_\alpha)_{\alpha \in G_2^0} $ from \cref{pr:G2-graded}. Note that we can also describe $\rootproj$ by the formula
	\[ \rootproj\brackets[\big]{(p,i,j,l)} = \brackets*{p, \, \frac{p+i+j+l}{2}} \]
	for all $(p,i,j,l) \in F_4$.
\end{remark}

\begin{remark}[{cf. \ref{rem:3-graded rs}}]
	Keep the notation of \cref{rem:pi in grid} and let $\beta \in G_2$ be short. Then by \cref{le:index fibers}\cref{le:index fibers:short}, the preimage $C_\beta \coloneq \pi^{-1}(\{\beta, -\beta, 0\}) = C_{-\beta}$ is a subsystem of $F_4$ of type $C_3$. Thus $L' \coloneq \bigoplus_{\alpha \in C_\beta} L_\alpha$ is a $C_3$-graded subalgebra of $L$. Specifically, the three possibilities for $L'$ are
	\[ L_{-1,0} \oplus L_{0,0} \oplus L_{1,0}, \quad L_{-1,-1} \oplus L_{0,0} \oplus L_{1,1}, \quad L_{0,-1} \oplus L_{0,0} \oplus L_{0,1}. \]
\end{remark}

\subsection{A Chevalley-type basis of the Lie algebra}\label{subsec:chev}

By specializing the commutator formulas in \cref{ta:Lie} to the case that $ J $ is a cubic Jordan matrix algebra, we obtain commutator formulas for the $ F_4 $-graded Lie algebra from \cref{thm:lie F4 grading}. For example,
\[ [(0, \cubel{c}{12}, 0, 0)_-, (0, 0, \cubel{d}{12}^\iota, 0)_-] = T(\cubel{c}{12}, \cubel{d}{12}^\iota)x = \gamma_1 \gamma_2 \contr(\conconj{c} d) \]
for all $ c,d \in C $ by \cref{le:T formulas}\cref{le:T formulas:dij}. It is not surprising that the constants $ \gamma_1 $, $ \gamma_2 $, $ \gamma_3 $ may appear in these formulas because the structural maps of $ (J,J') $ depend on them. However, we have already observed in \cref{rem:gamma independent} that the isomorphism type of the Lie algebra $ L(J,J') $ does not depend on $ \gamma_1, \gamma_2, \gamma_3 $ as long as they are assumed to be invertible. Hence one should expect that, under this assumption, there exists a family $ (\rhomlie{\alpha})_{\alpha \in F_4} $ of root homomorphisms (each with domain $ k $ or $ C $ and codomain $ L_\alpha $) such that the commutator relations with respect to this family do not involve $ \gamma_1 $, $ \gamma_2 $, $ \gamma_3 $. The goal of this subsection is precisely to find such a family of root homomorphisms. We construct this family by multiplying the \enquote{natural} root homomorphisms in $ L $ by appropriate products of $ \gamma_1 $, $ \gamma_2 $, $ \gamma_3 $ and their inverses.

We will also observe that the elements $ (\rhomlie{\alpha}(1))_{\alpha \in F_4} $ defined by our root homomorphisms share several crucial properties with Chevalley bases of complex semisimple Lie algebras. To achieve this, some root homomorphisms will have to be multiplied not only with a product of $ \gamma_1 $, $ \gamma_2 $, $ \gamma_3 $ and their inverses, but with $ -1 $ as well. An immediate consequence of the existence of a \enquote{Chevalley-type basis} is that if $C$ is faithful as a $k$-algebra, then $L$ contains a copy of the split Chevalley Lie algebra of type $F_4$ over $k$ as a subalgebra, which implies that $L$ is $F_4$-graded in the stronger sense of \cite{BM92}. See \cref{rem:stronger grading,rem:is stronger grading} for details.

In \cref{sec:refining aut}, we will \enquote{exponentiate} the root homomorphisms from this subsection to obtain root homomorphisms in the corresponding group of automorphisms of $ L(J,J') $ (which will then proven to be $ F_4 $-graded as well). Consequently, the commutator formulas in this $ F_4 $-graded group will not involve $ \gamma_1, \gamma_2, \gamma_3 $ either (see \cref{pr:F4 comm}). This aligns with the fact that in the coordinatization theorem for arbitrary $ F_4 $-graded groups (\cref{thm:rgg-param}), the algebraic data coordinatising the group involves only a commutative ring and a multiplicative conic alternative algebra but no choice of constants in the commutative ring. Further, using a Chevalley-type basis has benefits concerning the description of Weyl elements in the automorphism group, which will become apparent in \cref{rem:F4 weyl def}.

\begin{notation}\label{not:F4 grid}
	From now on, we assume that $ \gamma_1, \gamma_2, \gamma_3 $ are invertible. In the following, we will always use the convention from \cref{rem:G2 grid,rem:F4 grid} to identify $ G_2^0 $ and $ F_4^0 $ with subsets of $ \Z^2 $ and $ \Z^4 $, respectively. We will sometimes use the convention from \cref{fig:F4 grading} to write, e.g., $ 11\bar{1} \bar{1} $ as an abbreviation of $ (1,1,-1,-1) $. Further, we denote by $ \pi \colon F_4^0 \to G_2^0 $ the map defined in \cref{rem:pi in grid}, so that $ L_\alpha \le L_{\pi(\alpha)_1} \cap L_{\pi(\alpha)_2} $ for all $ \alpha \in F_4^0 $.
	Finally, for all $\alpha \in F_4$, we put $\pargrp{\alpha} \coloneq k$ if $\alpha$ is long and $\pargrp{\alpha} \coloneq C$ if $\alpha$ is short, so that $L_\alpha \cong \pargrp{\alpha}$ as $k$-modules.
\end{notation}

We begin by defining the root homomorphisms $ (\rhomlie{\alpha})_{\alpha \in F_4} $. It turns out that they can be constructed step by step: We first define root homomorphisms in $ J $ and $ J' $, then root homomorphisms in $ (k,J,J') \cong L_{\pm 1} $ and then root homomorphisms in $ L $. Note that the root homomorphisms in $ J' $ are defined in a different way than those in $ J $ even though $ J \cong J' $. 

\begin{definition}[Root homomorphisms in $J$ and $J'$]
	For any cyclic permutation $(i\,j\,l)$ of $(1\,2\,3)$, we define homomorphisms
	\begin{align*}
		\rhomcub{l} \colon (k,+) \to J &\colon t \mapsto \gamma_i^{-1} \gamma_j t e_l, \\
		\rhomcubop{l} \colon (k,+) \to J' &\colon t \mapsto \gamma_i \gamma_j^{-1} t e_l', \\
		\rhomcub{l} \colon (C,+) \to J &\colon a \mapsto \cubel{\gamma_j^{-1} a}{ij}, \\
		\rhomcubop{l} \colon (C,+) \to J' &\colon a \mapsto \cubel{\gamma_i^{-1} a}{ij}.
	\end{align*}
	Note that we use the same name for the maps defined on $k$ as for the maps defined on $C$, but there will be no danger of confusion in the sequel.
\end{definition}

\begin{definition}[Root homomorphisms in $(k,J,J',k)$]
	Let $p_1,p_2,p_3 \in \{-1,0,1\}$ with $(1,p_1,p_2,p_3) \in F_4$ (or equivalently, $(-1,p_1,p_2,p_3) \in F_4$) and put $s \coloneq p_1+p_2+p_3$. We define a homomorphism $\rhombr{p_1p_2p_3}$ as follows:
	\begin{enumerate}
		\item If $s=-3$, define $\rhombr{p_1p_2p_3} \colon k \to (k,J,J',k) \colon t \mapsto (t, 0, 0, 0)$.
		
		\item If $s=-1$, then there is a unique cyclic permutation $(i\,j\,l)$ of $(1\,2\,3)$ with $p_l \ne p_i = p_j$, and we define
		\[ \rhombr{p_1p_2p_3} \colon \pargrp{(1,p_1,p_2,p_3)} = \pargrp{(-1,p_1,p_2,p_3)} \to (k,J,J',k) \colon a \mapsto (0, \rhomcub{l}(a), 0, 0) \]
		
		\item If $s=1$, then there is a unique cyclic permutation $(i\,j\,l)$ of $(1\,2\,3)$ with $p_l \ne p_i = p_j$, and we define
		\[ \rhombr{p_1p_2p_3} \colon \pargrp{(1,p_1,p_2,p_3)} = \pargrp{(-1,p_1,p_2,p_3)} \to (k,J,J',k) \colon a \mapsto (0, 0, \rhomcubop{l}(a), 0) \]
		
		\item If $s=3$, define $\rhombr{p_1p_2p_3} \colon k \to (k,J,J',k) \colon t \mapsto (0, 0, 0, t)$.
	\end{enumerate}
\end{definition}

\begin{definition}[Root homomorphisms in $L$]\label{def:roothom lie}
	Let $\alpha = (p_0, p_1, p_2, p_3) \in F_4$ and put $ \mathcal{S} \coloneq \mathcal{S}^{\text{sh}} \cup \mathcal{S}^{\text{lo}} $ where
	\begin{align*}
		\mathcal{S}^{\text{lo}} &\coloneq \{\alpha \in F_4 \mid \alpha \text{ is long}\} \setminus \{\bar{1}\bar{1}\bar{1}1, 11\bar{1}1, 1\bar{1}11, \bar{1}111\} \\
		\mathcal{S}^{\text{sh}} &\coloneq \{10\bar{1}0, 0101, 0011, 0\bar{1}\bar{1}0, \bar{1}100, \bar{1}001\}.
	\end{align*}
	We first define an auxiliary map $\rhomlien{\alpha} \colon \pargrp{\alpha} \to L$ as follows.
	\begin{enumerate}
		\item If $\rootproj(\alpha)_1 = -2$, define $\rhomlien{\alpha} \colon k \to L \colon t \mapsto tx$.
		
		\item If $\rootproj(\alpha)_1 = \varepsilon 1$ for $\varepsilon \in \{\mathord{+}, \mathord{-}\}$, define $\rhomlien{\alpha} \colon \pargrp{\alpha} \to L \colon a \mapsto \rhombr{p_1p_2p_3}(a)_\varepsilon$.
		
		\item If $\rootproj(\alpha) = (0,\varepsilon 1)$ for $\varepsilon \in \{\mathord{+}, \mathord{-1}\}$, define $\rhomlien{\alpha} \colon \pargrp{\alpha} \to L \colon a \mapsto \rhomcubvar{l}{\varepsilon}(a) \in J+J' \le L_0$ where $l \in \{1,2,3\}$ is the unique index with $\vert p_l \vert=2$ if $\alpha$ is long and $p_l = 0$ if $\alpha$ is short.
		
		\item If $\rootproj(\alpha) = (0,0)$, then there exist unique $i,j,l \in \{1,2,3\}$ with $p_i = 1$, $p_j=-1$, $p_l=0$ and we define
		\begin{equation}\label{eq:roothom-lie-00part}
			\rhomlien{\alpha} \colon C \to L \colon a \mapsto \begin{cases}
				\lambda_{ij} \dd_{e_i,\cubel{a}{ij}} & \text{if } i<j, \\
				\lambda_{ij} \dd_{e_i, \cubel{\conconj{a}}{ij}} & \text{if } i>j
			\end{cases}
		\end{equation}
		where
		\begin{align*}
			\lambda_{12} &\coloneq \gamma_1^{-1} \gamma_2^{-1} \gamma_3, & \lambda_{13} &\coloneq \gamma_2^{-1}, & \lambda_{23} &\coloneq \gamma_1 \gamma_2^{-1} \gamma_3^{-1}, \\
			\lambda_{21} &\coloneq \gamma_3^{-1}, & \lambda_{31} &\coloneq \gamma_1^{-1} \gamma_2 \gamma_3^{-1}, & \lambda_{32} &\coloneq \gamma_1^{-1}.
		\end{align*}
		
		\item If $\rootproj(\alpha)_1 = 2$, define $\rhomlien{\alpha} \colon k \to L \colon t \mapsto ty$.
	\end{enumerate}
	Further, the \emph{root homomorphism in $L$ with respect to $\alpha$} is
	\[ \rhomlie{\alpha} \colon \pargrp{\alpha} \to L \colon a \mapsto \begin{cases}
		\rhomlien{\alpha}(a) & \text{if } \alpha \notin \mathcal{S}, \\
		\rhomlien{\alpha}(-a) & \text{if } \alpha \in \mathcal{S}.
	\end{cases} \]
\end{definition}

The reason to use the case distinction in \cref{eq:roothom-lie-00part} will only become apparent at a later point, see \cref{rem:parmap negative}.

The following definition is motivated by the classical notion of Chevalley bases in complex semisimple Lie algebras. See, for example, \cite[25.2]{HumphreysLie}.

\begin{definition}\label{def:chev basis}
	\begin{enumerate}
		\item \label{def:chev basis:h}Recall from \cref{def:L0 bar} that $ M $ denotes the Lie algebra $ \bigoplus_{i \ne 0} \gl(L_i) $ which contains $ L_0 $. For all $ \alpha \in F_4 $, we denote by $ h_\alpha $ the element of $ M $ which is defined by $ [h_\alpha, y_\beta] = \cartan{\beta}{\alpha} y_\beta $ for all $ y_\beta \in L_\beta $ and all $ \beta \in F_4 $ with $ \rootproj(\beta)_1 \ne 0 $. Here $ \cartan{\beta}{\alpha} \coloneq 2 \frac{\beta \cdot \alpha}{\alpha \cdot \alpha} $ denotes the Cartan integer.
		
		\item \label{def:chev basis:chev}A \emph{Chevalley-type basis of $L$} is a family $(\cbas{\alpha} \in L_\alpha)_{\alpha \in F_4}$ with the following properties:
		\begin{enumerate}[(a)]
			
			\item For all $\alpha \in F_4$, we have $ [\cbas{\alpha}, \cbas{-\alpha}] = h_\alpha $.
			
			\item For all $\alpha, \beta \in F_4$ with $\alpha+\beta \in F_4$, there exists $c_{\alpha,\beta} \in \Z$ such that $ [\cbas{\alpha}, \cbas{\beta}] = c_{\alpha, \beta} \cbas{\alpha+\beta} $.
			
			\item It is possible to choose the constants $ c_{\alpha,\beta} $ above so that for all $\alpha, \beta \in F_4$ with $\alpha+\beta \in F_4$, we have $c_{\alpha, \beta} = -c_{-\alpha, -\beta}$.
		\end{enumerate}
		A family of integers $ c_{\alpha,\beta} $ for $ \alpha,\beta \in F_4 $ with $ \alpha+\beta \in F_4 $ which satisfies the properties above is said to be \emph{conforming to $(\cbas{\alpha})_{\alpha \in F_4}$}.
	\end{enumerate}
\end{definition}

\begin{remark}
	\begin{enumerate}
		\item Note that a Chevalley-type basis $ (\cbas{\alpha})_{\alpha \in F_4} $ of $ L $ is not actually a basis of the $ k $-module $ L $: Its $ k $-linear span intersects $ L_{0000} $ trivially and for any short root $ \alpha \in F_4 $, the $ k $-span of $ \cbas{\alpha} $ is, in general, a proper submodule of $ L_\alpha \cong C $.
		
		\item Let $ L' $ be a complex semisimple Lie algebra, let $ H $ denote a maximal toral subalgebra of $ L' $, let $ \kappa $ denote the Killing form on $ L' $, let $ \Phi \subseteq H^* $ denote the root system of $ L' $ and let $ h_\alpha, t_\alpha \in L' $ for $ \alpha \in \Phi $ be defined as in \cite[8.2, 8.3]{HumphreysLie}. Then for all $ \alpha, \beta \in \Phi $ and all $ y_\beta $ in the $ \beta $-root space $ L'_\beta $, we have
		\[ [h_\alpha, y_\beta] = \frac{2}{\kappa(t_\alpha, t_\alpha)} [t_\alpha, y_\beta] = \frac{2}{\kappa(t_\alpha, t_\alpha)}\beta(t_\alpha) y_\beta = 2\frac{\kappa(t_\beta, t_\alpha)}{\kappa(t_\alpha, t_\alpha)} y_\beta = \cartan{\beta}{\alpha} y_\beta. \]
		Hence the elements $ h_\alpha $ in the classical theory have the same property as in \cref{def:chev basis}\cref{def:chev basis:h}. In fact, they are the unique elements of $H$ with this property because $ H $ acts faithfully on $ L' $.
		
		\item The numbers $ c_{\alpha, \beta} $ in \cref{def:chev basis}\cref{def:chev basis:chev} are not uniquely determined in general because $ L $ may not be faithful as a $ \Z $-module. For example, this is the case if $ k $ is a field of positive characteristic. Even the image of $ c_{\alpha,\beta} $ in $ k $ need not be uniquely determined because $ C $ (and hence $ L_\gamma $ for all short $ \gamma \in F_4 $) need not be faithful as a $ k $-module.
	\end{enumerate}
\end{remark}

\begin{remark}\label{rem:root-length-sum}
	Let $ \alpha,\beta \in F_4 $ such that $ \alpha+\beta \in F_4 $. If $ \beta $ and $ \alpha+\beta $ are long, then $ \alpha $ is long as well. If $ \beta $ is long and $ \alpha+\beta $ is short, then $ \alpha $ must be short. It follows that, unless $ \alpha, \beta $ are short and $ \alpha + \beta $ is long, we always have a multiplication map $ P_\alpha \times P_\beta \to P_{\alpha+\beta} $ which is one of the maps $ k \times k \to k $, $ C \times C \to C $, $ k \times C \to C $. In the situation that $ \alpha, \beta $ are short and $ \alpha + \beta $ is long (so that $ P_\alpha = P_\beta = C $ and $ P_{\alpha+\beta} = k $), however, the multiplication defined on $ P_\alpha \times P_\beta = C \times C $ has image in $ C $ and not in $ k = P_{\alpha+\beta} $.
\end{remark}

We can now state the crucial property of the root homomorphisms $ (\rhomlie{\alpha})_{\alpha \in F_4} $. We find not only that the commutator formulas in $ L $ are independent of $ \gamma_1, \gamma_2, \gamma_3 $ but also that \emph{all commutator formulas look the same} up to possibly \enquote{twisting} the short root spaces by the conjugation $ \conconj{\: \cdot \:} $ of $ C $. This is the same kind of \enquote{twisting} that we have already seen in \cref{def:twistgrp}.

\begin{proposition}\label{pr:chev-basis}
	There exists a unique family $ c=(c_{\alpha,\beta})_{\alpha, \beta \in F_4, \alpha+\beta \in F_4} $ of integers in $ \{-2,\ldots, 2\} $ such that for all commutative rings $ k $, all multiplicative conic alternative $ k $-algebras $ C $ and all invertible $ \Gamma = (\gamma_1, \gamma_2, \gamma_3) \in \Mat_3(k) $, the following hold:
	\begin{enumerate}
		\item \label{pr:chev-basis:comm}Let $ \alpha, \beta \in F_4 $ such that $ \alpha+\beta \in F_4 $. For each $ \gamma \in F_4 $, put $ T_\gamma \coloneq \{\id_C, \conconj{\:\cdot\:}\} \le \Aut(P_\gamma) $ if $ \gamma $ is short and $ T_\gamma \coloneq \{\id_k\} \le \Aut(P_\gamma) $ if $ \gamma $ is long. If $ \alpha, \beta $ are short and $ \alpha+\beta $ is long, then $ c_{\alpha,\beta} $ is even and there exists $ \mathord{*} \in T_\alpha = \{\id_C, \conconj{\:\cdot\:}\} $ such that the following commutator relation is satisfied for all $ c,d \in C $:
		\[ \bigl[ \rhomlie{\alpha}(c), \rhomlie{\beta}(d)\bigr] = \rhomlie{\alpha+\beta}\brackets[\big]{c_{\alpha,\beta}/2 \cdot \contr(c^*d)}. \]
		Otherwise, there exist $ \mathord{*} \in T_\alpha $, $ \mathord{\star} \in T_\beta $ and $ \sigma \in T_{\alpha+\beta} $ such that the following commutator relation is satisfied for all $ c \in P_\alpha $, $ d \in P_\beta $:
		\[ \bigl[ \rhomlie{\alpha}(c), \rhomlie{\beta}(d)\bigr] = \rhomlie{\alpha+\beta}\brackets[\big]{c_{\alpha,\beta} (c^* \cdot d^\star)^\sigma}. \]
		(Here $ c \cdot d \in P_{\alpha+\beta} $ by \cref{rem:root-length-sum}.)
		
		\item The family $ (\rhomlie{\alpha}(1))_{\alpha \in F_4} $ (where \enquote{$ 1 $} should be read as $ 1_k $ or as $ 1_C $, depending on $ \alpha $) is a Chevalley-type basis of $ L $ to which $ c $ conforms.
	\end{enumerate}
	Further, for all $ \alpha \in F_4 $, the element $ h_\alpha $ lies in $ L_0 $ and satisfies $ [h_\alpha, y_\beta] = \cartan{\beta}{\alpha} y_\beta $ for all $ y_\beta \in L_\beta $ and all $ \beta \in F_4 $.
\end{proposition}
\begin{proof}
	A computer calculation (see \cref{subsec:computer}) shows that for all $ \alpha, \beta \in F_4 $ such that $ \alpha+\beta \in F_4 $, there exists $ c_{\alpha,\beta} \in \{-2, \ldots, 2\} $ such that $ [\rhomlie{\alpha}(1), \rhomlie{\beta}(1)] = \rhomlie{\alpha+\beta}(c_{\alpha,\beta}) $ for all possible choices of $ k $, $ C $, $ \Gamma $. In particular, the validity of this equation for $ k=C=\mathbb{C} $ (or any field of characteristic~$ 0 $) shows that the family $ c $ is uniquely determined by this property (because then $ L $ is free as a $ \Z $-module). A further computation shows that $ [[\cbas{\alpha}, \cbas{-\alpha}], y_\beta] = \cartan{\beta}{\alpha} y_\beta $ for all $ \alpha, \beta \in F_4 $ and $ y_\beta \in L_\beta $. This implies that $ h_\alpha = [\cbas{\alpha}, \cbas{-\alpha}] \in [L_\alpha, L_{-\alpha}] \subseteq L_0 $ for all $ \alpha \in L $. The remaining assertions also follow from a computation.
\end{proof}

\begin{remark}\label{rem:is stronger grading}
	Recall from \cref{rem:stronger grading} the stronger notion of root graded Lie algebras from \cite{BM92}. It follows from the existence of a Chevalley-type basis in $L$ that $L$ is $F_4$-graded in this stronger sense if $C$ is faithful as a $k$-algebra. This additional condition, which  by \cite[11.14]{GPR24} is always satisfied if $2$ is invertible in $k$, is need to ensure that the $k$-submodule $k\cbas{\alpha}$ for a short root $\alpha \in F_4$ is isomorphic to $k$ and not to a quotient of $k$.
	
	Since an arbitrary cubic norm pair $(\bar{J}, \bar{J}')$ need not contain a distinguished base point, there is no evident way to define a Chevalley-type basis (in the sense of \cref{def:chev basis}, but with $F_4$ replaced by $G_2$) of the corresponding Lie algebra $L(\bar{J}, \bar{J}')$. Hence there is no reason to think that $L(\bar{J}, \bar{J}')$ is always $G_2$-graded in the stronger sense. However, assume now that $(\bar{J}, \bar{J})$ is (induced by) a cubic norm structure and that its distinguished base point $1_{\bar{J}}$ is unimodular (see \cref{rem:CNP axioms}\cref{rem:CNP axioms:unimod}). Denote by $(\rhomlieG{\alpha})_{\alpha \in G_2}$ the root homomorphisms in $L(\bar{J}, \bar{J})$ from \cref{def:lie G2 natural rhom}. For all roots $\alpha, \beta \in G_2$ where $\alpha$ is long and $\beta$ is short, put
	\begin{align*}
		\mathcal{S}_1 &\coloneq \{(-2, -1), (-1,-2)\}, &\mathcal{S}_2 &\coloneq \{(-1, 0), (0, -1)\}, \\
		\cbas{\alpha} &\coloneq \begin{cases}
			\rhomlieG{\alpha}(-1_k) & \text{if } \alpha \in \mathcal{S}_1, \\
			\rhomlieG{\alpha}(1_k) & \text{otherwise,}
		\end{cases} & \cbas{\beta} &\coloneq \begin{cases}
			\rhomlieG{\beta}(-1_{\bar{J}}) & \text{if } \alpha \in \mathcal{S}_2, \\
			\rhomlieG{\beta}(1_{\bar{J}}) & \text{otherwise,}
		\end{cases}
	\end{align*}
	Using that $T(1_{\bar{J}}, 1_{\bar{J}}) = 3_k$ and $D(1_{\bar{J}}, 1_{\bar{J}}) = 2\id_{\bar{J}}$, one can show that $(\cbas{\alpha})_{\alpha \in G_2}$ is a Chevalley-type basis of $L(\bar{J}, \bar{J})$ of type $G_2$. Indeed, it is straightforward to see that there exist unique integers $c_{\alpha,\beta}$ for all $\alpha,\beta \in G_2$ with $\alpha+\beta \in G_2$ and unique integers $d_{\alpha,\beta}$ for all $\alpha,\beta \in G_2$ such that for all cubic norm structures $(\bar{J}, \bar{J})$, we have $[\cbas{\alpha}, \cbas{\beta}]=c_{\alpha,\beta} \cbas{\alpha+\beta}$ and $[[\cbas{\alpha}, \cbas{-\alpha}], y_\beta] = d_{\alpha,\beta} y_\beta$. Proving that $(\cbas{\alpha})_{\alpha \in G_2}$ is a Chevalley-type basis reduces to proving that these integers satisfy certain properties, so it suffices to show that it is a Chevalley-type basis for some cubic norm structure that is free as a $\Z$-module. This can be done for cubic Jordan matrix algebras with the strategy outlined in \cref{subsec:computer}.
\end{remark}

\section{Refining the \texorpdfstring{$ G_2 $}{G2}-grading of the automorphism group}\label{sec:refining aut}

We keep up the notation from \cref{sec:lie F4}, including \cref{not:F4 grid}. By \cref{thm:G2 graded group}, we have a $ G_2 $-grading $ (U_\alpha)_{\alpha \in G_2} $ in the automorphism group of the Lie algebra $ L=L(J,J') $ where $ J = J' = \Her_3(C, \Gamma) $ is a cubic Jordan matrix algebra over the conic $ k $-algebra $ (C, \connorm) $. We put $ G \coloneq \langle U_\alpha \mid \alpha \in G_2 \rangle \le \Aut(L) $. Our goal is to refine $ (U_\alpha)_{\alpha \in G_2} $ to an $ F_4 $-grading $ (U_\alpha)_{\alpha \in F_4} $ of $ G $, using the $ F_4 $-grading of $ L(J,J') $ constructed in \cref{sec:lie F4}.

The system $(f_1,f_2,f_3,f_4)$ of simple roots in $F_4$ that induces the map $\pi \colon F_4^0 \to G_2^0$ (see \cref{rem:pi in grid}) is impractical to work with because it contains two simple roots that map to zero. Instead, we will consider the following system of simple roots, which in addition has the property that three of its elements map to the positive system in $G_2$ for which we have determined the commutator relations in \cref{pr:G2 comm}.

\begin{notation}
	We denote by $\Delta_F=(\delta_1,\delta_2,\delta_3,\delta_4)$ the standard system of simple roots in $F_4$ given by
	\[ \delta_1 = 11 \bar{1} \bar{1}, \quad \delta_2 = \bar{2}000, \quad \delta_3 \coloneq 1\bar{1}00, \quad \delta_4 \coloneq 0110. \]
	Further, we denote by $T \coloneq \{\pm 1\} \times \{\pm 1\}$ the twisting group of $(k,C)$, which acts on $k$ and $C$ (see \cref{def:twistgrp}).
\end{notation}

For roots $ \alpha \in F_4 $ with $ \rootproj(\alpha) \ne (0,0) $, we can define root homomorphisms in $ \Aut(L) $ by exponentiating the root homomorphisms in $ L $ from \cref{subsec:chev}.

\begin{definition}\label{def:F4 nonzero rhom}
	Let $ \alpha \in F_4 $ such that $ \pi(\alpha) \ne (0,0) $ and let $ \rhomlie{\alpha} \colon P_\alpha \to L_\alpha \le L_{\pi(\alpha)} $ denote the root homomorphism from \cref{def:roothom lie}. Then the \emph{root homomorphism in $ G $ with respect to $ \alpha $} is
	\[ \rhomgrp{\alpha} \colon P_\alpha \to G \colon a \mapsto \exp_{\pi(\alpha)}\brackets[\big]{\rhomlie{\alpha}(a)} \]
	where $ \exp_{\pi(\alpha)} \colon L_{\pi(\alpha)} \to \Aut(L) $ is the $ \pi(\alpha) $-parametrization (in the $ G_2 $-graded Lie algebra $ L $) from \cref{pr:param}. Further, the \emph{root group associated to $ \alpha $} is $ U_\alpha \coloneq \rhomgrp{\alpha}(P_\alpha) $.
\end{definition}

For the remaining six roots that map to $ 0 $ in $ G_2 $, we define the corresponding root homomorphisms by conjugating root homomorphisms from \cref{def:F4 nonzero rhom} by appropriate reflections.

\begin{definition}\label{def:F4 weyl}
	For all $ \alpha \in F_4 $ such that $ \pi(\alpha) \ne (0,0) $, we define
	\[ w_\alpha \coloneq \rhomgrp{-\alpha}(-1_\alpha) \rhomgrp{\alpha}(1_\alpha) \rhomgrp{-\alpha}(-1_\alpha) \in G \le \Aut(L) \]
	where $ 1_\alpha = 1_k $ if $ \alpha $ is long and $ 1_\alpha = 1_C $ if $ \alpha $ is short.
\end{definition}

\begin{lemma}\label{le:weyl F4 refl}
	For all $ \alpha \in F_4 $ such that $ \pi(\alpha) \ne (0,0) $, the element $ w_\alpha \in \Aut(L) $ is an $ \alpha $-reflection.
\end{lemma}
\begin{proof}
	For all $ \alpha, \beta \in F_4 $, we have to show that $ w_\alpha(L_\beta) = L_{\beta^{\reflbr{\alpha}}} $. We have verified these $ 48 \cdot 48 $ assertions on a computer. See \cref{subsec:computer} for details.
\end{proof}

\begin{remark}
	We will only need the statement of \cref{le:weyl F4 refl} for $ \alpha \in \{\delta_1, \delta_4\} $.
\end{remark}

\begin{remark}\label{rem:F4 weyl def}
	Let $\alpha \in F_4$ with $\rootproj(\alpha) \ne (0,0)$. While the root group $U_\alpha = \exp_{\pi(\alpha)}(L_\alpha)$ does not depend on the choice of the root homomorphism $\rhomlie{\alpha} \colon P_\alpha \to L_\alpha$ in $L$, the root homomorphism $\rhomgrp{\alpha} \colon P_\alpha \to U_\alpha$ in $G$ does. In fact, the \enquote{twisting} of the $\rhomlie{\alpha}$ by products of $\gamma_1$, $\gamma_2$, $\gamma_3$, $-1$ in \cref{def:roothom lie} is precisely what ensures that $w_\alpha$ is indeed a reflection. If we had chosen a different map $\tilde{\vartheta}_\alpha \colon P_\alpha \to L_\alpha$ and defined $\tilde{\theta}_\alpha \colon P_\alpha \to G \colon a \mapsto \exp_{\rootproj(\alpha)}(\tilde{\vartheta}_\alpha(a))$ accordingly, then the desired reflections might, for example, have the form $ w_\alpha = \tilde{\theta}_{-\alpha}(1_\alpha) \tilde{\theta}_\alpha(1_\alpha) \tilde{\theta}_{-\alpha}(1_\alpha) $ (as in \cref{le:phihor factor} for the $ G_2 $-grading). Even worse, their expression in terms of $ \tilde{\theta}_\alpha $, $\tilde{\theta}_{-\alpha}$ may involve the structure constants $ \gamma_1, \gamma_2, \gamma_3 $. It is this obstruction that originally prompted our search for a Chevalley-type basis of $ L $.
\end{remark}

\begin{remark}\label{rem:F4 0 rhom}
	Each of the six roots $ \alpha \in F_4 $ with $ \pi(\alpha) = (0,0) $ can be expressed as $ \alpha = \beta^{\reflbr{\delta_i}} $ for some $ i \in \{1,4\} $ and some $ \beta \in F_4 $ with $ \pi(\alpha) \ne (0,0) $. Explicitly, we may choose $ \beta $ and $ i $ as in \cref{ta:0rhom choice}.
	\begin{table}
		\centering\begin{tabular}{ccccccc}
			\toprule
			$ \alpha $ & $ 01\bar{1}0 $ & $ 0\bar{1}10 $ & $ 010\bar{1} $ & $ 0\bar{1}01 $ & $ 001\bar{1} $ & $ 00\bar{1}1 $ \\
			\midrule
			$ \beta $ & $ \bar{1} 001 $ & $ 100\bar{1} $ & $ \bar{1}010 $ & $ 10\bar{1}0 $ & $ 0\bar{1}0\bar{1} $ & $ 0101 $ \\
			\midrule
			$ i $ & 1 & 1 & 1 & 1 & 4 & 4 \\
			\midrule
			$\eta_\alpha$ & $(-1,1)$ & $(-1,1)$ & $(1,-1)$ & $(1,-1)$ & $(1,-1)$ & $(1,-1)$ \\
			\bottomrule
		\end{tabular}
		\bigskip
		\caption{The choice of $ \beta $ and $ i $ in \cref{rem:F4 0 rhom}, in dependence of $ \alpha $}
		\label{ta:0rhom choice}
	\end{table}
	Hence for each such $ \alpha $, it follows from \cref{le:conj by refl,le:weyl F4 refl} that
	\[ \exp_\alpha \colon L_\alpha \to \Aut(L) \colon b \mapsto w_i^{-1} \circ \exp_{\beta}\brackets[\big]{w_i(b)} \circ w_i \]
	is an $ \alpha $-parametrization where $ \beta \in F_4 $ and $ i \in \{1,4\} $ are chosen as in \cref{ta:0rhom choice} (in dependence of $ \alpha $) and $ w_i \coloneq w_{\delta_i} $. Hence we may define
	\[ \rhomgrp{\alpha} \colon C \to \Aut(L) \colon a \mapsto \exp_\alpha\brackets[\big]{\rhomlie{\alpha}(a)} = w_i^{-1} \circ \exp_\beta\brackets[\big]{w_i(\rhomlie{\alpha}(a))} \circ w_i. \]
	Further, a computer calculation (see \cref{subsec:computer}) shows that $ w_i(\rhomlie{\alpha}(a)) = \rhomlie{\beta}(\eta_\alpha.a) $ for all $ a \in C $ where $ \eta_\alpha \in T $ is given by the last row in \cref{ta:0rhom choice}.
\end{remark}

\begin{definition}\label{def:F4 0 rhom}
	Let $ \alpha \in F_4 $ with $ \pi(\alpha) = (0,0) $. Then the map $ \rhomgrp{\alpha} \colon C \to \Aut(L) $ from \cref{rem:F4 0 rhom} (with $ \beta $ and $ i $ chosen as in \cref{ta:0rhom choice}) is called the \emph{root homomorphism with respect to $ \alpha $}, and its image $ U_\alpha $ is called the \emph{root group associated to~$ \alpha $}.
\end{definition}

\begin{remark}\label{pr:F4 gen}
	For all $ \alpha \in F_4 $ with $ \pi(\alpha) = (0,0) $, we have that $ w_i $ and $ U_\beta $ are contained in $ G $, the group generated by $ (U_\gamma)_{\gamma \in G_2} $. Hence the group generated by $ (U_\alpha)_{\alpha \in F_4} $ is exactly $ G $ (and not a bigger group).
\end{remark}

Having defined all root homomorphisms $ (\rhomgrp{\alpha})_{\alpha \in F_4} $, we now verify that $ (U_\alpha)_{\alpha \in F_4} $ satisfies the axioms of an $ F_4 $-grading of $G$. We begin with the existence of Weyl elements.

\begin{notation}[cf.~\ref{rem:coordinatization parmap}]\label{def:existence parmap}
	We denote by $\etae \colon F_4 \times \Delta_F \to T$ the $\Delta_F$-parity map with values in $T$ (in the sense of \cref{def:parmap}) that is given by the values in \cref{ta:existence parmap}. Note that the definition of $\etae$ does not involve the actions of $T$ on $k$ and $C$, so $\etae$ is independent of $k$, $C$ and $\Gamma$.
\end{notation}

\begin{proposition}\label{pr:F4 weyl}
	For all $ \delta \in \Delta_F $, $ \alpha \in F_4 $ and $ a \in P_\alpha $, we have
	\begin{equation}\label{eq:F4 weyl conj}
		w_\delta^{-1} \circ \rhomgrp{\alpha}(a) \circ w_\delta = \rhomgrp{\alpha^{\reflbr{\delta}}}(\etae_{\alpha, \delta}.a).
	\end{equation}
	In particular, $ w_\delta $ is a $ \delta $-Weyl element with respect to $ (U_\beta)_{\beta \in F_4} $, $ (\rhomgrp{\alpha})_{\alpha \in F_4} $ is a parametrization of $ (U_\beta)_{\beta \in F_4} $ by $ (k,C) $ with respect to $ (w_\delta)_{\delta \in \Delta_F} $ and $ \etae $ (in the sense of \cref{def:rgg-param}) and $ (U_\beta)_{\beta \in F_4} $ satisfies Axiom~\ref{def:rgg}\cref{def:rgg:weyl}.
\end{proposition}
\begin{proof}
	The first assertion is a straightforward computation, but a very long one: For all $ \alpha \in F_4 $ and for arbitrary $ a \in P_\alpha $, we have to compute the action of $ w_\delta^{-1} \circ \rhomgrp{\alpha}(a) \circ w_\delta $ and $ \rhomgrp{\alpha^{\reflbr{\delta}}}(\etae_{\alpha, \delta}.a) $ on a generating set of $ L $ (such as $ L_{-2} \cup L_1 $, which is spanned by $ 15 $ $ F_4 $-root spaces) and show that the resulting expressions are identical. (This procedure is precisely how we obtained the values of $ \etae$.) We have performed these calculations on a computer (see \cref{subsec:computer} for details). By the definition of the root groups, we infer that $ w_\delta $ is a $ \delta $-Weyl element. Since $ \Delta_F $ is a system of simple roots, it follows from \cref{rem:weyl basic} that Axiom~\ref{def:rgg}\cref{def:rgg:weyl} is satisfied.
\end{proof}

\premidfigure
\begin{table}[htb]
	\centering\begin{tabular}{ccccc}
		\toprule
		$\alpha$ & $\etae(\alpha, \delta_1)$ & $\etae(\alpha, \delta_2)$ & $\etae(\alpha, \delta_3)$ & $\etae(\alpha, \delta_4)$ \\
		\midrule
		1000 & $ (-1, 1) $ & $ (1, 1) $ & $ (1, 1) $ & $ (1, 1) $\\
		0100 & $ (-1, 1) $ & $ (-1, 1) $ & $ (1, 1) $ & $ (1, 1) $\\
		0010 & $ (1, 1) $ & $ (1, 1) $ & $ (-1, -1) $ & $ (-1, -1) $\\
		0001 & $ (1, 1) $ & $ (1, 1) $ & $ (1, -1) $ & $ (-1, -1) $\\
		1100 & $ (1, 1) $ & $ (-1, 1) $ & $ (1, 1) $ & $ (1, 1) $\\
		\midrule
		0110 & $ (-1, 1) $ & $ (-1, 1) $ & $ (-1, -1) $ & $ (-1, -1) $\\
		0011 & $ (1, 1) $ & $ (1, 1) $ & $ (-1, -1) $ & $ (1, -1) $\\
		1110 & $ (1, 1) $ & $ (1, 1) $ & $ (-1, -1) $ & $ (1, 1) $\\
		0120 & $ (-1, 1) $ & $ (1, 1) $ & $ (1, 1) $ & $ (1, 1) $\\
		0111 & $ (1, -1) $ & $ (-1, 1) $ & $ (1, -1) $ & $ (1, -1) $\\
		\midrule
		1120 & $ (1, 1) $ & $ (1, 1) $ & $ (1, 1) $ & $ (-1, 1) $\\
		1111 & $ (-1, -1) $ & $ (1, 1) $ & $ (1, 1) $ & $ (-1, 1) $\\
		0121 & $ (1, 1) $ & $ (1, 1) $ & $ (-1, -1) $ & $ (-1, -1) $\\
		1220 & $ (1, 1) $ & $ (-1, 1) $ & $ (1, 1) $ & $ (-1, 1) $\\
		1121 & $ (-1, 1) $ & $ (1, 1) $ & $ (-1, 1) $ & $ (-1, -1) $\\
		\midrule
		0122 & $ (1, 1) $ & $ (1, 1) $ & $ (1, 1) $ & $ (1, 1) $\\
		1221 & $ (1, 1) $ & $ (-1, 1) $ & $ (1, -1) $ & $ (-1, -1) $\\
		1122 & $ (-1, 1) $ & $ (1, 1) $ & $ (1, 1) $ & $ (-1, 1) $\\
		1231 & $ (1, 1) $ & $ (1, 1) $ & $ (-1, -1) $ & $ (1, -1) $\\
		1222 & $ (1, 1) $ & $ (-1, 1) $ & $ (1, 1) $ & $ (-1, 1) $\\
		\midrule
		1232 & $ (1, 1) $ & $ (1, 1) $ & $ (-1, -1) $ & $ (-1, -1) $\\
		1242 & $ (1, 1) $ & $ (1, 1) $ & $ (1, 1) $ & $ (1, 1) $\\
		1342 & $ (1, 1) $ & $ (-1, 1) $ & $ (1, 1) $ & $ (1, 1) $\\
		2342 & $ (-1, 1) $ & $ (1, 1) $ & $ (1, 1) $ & $ (1, 1) $\\
		\bottomrule
	\end{tabular}
	\bigskip
	\caption{The values of $\etae \colon F_4 \times \Delta \to T$ in \cref{pr:F4 weyl}. In the first column, a sequence $a_1 a_2 a_3 a_4$ represents $\sum_{i=1}^4 a_i \delta_i$. Further, $\etae(\alpha, \delta) = \etae(-\alpha, \delta)$ for all $\alpha \in F_4$ and $\delta \in \Delta$.}
	\label{ta:existence parmap}
\end{table}
\postmidfigure

\begin{remark}\label{rem:parmap negative}
	The map $ \etae $ in \cref{pr:F4 weyl} has the property that $\etae(\alpha, \delta) = \etae(-\alpha, \delta)$ for all $\alpha \in F_4$ and $\delta \in \Delta$. As in \cref{rem:F4 weyl def}, this property relies on our careful choice of the root homomorphisms $ (\rhomgrp{\alpha})_{\alpha \in F_4} $ in \cref{def:roothom lie}. In particular, the validity of this property for $\alpha \in F_4$ with $\rootproj(\alpha) = (0,0)$ relies on the case distinction in \cref{eq:roothom-lie-00part}.
\end{remark}

\begin{proposition}\label{pr:F4 comm}
	The following commutator relations hold for all $ t,s \in k $ and $ c,d \in C $:
	\begin{align*}
		[\rhomgrp{\delta_1}(t), \rhomgrp{\delta_2}(s)] &= \rhomgrp{\delta_1 + \delta_2}(ts), \\
		[\rhomgrp{\delta_2}(t), \rhomgrp{\delta_3}(c)] &= \rhomgrp{\delta_2 + \delta_3}(-tc) \rhomgrp{\delta_2 + 2\delta_3}\brackets[\big]{t \connorm(c)}, \\
		[\rhomgrp{\delta_2 + \delta_3}(c), \rhomgrp{\delta_3}(d)] &= \rhomgrp{\delta_2 + 2\delta_3}\brackets[\big]{-\contr(c \conconj{d})}, \\
		[\rhomgrp{\delta_2}(t), \rhomgrp{\delta_2 + 2\delta_3}(s)] &= 1_G, \\
		[\rhomgrp{\delta_4}(c), \rhomgrp{\delta_3}(d)] &= \rhomgrp{\delta_3+\delta_4}(\conconj{d} \conconj{c}).
	\end{align*}
	In particular, Axiom~\ref{def:rgg}\cref{def:rgg:comm} is satisfied.
\end{proposition}
\begin{proof}
	The first four commutator formulas follow from the ones established in \cref{pr:G2 comm} for arbitrary cubic norm pairs. For example, for $t,s \in k$:
	\begin{align*}
		[\rhomgrp{\delta_1}(t), \rhomgrp{\delta_2}(s)] &= \bigl[\exp_{1,0}\brackets[\big]{\rhomlie{11\bar{1} \bar{1}}(s)}, \exp_{-2,-1}\brackets[\big]{\rhomlie{\bar{2}000}(t)}\bigr] \\
		&= \bigl[\exp_{1,0}\brackets[\big]{(0, -\gamma_2^{-1} \gamma_3 se_1, 0, 0)_+}, \exp_{-2,-1}(tx)\bigr] \\
		&= \bigl[x_1\brackets[\big]{-\gamma_2^{-1} \gamma_3 se_1}, x_6(t)\bigr] = x_5\brackets[\big]{-t\gamma_2^{-1} \gamma_3 se_1}.
	\end{align*}
	Here we have used that $b \coloneq -\gamma_2^{-1} \gamma_3 se_1$ satisfies $N(b)=0$ and $b^\sharp = 0$. This yields the desired commutator relation because
	\begin{align*}
		\rhomgrp{\delta_1+\delta_2}(ts) &= \exp_{-1,-1}\brackets[\big]{\rhomlie{\bar{1}1\bar{1}\bar{1}}(ts)} = x_5\brackets[\big]{-t\gamma_2^{-1} \gamma_3 se_1}.
	\end{align*}
	Alternatively, the commutator formulas follow from a direct computer calculation (see \cref{subsec:computer} for details), and this is also how establish the last formula.
	
	It follows from the above that the commutator relation $ [U_\alpha, U_\beta] \subseteq U_{\rootint{\alpha}{\beta}} $ is satisfied for all
	\[ (\alpha, \beta) \in X \coloneq \{(\delta_1, \delta_2), (\delta_2, \delta_3), (\delta_2 + \delta_3, \delta_3), (\delta_2, \delta_2 + 2\delta_3), (\delta_3, \delta_4)\}. \]
	By similar arguments as in \cref{rem:rgg F4 comrel}, we infer that $ [U_\alpha, U_\beta] \subseteq U_{\rootint{\alpha}{\beta}} $ holds for all non-proportional $ \alpha, \beta \in F_4 $. Hence Axiom~\ref{def:rgg}\cref{def:rgg:comm} is satisfied.
\end{proof}

\begin{proposition}\label{pr:F4 nondeg}
	Let $ \Pi $ be any positive system of roots in $ F_4 $ and denote by $ U_+ $ and $ U_- $ the groups generated by $ (U_\alpha)_{\alpha \in \Pi} $ and $ (U_\alpha)_{\alpha \in -\Pi} $, respectively. Then $ U_+ \cap U_- = 1 $. In particular, the root groups $ (U_\alpha)_{\alpha \in G_2} $ satisfy Axiom~\ref{def:rgg}\cref{def:rgg:nondeg}
\end{proposition}
\begin{proof}
	This follows from \cref{pr:G2 nondeg}.
\end{proof}

We conclude with our final main result.

\begin{theorem}\label{thm:F4 graded group}
	Let $ k $ be a commutative ring and let $ C $ be a multiplicative conic alternative algebra over $ k $. Let $ \Gamma = \operatorname{diag}(\gamma_1, \gamma_2, \gamma_3) \in \Mat_3(k)$ be invertible, let $ J \coloneq \Her_3(C, \Gamma) $ denote the associated cubic Jordan matrix algebra (see \cref{not:CJMA def}) and let $ L \coloneq L(J,J) $ denote the Lie algebra constructed in \cref{sec:lie const}. Then the root groups $ (U_\alpha)_{\alpha \in F_4} $ from \cref{def:F4 0 rhom,def:F4 nonzero rhom} constitute an $ F_4 $-grading of the group $ G \le \Aut(L) $ that they generate. Further, $ (\rhomgrp{\alpha})_{\alpha \in F_4} $ is a parametrization of $(G, (U_\alpha)_{\alpha \in F_4})$ by $(k,C)$ with respect to $(w_\delta)_{\delta \in \Delta_F}$ and $\etae$ which satisfies the commutator formulas in \cref{pr:F4 comm}.
\end{theorem}
\begin{proof}
	This is a consequence of \cref{pr:F4 comm,pr:F4 weyl,pr:F4 nondeg,pr:F4 gen}.
\end{proof}

\begin{remark}[cf. \ref{rem:different signs}]\label{rem:sign twist}
	The commutator formulas in \cref{pr:F4 comm} are the same as in \cref{thm:rgg-param} only \enquote{up to twisting}, that is, \enquote{up to signs and conjugation}. Similarly, the parity map $ \etae $ in this section is not the same as the parity map $ \etac $ from \cref{rem:coordinatization parmap} that arises in \cref{thm:rgg-param}. In this remark, we illustrate how to modify the root homomorphisms $(\rhomgrp{\alpha})_{\alpha \in F_4}$ to obtain a parametrization of $G$ which looks just like the one in \cref{thm:rgg-param}. 
	Put $ \Delta \coloneq \Delta_F $. Since we have shown in \cref{thm:F4 graded group} that $ (G, (U_\alpha)_{\alpha \in F_4}) $ is an $ F_4 $-graded group, we can apply \cref{thm:rgg-param}. Hence there exist a commutative ring $k'$, a multiplicative conic alternative $k'$-algebra $ (C',n') $ and a parametrization $ (\rhomgrp{\alpha}')_{\alpha \in F_4} $ of $(U_\alpha)_{\alpha \in F_4}$ by $(k',C')$ with respect to $(w_\delta)_{\delta\in \Delta}$ (the Weyl elements from \cref{def:F4 weyl}) and $\etac$. We aim to show that $k=k'$ (as rings) and $(C,n) = (C',n')$ (as conic algebras).

	By \cite[4.5.17]{WiedemannPhD}, we can achieve that $ k=k' $, $ C=C' $ as additive groups and $ \rhomgrp{\delta_i}' = \rhomgrp{\delta_i} $ for $ i \in \{2,3\} $. Our main tool is the following observation: For all $i \in \{2,3\}$, $\delta_1, \ldots, \delta_\ell \in \Delta$, $p \in P_{\delta_i}$ and $\alpha \coloneq \delta_i^{\reflbr{\delta_1} \cdots \reflbr{\delta_\ell}}$, we have (in the notation of \cref{def:parmap,rem:parmap prod}):
	\begin{equation}\label{eq:param compare}
		\rhomgrp{\alpha}(\etae_{\delta_i,\delta_1 \cdots \delta_\ell}.p) = \rhomgrp{\delta_i}(p)^{w_{\delta_1} \cdots w_{\delta_\ell}} = \rhomgrp{\delta_i}'(p)^{w_{\delta_1} \cdots w_{\delta_\ell}} = \rhomgrp{\alpha}'(\etac_{\delta_i,\delta_1 \cdots \delta_\ell}.p).
	\end{equation}
	Note that, while the twistings groups $T$ and $T'$ of $(k,C)$ and $(k',C')$ are the same as abstracts groups and act on the same sets $k=k'$ and $C=C'$, their actions on $C=C'$ are not (yet proven to be) the same because they involve the conjugation of $C$ and of $C'$, respectively. It is the action of $T$ that appears on the left side of \cref{eq:param compare} while the action of $T'$ appears on the right side. In the following, we will write $\conconj{\:\cdot\:}$ for the conjugation of $C$ and never (need to) refer to the conjugation of $C'$.

	We first show that $k=k'$ as rings. We will frequently use some explicit values of $\etac$, such as
	\begin{align*}
		\etae_{\delta_2,\delta_1} = (-1,1) = \etae_{\delta_1+\delta_2,\delta_2} \quad \text{and} \quad \etac_{\delta_2,\delta_1} = (1,1) = \etac_{\delta_1+\delta_2,\delta_2},
	\end{align*}
	which can be read of from \cref{ta:existence parmap} and \cite[Figure~10.4]{WiedemannPhD}, respectively. Using these values and \cref{def:parmap}, we compute that $\etae_{\delta_2,\delta_1 \delta_2} = (1,1) = \etac_{\delta_2, \delta_1 \delta_2}$. Hence by \cref{eq:param compare}, we have
	\[ \rhomgrp{\delta_1+\delta_2}(t) = \rhomgrp{\delta_1+\delta_2}'(-t) \quad \text{and} \quad \rhomgrp{\delta_1}(t) = \rhomgrp{\delta_1}(t) \]
	for all $t \in k$. This implies that
	\begin{align*}
		\rhomgrp{\delta_1+\delta_2}(t \cdot_k s) &= [\rhomgrp{\delta_1}(t), \rhomgrp{\delta_2}(s)] = [\rhomgrp{\delta_1}'(t), \rhomgrp{\delta_2}'(s)] = \rhomgrp{\delta_1+\delta_2}'(-t \cdot_{k'} s) = \rhomgrp{\delta_1+\delta_2}(t \cdot_{k'} s)
	\end{align*}
	for all $t,s \in k$. Hence $k'=k$.
	
	Similarly, to show that $C=C'$ as (nonassociative) rings, we observe that
	\[ \etae_{\delta_3,\delta_4} = (-1,-1) = \etae_{\delta_3+\delta_4,\delta_3} \quad \text{and} \quad \etac_{\delta_3,\delta_4} = (-1,1) = \etac_{\delta_3+\delta_4,\delta_3}. \]
	This implies that $\etae_{\delta_3,\delta_4 \delta_3} = (1,1) = \etac_{\delta_3,\delta_4 \delta_3}$ and, by \cref{eq:param compare},
	\[ \rhomgrp{\delta_3+\delta_4}(\conconj{c}) = \rhomgrp{\delta_3+\delta_4}'(c) \quad \text{and} \quad \rhomgrp{\delta_4}(c) = \rhomgrp{\delta_4}'(c) \]
	for all $c \in C$. Thus
	\begin{align*}
		\rhomgrp{\delta_3+\delta_4}'(c \cdot_C d) &= [\rhomgrp{\delta_4}'(c), \rhomgrp{\delta_3}'(d)] = [\rhomgrp{\delta_4}(c), \rhomgrp{\delta_3}(d)] = \rhomgrp{\delta_3+\delta_4}(\conconj{d} \cdot_C \conconj{c}) = \rhomgrp{\delta_3+\delta_4}'(c \cdot_C d)
	\end{align*}
	for all $c,d \in C$, so $C=C'$ as (nonassociative) rings.
	
	Finally, we have $\etae_{\delta_2,\delta_3} = (1,1)$, $\etac_{\delta_2,\delta_3} = (-1,1)$, $\etac_{\delta_3,\delta_2}=(1,1)=\etae_{\delta_3,\delta_2}$ and thus
	\[ \rhomgrp{\delta_2 + 2\delta_3}(t) = \rhomgrp{\delta_2+2\delta_3}'(-t) \quad \text{and} \quad \rhomgrp{\delta_2+\delta_3}(c) = \rhomgrp{\delta_2+\delta_3}'(c) \]
	for all $t \in k$, $c \in C$, which implies that
	\begin{align*}
		\rhomgrp{\delta_2+\delta_3}(-t \cdot_C c) \rhomgrp{\delta_2 + 2\delta_3}\brackets[\big]{t n(c)} &= [\rhomgrp{\delta_2}(t), \rhomgrp{\delta_3}(c)] = [\rhomgrp{\delta_2}'(t), \rhomgrp{\delta_3}'(c)] \\
		&= \rhomgrp{\delta_2+\delta_3}'(- t \cdot_{C'} c) \rhomgrp{\delta_2+2\delta_3}'\brackets[\big]{-t n'(c)} \\
		&= \rhomgrp{\delta_2+\delta_3}(- t \cdot_{C'} c) \rhomgrp{\delta_2+2\delta_3}\brackets[\big]{t n'(c)}
	\end{align*}
	for all $t \in k$, $c \in C$. Hence $n=n'$ and the $k$-algebra structures on $C$ and on $C'$ coincide. Altogether, we conclude the following: Starting from an arbitrary multiplicative conic alternative algebra $C$ over a commutative ring $k$, we can construct an $F_4$-graded group with a parametrization $(\rhomgrp{\alpha}')_{\alpha \in F_4}$ that behaves exactly like the one in \cref{thm:rgg-param}.
\end{remark}

\appendix

\section{Higher Leibniz rules}\label{sec:leibniz}

In this appendix, we develop the machinery needed to verify that certain ``exponential endomorphisms'' preserve the Lie bracket, over arbitrary rings. We will apply this machinery to show that the maps $e_{(0,b,0,0)_+}$ for $b \in J$ and $e_{\rho x}$ for $\rho \in k$ from \cref{def:exp L1,def:exp xy} are endomorphisms of the Lie algebra $L$.
\begin{lemma}\label{le:leibniz endo}
	Let $L$ be a Lie algebra over an arbitrary commutative ring $k$.
	Assume that $\delta_0,\delta_1,\dots,\delta_s$ are $k$-linear endomorphisms of $L$ such that
	\begin{equation}\label{eq:star l}
		\delta_\ell([v,w]) = \sum_{i+j = \ell} [\delta_i(v), \delta_j(w)]
		\tag*{$(*)_\ell$}
	\end{equation}
	for all $v,w \in L$ and all $\ell \in \mathbb{N}$, where $\delta_\ell = 0$ for all $\ell > s$.
	
	Then the map $\varphi \coloneq \sum_{\ell=0}^s \delta_\ell$ is a Lie algebra endomorphism, i.e., $\varphi([v,w]) = [\varphi(v), \varphi(w)]$ for all $v,w \in L$.
\end{lemma}
\begin{proof}
	We have $\varphi([v,w]) = \sum_\ell \delta_\ell([v,w])$, so by invoking \ref{eq:star l} for each $\ell$, we get
	\[ \varphi([v,w]) = \sum_\ell \sum_{i+j = \ell} [\delta_i(v), \delta_j(w)] = \Bigl[ \sum_i \delta_i(v), \sum_j \delta_j(w) \Bigr] = [\varphi(v), \varphi(w)] . \qedhere \] 
\end{proof}

In the case that the maps $ \delta_0, \ldots, \delta_s $ are related to a grading of $ L $ (an assumption that will always be satisfied in our situation), we have the following converse of \cref{le:leibniz endo}.

\begin{lemma}\label{le:leibniz endo converse}
	Let $ (A,+) $ be an abelian group and let $ L = \bigoplus_{a \in A} L_a $ be an $ A $-graded Lie algebra over a commutative ring $ k $. Assume that $ \delta_0, \delta_1, \ldots, \delta_s $ are $ k $-linear endomorphisms of $ L $ such that the map $ \varphi \coloneq \sum_{\ell = 0}^s \delta_\ell $ is a Lie algebra endomorphism and put $ \delta_\ell \coloneq 0 $ for all $ \ell > s $. Assume further that there exists $ a \in A $ such that $ \delta_\ell(L_b) \subseteq L_{b+\ell a} $ for all $ b \in A $ and $ \ell \in \mathbb{N} $. Then \ref{eq:star l} holds for all $ v,w \in L $ and all $ \ell \in \mathbb{N} $.
\end{lemma}
\begin{proof}
	It suffices to prove \ref{eq:star l} for homogeneous elements $ u,v $, so let $ b,c \in A $ and let $ v \in L_b $, $ w \in L_c $. Since $ \varphi $ is an endomorphism of $ L $, a slight rearrangement of the computation in \cref{le:leibniz endo} yields that
	\begin{align*}
		\sum_{\ell=0}^s \delta_\ell([v,w]) = \varphi([v,w]) = [\varphi(v), \varphi(w)] = \sum_\ell \sum_{i+j=\ell} [\delta_i(v), \delta_j(w)].
	\end{align*}
	Note that $ \delta_\ell([v,w]) $ and $ \sum_{i+j=\ell} [\delta_i(v), \delta_j(w)] $ both lie in $ L_{b+c+\ell a} $ for all $ \ell \in \mathbb{N} $. Since $ L = \bigoplus_{d \in A} L_d $, it follows that $ \delta_\ell([v,w]) = \sum_{i+j=\ell} [\delta_i(v), \delta_j(w)] $ for all $ \ell \in \mathbb{N} $.
\end{proof}

The guiding example we have in mind, is the case where $\delta$ is a derivation of $L$ with $\delta^{s+1}=0$ for some $s \in \mathbb{N}_0$, and where $\delta_\ell \coloneq \frac{1}{\ell!} \delta^\ell$ for each $\ell$, provided each $\ell!$ is invertible in $k$. Notice that
\[ \varphi = \sum_{\ell=0}^s \delta_\ell = \sum_{\ell=0}^s \frac{1}{\ell!} \delta^\ell \]
is then the usual exponential map corresponding to $\delta$.
In this case, $(*)_0$ is trivially satisfied (since $\delta_0 = \id$) and $(*)_1$ is simply the Leibniz rule that defines derivations.
By applying the Leibniz rule $\ell$ times, we get that
\[ \ell! \cdot \delta_\ell([v,w]) = \ell! \cdot \sum_{i+j = \ell} [\delta_i(v), \delta_j(w)] \]
for all $\ell$, so if each $\ell!$ is invertible in $k$, then \ref{eq:star l} follows for each $\ell$.

The key point now is that the maps we care about, are exponential maps of derivations defined over arbitrary rings.
We will still have that $\delta_0 = \id$ and $\delta_1 = \delta$ is a derivation, but we cannot deduce the ``higher Leibniz rules'' \ref{eq:star l} from the fact that $\delta$ is a derivation. Proving each of these \ref{eq:star l} is a computational and rather elaborate process, that we will describe below, providing enough details so that the courageous reader can reconstruct the full proof without too much effort.

\begin{remark}
	We have been wondering whether there is some general principle that would allow us to pass from $\Z$ to an arbitrary ring $k$ (e.g. by tensoring), so that we would be able to omit the lengthy proofs that follow. However, the fact that we need a lot of additional identities and lemmas in the process indicates that this seems unlikely.
	
	A related example from the literature is that in the appendix of \cite{Faulkner1989}, Faulkner constructs an $A_2$-graded Lie algebra and an $A_2$-graded group from an alternative ring. He faces the same problems as we do in \cite[(A.7)]{Faulkner1989} (except that his grading is a 3-grading, so only the invertibility of 2 is a problem). He proceeds precisely by proving the identities \ref{eq:star l} by hand, so apparently he did not see a way around this.
\end{remark}	

\subsection{The exponential maps \texorpdfstring{$e_{(0, b, 0, 0)_+}$}{e\_{(0,b,0,0)\_+}}}

We now proceed to show that the maps $e_{(0, b, 0, 0)_+}$ from \cref{def:exp L1} are indeed endomorphisms of the Lie algebra $L$.
So fix $b \in J$, let $\varphi \coloneq e_{(0, b, 0, 0)_+}$ and decompose $\varphi = \sum_{\ell = 0}^3 \delta_\ell$ according to the degree.
In particular, $\delta_0 \coloneq \id$ and $\delta_1 \coloneq \ad_{(0, b, 0, 0)_+}$. Explicitly, the maps $\delta_1$, $\delta_2$ and $\delta_3$ are as follows.
\begin{align}
	\delta_1 &\colon
	\begin{cases}
		x &\mapsto (0, b, 0, 0)_- \\
		(\nu, c, c', \rho)_- &\mapsto -b \times c + T(b, c') (\zeta - \xi) - \dd_{b, c'} + \rho b \\
		c' &\mapsto \bigl( T(b,c'), 0, 0, 0 \bigr)_+ \\
		\dd_{e, e'} &\mapsto \bigl( 0, -D_{e, e'}(b) + T(e, e')b, 0, 0 \bigr)_+ \\
		\xi &\mapsto (0, -b, 0, 0)_+ \\
		\zeta &\mapsto 0 \\
		c &\mapsto (0, 0, -b \times c, 0)_+ \\
		(\nu, c, c', \rho)_+ &\mapsto T(b, c') y \\
		y &\mapsto 0
	\end{cases} \\
	\delta_2 &\colon
	\begin{cases}
		x &\mapsto -b^\sharp \\
		(\nu, c, c', \rho)_- &\mapsto \bigl( -T(c, b^\sharp), \ T(b,c')b - b^\sharp \times c', \ -\rho b^\sharp, \ 0 \bigr)_+ \\
		c' &\mapsto 0 \\
		\dd_{e, e'} &\mapsto 0 \\
		\xi &\mapsto 0 \\
		\zeta &\mapsto 0 \\
		c &\mapsto - T(c, b^\sharp) y \\
		(\nu, c, c', \rho)_+ &\mapsto 0 \\
		y &\mapsto 0
	\end{cases} \\
	\delta_3 &\colon
	\begin{cases}
		x &\mapsto \bigl( -N(b), 0, 0, 0 \bigr)_+ \\
		(\nu, c, c', \rho)_- &\mapsto - \rho N(b) y \\
		c' &\mapsto 0 \\
		\dd_{e, e'} &\mapsto 0 \\
		\xi &\mapsto 0 \\
		\zeta &\mapsto 0 \\
		c &\mapsto 0 \\
		(\nu, c, c', \rho)_+ &\mapsto 0 \\
		y &\mapsto 0
	\end{cases}
\end{align}

\begin{remark}\label{rem:delta grading}
	\begin{enumerate}
		\item 
			To make the expressions more readable, we will usually omit parentheses and write $\delta_i v$ in place of $\delta_i(v)$.
		\item
			Observe that $\delta_\ell(L_{i,j}) \subseteq L_{i+\ell,j}$ for all $i,j,\ell$.
			In particular, \ref{eq:star l} is trivially valid if $v \in L_{ij}$ and $w \in L_{rs}$ when $L_{i+r+\ell, j+s} = 0$.
			This will greatly reduce the number of cases we have to verify.
	\end{enumerate}
\end{remark}

The following little trick will severely reduce the required computations.
\begin{lemma}\label{le:Leibniz trick}
	Let $n \geq 2$ and let $a, b, c \in L$.
	Assume that \ref{eq:star l} holds for all $\ell \leq n$ for each of the following five pairs:
	\begin{enumerate}
		\item $v = a$ and $w = b$;
		\item $v = a$ and $w = c$;
		\item $v = b$ and $w = c$;
		\item $v = b$ and $w = [a, c]$;
		\item $v = c$ and $w = [a, b]$.
	\end{enumerate}
	Then \ref{eq:star l} also holds for $v = a$ and $w = [b,c]$, for all $\ell \leq n$.
\end{lemma}
\begin{proof}
	By the Jacobi identity, both sides of \ref{eq:star l} for $v = a$ and $w = [b, c]$ can be simplified to the same expression
	\[ \sum_{i+j+m = \ell} \bigl[ [ \delta_i a, \delta_j b], \delta_m c \bigr] - \sum_{i+j+m = \ell} \bigl[ [ \delta_i a, \delta_m c], \delta_j b \bigr] . \qedhere \]
\end{proof}

Here is one specific instance that we will use several times.
\begin{corollary}\label{le:dd trick}
	Let $r \in \{ -2, -1, 0, 1, 2 \}$, let $n \geq 2$, let $e \in J$ and $e' \in J'$.
	Assume that \ref{eq:star l} holds for all $\ell \leq n$ in each of the following situations:
	\begin{enumerate}
		\item $v \in L_r$ and $w = e$;
		\item $v \in L_r$ and $w = e'$.
	\end{enumerate}
	Then \ref{eq:star l} also holds for all $v \in L_r$ and $w = \dd_{e, e'}$, for all $\ell \leq n$.
\end{corollary}
\begin{proof}
	By definition, we have $\dd_{e, e'} = [e', e]$.
	Observe that for each $v \in L_r$, the elements $[v, e']$ and $[v, e]$ also belong to $L_r$.
	Moreover, \ref{eq:star l} also holds for $v=e$ and $w=e'$: it holds trivially for $\ell = 0$ and $\ell = 1$, and once $\ell \geq 2$, it holds by \cref{rem:delta grading} because $L_{\ell,0} = 0$.
	The result now follows from \cref{le:Leibniz trick}.
\end{proof}

We are now ready to prove \ref{eq:star l} for each $\ell$, starting with $\ell = 2$.
(Recall that $(*)_1$ is just the Jacobi identity since $\delta_0 = \id$ and $\delta_1$ is an inner derivation.)
For each $\ell$, this will require a large number of case distinctions depending on $v$ and $w$, where we will go over each of the different cases as in \cref{ta:Lie}.

\subsection*{The case \texorpdfstring{$ \ell=2 $}{l=2}.}

We have to show that
\begin{equation}\label{eq:star 2}
	\delta_2 [v,w] = [\delta_2 v, w] + [\delta_1 v, \delta_1 w] + [v, \delta_2 w]
		\tag*{$(*)_2$}
\end{equation}
for all $v,w \in L$.
We first list the different cases we have to consider.
\begin{enumerate}[(a)]
	\item\label{2a} $v = x$, $w = (\nu, c, c', \rho)_-$\,,
	\item\label{2b} $v = x$, $w = \dd_{e,e'}$\,,
	\item\label{2c} $v = x$, $w = \xi$\,,
	\item\label{2d} $v = x$, $w = \zeta$\,,
	\item\label{2e} $v = x$, $w = c \in J$\,,
	\item\label{2f} $v = x$, $w = (\nu, c, c', \rho)_+$\,,
	\item\label{2g} $v = (\lambda, a, a', \mu)_-$, $w = (\nu, c, c', \rho)_-$\,,
	\item\label{2h} $v = (\lambda, a, a', \mu)_-$, $w = c' \in J'$\,,
	\item\label{2i} $v = (\lambda, a, a', \mu)_-$, $w = \dd_{e,e'}$\,,
	\item\label{2j} $v = (\lambda, a, a', \mu)_-$, $w = \xi$\,,
	\item\label{2k} $v = (\lambda, a, a', \mu)_-$, $w = \zeta$\,,
	\item\label{2l} $v = (\lambda, a, a', \mu)_-$, $w = c \in J$\,,
	\item\label{2m} $v = (\lambda, a, a', \mu)_-$, $w = (\nu, c, c', \rho)_+$\,,
	\item\label{2n} $v = \dd_{e,e'}$, $w = c \in J$\,,
	\item\label{2o} $v = \xi$, $w = c \in J$\,,
	\item\label{2p} $v = \zeta$, $w = c \in J$\,.
\end{enumerate}
Observe that several cases have been omitted because of \cref{rem:delta grading}.
For example, when $v=x$ and $w=c' \in J'$, we have $v \in L_{-2, -1}$ and $w \in L_{0, -1}$, but $L_{-2 + 0 + 2, -1 -1} = L_{0, -2} = 0$.
As another example, note that when $v \in L_{0,i}$ and $w \in L_{0,j}$, then we end up in $L_{2,i+j}$, which is $0$ unless $i+j=1$; this implies that only cases \cref{2n,2o,2p} arise (up to symmetry) for $v,w \in L_0$.

We will now deal with each of these cases separately. In each case, we will simply give the result of the computation of each of the four terms involved in \ref{eq:star 2} and mention if the computation or the verification of \ref{eq:star 2} involves additional identities.
\begin{enumerate}[(a)]
	\item $v = x$, $w = (\nu, c, c', \rho)_-$\,. Then
		\begin{align*}
			\delta_2 [v,w] &= 0 \\
			[\delta_2 v, w] &= \bigl( T(c, b^\sharp),\ b^\sharp \times c',\ \rho b^\sharp,\ 0 \bigr)_- \\
			[\delta_1 v, \delta_1 w] &= \bigl( -2 T(c, b^\sharp),\ T(b, c')b - 2 b^\sharp \times c',\ -2 \rho b^\sharp,\ 0 \bigr)_- \\
			[v, \delta_2 w] &= \bigl( T(c, b^\sharp),\ -T(b, c')b + b^\sharp \times c',\ \rho b^\sharp,\ 0 \bigr)_-.
		\end{align*}
	\item $v = x$, $w = \dd_{e,e'}$\,. By \cref{le:dd trick}, we can skip this case, as it will follow from~\cref{2e}. (Notice that the case $v=x$, $w=e' \in J'$ holds by \cref{rem:delta grading}.)
	\item $v = x$, $w = \xi$\,. Then
		\begin{align*}
			\delta_2 [v,w] &= -2 b^\sharp \\
			[\delta_2 v, w] &= 0 \\
			[\delta_1 v, \delta_1 w] &= -2 b^\sharp \\
			[v, \delta_2 w] &= 0.
		\end{align*}
	\item $v = x$, $w = \zeta$\,. Then
		\begin{align*}
			\delta_2 [v,w] &= -b^\sharp \\
			[\delta_2 v, w] &= -b^\sharp \\
			[\delta_1 v, \delta_1 w] &= 0 \\
			[v, \delta_2 w] &= 0.
		\end{align*}
	\item $v = x$, $w = c \in J$\,. Then
		\begin{align*}
			\delta_2 [v,w] &= 0 \\
			[\delta_2 v, w] &= -\dd_{c, b^\sharp} \\
			[\delta_1 v, \delta_1 w] &= 2 T(c, b^\sharp) \zeta - \dd_{b, b \times c} \\
			[v, \delta_2 w] &= - T(c, b^\sharp) \xi.
		\end{align*}
		For the computation of $[\delta_1 v, \delta_1 w]$, we used the fact that $T(b, b \times c) = T(c, b \times b) = 2 T(c, b^\sharp)$.
		Moreover, to verify \ref{eq:star 2}, we require \cref{le:triple}\cref{le:triple:norm r}.
	\item $v = x$, $w = (\nu, c, c', \rho)_+$\,. Then
		\begin{align*}
			\delta_2 [v,w] &= \bigl( T(c, b^\sharp),\ b^\sharp \times c' - T(b, c')b,\ \rho b^\sharp,\ 0 \bigr)_+ \\
			[\delta_2 v, w] &= \bigl( T(c, b^\sharp),\ b^\sharp \times c',\ \rho b^\sharp,\ 0 \bigr)_+ \\
			[\delta_1 v, \delta_1 w] &= \bigl( 0,\ - T(b, c')b,\ 0,\ 0 \bigr)_+ \\
			[v, \delta_2 w] &= 0.
		\end{align*}
	\item $v = (\lambda, a, a', \mu)_-$, $w = (\nu, c, c', \rho)_-$\,.
		A direct computation is possible but quite involved, and it is much more efficient to invoke \cref{le:Leibniz trick}.
		So write $w = [z, x]$ with $z = (\nu, c, c', \rho)_+$.
		Then \ref{eq:star 2} for the given pair $(v, w)$ will follow once we know it for each of the five pairs $(v, z)$, $(v, x)$, $(z, x)$, $(x, [v, z])$ and $(z, [v, x])$.
		For $(v,z)$, this is case \cref{2m} below. For $(v,x)$ and $(z,x)$, this has been dealt with in the respective cases \cref{2a,2f} above.
		Since $[v,z] \in L_0$, the case $(x, [v,z])$ follows from \cref{2b,2c,2d,2e} above. The case $(z, [v,x])$, finally, is trivial because $[v,x] = 0$.
	\item $v = (\lambda, a, a', \mu)_-$, $w = c' \in J'$\,. Then
		\begin{align*}
			\delta_2 [v,w] &= \bigl( -T(c' \times a', b^\sharp),\ \mu T(b, c') b - \mu b^\sharp \times c',\ 0,\ 0 \bigr)_+ \\
			[\delta_2 v, w] &= \bigl( T(b, a') T(b, c') - T(b^\sharp \times a', c'),\ -\mu c' \times b^\sharp,\ 0,\ 0 \bigr)_+ \\
			[\delta_1 v, \delta_1 w] &= \bigl( - T(b, a') T(b, c'),\ \mu T(b, c') b,\ 0,\ 0 \bigr)_+ \\
			[v, \delta_2 w] &= 0.
		\end{align*}
	\item $v = (\lambda, a, a', \mu)_-$, $w = \dd_{e,e'}$\,.
		By \cref{le:dd trick}, we can skip this case, as it will follow from~\cref{2h,2l}.
	\item $v = (\lambda, a, a', \mu)_-$, $w = \xi$\,. Then
		\begin{align*}
			\delta_2 [v,w] &= \bigl( -T(a, b^\sharp),\ T(b, a')b - b^\sharp \times a',\ -\mu b^\sharp,\ 0 \bigr)_+ \\
			[\delta_2 v, w] &= \bigl( T(a, b^\sharp),\ -T(b, a')b + b^\sharp \times a',\ \mu b^\sharp,\ 0 \bigr)_+ \\
			[\delta_1 v, \delta_1 w] &= \bigl( -2 T(a, b^\sharp),\ D_{b, a'}(b),\ -2 \mu b^\sharp,\ 0 \bigr)_+ \\
			[v, \delta_2 w] &= 0.
		\end{align*}
		By definition, $D_{b,a'}(b) = 2 T(b, a') b - 2 b^\sharp \times a'$.
	\item $v = (\lambda, a, a', \mu)_-$, $w = \zeta$\,. Then
		\begin{align*}
			\delta_2 [v,w] &= \bigl( -T(a, b^\sharp),\ 0,\ \mu b^\sharp,\ 0 \bigr)_+ \\
			[\delta_2 v, w] &= \bigl( -T(a, b^\sharp),\ 0,\ \mu b^\sharp,\ 0 \bigr)_+ \\
			[\delta_1 v, \delta_1 w] &= 0 \\
			[v, \delta_2 w] &= 0.
		\end{align*}
	\item $v = (\lambda, a, a', \mu)_-$, $w = c \in J$\,. We can again invoke \cref{le:Leibniz trick} to avoid computations.
		Indeed, again write $v = [z, x]$ with $z = (\lambda, a, a', \mu)_+$.
		Then \ref{eq:star 2} for the given pair $(v, w)$ will follow once we know it for each of the five pairs $(x, z)$, $(c, z)$, $(x, c)$, $(z, [x, c])$ and $(x, [z, c])$.
		For $(x, c)$, this is case \cref{2e} above. For $(x, z)$ and $(x, [z,c])$, this is case \cref{2f} above.
		For $(c, z)$, this is trivial because of \cref{rem:delta grading}. The case $(z, [x,c])$, finally, is trivial because $[x,c] = 0$.
	\item $v = (\lambda, a, a', \mu)_-$, $w = (\nu, c, c', \rho)_+$\,. Then
		\begin{align*}
			\delta_2 [v,w] &= T(\rho a - a' \times c' + \mu c,\ b^\sharp) y \\
			[\delta_2 v, w] &= \rho T(a, b^\sharp) y + T(b, a') T(b, c') y - T(b^\sharp \times a', c') y + \mu T(c, b^\sharp) y \\
			[\delta_1 v, \delta_1 w] &= - T(b, a') T(b, c') y \\
			[v, \delta_2 w] &= 0.
		\end{align*}
	\item $v = \dd_{e,e'}$, $w = c \in J$\,. Then
		\begin{align*}
			\delta_2 [v,w] &= -T\bigl( D_{e,e'}(c), b^\sharp \bigr) y \\
			[\delta_2 v, w] &= 0 \\
			[\delta_1 v, \delta_1 w] &= T\bigl( D_{e,e'}(b) - T(e, e')b, b \times c \bigr) y \\
			[v, \delta_2 w] &= 0.
		\end{align*}
		To prove equality of the two non-zero right hand sides, we expand
		\begin{gather*}
			-T\bigl( D_{e,e'}(c), b^\sharp \bigr) = - T(e, e') T(c, b^\sharp) - T(c, e') T(e, b^\sharp) + T\bigl( (c \times e) \times e', b^\sharp \bigr), \\
			T\bigl( D_{e,e'}(b) - T(e, e')b, b \times c \bigr) = T(b, e') T(e, b \times c) - T \bigl( (b \times e) \times e', b \times c \bigr) .
		\end{gather*}
		Now invoke \cref{le:CNP identities}\cref{le:CNP identities:9} to rewrite
		\begin{align*}
			&T \bigl( (b \times e) \times e', b \times c \bigr)
				= T \bigl( (b \times e) \times (b \times c), e' \bigr) \\
			&= T(c, b \times e) T(b, e') + T(c, b^\sharp) T(e, e') + T(e, b^\sharp) T(c, e') - T\bigl( (c \times e) \times b^\sharp, e' \bigr)
		\end{align*}
		and the result follows.
		
		(Notice that we could not apply \cref{le:Leibniz trick} or \cref{le:dd trick} since we would run into a circular argument.)
	\item $v = \xi$, $w = c \in J$\,. Then
		\begin{align*}
			\delta_2 [v,w] &= 0 \\
			[\delta_2 v, w] &= 0 \\
			[\delta_1 v, \delta_1 w] &= 2 T(c, b^\sharp) y \\
			[v, \delta_2 w] &= -2 T(c, b^\sharp) y.
		\end{align*}		
	\item $v = \zeta$, $w = c \in J$\,. Then
		\begin{align*}
			\delta_2 [v,w] &= -T(c, b^\sharp) y \\
			[\delta_2 v, w] &= 0 \\
			[\delta_1 v, \delta_1 w] &= 0 \\
			[v, \delta_2 w] &= -T(c, b^\sharp) y.
		\end{align*}	
\end{enumerate}

\subsection*{The case \texorpdfstring{$ \ell=3 $}{l=3}.}

We have to show that
\begin{equation}\label{eq:star 3}
	\delta_3 [v,w] = [\delta_3 v, w] + [\delta_2 v, \delta_1 w] + [\delta_1 v, \delta_2 w] + [v, \delta_3 w]
		\tag*{$(*)_3$}
\end{equation}
for all $v,w \in L$.
Again, we first list the different cases we have to consider.
\begin{enumerate}[(a)]
	\item\label{3a} $v = x$, $w = (\nu, c, c', \rho)_-$\,,
	\item\label{3b} $v = x$, $w = \dd_{e,e'}$\,,
	\item\label{3c} $v = x$, $w = \xi$\,,
	\item\label{3d} $v = x$, $w = \zeta$\,,
	\item\label{3e} $v = x$, $w = c \in J$\,,
	\item\label{3f} $v = x$, $w = (0, 0, 0, \rho)_+$\,,
	\item\label{3g} $v = (\lambda, a, a', \mu)_-$, $w = (\nu, c, c', \rho)_-$\,,
	\item\label{3h} $v = (\lambda, a, a', \mu)_-$, $w = \dd_{e,e'}$\,,
	\item\label{3i} $v = (\lambda, a, a', \mu)_-$, $w = \xi$\,,
	\item\label{3j} $v = (\lambda, a, a', \mu)_-$, $w = \zeta$\,,
	\item\label{3k} $v = (\lambda, a, a', \mu)_-$, $w = c \in J$\,.
\end{enumerate}
Of course, we have again applied \cref{rem:delta grading}.
In particular, notice that for case~\cref{3f}, where $v=x \in L_{-2,-1}$ and $w \in L_1$, \ref{eq:star 3} is trivial when $w \in L_{1j}$ with $j \neq 1$, which is why we can restrict to the case $w = (0,0,0,\rho)_+$.

\begin{enumerate}[(a)]
	\item $v = x$, $w = (\nu, c, c', \rho)_-$\,. Then
		\begin{align*}
			\delta_3 [v,w] &= 0 \\
			[\delta_3 v, w] &= -N(b)c' - \rho N(b) (\xi - \zeta) \\
			[\delta_2 v, \delta_1 w] &= -T(b,c') b^\sharp + D_{c',b}(b^\sharp) - \rho \dd_{b, b^\sharp} \\
			[\delta_1 v, \delta_2 w] &= T(b, c') b^\sharp - N(b) c' + 3 \rho N(b) \zeta - \rho \dd_{b, b^\sharp} \\
			[v, \delta_3 w] &= - \rho N(b) \xi ,
		\end{align*}	
		where we have used \cref{def:CNP new}\cref{eq:CNP cubic,eq:CNP 19} in the computation of $[\delta_1 v, \delta_2 w]$.
		Now these two identities also imply that $D_{c',b}(b^\sharp) = 2N(b) c'$.
		Together with the fact that $\dd_{b, b^\sharp} = N(b)(2\zeta - \xi)$ by \cref{le:triple}\cref{le:triple:aasharp}, \ref{eq:star 3} follows.
	\item $v = x$, $w = \dd_{e,e'}$\,. By \cref{le:dd trick}, we can skip this case, as it will follow from~\cref{3e}. (Notice that the case $v=x$, $w=e' \in J'$ holds by \cref{rem:delta grading}.)
	\item $v = x$, $w = \xi$\,. Then
		\begin{align*}
			\delta_3 [v,w] &= \bigl( -2N(b),\ 0,\ 0,\ 0 \bigr)_+ \\
			[\delta_3 v, w] &= \bigl( N(b),\ 0,\ 0,\ 0 \bigr)_+ \\
			[\delta_2 v, \delta_1 w] &= \bigl( -3N(b),\ 0,\ 0,\ 0 \bigr)_+ \\
			[\delta_1 v, \delta_2 w] &= 0 \\
			[v, \delta_3 w] &= 0 .
		\end{align*}	
	\item $v = x$, $w = \zeta$\,. Then
		\begin{align*}
			\delta_3 [v,w] &= \bigl( -N(b),\ 0,\ 0,\ 0 \bigr)_+ \\
			[\delta_3 v, w] &= \bigl( -N(b),\ 0,\ 0,\ 0 \bigr)_+ \\
			[\delta_2 v, \delta_1 w] &= 0 \\
			[\delta_1 v, \delta_2 w] &= 0 \\
			[v, \delta_3 w] &= 0 .
		\end{align*}	
	\item $v = x$, $w = c \in J$\,. Then
		\begin{align*}
			\delta_3 [v,w] &= 0 \\
			[\delta_3 v, w] &= \bigl( 0,\ N(b)c,\ 0,\ 0 \bigr)_+ \\
			[\delta_2 v, \delta_1 w] &= \bigl( 0,\ - b^\sharp \times (b \times c),\ 0,\ 0 \bigr)_+ \\
			[\delta_1 v, \delta_2 w] &= \bigl( 0,\ T(c, b^\sharp)b,\ 0,\ 0 \bigr)_+ \\
			[v, \delta_3 w] &= 0 .
		\end{align*}	
		The result now follows from \cref{le:CNP identities}\cref{le:CNP identities:8}.
	\item $v = x$, $w = (0, 0, 0, \rho)_+$\,. Then
		\begin{align*}
			\delta_3 [v,w] &= \rho N(b) y \\
			[\delta_3 v, w] &= \rho N(b) y \\
			[\delta_2 v, \delta_1 w] &= 0 \\
			[\delta_1 v, \delta_2 w] &= 0 \\
			[v, \delta_3 w] &= 0 .
		\end{align*}	
	\item $v = (\lambda, a, a', \mu)_-$, $w = (\nu, c, c', \rho)_-$\,. Again, we will invoke \cref{le:Leibniz trick}, writing $w = [z, x]$ with $z = (\nu, c, c', \rho)_+$.
		Then \ref{eq:star 3} for the given pair $(v, w)$ will follow once we know it for each of the five pairs $(v, z)$, $(v, x)$, $(z, x)$, $(x, [v, z])$ and $(z, [v, x])$.
		For $(v,z)$, this is trivial by \cref{rem:delta grading}. For $(v,x)$ and $(z,x)$, this has been dealt with in the respective cases \cref{3a,3f} above.
		Since $[v,z] \in L_0$, the case $(x, [v,z])$ follows from \cref{3b,3c,3d,3e} above. The case $(z, [v,x])$, finally, is trivial because $[v,x] = 0$.
	\item $v = (\lambda, a, a', \mu)_-$, $w = \dd_{e,e'}$\,.
		By \cref{le:dd trick}, we can skip this case, as it will follow from~\cref{3k}.
		(Notice that the case $v = (\lambda, a, a', \mu)_-$, $w=e' \in J'$ holds by \cref{rem:delta grading}.)
	\item $v = (\lambda, a, a', \mu)_-$, $w = \xi$\,. Then
		\begin{align*}
			\delta_3 [v,w] &= -\mu N(b) y \\
			[\delta_3 v, w] &= 2 \mu N(b) y \\
			[\delta_2 v, \delta_1 w] &= -3 \mu N(b) y \\
			[\delta_1 v, \delta_2 w] &= 0 \\
			[v, \delta_3 w] &= 0 .
		\end{align*}	
	\item $v = (\lambda, a, a', \mu)_-$, $w = \zeta$\,. Then
		\begin{align*}
			\delta_3 [v,w] &= \mu N(b) y \\
			[\delta_3 v, w] &= \mu N(b) y \\
			[\delta_2 v, \delta_1 w] &= 0 \\
			[\delta_1 v, \delta_2 w] &= 0 \\
			[v, \delta_3 w] &= 0 .
		\end{align*}	
	\item $v = (\lambda, a, a', \mu)_-$, $w = c \in J$\,. Then
		\begin{align*}
			\delta_3 [v,w] &= T(c, a') N(b) y \\
			[\delta_3 v, w] &= 0 \\
			[\delta_2 v, \delta_1 w] &= T(c, a') N(b) y - T(b, a') T(c, b^\sharp) y \\
			[\delta_1 v, \delta_2 w] &= T(b, a') T(c, b^\sharp) y \\
			[v, \delta_3 w] &= 0 ,
		\end{align*}	
		where we have used \cref{le:CNP identities}\cref{le:CNP identities:8} in the computation of $[\delta_2 v, \delta_1 w]$.
\end{enumerate}

\subsection*{The case \texorpdfstring{$ \ell=4 $}{l=4}.}

Notice that $\delta_4 = 0$, so we have to show that
\begin{equation}\label{eq:star 4}
	0 = [\delta_3 v, \delta_1 w] + [\delta_2 v, \delta_2 w] + [\delta_1 v, \delta_3 w]
		\tag*{$(*)_4$}
\end{equation}
for all $v,w \in L$.
By \cref{rem:delta grading}, there are not many cases to consider.
\begin{enumerate}[(a)]
	\item\label{4a} $v = x$, $w = (\nu, c, c', \rho)_-$\,,
	\item\label{4b} $v = (\lambda, a, a', \mu)_-$, $w = (\nu, c, c', \rho)_-$\,.
\end{enumerate}
Notice that when $v = x \in L_{-2,-1}$ and $w \in L_{0j}$, \cref{rem:delta grading} always applies because $L_{2,-1+j}=0$ unless $j=2$, but then $L_{0j} = L_{02} = 0$.
\begin{enumerate}[(a)]
	\item $v = x$, $w = (\nu, c, c', \rho)_-$\,. Then
		\begin{align*}
			[\delta_3 v, \delta_1 w] &= \bigl( -N(b) T(b, c'),\ \rho N(b) b,\ 0,\ 0 \bigr)_+ \\
			[\delta_2 v, \delta_2 w] &= \bigl( T(b, c') T(b, b^\sharp) - T(b^\sharp \times c, b^\sharp),\ -\rho b^\sharp \times b^\sharp,\ 0,\ 0 \bigr)_+ \\
			[\delta_1 v, \delta_3 w] &= \bigl( 0,\ \rho N(b) b,\ 0,\ 0 \bigr)_+ .
		\end{align*}	
		The result now follows since $b^\sharp \times b^\sharp = 2 b^{\sharp\sharp} = 2 N(b) b$ and $T(b, b^\sharp) = 3 N(b)$.
	\item $v = (\lambda, a, a', \mu)_-$, $w = (\nu, c, c', \rho)_-$\,. We could, of course, directly compute that \ref{eq:star 4} holds, but this is not necessary by the same argument as before, invoking \cref{le:Leibniz trick}.
\end{enumerate}

\subsection*{The case \texorpdfstring{$ \ell \ge 5 $}{l≥5}.}

Since $\delta_j = 0$ for all $j \geq 4$, the identities \ref{eq:star l} become trivial as soon as $\ell \geq 7$. For $\ell = 5$ and $\ell = 6$, they simplify to
\begin{align*}
	0 &= [\delta_3 v, \delta_2 w] + [\delta_2 v, \delta_3 w] \quad \text{and}
		\tag*{$(*)_5$} \\
	0 &= [\delta_3 v, \delta_3 w] ,
		\tag*{$(*)_6$}
\end{align*}
respectively. By \cref{rem:delta grading}, however, they are trivially satisfied if $v \in L_i$ and $w \in L_j$ with $i + j > -3$.
So the only remaining case is $v = x$ and $w \in L_{-1}$, but since $L_{-1,2} = 0$, this is again dealt with by the same \cref{rem:delta grading}.

\medskip

This finishes the proof that the maps $e_{(0,b,0,0)_+}$ are automorphisms of the Lie algebra~$L$.

\subsection{The exponential maps \texorpdfstring{$e_{\rho x}$}{e\_{ρx}}}

We now show that the maps $e_{\rho x}$ from \cref{def:exp xy} are automorphisms of the Lie algebra $L$. This works in the same way as for the maps $e_{(0,b,0,0)_+}$, but the number of cases we have to verify will be much lower. We fix $\rho \in k$, put $\varphi \coloneq e_{\rho x} $ and decompose $ \varphi = \sum_{\ell = 0}^2 \delta_\ell$. Here $\delta_0 = \id$, $\delta_1 = \ad_{\rho x}$ and $\delta_2$ is defined by $\delta_2(y) = \rho^2 x$ and $\delta_2 (L_{ij}) = 0$ for all $(i,j) \ne (2,1)$. Further, $\delta_1$ is trivial on all $L_{ij}$ except for the following cases:
\begin{align*}
	\delta_1(\xi) = 2\rho x, \quad \delta_1(\zeta) = \rho x, \quad \delta_1\brackets[\big]{(\lambda,b, b', \mu)_+} = -\rho(\lambda, b, b', \mu)_, \quad \delta_1(y) = \rho \xi.
\end{align*}

\begin{remark}\label{rem:delta grading x}
	Similarly to \cref{rem:delta grading}, we have $\delta_\ell(L_{i,j}) = L_{i-2\ell, j-\ell}$ for all $i,j,\ell$. Hence \ref{eq:star l} is trivially satisfied for all $v \in L_{ij}$, $w \in L_{rs}$ with $L_{i+r-2\ell, j+s-\ell} = 0$.
\end{remark}

In order to verify
\begin{equation*}
	\delta_2 [v,w] = [\delta_2 v, w] + [\delta_1 v, \delta_1 w] + [v, \delta_2 w],
			\tag*{$(*)_2$}
\end{equation*}
only the following cases have to be considered by \cref{rem:delta grading x}.
\begin{enumerate}[(a)]
	\item $v \in L_0$, $w = y$.
	\item $v = (0,0,0,\mu)_+$, $w=y$.
	\item $v = (0,0,a',0)_+$, $w = (0,0,b',0)_+$.
	\item\label{xd} $v = (\lambda, b, b', \mu)_+$, $w = y$.
	\item \label{xe}$v=y$, $w=y$.
\end{enumerate}
In case~\cref{xd}, we compute that
\begin{align*}
	[\delta_1 v, \delta_1 w] &= [-\rho (\lambda, b, b', \mu)_-, \rho \xi] = -\rho^2 (\lambda, b, b', \mu)_-, \\
	[v, \delta_2 w] &=[(\lambda, b, b', \mu)_+, \rho^2 x] = \rho^2 (\lambda, b, b', \mu)_-
\end{align*}
and $\delta_2[v,w]=[\delta_2 v,w] = [v, \delta_2 w] = 0$. In case~\cref{xe}, we have $\delta_2[v,w] = 0 = [\delta_1 v, \delta_1 w]$ and $[\delta_2 v, w] = -[v, \delta_2 w]$ because $v=w$. In all remaining cases,
\[ \delta_2 [v,w] = [\delta_2 v, w] = [\delta_1 v, \delta_1 w] = [v, \delta_2 w] = 0. \]
Thus $(*)_2$ is satisfied.

We now turn to $(*)_3$. Using \cref{rem:delta grading x}, we can omit all cases except for $v=y=w$, and this case is trivial by the same arguments as for $(*)_2$. The relations $(*)_\ell$ for $\ell \ge 4$ hold by \cref{rem:delta grading x} as well. We conclude that $e_{\rho x}$ is indeed an automorphism of $L$.

\section{Weyl elements in the \texorpdfstring{$ G_2 $}{G2}-grading}\label{sec:weyl G2}

We can summarize the results of \cref{sec:leibniz} as follows.

\begin{lemma}\label{le:exp x auto}
	We have $e_{\lambda x} \in \End(L)$ for all $\lambda \in k$.
\end{lemma}

\begin{lemma}\label{le:exp L10 auto}
	We have $e_{(0, b, 0, 0)_+} \in \End(L)$ for all $b \in J$.
\end{lemma}

The goal of this appendix is to show that all the remaining maps defined in \cref{def:exp xy,def:exp L1,def:exp L-1,def:exp JJ'} can be expressed as the conjugate of $ e_{\lambda x} $ or $ e_{(0, b, 0, 0)_+} $ by a suitable product of the reflections defined in \cref{sec:reflections}. In particular, we infer that all these maps are compatible with the Lie bracket, and hence lie in $ \End(L) $. Once this is done, only a few additional computations are required to show that the reflections are actually Weyl elements in the sense of \cref{def:rgg}\cref{def:rgg:weyl}.

\begin{lemma}\label{le:exp y conj}
	For all $\lambda \in k$, we have $\phihor^{-1} \circ e_{\lambda x} \circ \phihor = e_{\lambda y}$. In particular, $e_{\lambda y} \in \End(L)$.
\end{lemma}
\begin{proof}
	Let $\lambda \in k$. Using the formulas in \cref{def:reflections,rem:reflections inv,def:exp xy}, we can compute that $\phihor^{-1} \circ e_{\lambda x} \circ \phihor$ acts as the identity on $ L_1+L_2 $ and on $\dd_{c,c'}$ for all $c \in J$, $c' \in J'$, and that it acts on the remaining generators of the $ k $-modules $ L $ as follows:
	\begin{align*}
		x &\mapsto x-\lambda \xi + \lambda^2 y, & (\nu, c, c', \rho)_- &\mapsto  (\nu, c, c', \rho)_- + \lambda (\nu, c, c', \rho)_+, \\
		\xi &\mapsto \xi - 2\lambda y, & \zeta &\mapsto \zeta - \lambda y.
	\end{align*}
	Comparing the results of our computation to the definition of $e_{\lambda y}$, we conclude that $\phihor^{-1} \circ e_{\lambda x} \circ \phihor = e_{\lambda y}$. By \cref{le:exp x auto,le:reflections auto}, it follows that $e_{\lambda y} \in \End(L)$.
\end{proof}

\begin{lemma}\label{le:exp y conj inv}
	For all $\lambda \in k$, we have $\phihor^{-1} \circ e_{\lambda y} \circ \phihor = e_{\lambda x}$.
\end{lemma}
\begin{proof}
	Let $\lambda \in k$. It follows from \cref{le:exp y conj} that $\phihor^{-1} \circ e_{\lambda y} \circ \phihor = \phihor^{-2} \circ e_{\lambda x} \circ \phihor^2$. Since $\phihor^2$ is the parity automorphism of the grading $(L_i)_{-2 \le i \le 2}$ by \cref{rem:reflection grading}, the assertion follows from \cref{le:exp parity conj}.
\end{proof}

\begin{lemma}\label{le:exp JJ' phihor}
	For all $a \in J$ and $a' \in J'$, we have $\phihor^{-1} \circ e_a \circ \phihor = e_a$ and $\phihor^{-1} \circ e_{a'} \circ \phihor = e_{a'}$.
\end{lemma}
\begin{proof}
	It suffices to check that $e_a \circ \phihor = \phihor \circ e_a$ and $e_{a'} \circ \phihor = \phihor \circ e_{a'}$, which is straightforward. The main argument is the observation that the formulas for the action of $e_a$ on $L_1$ and $L_{-1}$ are identical, and similarly for $e_{a'}$.
\end{proof}

\begin{lemma}\label{le:phihor factor}
	We have $\phihor = e_y \circ e_x \circ e_y = e_x \circ e_y \circ e_x$.
\end{lemma}
\begin{proof}
	As in \cref{le:exp y conj}, we can compute that
	\begin{align*}
		e_y \circ e_x \circ e_y(x) &= y, \qquad e_y \circ e_x \circ e_y\brackets[\big]{(\nu, c, c', \rho)_+} = -(\nu, c, c', \rho)_-
	\end{align*}
	for all $\nu, \rho \in k$, $c \in J$ and $c' \in J'$. Since all the maps $ \phihor $, $ e_x $ and $ e_y $ lie in $ \End(L) $, it follows from \cref{rem:aut equality} that $\phihor = e_y \circ e_x \circ e_y$. This implies that
	\[ \phihor = \phihor^{\phihor} = e_y^{\phihor} \circ e_x^{\phihor} \circ e_y^{\phihor} = e_x \circ e_y \circ e_x \]
	by \cref{le:exp y conj,le:exp y conj inv}.
\end{proof}

\begin{lemma}\label{le:exp L-1-1 conj}
	For all $ \lambda \in k $, we have $ \phisl^{-1} \circ e_{\lambda x} \circ \phisl = e_{(-\lambda, 0, 0, 0)_-} $. In particular, $e_{(\lambda, 0, 0, 0)_-} \in \End(L)$.
\end{lemma}
\begin{proof}
	Let $\lambda \in k$. Similarly as in the proof of \cref{le:exp y conj}, we can compute that $\phisl^{-1} \circ e_{\lambda x} \circ \phisl$ acts on the generators of the $ k $-module $ L $ as follows:
	\begin{align*}
		y &\mapsto y+(\lambda,0,0,0)_-, \qquad x \mapsto x, & c'  &\mapsto c', \\
		(\nu, c, c', \rho)_- &\mapsto (\nu, c, c', \rho)_- + \lambda \rho x, & \xi &\mapsto \xi + (-\lambda, 0, 0, 0)_-, \\
		\dd_{c,c'} &\mapsto \dd_{c,c'} + \brackets[\big]{-T(c,c')\lambda, 0, 0, 0}_-, & \zeta &\mapsto \zeta + (-2\lambda, 0, 0, 0)_-, \\
		(\nu, c, c', \rho)_+ &\mapsto (\nu, c, c', \rho)_+ + \lambda c' - \lambda \rho \zeta + (\lambda^2 \rho, 0, 0, 0)_-, & c & \mapsto c + (0,\lambda c, 0, 0)_-
	\end{align*}
	for all $\nu, \rho \in k$, $c \in J$, $c' \in J'$. We conclude that $ \phisl^{-1} \circ e_{\lambda x} \circ \phisl = e_{(-\lambda, 0, 0, 0)_-} $. By \cref{le:exp x auto,le:reflections auto}, it follows that $e_{(\lambda, 0, 0, 0)_-} \in \End(L)$.
\end{proof}

\begin{lemma}\label{le:exp L-1-2 conj}
	For all $ \lambda \in k $, we have $ \phisl^{-1} \circ e_{\lambda y} \circ \phisl = e_{(0, 0, 0, -\lambda)_+} $. In particular, $e_{(0, 0, 0, \lambda)_+} \in \End(L)$.
\end{lemma}
\begin{proof}
	Let $\lambda \in k$. Then $\phisl^{-1} \circ e_{\lambda y} \circ \phisl = e_{(0, 0, 0, -\lambda)_+}$ acts on $ L $ as follows:
	\begin{align*}
		x & \mapsto x + (0,0,0,-\lambda)_-, & (\nu, c, c', \rho)_- & \mapsto (\nu, c, c', \rho)_- + \lambda \nu \zeta + (0,0,0,\lambda^2 \nu)_+ - \lambda c, \\
		c' &\mapsto c' + (0,0,-\lambda c',0)_+, & (\nu, c, c', \rho)_+ &\mapsto (\nu, c, c', \rho)_+ - \lambda \nu y, \\
		\xi &\mapsto \xi + (0,0,0,\lambda)_+, & \dd_{c,c'} &\mapsto \dd_{c,c'} + \brackets[\big]{0,0,0, T(c,c') \lambda}_+, \\
		\zeta &\mapsto \zeta + (0,0,0,2\lambda)_+, & c &\mapsto c, \qquad y \mapsto y 
	\end{align*}
	for all $\nu, \rho \in k$, $c \in J$, $c' \in J'$. We conclude that $ \phisl^{-1} \circ e_{\lambda y} \circ \phisl = e_{(0, 0, 0, -\lambda)_+} $. By \cref{le:exp y conj,le:reflections auto}, it follows that $e_{(0, 0, 0, \lambda)_+} \in \End(L)$.
\end{proof}

\begin{lemma}\label{le:exp L-1 conj}
	\begin{enumerate}
		\item Let $\lambda, \mu \in k$, $b \in J$, $b' \in J'$ such that at most one of $\lambda, b, b', \mu$ is non-zero. Then $\phihor^{-1} \circ e_{(\lambda, b, b', \mu)_+} \circ \phihor = e_{(\lambda, b, b', \mu)_-}$.
		
		\item We have $e_{(\lambda,0,0,0)_\varepsilon}, e_{(0,b,0,0)_\varepsilon}, e_{(0,0,0,\lambda)_\varepsilon} \in \End(L)$ for all $\lambda \in k$, $b \in J$ and $\varepsilon \in \{\mathord{\pm}\}$.
	\end{enumerate}
\end{lemma}
\begin{proof}
	Let $\lambda, \mu \in k$, $b \in J$, $b' \in J'$ such that at most one of $\lambda, b, b', \mu$ is non-zero. Then a similar computation as before shows that $\phihor^{-1} \circ e_{(\lambda, b, b', \mu)_+} \circ \phihor$ agrees with $ e_{(\lambda, b, b', \mu)_-} $ on the usual set of generators of the $ k $-module $ L $.
	The final assertion follows from \cref{le:exp x auto,le:exp L10 auto,le:exp L-1-1 conj,le:exp L-1-2 conj}.
\end{proof}

\begin{lemma}\label{le:exp L-1 conj inv}
	Let $\lambda, \mu \in k$, $b \in J$, $b' \in J'$ such that at most one of $\lambda, b, b', \mu$ is non-zero. Then $\phihor^{-1} \circ e_{(\lambda, b, b', \mu)_-} \circ \phihor = e_{-(\lambda, b, b', \mu)_+}$.
\end{lemma}
\begin{proof}
	We have $\phihor^{-1} \circ e_{(\lambda, b, b', \mu)_-} \circ \phihor = \phihor^{-2} \circ e_{(\lambda, b, b', \mu)_+} \circ \phihor^2$ by \cref{le:exp L-1 conj}. Thus the assertion follows from \cref{rem:reflection parity,le:exp parity conj}.
\end{proof}

The following \cref{le:exp J conj,le:exp J' conj} follow from the same kind of computations as before.

\begin{lemma}\label{le:exp J conj}
	For all $a \in J$, we have $\phisl^{-1} \circ e_{(0,a,0,0)_+} \circ \phisl = e_{-a}$. In particular, $e_a \in \End(L)$.
\end{lemma}

\begin{lemma}\label{le:exp J' conj}
	For all $ a' \in J' $, we have $\phisl^{-1} \circ e_{(0,0,a',0)_-} \circ \phisl = e_{-a'}$.
\end{lemma}

\begin{lemma}\label{le:exp J' psi conj}
	Using the notation from \cref{rem:exp CNP notation}, we have
	\[ (\psibs^{(J',J)})^{-1} \circ e_{(0,a,0,0)_+}^{(J,J')} \circ \psibs^{(J',J)} = e_{-a}^{(J', J)} \]
	for all $a \in J$. In particular, $e_{a'}^{(J,J')} \in \End(L(J,J'))$ for all $a' \in J'$.
\end{lemma}
\begin{proof}
	The first assertion follows from the same kind of computation as before. By \cref{le:exp L10 auto,le:reflections auto}, it follows that for all cubic norm pairs $(J,J')$ and all $a \in J$, the map $e_a^{(J',J)}$ is an endomorphism of $L(J',J)$. Applying this observation to the cubic norm pair $(J',J)$, the second assertion follows.
\end{proof}

The results in this section up to this point can be summarized as follows.

\begin{proposition}\label{pr:exp conj aut summary}
	\begin{enumerate}
		\item Let $ \lambda \in k $ and
		\[ e \in \{e_{\lambda x}, e_{(\lambda, 0, 0, 0)_\pm}, e_{(0,0,0,\lambda)_\pm}, e_{\lambda y}\}. \]
		Then there exist $ n \in \{0,1,2,3\} $, $ \epsilon \in \{\pm 1\} $ and $ \varphi_1, \ldots, \varphi_n \in \{\phihor, \phisl\} $ such that $ e = \varphi^{-1} \circ e_{\epsilon \lambda x} \circ \varphi $ where $ \varphi \coloneq \varphi_1 \cdots \varphi_n $ (which should be interpreted as $ \id_L $ if $ n=0 $). In particular, $ e \in \End(L) $.
		
		\item Let $ a \in J $ and $ e \in \{e_a, e_{(0,a,0,0)_\pm}\} $. Then there exist $ \varphi \in \{\phihor, \phisl, \id_L\} $ and $ \epsilon \in \{\pm 1\} $ such that $ e = \varphi^{-1} \circ e_{(0,\epsilon a,0,0)_+} \circ \varphi $. In particular, $ e \in \End(L) $.
		
		\item Let $ a' \in J' $ and
		\[ e \in \{e_{a'}^{(J,J')}, e_{(0,0,a',0)_\pm}^{(J,J')}\}. \]
		Then there exist
		\[ \varphi^{(J',J)} \in \{\id_L, \phisl^{(J,J')} \circ \phihor^{(J,J')}\} \]
		and $ \epsilon \in \{\pm 1\} $ such that
		\[ (\varphi^{(J,J')})^{-1} \circ e \circ \varphi^{(J,J')} = (\psibs^{(J,J')})^{-1} \circ e_{(0,\epsilon a',0,0)}^{(J',J)} \circ \psibs^{(J,J')}. \]
		In particular, $ e \in \End(L) $.
	\end{enumerate}
\end{proposition}
\begin{proof}
	The first assertion follows from \cref{le:exp x auto,le:exp y conj,le:exp L-1-1 conj,le:exp L-1-2 conj,le:exp L-1 conj,le:exp L-1 conj inv}. The second assertion follows from \cref{le:exp L10 auto,le:exp L-1 conj,le:exp J conj}. Finally, note that
	\[ (\psibs^{(J,J')})^{-1} \circ e_{(0,\epsilon a',0,0)}^{(J',J)} \circ \psibs^{(J,J')} = e_{-\epsilon a'}^{(J,J')} \]
	by \cref{le:exp J' psi conj} and that the left-hand map lies in $ \End(L) $ by the second assertion. Thus the third assertion follows from \cref{le:exp J' conj,le:exp L-1 conj}.
\end{proof}

With \cref{pr:exp conj aut summary} established, it now follows from \cref{pr:param} that $ \exp_\alpha \colon L_\alpha \to \Aut(L) \colon a \mapsto e_a $ is a well-defined $ \alpha $-parametrization in $ L $, and we can define root groups $ (U_\alpha)_{\alpha \in G_2} $ in $ \Aut(L) $ as in \cref{def:G2 rootgr}. In the remaining part of this appendix, we will show that these root groups have Weyl elements in the sense of \cref{def:rgg}\cref{def:rgg:weyl}. In fact, one half of the work is already done.

\begin{proposition}\label{le:phihor weyl}
	$\phihor$ is a $(2,1)$-Weyl element with respect to $ (U_\alpha)_{\alpha \in G_2} $.
\end{proposition}
\begin{proof}
	This follows from \cref{le:phihor factor,le:exp y conj,le:exp y conj inv,le:exp L-1 conj,le:exp L-1 conj inv}.
\end{proof}

By \cref{rem:weyl basic}, it remains to find a $ \delta $-Weyl element for a root $ \delta $ such that $ \{(2, 1), \delta\} $ is a system of simple roots in $ G_2 $. We opt to find a $ (-1,-1) $-Weyl element.

\begin{convention}\label{conv:J invertible}
	From now on, we assume that there exists $b \in J$ for which $N(b)$ is invertible. We fix such a $b$ and put $b' \coloneq N(b)^{-1} b^\sharp$. By \cref{rem:isotopes}, the pair $(N(b) U_{b'}, N(b)^{-1} U_b)$ is then a $t$-involution of $(J,J')$ for $t \coloneq N(b)$, which we denote by $(\iota, \iotainv)$.
\end{convention}

%

We aim to show that the reflection $ \phibs \coloneq \phibs^{(\iota, \iotainv,t)} = \psibs^{(J',J)} \circ L(\iota, \iotainv,t) $ from \cref{def:phibs} is a $ (-1,-1) $-Weyl element. We begin by showing that it has the desired product decomposition.

\begin{lemma}\label{le:phibs prod}
	We have $\phibs^{-1} = e_{(0,0,-b',0)_+} \circ e_{(0,b,0,0)_-} \circ e_{(0,0,-b',0)_+} $.
\end{lemma}
\begin{proof}
	We compute the action of $g \coloneq e_{(0,0,-b',0)_+} \circ e_{(0,b,0,0)_-} \circ e_{(0,0,-b',0)_+}$ on the generating set
	\[ \phihor \circ \phisl(kx+ L_1) = kx + (0,0,J',0)_- + J + (k,0,0,k)_+ \]
	of $L$. Since this computation relies on several identities, we perform it in full detail, with each arrow \enquote{$ \mapsto $} denoting the application of the next consecutive factor of $ g $. We will frequently use that $T(b,b^\sharp) = 3N(b) = 3t$ by Axiom~\ref{def:CNP new}\cref{eq:CNP cubic}, that $(b')^\sharp = t^{-1}b$ and $N(b') = -t^{-1}$ and that $b \times b = 2b^\sharp$ because $\times$ is the linearization of $\sharp$. We will also apply \cref{le:CNP identities}\cref{le:CNP identities:7,le:CNP identities:8} and Axioms~\ref{def:CNP new}\cref{eq:CNP 18,eq:CNP 19} in many cases.
	
	Each factor of $g$ fixes $(k,0,0,0)_+$ pointwise. The action of $ g $ on $x$ is given by
	\begin{align*}
		x &\mapsto x + (0,0,-t^{-1}b^\sharp, 0)_- - t^{-1}b + (0,0,0,-t^{-1})_+ \\
		&\mapsto x + (0,0,-t^{-1}b^\sharp, 0)_- - t^{-1} T(b,b^\sharp)x - t^{-1}b + (0,0, t^{-1} b \times b, 0)_- \\
		&\qquad{}+ t^{-1} T(b,b^\sharp)x + (0,0,0,-t^{-1})_+ + t^{-1}b + (0,0,-t^{-1}b^\sharp, 0)_- - t^{-1}N(b)x \\
		&= (0,0,0,-t^{-1})_+ \mapsto (0,0,0,-t^{-1})_+.
	\end{align*}
	The action of $ g $ on $(0,0,c',0)_-$ for an arbitrary $c' \in J'$ is given by
	\begin{align*}
		(0,0,c',0)_- &\mapsto (0,0,c',0)_- + t^{-1} c' \times b^\sharp + \brackets[\big]{0,0,0,t^{-1} T(b,c')}_+ \\
		&\mapsto (0,0,c',0)_- + T(b,c')x + t^{-1} c' \times b^\sharp + \brackets[\big]{0,0,-t^{-1} (c' \times b^\sharp) \times b, 0}_- \\
		&\qquad{}- t^{-1} T(c' \times b^\sharp, b^\sharp)x + \brackets[\big]{0,0,0,t^{-1}T(b,c')}_+ - t^{-1}T(b,c')b \\
		&\qquad{}+ \brackets[\big]{0,0, t^{-1}T(b,c') b^\sharp, 0}_- + T(b,c')x \\
		&=t^{-1} c' \times b^\sharp -t^{-1}T(b,c')b + \brackets[\big]{0,0,0,t^{-1}T(b,c')}_+ \\
		&\mapsto t^{-1} c' \times b^\sharp + \brackets[\big]{0,0,0, t^{-2} T(c' \times b^\sharp, b^\sharp)}_+ -t^{-1} T(b,c')b \\
		&\qquad{}+ \brackets[\big]{0,0,0, -t^{-2} T(b,c') T(b,b^\sharp)}_+ + \brackets[\big]{0,0,0,t^{-1} T(b,c')}_+ \\
		&=-t^{-1} \brackets[\big]{T(b,c')b - c' \times b^\sharp} = -t^{-1} U_b c'
	\end{align*}
	Here we have used twice that $T(c' \times b^\sharp, b^\sharp) = T(c', b^\sharp \times b^\sharp) = 2 t T(c',b)$, and also that
	\[ -t^{-1} (c' \times b^\sharp) \times b = -c' - t^{-1}T(b,c')b^\sharp \]
	by Axiom~\ref{def:CNP new}\cref{eq:CNP 19}. For arbitrary $c \in J$, the action of $ g $ on $c$ is given by
	\begin{align*}
		c &\mapsto c + \brackets[\big]{0,0,0, t^{-1} T(c,b^\sharp)}_+ \\
		&\mapsto c + (0,0,-c \times b, 0)_- - T(c,b^\sharp)x + \brackets[\big]{0,0,0,t^{-1} T(c,b^\sharp)}_+ - t^{-1} T(c,b^\sharp)b \\
		&\qquad{}+ \brackets[\big]{0,0, t^{-1} T(c,b^\sharp)b^\sharp, 0}_- + T(c,b^\sharp)x \\
		&= \brackets[\big]{0,0,t^{-1} U_{b^\sharp}c, 0}_- + c - t^{-1} T(c,b^\sharp) b + \brackets[\big]{0,0,0,t^{-1} T(c,b^\sharp)}_+ \\
		&\mapsto \brackets[\big]{0,0,t^{-1} U_{b^\sharp}c, 0}_- + t^{-2} U_{b^\sharp}(c) \times b^\sharp + \brackets[\big]{0,0,0, t^{-2} T(b, U_{b^\sharp} c)}_+ + c \\
		&\qquad{}+ \brackets[\big]{0,0,0, t^{-1} T(c,b^\sharp)}_+ - t^{-1} T(c,b^\sharp)b + \brackets[\big]{0,0,0,-t^{-2} T(c,b^\sharp) T(b,b^\sharp)}_+ \\
		&\qquad{} +\brackets[\big]{0,0,0,t^{-1} T(c,b^\sharp)}_+ \\
		&= (0,0, t^{-1} U_{b^\sharp}c, 0)_-.
	\end{align*}
	Here we have used that $U_{b^\sharp}(c) \times b^\sharp = T(b^\sharp, c) b^{\sharp \sharp} - N(b^\sharp)c = t T(b^\sharp,c)b - t^2c$ by \eqref{eq:new 6} in \cref{le:CNP new} and $T(b,U_{b^\sharp}c) = t T(b^\sharp, c)$ by \eqref{eq:new 5} in \cref{le:CNP new} (applied to $b^\sharp$, $c$ in place of $b$, $c'$ and using that $t=N(b)$ is invertible). Finally, the action of $ g $ on $(0,0,0,1)_+$ is given by
	\begin{align*}
		(0,0,0,1)_+ &\mapsto (0,0,0,1)_+ \mapsto (0,0,0,1)_+ - b + (0,0,b^\sharp, 0)_- + tx \\
		&\mapsto (0,0,0,1)_+ - b + \brackets[\big]{0,0,0,-t^{-1}T(b,b^\sharp)}_+ + (0,0,b^\sharp, 0)_- + t^{-1} b^\sharp \times b^\sharp \\*
		&\qquad{}+ \brackets[\big]{0,0,0, t^{-1} T(b,b^\sharp)}_+ + tx + (0,0,-b^\sharp, 0)_- - b + (0,0,0,-1)_+ \\
		&=tx.
	\end{align*}
	By \cref{rem:phibs formulas}, we conclude that $\phibs^{-1} = e_{(0,0,-b',0)_+} \circ e_{(0,b,0,0)_-} \circ e_{(0,0,-b',0)_+} $.
\end{proof}

It remains to show that conjugation by $\phibs$ induces the desired permutation of the root groups. We first verify this on the short root groups.

\begin{lemma}\label{le:phibs short 1}
	We have $ \phibs^{-1} \circ e_{(0,a,0,0)_+} \circ \phibs = e_{-a^\iota} $ for all $ a \in J $ and $\phibs^{-1} \circ e_{a'} \circ \phibs = e_{(0,(a')^\iotainv, 0, 0)_+}$ for all $a' \in J'$.
\end{lemma}
\begin{proof}
	It follows from \cref{le:exp J' psi conj} that
	\[ \phibs^{-1} \circ e_{(0,a,0,0)_+} \circ \phibs = L(\iota, \iotainv, t)^{-1} \circ e_{-a}^{(J',J)} \circ L(\iota, \iotainv, t). \]
	The map on the right-hand side fixes $ x $ and maps $ (\nu, c, c', \rho)_+ $ to
	\begin{align*}
		\bigl(\nu + T(c^\iota, a) + t T(a^\sharp, (c')^\iotainv) + t\rho N(a),c + a^\iota \times c' + t\rho (a^\sharp)^\iotainv, \, c' + \rho a^\iota, \, \rho \bigr)
	\end{align*}
	Here, since $(\iota, \iotainv)$ is a $t$-involution,
	\begin{gather*}
		T(c^\iota, a) = T(c, a^\iota), \qquad tT\brackets[\big]{a^\sharp, (c')^\iotainv} = T\brackets[\big]{(a^\iota)^\sharp, c'}, \\
		t\rho N(a) = \rho N(a^\iota), \qquad t\rho (a^\sharp)^\iotainv = \rho (a^\iota)^\sharp.
	\end{gather*}
	The first assertion follows. This implies that
	\[ \phibs^{-1} \circ e_{a'} \circ \phibs = \phibs^{-2} \circ e_{(0,-(a')^\iotainv,0,0)_+} \circ \phibs^2 \]
	for all $a' \in J'$, so the second assertion follows from \cref{rem:phibs formulas,le:exp parity conj}.
\end{proof}

\begin{lemma}\label{le:phibs short 2}
	We have $\phibs^{-1} \circ e_{(0,0,a',0)_+} \circ \phibs = e_{(0,-t(a')^\iotainv, 0, 0)_-}$ for all $a' \in J$ and $\phibs^{-1} \circ e_{(0,a, 0, 0)_-} \circ \phibs = e_{(0,0,-t^{-1} a^\iota,0)_+}$ for all $a \in J$.
\end{lemma}
\begin{proof}
	Let $ a' \in J $. We compute the action of $\phibs^{-1} \circ e_{(0, 0. a', 0)_-} \circ \phibs$ on the generating set
	\[ \phiver(kx + L_1) = (k,0,0,k)_- + J' + (0,J,0,0)_+ + ky \]
	of $L$. The module $(0,0,0,k)_-$ is fixed pointwise and the action on the remaining generators is given as follows, where the final equality for each generator holds because $ (\iota, \iotainv) $ is a $ t $-homotopy:
	\begin{align*}
		(1,0,0,0)_- &\mapsto (1,0,0,0)_-, \\
		c' &\mapsto c' + \brackets[\big]{-T((c')^\iotainv, a')t, 0,0,0}_- = c' + \brackets[\big]{T((-ta')^\iotainv, c'), 0, 0, 0}_-, \\
		(0,c,0,0)_+ &\mapsto (0,c,0,0)_+ - (c^\iota \times a')^\iota + \brackets[\big]{tT((a')^\sharp, c^\iota), 0, 0, 0}_- \\
		&\qquad \text{where} \quad tT\brackets[\big]{(a')^\sharp, c^\iota} = T\brackets[\big]{((-ta')^\iotainv)^\sharp, c}, \\
		y &\mapsto y + \brackets[\big]{0, t(a')^\iotainv, 0, 0}_+ - t \brackets[\big]{(a')^\sharp}^\iota + \brackets[\big]{t^2 N(a'), 0, 0, 0}_- \\
		&\qquad \text{where} \quad - t \brackets[\big]{(a')^\sharp}^\iota = -\brackets[\big]{(-ta')^\iotainv}^\sharp, \quad t^2 N(a') = -N\brackets[\big]{(-ta')^\iotainv}.
	\end{align*}
	where $ c \in J $ and $ c' \in J' $.
	The first assertion follows. This implies that
	\begin{align*}
		\phibs^{-1} \circ e_{(0,a, 0, 0)_-} \circ \phibs = \phibs^{-2} \circ e_{(0,0,-t^{-1}a^\iota,0)_+} \circ \phibs^2
	\end{align*}
	for all $a \in J$, so the second assertion follows from \cref{rem:phibs formulas,le:exp parity conj}.
\end{proof}

\begin{lemma}\label{le:phibs short 3}
	We have $ \phibs^{-1} \circ e_a \circ \phibs = e_{(0,0,a^\iota, 0)_-} $ for all $ a \in J $ and $ \phibs^{-1} \circ e_{(0,0,a',0)_-} \circ \phibs = e_{-(a')^\iotainv} $ for all $ a' \in J' $.
\end{lemma}
\begin{proof}
	Let $ a \in J $. We compute the action of $\phibs^{-1} \circ e_a \circ \phibs$ on the generating set
	\[ \phibs(kx + L_1) = (k,0, 0, k)_- + kx + (0,J,0,0)_- + J' \]
	of $ L $. The module $ (0,0,0,k)_+ $ is fixed pointwise.
	The action on the remaining part is given by
	\begin{align*}
		&\rho x + (0,c,0,0)_- + c' + (\nu, 0, 0, 0)_+ \\
		&\qquad{} \mapsto (\nu, 0, 0, 0)_+ + \brackets[\big]{c' - \nu a^\iota} + \brackets[\big]{0, c + t (a \times (c')^\iotainv)^\iota - t \nu (a^\sharp)^\iotainv, 0, 0}_- \\
		&\qquad\qquad{} + \brackets[\big]{\rho - T(a,c^\iota) - t T\brackets[\big]{(c')^\iotainv, a^\sharp} + t\nu N(a)}
	\end{align*}
	for all $ \rho, \nu \in k $, $ c \in J $ and $ c' \in J' $. Here, since $ (\iota, \iotainv) $ is a $ t $-homotopy,
	\begin{gather*}
		t \brackets[\big]{a \times (c')^\iotainv}^\iota = a^\iota \times c', \qquad -t\nu (a^\sharp)^\iotainv = -\nu (a^\iota)^\sharp, \\
		-T(a,c^\iota) = -T(a^\iota, c), \quad -tT\brackets[\big]{(c')^\iotainv, a^\sharp} = -T\brackets[\big]{(a^\iota)^\sharp, c'}, \quad t\nu N(a) = \nu N(a^\iota).
	\end{gather*}
	The first assertion follows. Hence
	\begin{align*}
		\phibs^{-1} \circ e_{(0,0,a',0)} \circ \phibs = \phibs^{-2} \circ e_{(a')^\iotainv} \circ \phibs^2 = e_{-(a')^\iotainv}
	\end{align*}
	for all $ a' \in J' $ by \cref{rem:phibs formulas,le:exp parity conj}.
\end{proof}

On the long root groups, the required computations are much shorter.

\begin{lemma}\label{le:phibs long 1}
	For all $\nu, \rho \in k$, we have
	\begin{align*}
		\phibs^{-1} \circ e_{(0,0,0,\rho)_-} \circ \phibs = e_{(0,0,0,\rho)_-} \quad \text{and} \quad \phibs^{-1} \circ e_{(\nu,0,0,0)_+} \circ \phibs = e_{(\nu,0,0,0)_+}.
	\end{align*}
\end{lemma}
\begin{proof}
	This follows from \cref{le:phibs prod,le:trivial comrels} because the roots $(-1,1)$ and $(1,-1)$ are adjacent to $(1,1)$ and $(-1,-1)$.
\end{proof}

\begin{lemma}\label{le:phibs long 2}
	We have $ \phibs^{-1} \circ e_{\lambda x} \circ \phibs = e_{(0,0,0,-t^{-1} \lambda)_+} $ and $ \phibs^{-1} \circ e_{(0,0,0,\lambda)_+} \circ \phibs = e_{t\lambda x} $ for all $ \lambda \in k $.
\end{lemma}
\begin{proof}
	Let $ \lambda \in k $. Since $ e_{\lambda x} $ is the identity on $ L_{-1} $, $ \phibs^{-1} \circ e_{\lambda x} \circ \phibs $ is the identity on $ \phibs^{-1}(L_{-1}) $. Further, $ \phibs^{-1} \circ e_{\lambda x} \circ \phibs $ acts on $ (k,0,0,0)_- = \phibs^{-1}(ky) $ by
	\begin{align*}
		(1,0,0,0)_- &\mapsto (1,0,0,0)_- + t^{-1} \lambda \zeta + (0,0,0,t^{-2} \lambda^2)_+.
	\end{align*}
	Since $ \phibs^{-1}(ky + L_{-1}) $ generates $ L $, the first assertion follows. Hence
	\begin{align*}
		\phibs^{-1} \circ e_{(0,0,0,\lambda)_+} \circ \phibs &= \phibs^{-2} \circ e_{-t\lambda x} \circ \phibs^2 = e_{t\lambda x}.
	\end{align*}
	by \cref{rem:phibs formulas,le:exp parity conj}.
\end{proof}

\begin{lemma}\label{le:phibs long 3}
	We have $ \phibs^{-1} \circ e_{\lambda y} \circ \phibs = e_{(-t\lambda, 0, 0, 0)_-} $ and $ \phibs^{-1} \circ e_{(\lambda, 0, 0, 0)_-} \circ \phibs = e_{t^{-1} \lambda y} $ for all $ \lambda \in k $.
\end{lemma}
\begin{proof}
	Let $ \lambda \in k $. Since $ e_{\lambda y} $ is the identity on $ L_1 $, $ \phibs^{-1} \circ e_{\lambda y} \circ \phibs $ is the identity on $ \phibs^{-1}(L_1) $. Further, $ \phibs^{-1} \circ e_{\lambda y} \circ \phibs $ acts on $ (0,0,0,k) = \phibs^{-1}(kx) $ by
	\begin{align*}
		(0,0,0,1)_+ &\mapsto (0,0,0,1) - t\lambda \zeta + (t^2 \lambda^2, 0, 0, 0)_-.
	\end{align*}
	Since $ \phibs^{-1}(ky+L_{-1}) $ generates $ L $, the first assertion follows. Thus
	\begin{align*}
		\phibs^{-1} \circ e_{(\lambda, 0, 0, 0)_-} \circ \phibs &= \phibs^{-2} \circ e_{-t^{-1} \lambda y} \circ \phibs^{2} = e_{t^{-1} \lambda y}
	\end{align*}
	by \cref{rem:phibs formulas,le:exp parity conj}.
\end{proof}

\begin{proposition}\label{le:phibs weyl}
	Assume that there exists $ b \in J $ such that $ N(b) $ is invertible and define $(\iota, \iotainv, t)$ as in \cref{conv:J invertible}. Then $ \phibs^{(\iota, \iotainv, t)} $ is a $(-1,-1)$-Weyl element with respect to $ (U_\alpha)_{\alpha \in G_2} $.
\end{proposition}
\begin{proof}
	This follows from \cref{le:phibs prod,le:phibs long 1,le:phibs long 2,le:phibs long 3,le:phibs short 1,le:phibs short 2,le:phibs short 3}.
\end{proof}

\begin{proposition}\label{pr:G2 weyl}
	Assume that there exists $p \in J$ such that $N(p)$ is invertible. Then for all $ \alpha \in G_2 $, there exists an $ \alpha $-Weyl element with respect to $ (U_\alpha)_{\alpha \in G_2} $.
\end{proposition}
\begin{proof}
	By \cref{le:phibs weyl,le:phihor weyl}, there exist Weyl elements for the system $\{(2,1), (-1,-1)\}$ of simple roots in $G_2$. The assertion now follows from \cref{rem:weyl basic}.
\end{proof}

\begin{sidewaystable}
\vspace{70ex}
\resizebox{\textwidth}{!}{	
	
\renewcommand{\arraystretch}{2.5}

$
\begin{array}{c||c|c|ccccc|c|c|}
	\varphi_{\leftrightarrow} & x & (\nu, c, c', \rho)_- & c' & \dd_{e,e'} & \xi & \zeta & c & (\nu, c, c', \rho)_+ & y \\
	\hline \hline
	x & 0 &  &  &  &  &  &  &  &  \\
	\hline
	(\lambda, b, b', \mu)_- & 0 & \bigl( T(b, c') - T(c, b') + \mu \nu - \lambda \rho \bigr) y &  &	&  &  &  &	&  \\
	\hline
	b' & 0 & - \bigl( T(c, b'), \ b' \times c', \ \rho b', \ 0 \bigr)_+ & 0 &  &  &  &  &  &  \\
	\dd_{a,a'} & 0 &
			\makecell*[l]{\bigl( - \nu T(a, a'), \\[.3ex] \quad D_{a,a'}(c) - T(a, a')c, \\[.3ex] \quad - D_{a',a}(c') + T(a, a')c', \\[.3ex] \hspace*{15ex} \rho T(a, a') \bigr)_+}
		& -D_{a', a}(c') & \dd_{D_{a,a'}(e), e'} - \dd_{e, D_{a',a}(e')} &  &  &  & &  \\
	\xi & -2y & -(\nu, c, c', \rho)_- & 0 & 0 & 0 &  &  &  & \\
	\zeta & -y & (-2 \nu, - c, 0, \rho)_+ & -c' & 0 & 0 & 0 & &  &  \\
	b & 0 & \bigl( 0, \ \nu b, \ b \times c, \ T(b, c') \bigr)_+ & - \dd_{b, c'} & -D_{e,e'}(b) & 0 & -b & 0 &  & \\
	\hline
	(\lambda, b, b', \mu)_+ & (\lambda, b, b', \mu)_+ &
		\makecell*[l]{\ \bigl( \lambda c' - b \times c + \nu b' \bigr) \\[.8ex]
		\quad {} + (T(c, b') - \mu \nu) \cdot (\zeta-\xi) \\
		\quad {} + (\lambda \rho - T(b, c')) \cdot \zeta \\
		\quad {} - \dd_{b, c'} - \dd_{c, b'} \\[.6ex]
		\hspace*{9ex} {}+ \bigl( \rho b - b' \times c' + \mu c \bigr) }
		& \bigl( -T(b, c'), \ -c' \times b', \ -\mu c', \ 0 \bigr)_- &
		\makecell*[l]{\bigl( -\lambda T(e, e'), \\[.3ex] \quad D_{e,e'}(b) - T(e, e')b, \\[.3ex] \quad -D_{e',e}(b') + T(e, e')b', \\[.3ex] \hspace*{12ex} \mu T(e, e') \bigr)_-}

		& (\lambda, b, b', \mu)_- & (-\lambda, 0, b', 2 \mu)_- & \bigl( 0, \ \lambda c, \ c \times b, \ T(c, b') \bigr)_- & \bigl( T(b, c') - T(c, b') + \mu \nu - \lambda \rho \bigr) x & \\
	\hline
	y & \xi & -(\nu, c, c', \rho)_- & 0 & 0 & -2x & -x & 0 & 0 & 0 \\
	\hline
\end{array}
$

}
\bigskip
\caption{Computation in \cref{le:reflections auto}: The entry at row $a$ and column $b$ shows $\phihor([a,b])=[\phihor(a), \phihor(b)]$}\label{ta:phihor auto}
\end{sidewaystable}

\begin{sidewaystable}
\vspace{70ex}
\resizebox{\textwidth}{!}{	
	
\renewcommand{\arraystretch}{2.5}

$
\begin{array}{c||c|c|ccccc|c|c|}
	\varphi_{\leftrightarrow} & x & (\nu, c, c', \rho)_- & c' & \dd_{e,e'} & \xi & \zeta & c & (\nu, c, c', \rho)_+ & y \\
	\hline \hline
	x & 0 &  &  &  &  &  &  &  &  \\
	\hline
	(\lambda, b, b', \mu)_- & 0 & \bigl( T(b, c') - T(c, b') + \mu \nu - \lambda \rho, 0 ,0,0 \bigr)_- &  &	&  &  &  &	&  \\
	\hline
	b' & 0 & T(c,b')x + \bigl( 0, \ -b' \times c', \ 0, \ 0 \bigr)_- - \rho b' & 0 &  &  &  &  &  &  \\
	\dd_{a,a'} & 0 &
			\makecell*[l]{\nu T(a, a')x \\[.3ex] \mathord{}+\bigl( 0, \ D_{a,a'}(c) - T(a,a')c, \ 0, \ 0 \bigr)_- \\[.3ex] \mathord{} - D_{a',a}(c') + T(a, a')c' \\[.3ex] \mathord{} + \bigl( -\rho T(a, a'), \ 0, \ 0, \ 0 \bigr)_+}
		& \bigl( 0, \ 0, \ D_{a', a}(c'), \ 0 \bigr)_- & \makecell*[c]{\dd_{D_{a,a'}(e), e'} \\[.3ex] \quad {}- \dd_{e, D_{a',a}(e')}} &  &  &  & &  \\
	\xi & (-2,0,0,0)_- & \nu x - (0,c,0,0)_- -c' + (\rho, 0, 0, 0)_+ & 0 & 0 & 0 &  &  &  & \\
	\zeta & (-1,0,0,0)_- & 2\nu x - (0,c,0,0)_-  - (\rho,0,0,0)_+ & (0,0,c',0)_- & 0 & 0 & 0 & &  &  \\
	b & 0 & (0,\nu b, 0, 0)_- + b \times c - \bigl(T(b,c'),0,0,0\bigr)_+ & - \dd_{b, c'} + T(b,c') (\zeta-\xi) & \bigl(0,-D_{e,e'}(b),0,0\bigr)_+ & 0 & (0,b,0,0)_+ & 0 &  & \\
	\hline
	(\lambda, b, b', \mu)_+ & \makecell*[l]{-\lambda x + (0,b,0,0)_- \\[.3ex] \mathord{}+ b' - (\mu, 0, 0, 0)_+} &
		\makecell*[l]{\bigl(0, 0, -\lambda c' + b \times c - \nu b', 0\bigr)_- \\[.3ex]
		\quad {} + (T(c, b') - \mu \nu) \cdot \xi \\
		\quad {} + (\lambda \rho + T(c,b')) \cdot (\zeta-\xi) \\
		\quad {} - \dd_{b, c'} - \dd_{c, b'} \\[.6ex]
		\quad {}+ \bigl(0,  -\rho b + b' \times c' - \mu c, 0, 0 \bigr)_+ }
		& \makecell*[l]{\bigl(0,0,0,-T(b,c')\bigr)_- \\[.3ex] \quad {}+c' \times b' \\[.3ex] \quad {}+ (0,0,\mu c', 0)_+} &
		\makecell*[l]{\bigl(0,0,0, -\lambda T(e,e')\bigr)_- \\ \quad {} -D_{e,e'}(b) + T(e,e')b \\ \quad {}+ \bigl(0,0, D_{e',e}(b') - T(e,e')b', 0\bigr)_+ \\ \quad {} + \mu T(e,e') y}
		& \makecell*[l]{(0,0,0,\lambda)_- \\ \; {} -b \\ \; {} + (0,0,-b',0)_+ \\ \; {} +\mu y} & \makecell*[l]{(0,0,0,-\lambda)_- \\ \; {} + (0,0,-b',0)_+ \\ \; {} + 2\mu y} & \makecell*[l]{(0,0,-c \times b, 0)_+ \\ \; {} -\lambda c \\ \; {} -T(c,b') y} & \makecell*[l]{\bigl( 0,0,0, \\ \quad T(b,c') - T(c,b') \\ \quad {} + \mu\nu - \lambda\rho\bigr)_+} & \\
	\hline
	y & -\zeta & -(0,0,0,\nu)_- + c  + (0,0,c',0)_+ - \rho y & 0 & 0 & (0,0,0,-2)_+ & (0,0,0,-1)_+ & 0 & 0 & 0 \\
	\hline
\end{array}
$

}
\bigskip
\caption{Computation in \cref{le:reflections auto}: The entry at row $a$ and column $b$ shows $\phisl([a,b])=[\phisl(a), \phisl(b)]$}\label{ta:phisl auto}
\end{sidewaystable}

\begin{sidewaystable}
\vspace{70ex}
\resizebox{\textwidth}{!}{	
	
\renewcommand{\arraystretch}{2.5}

$
\begin{array}{c||c|c|ccccc|c|c|}
	\psi_{\ensuremath\operatorname{\rotatebox{120}{\!$_\leftrightarrow$}}} & x & (\nu, c, c', \rho)_- & c' & \dd_{e,e'} & \xi & \zeta & c & (\nu, c, c', \rho)_+ & y \\
	\hline \hline
	x & 0 &  &  &  &  &  &  &  &  \\
	\hline
	(\lambda, b, b', \mu)_- & 0 & \bigl( 0,0,0,T(b, c') - T(c, b') + \mu \nu - \lambda \rho \bigr)_+ &  &	&  &  &  &	&  \\
	\hline
	b' & 0 & T(c,b')y + \bigl( 0, \ 0, \ b' \times c', \ 0 \bigr)_+ - \rho b' & 0 &  &  &  &  &  &  \\
	\dd_{a,a'} & 0 &
			\makecell*[l]{\nu T(a, a')y \\[.3ex] \mathord{}+\bigl( 0, \ 0, \ -D_{a,a'}(c) + T(a,a')c, \ 0 \bigr)_+ \\[.3ex] \mathord{} - D_{a',a}(c') + T(a, a')c' \\[.3ex] \mathord{} + \bigl( 0, \ 0, \ 0, \ \rho T(a, a') \bigr)_-}
		& \bigl( 0, \ D_{a', a}(c'), \ 0,\ 0 \bigr)_- & -\dd_{e', D_{a,a'}(e)} + \dd_{D_{a',a}(e'), e} &  &  &  & &  \\
	\xi & (0,0,0,-2)_+ & \nu y + (0,0,-c,0)_+ -c' + (0, 0, 0, -\rho)_- & 0 & 0 & 0 &  &  &  & \\
	\zeta & (0,0,0,-1)_+ & 2\nu y + (0,0,-c,0)_+  + (0,0,0,\rho)_- & (0,c',0,0)_+ & 0 & 0 & 0 & &  &  \\
	b & 0 & (0,0, -\nu b, 0)_+ + b \times c + \bigl(0,0,0,T(b,c')\bigr)_- & \dd_{c', b} - T(b,c') (\zeta-\xi) & \bigl(0,0,D_{e,e'}(b),0\bigr)_- & 0 & (0,0,b,0)_- & 0 &  & \\
	\hline
	(\lambda, b, b', \mu)_+ & \makecell*[l]{-\lambda y + (0,0,-b,0)_+ \\[.3ex] \mathord{}+ b' + (0, 0, 0, \mu)_-} &
		\makecell*[l]{\bigl(0, -\lambda c' + b \times c - \nu b', 0, 0\bigr)_+ \\[.3ex]
		\quad {} + (-T(c, b') + \mu \nu) \cdot \xi \\
		\quad {} + (\lambda \rho + T(c,b')) \cdot (\xi-\zeta) \\
		\quad {} + \dd_{c', b} + \dd_{b', c} \\[.6ex]
		\quad {}+ \bigl(0, 0,  -\rho b + b' \times c' - \mu c, 0 \bigr)_- }
		& \makecell*[l]{\bigl(T(b,c'),0,0,0\bigr)_+ \\[.3ex] \quad {}+c' \times b' \\[.3ex] \quad {}+ (0,-\mu c', 0, 0)_-} &
		\makecell*[l]{\bigl(\lambda T(e,e'),0,0, 0\bigr)_+ \\ \quad {} -D_{e,e'}(b) + T(e,e')b \\ \quad {}+ \bigl(0, -D_{e',e}(b') + T(e,e')b', 0, 0\bigr)_- \\ \quad {} + \mu T(e,e') x}
		& \makecell*[l]{(-\lambda,0,0,0,)_+ \\ \; {} -b \\ \; {} + (0,b',0,0)_- \\ \; {} +\mu x} & \makecell*[l]{(\lambda,0,0,0)_+ \\ \; {} + (0,b',0,0)_- \\ \; {} + 2\mu x} & \makecell*[l]{\bigl(0,c \times b, 0, 0\bigr)_- \\ \; {} -\lambda c \\ \; {} +T(c,b') x} & \makecell*[l]{\bigl(T(b,c') - T(c,b') \\ \quad {} + \mu\nu - \lambda\rho, \\ \qquad 0,0,0\bigr)_-} & \\
	\hline
	y & \zeta & (\nu, 0, 0, 0)_+ + c + \bigl(0, -c', 0, 0\bigr)_- - \rho x & 0 & 0 & (-2,0,0,0)_- & (-1,0,0,0)_- & 0 & 0 & 0 \\
	\hline
\end{array}
$

}
\bigskip
\caption{Computation in \cref{le:reflections auto}: The entry at row $a$ and column $b$ shows $\psibs([a,b])=[\psibs(a), \psibs(b)]$}\label{ta:psibs auto}
\end{sidewaystable}

\clearpage 

\bibliographystyle{alpha}
\bibliography{F4}

\end{document}